\documentclass[9pt]{article}
\usepackage{amssymb}
\usepackage{tikz}
\usetikzlibrary{arrows, calc, decorations.pathmorphing, backgrounds, positioning, fit, petri, automata}
\definecolor{grey1}{rgb}{0.5,0.5,0.5}
\usepackage[top=2.54cm, bottom=2.54cm, left=2.2cm, right=2.2cm]{geometry}
\usepackage{algorithmicx}
\usepackage{algcompatible}
\usepackage{algorithm}

\usepackage{changebar}

\usepackage{hyperref}
\hypersetup{ 
  colorlinks,
  citecolor=blue,
  linkcolor=red,
  urlcolor=magenta}

\usepackage{mathtools}
\usepackage{graphicx}
\usepackage{subfigure}
\usepackage{amsmath,amsthm,bm,bbm}
\usepackage{rotating}
\usepackage{mathrsfs}
\usepackage{chngcntr}
\usepackage{apptools}
\usepackage[titletoc,title]{appendix}
\usepackage{comment}
\renewcommand{\baselinestretch}{1.3}
\usepackage{xcolor}         
\definecolor{grau}{rgb}{0.8,0.8,0.8}

\newcommand{\chen}[1]{\color{orange}}
\usepackage{setspace,lscape,longtable}
\usepackage{amsmath,amsthm,bm}
\usepackage{fullpage}  

\usepackage{titlesec}
\usepackage[english]{babel}
\usepackage{xcolor}         
\definecolor{grau}{rgb}{0.8,0.8,0.8}

\newtheoremstyle{mystyle}
{1ex} 
{1ex} 
{\itshape} 
{} 
{\bfseries} 
{} 
{1ex} 
{} 
\theoremstyle{mystyle}

\newtheorem{theorem}{Theorem}
\newtheorem{lemma}{Lemma}

\newtheorem{definition}{Definition}
\newtheorem*{example}{Example}

\newtheorem{remark}{Remark}

\usepackage[comma,authoryear]{natbib}

\numberwithin{equation}{section}

\usepackage{multirow}

\bibpunct{(}{)}{;}{a}{,}{,}

\DeclareMathOperator*{\argmin}{arg\,min}
\DeclareMathOperator*{\argmax}{arg\,max}

\newcommand{\prob}{{\mathbb{P}}}
\newcommand{\var}{{\mathrm{var}}}
\newcommand{\expect}{\mathbb{E}}
\newcommand{\vect}{\mathrm{vec}}

\newcommand{\iidsim}{{\overset{\mathrm{i.i.d.}}{\sim}}}
\newcommand{\transpose}{^{\mathrm{T}}}

\newcommand{\bLambda}{{\bm{\Lambda}}}

\newcommand{\calC}{{\mathcal{C}}}
\newcommand{\calD}{{\mathcal{D}}}

\newcommand{\calL}{{\mathcal{L}}}

\newcommand{\calX}{{\mathcal{X}}}

\newcommand{\balpha}{{\boldsymbol{\alpha}}}
\newcommand{\bT}{{\mathbf{T}}}
\newcommand{\bU}{{\mathbf{U}}}
\newcommand{\bP}{{\mathbf{P}}}

\newcommand{\bY}{{\mathbf{Y}}}
\newcommand{\bv}{{\mathbf{v}}}
\newcommand{\bx}{{\mathbf{x}}}
\newcommand{\bX}{{\mathbf{X}}}
\newcommand{\bA}{{\mathbf{A}}}
\newcommand{\bB}{{\mathbf{B}}}

\newcommand{\bD}{{\mathbf{D}}}
\newcommand{\bE}{{\mathbf{E}}}

\newcommand{\bG}{{\mathbf{G}}}
\newcommand{\bH}{{\mathbf{H}}}
\newcommand{\bM}{{\mathbf{M}}}
\newcommand{\bN}{{\mathbf{N}}}

\newcommand{\bR}{{\mathbf{R}}}
\newcommand{\bQ}{{\mathbf{Q}}}
\newcommand{\bS}{{\mathbf{S}}}
\newcommand{\bW}{{\mathbf{W}}}
\newcommand{\bZ}{{\mathbf{Z}}}

\newcommand{\bh}{{\mathbf{h}}}

\newcommand{\br}{{\mathbf{r}}}

\newcommand{\bu}{{\mathbf{u}}}

\newcommand{\by}{{\mathbf{y}}}
\newcommand{\bz}{{\mathbf{z}}}

\newcommand{\be}{{\mathbf{e}}}

\newcommand{\bV}{{\mathbf{V}}}
\newcommand{\bgamma}{{\bm{\gamma}}}
\newcommand{\bGamma}{{\bm{\Gamma}}}
\newcommand{\bbeta}{{\bm{\beta}}}
\newcommand{\bDelta}{{\bm{\Delta}}}

\newcommand{\bSigma}{{\bm{\Sigma}}}
\newcommand{\bxi}{{\boldsymbol{\xi}}}

\newcommand{\eye}{{\mathbf{I}}}
\newcommand{\one}{{\mathbf{1}}}
\newcommand{\btheta}{{\bm{\theta}}}
\newcommand{\bzeta}{{\bm{\zeta}}}
\newcommand{\bmu}{{\bm{\mu}}}
\newcommand{\bnu}{{\bm{\nu}}}
\newcommand{\bPi}{{\bm{\Pi}}}

\newcommand{\bPsi}{{\bm{\Psi}}}
\newcommand{\bphi}{{\bm{\phi}}}

\newcommand{\bpi}{{\bm{\pi}}}
\newcommand{\zero}{{\bm{0}}}
\newcommand{\eps}{\epsilon}
\usepackage{lipsum}
\usepackage{enumitem}
\setlist[enumerate]{itemsep=0mm}
\setlist[itemize]{itemsep=0mm}

\begin{document}

\author{Fangzheng Xie
\thanks{Department of Applied Mathematics and Statistics, Johns Hopkins University}
\and Yanxun Xu
\footnotemark[1] \thanks{Correspondence should be addressed to Yanxun Xu (yanxun.xu@jhu.edu)}}

 \date{}

\title{\bf Efficient Estimation for Random Dot Product Graphs via a One-step Procedure}
\maketitle

\begin{abstract}
We propose a one-step procedure to estimate the latent positions in random dot product graphs efficiently. Unlike the classical spectral-based methods such as the adjacency and Laplacian spectral embedding, the proposed one-step procedure takes advantage of both the low-rank structure of the expected adjacency matrix and the Bernoulli likelihood information of the sampling model simultaneously. We show that for each vertex, the corresponding row of the one-step estimator converges to a multivariate normal distribution after proper scaling and centering up to an orthogonal transformation, with an efficient covariance matrix.
The initial estimator for the one-step procedure needs to satisfy the so-called approximate linearization property. The one-step estimator improves the commonly-adopted spectral embedding methods in the following sense: Globally for all vertices, it yields an asymptotic sum of squares error no greater than those of the spectral methods, and locally for each  vertex, the asymptotic covariance matrix of the corresponding row of the one-step estimator dominates those of the spectral embeddings in spectra. 
The usefulness of the proposed one-step procedure is demonstrated via numerical examples and the analysis of a real-world Wikipedia graph dataset.  
\end{abstract}

{\it Keywords:} Approximate linearization property; Asymptotic normality; Bernoulli likelihood information; Latent position estimation; Normalized Laplacian.

\setcounter{tocdepth}{1}
\renewcommand{\baselinestretch}{1}
\tableofcontents
\renewcommand{\baselinestretch}{1.3}

\allowdisplaybreaks
\section{Introduction} 
\label{sec:introduction}
Statistical inference on graph data, an important topic in statistics and machine learning, has been pervasive in a variety of application domains, such as social networks \citep{young2007random,Girvan7821,wasserman1994social}, brain connectomics \citep{priebe2017semiparametric,8570772}, political science \citep{doi:10.1146/annurev.polisci.12.040907.115949}, computer networks \citep{6623779,7745482}, etc. Due to the high dimensional nature and the complex structure of graph data, classical statistical  methods typically begin with
finding a low-dimensional representation for the vertices in a graph using a collection of points in some Euclidean space, referred to as  \emph{latent positions} of the vertices. 
These latent positions are further used as features for subsequent inference tasks, such as vertex clustering \citep{sussman2012consistent} and classification \citep{6565321,tang2013}, regression \citep{mele2019spectral}, and nonparametric graph testing \citep{tang2017}. 

\cite{doi:10.1198/016214502388618906} proposed the latent position graphs to formalize the idea of  latent positions: 
Each vertex $i$ in the graph is assigned a Euclidean vector $\bx_i\in\mathbb{R}^d$, and the occurrence of an edge linking vertices $i$ and $j$ is a Bernoulli random variable with the success probability $\kappa(\bx_i,\bx_j)$, where $\kappa:\mathbb{R}^d\times\mathbb{R}^d\to [0, 1]$ is a symmetric link function. In this work, we study the random dot product graphs \citep{young2007random}, a particular class of latent position graphs   taking the link function to be the dot product of  latent positions: 
 $\kappa(\bx_i,\bx_j) = \bx_i\transpose{}\bx_j$. Random dot product graphs are of special interest due to the following two reasons: Firstly, the adjacency matrix of a random dot product graph can be viewed as the sum of a low-rank matrix and a mean-zero noise matrix, which facilitates the use of low-rank matrix factorization techniques for statistical inference; Secondly, random dot product graphs are sufficiently flexible as they can approximate general latent position graphs with symmetric positive definite link functions when the dimension $d$  of the latent positions grows with the number of vertices at a certain rate \citep{tang2013}. The readers are referred to the survey paper \cite{JMLR:v18:17-448} for a thorough review on the recent development of random dot product graphs. 

Low-rank matrix factorization methods, or more precisely, spectral-based methods, have been broadly used for estimating latent positions for
 random dot product graphs due to the low expected rank of the observed adjacency matrix. 
\cite{6565321} proposed to estimate  latent positions using the eigenvectors associated with the top $d$
 eigenvalues of the adjacency matrix. 
The resulting estimator is referred to as the \emph{adjacency spectral embedding} (ASE). Asymptotic characterization of the global behavior of the ASE for all vertices have been established, including the consistency \citep{6565321} and the limit of the sum of squares error \citep{doi:10.1080/10618600.2016.1193505} as the number of vertices goes to infinity. Locally, for each vertex, \cite{athreya2016limit} proved that the distribution of the corresponding row of the adjacency spectral embedding converges to a mean-zero multivariate normal mixture distribution after proper scaling and centering, up to an orthogonal transformation, as the number of vertices goes to infinity. Another popular spectral-based method is the \emph{Laplacian spectral embedding} (LSE), which does not  estimate  latent positions directly, but computes the eigenvectors of the normalized Laplacian matrix of the adjacency matrix associated with the top $d$ eigenvalues \citep{rohe2011}. 
The asymptotic theory of the LSE has also been established \citep{sarkar2015,tang2018}. Notably, 
\cite{tang2018} showed that each row of the LSE converges to a mean-zero multivariate normal mixture distribution after proper scaling and centering, up to an orthogonal transformation. 
These theoretical studies of the spectral-based methods lay a solid foundation 
for the development of  subsequent inference tasks, such as vertex clustering \citep{sussman2012consistent,rohe2011,sarkar2015}, vertex classification \citep{6565321,tang2013}, testing between graphs \citep{doi:10.1080/10618600.2016.1193505,tang2017}, and parameter estimation in latent structure random graphs \citep{athreya2018estimation}.

Despite the great success of the spectral-based methods for random dot product graphs, it has been pointed out in \cite{xie2019optimal} that they are formulated in a low-rank matrix factorization fashion, whereas the Bernoulli likelihood information contained in the sampling model has been neglected. A fundamental question remains open: 
whether or not the adjacency/Laplacian spectral embedding is optimal for estimating  latent positions (or the transformation of them) due to the negligence of the  the likelihood information? In this paper, we prove the sub-optimality of the ASE by showing that the asymptotic covariance matrix of each row of the ASE is sub-optimal. 
We propose
 a novel one-step procedure for estimating  latent positions, and show that for each  vertex, the corresponding row of the proposed one-step estimator converges to a multivariate normal distribution after $\sqrt{n}$-scaling and centering at the underlying true latent position, up to an orthogonal transformation. More importantly, the corresponding asymptotic covariance matrix is the same as the maximum likelihood estimator as if the rest of the latent positions are known, provided that the procedure is initialized at an estimator satisfying the approximate linearization property, which will be defined later. This phenomenon of the one-step estimator is referred to as the local efficiency, the formal definition of which is provided in Section \ref{sec:an_efficient_one_step_estimator}. 
 In particular, we show that the efficient covariance matrix is no greater than the asymptotic covariance matrix of the corresponding row of the ASE in spectra. We also provide an example where the difference between the efficient covariance matrix and the asymptotic covariance matrix of the ASE has at least one negative eigenvalue. 
 Besides the local efficiency for each vertex, the proposed one-step estimator for   latent positions has a smaller sum of squares error than that of the ASE globally for all vertices as well. 

The general one-step procedure, which finds a new estimator via a single iteration of the Newton-Raphson update given a $\sqrt{n}$-consistent initial estimator, has been applied to M-estimation theory in classical parametric models to produce an efficient estimator \citep{van2000asymptotic}. Even when the maximum likelihood estimator does not exist (\emph{e.g.}, Gaussian mixture models), the one-step estimator could still be efficient.
This motivates us to extend the one-step procedure from classical parametric models to efficient estimation in high-dimensional random graphs, because neither the existence nor the uniqueness of the maximum likelihood estimator for random dot product graphs has been established. Unlike the ASE, the proposed one-step procedure takes  both the low-rank structure of the mean matrix and the likelihood information of the sampling model into account simultaneously. 
This work represents, to the best of our knowledge, the first effort in the literature addressing the efficient estimation problem for random dot product graphs. 

Moreover, we prove the asymptotic sub-optimality of the widely adopted LSE  by applying 
 the one-step procedure to construct an estimator for the population version of the LSE, and show that it dominates the LSE in the following sense: Locally for each vertex, the corresponding row of the new estimator converges to a mean-zero multivariate normal distribution after proper scaling and centering, up to an orthogonal transformation, and the asymptotic covariance matrix is no greater than that of the corresponding row of the LSE in spectra; Globally for all vertices, it yields a sum of squares error no greater than that of the LSE. 


Recently, there has been substantial progress on generalized random dot product graphs \citep{rubin2017statistical}, which fall into the category of general latent position graphs as well but allow for a more general link function than random dot product graphs. The link function of a generalized random dot product graph is of the form $\kappa(\bx_i, \bx_j) = \bx_i\transpose\eye_{p,q}\bx_j$, where $\eye_{p,q}$ is a diagonal matrix with $p$ ones and $q$ minus ones on its diagonals and $p,q$ are non-negative integers such that $p + q = d$. 
This class of random graphs 
 include a broad class of popular network models (\emph{e.g.}, mixed-membership stochastic block models). We remark that the theory and method established in this work can be extended to   generalized random dot product graphs as long as $p,q$ are either provided or can be estimated consistently. 

The remaining part of the paper is structured as follows. We review the background on random dot product graphs and present the limit theorem for the ASE (modified theorem from \citealp{athreya2016limit}) in Section \ref{sub:background_on_random_dot_product_graphs}. The theory for the maximum likelihood estimation of a single latent position with the rest of the latent positions being known, which motivates us to pursue the efficient estimation task, is established in Section \ref{sub:a_motivating_one_dimensional_example}. 
Section \ref{sec:an_efficient_one_step_estimator} elaborates on the proposed one-step procedure for estimating the entire latent position matrix, establishes its 
 asymptotic theory, and  shows that it dominates the ASE as the number of vertices goes to infinity. In Section \ref{sub:a_plug_in_estimator_for_the_normalized_laplacian_matrix}, we apply the proposed one-step procedure to construct an estimator for the population version of the LSE, and show that it dominates the LSE asymptotically. Section \ref{sec:numerical_examples} demonstrates the usefulness of the proposed one-step procedure via numerical examples and the analysis of a real-world Wikipedia graph data. We conclude the paper with discussion in Section \ref{sec:discussion}.

\vspace*{1ex}
\noindent\textbf{Notations:} The $d\times d$ identity matrix is denoted by $\eye_d$ and the vector with  all entries being $1$ is denoted by the boldface $\one$. 
We define the notation $[n]$   to be the set of all consecutive positive integers from $1$ to $n$: $[n]:=\{1,2,\ldots,n\}$. The symbols $\lesssim$ and $\gtrsim$ mean the corresponding inequality up to a constant, \emph{i.e.}, $a \lesssim b$ ($a \gtrsim b$) if $a\leq Cb$ ($a \geq Cb$) for some constant $C > 0$, and we denote $a\asymp b$ if $a\lesssim b$ and $a\gtrsim b$. The shorthand notation $a\vee b$ denotes the maximum value between $a$ and $b$, namely, $a\vee b = \max(a,b)$ for any $a,b\in\mathbb{R}$.  
We use the notation $\mathbb{O}(n, d)$ to denote the set of all orthonormal $d$-frames in $\mathbb{R}^n$, \emph{i.e.}, $\mathbb{O}(n, d) = \{\bU\in\mathbb{R}^{n\times d}:\bU\transpose{}\bU = \eye_d\}$, where $n\geq d$, and write $\mathbb{O}(d) = \mathbb{O}(d, d)$. The notation $\|\bx\|$ is used to denote the Euclidean norm of a vector $\bx = [x_1,\ldots,x_d]\transpose\in\mathbb{R}^d$, \emph{i.e.}, $\|\bx\| = (\sum_{k = 1}^dx_k^2)^{1/2}$. For any two vectors $\bx = [x_1,\ldots,x_d]\transpose$ and $\by = [y_1,\ldots,y_d]\transpose$ in $\mathbb{R}^d$, the inequality $\bx\leq \by$ means that $x_k \leq y_k$ for all $k = 1,2,\ldots,d$. For any two positive semidefinite matrices $\bSigma_1$ and $\bSigma_2$ of the same dimension, the notation $\bSigma_1\preceq \bSigma_2$ ($\bSigma_1\succeq \bSigma_2$) means that $\bSigma_2 - \bSigma_1$ ($\bSigma_1 - \bSigma_2$) is positive semidefinite, and we say that $\bSigma_1$ is no greater (no less) than $\bSigma_2$ in spectra. 
For any rectangular matrix $\bX$, we use $\sigma_k(\bX)$ to denote its $k$th largest singular value. 
For a matrix $\bX = [x_{ik}]_{n\times d}$, we use $\|\bX\|_2$ to denote the spectral norm $\|\bX\|_2 = \sigma_1(\bX)$,  $\|\bX\|_{\mathrm{F}}$ to denote the Frobenius norm $\|\bX\|_{\mathrm{F}} = (\sum_{i = 1}^n\sum_{k = 1}^dx_{ik}^2)^{1/2}$, and $\|\bX\|_{2\to\infty}$ to denote the two-to-infinity norm  $\|\bX\|_{2\to\infty} = \max_{i\in[n]}(\sum_{k = 1}^dx_{ik}^2)^{1/2}$. 


\section{Preliminaries} 
\label{sec:preliminaries}
\subsection{Background on random dot product graphs} 
\label{sub:background_on_random_dot_product_graphs}


Denote $\calX = \{\bx = [x_1,\ldots,x_d]\transpose\in\mathbb{R}^d:x_1,\ldots,x_d > 0, \|\bx\| < 1\}$ the space of latent positions, and $\calX^n$ the $n$-fold Cartesian product of $\calX$, \emph{i.e.}, $\calX^n = \{\bX = [\bx_1,\ldots,\bx_n]\transpose\in\mathbb{R}^{n\times d}:\bx_1,\ldots,\bx_n\in\calX\}$. For any $\delta \in (0, 1/2)$, denote $\calX(\delta)$ the set of all $\bx\in\calX$ such that $\bx\transpose\bu\in[\delta, 1- \delta]$ for all $\bu\in\calX(\delta)$. Given an $n\times d$ matrix $\bX = [\bx_1,\ldots,\bx_n]\transpose\in\calX^n$ and a sparsity factor $\rho_n\in(0, 1]$, a symmetric and hollow (\emph{i.e.}, the diagonal entries are zeros) random   matrix $\bA = [A_{ij}]_{n\times n}\in\{0,1\}^{n\times n}$ is said to be the adjacency matrix of a random dot product graph on $n$ vertices $[n] = \{1,2,\ldots,n\}$ with latent positions $\bx_1,\ldots,\bx_n$, denoted by $\bA\sim\mathrm{RDPG}(\bX)$, if  $A_{ij}\sim\mathrm{Bernoulli}(\rho_n\bx_i\transpose\bx_j)$ independently, $1\leq i<j\leq n$. We refer to the matrix $\bX$ as the latent position matrix. Namely, the distribution of $\bA$ can be written as
\[
p_\bX(\bA) = \prod_{i < j}(\rho_n\bx_i\transpose\bx_j)^{A_{ij}}(1 - \rho_n\bx_i\transpose\bx_j)^{1 - A_{ij}}.
\]
When $\rho_n\equiv 1$ for all $n$, the resulting graph is dense, in the sense that the expected number of edges $\expect(\sum_{i < j}A_{ij})$ grows quadratically in $n$, and when $\rho_n\to 0$ as $n\to\infty$, the corresponding graph is sparse, namely, the expected number of edges is sub-quadratic in $n$ ($\expect(\sum_{i < j}A_{ij}) = o(n^2)$). 

The goal of this work is to estimate the latent positions $\bx_1,\ldots,\bx_n$, which are treated as deterministic parameters. In some cases, the latent positions $\bx_1,\ldots,\bx_n$ are considered as  latent random variables that are independently sampled from some underlying distribution $F$ on $\calX$ (see, for example, \citealp{athreya2016limit,6565321,tang2017,tang2018}). For deterministic latent positions, 
we require that there exists some cumulative distribution function $F$ on $\calX$, such that 
\begin{align}\label{eqn:strong_convergence_measure}
\sup_{\bx\in\calX}\left|F_n(\bx) - F(\bx)\right|\to 0\quad\text{as }n\to\infty,
\end{align}
where $F_n(\bx) = (1/n)\sum_{i = 1}^n\mathbbm{1}\{\bx_i\leq \bx\}$ is the empirical distribution function obtained by treating the latent positions $\bx_1,\ldots,\bx_n$ as independent and identically distributed (i.i.d.) data. Condition \eqref{eqn:strong_convergence_measure} is similar to the case where $\bx_i$'s are random in the following sense: When $\bx_1,\ldots,\bx_n$ are independent random variables sampled from $F$, the Glivenko-Cantelli theorem asserts that \eqref{eqn:strong_convergence_measure} holds with probability one with respect to the randomness of the infinite i.i.d. sequence $(\bx_i)_{i = 1}^\infty$. 

\begin{remark}
The latent position matrix $\bX$ can only be identified up to an orthogonal transformation since for any orthogonal matrix $\bW\in\mathbb{O}(d)$ and $i, j\in[n]$, $\bx_i\transpose\bx_j = (\bW\bx_i)\transpose(\bW\bx_j)$. Furthermore, for any $d' > d$ and any latent position matrix $\bX\in\mathbb{R}^{n\times d}$, there exists another matrix $\bX'\in\mathbb{R}^{n\times d'}$, such that $\mathrm{RDPG}(\bX)$ and $\mathrm{RDPG}(\bX')$ yield the same distribution of $\bA$. The latter source of non-identifiability can be avoided for large $n$ by requiring the second moment matrix $\bDelta = \int_\calX \bx\bx\transpose F(\mathrm{d}\bx)$ to be non-singular \citep{tang2018}. \end{remark}

Random dot product graphs have connections with the simplest Erd\H{o}s-R\'enyi models and the popular stochastic block models. 
When $F$ is a point mass at some $p\in(0, 1)$, namely, $F(\mathrm{d}x) = \delta_p(\mathrm{d}x)$, the resulting random dot product graph coincides with an Erd\H{o}s-R\'enyi graph, with $(A_{ij})_{i < j}$ being independent $\mathrm{Bernoulli}(p^2)$ random variables.
When $F$ is a finitely discrete distribution on $\calX$: $F(\mathrm{d}\bx) = \sum_{k = 1}^K\pi_k\delta_{\bnu_k}(\mathrm{d}\bx)$ for $\bnu_1,\ldots,\bnu_k\in\calX$ and $\sum_{k = 1}^K\pi_k = 1$, 
there exists a cluster assignment function $\tau:[n]\to[K]$ such that $(1/n)\sum_{i = 1}^n\mathbbm{1}\{\tau(i) = k\} \to \pi_k$ for all $k = 1,2,\ldots,K$ as $n\to\infty$. Denoting $\bB = [B_{kl}]_{K\times K} := [\bnu_k\transpose\bnu_l]_{K\times K}$ and $\bx_i = \bnu_{\tau(i)}$, $i\in[n]$,  we see that $A_{ij}$ follows $\mathrm{Bernoulli}(B_{\tau(i)\tau(j)}) = \mathrm{Bernoulli}(\bx_i\transpose\bx_j)$ for $i < j$ independently, where $i,j\in[n]$. In this case, the random dot product graph $\mathrm{RDPG}(\bX)$ with $\bX = [\bx_1,\ldots,\bx_n]\transpose$ becomes a stochastic block model with a positive semidefinite block probability matrix $\bB$ and a cluster assignment function $\tau$.


To estimate the latent positions, \cite{6565321} proposed to solve the least squares problem
\begin{align}\label{eqn:ASE_least_squared_problem}
\widehat\bX^{(\mathrm{ASE})} = \argmin_{\bX\in\mathbb{R}^{n\times d}}\|\bA - \bX\bX\transpose\|_{\mathrm{F}}^2.
\end{align}
The resulting solution $\widehat\bX^{(\mathrm{ASE})}$ to \eqref{eqn:ASE_least_squared_problem} is referred to as the \emph{adjacency spectral embedding} (ASE) of $\bA$ into $\mathbb{R}^d$. Note that $\expect(\bA)$ is a positive semidefinite low-rank matrix modulus the diagonal entries and $\|\bA - \bX\bX\transpose\|_{\mathrm{F}}^2 = \sum_{i = 1}^n\sum_{j = 1}^n(A_{ij} - \bx_i\transpose\bx_j)^2$ is exactly the empirical squared-error loss. Hence the problem \eqref{eqn:ASE_least_squared_problem} becomes a naive empirical risk minimization problem if we regard $\widehat\bX^{(\mathrm{ASE})}$ as an estimator for $\rho_n^{1/2}\bX$, and the solution to \eqref{eqn:ASE_least_squared_problem} can be conveniently computed \citep{Eckart1936}: $\widehat\bX^{(\mathrm{ASE})}$ is the matrix of eigenvectors associated with the top $d$ eigenvalues of $\bA$, scaled by the square roots of these eigenvalues.

\cite{6565321}  proved that $\widehat\bX^{(\mathrm{ASE})} = [\widehat\bx_1^{(\mathrm{ASE})},\ldots,\widehat\bx_n^{(\mathrm{ASE})}]\transpose$ is a consistent estimator for $\rho_n^{1/2}\bX$ globally for all vertices:  $(1/n)\|\widehat\bX^{(\mathrm{ASE})}\bW_n - \bX\|_{\mathrm{F}}^2$ converges to $0$ in probability as $n\to\infty$ for a sequence of orthogonal $(\bW_n)_{n = 1}^\infty\subset\mathbb{O}(d)$. Furthermore, 
for each fixed vertex $i\in[n]$, the asymptotic distribution of $\widehat\bx_i^{(\mathrm{ASE})}$ after proper scaling and centering has been established \citep{athreya2016limit,tang2018} in the case where $\bx_{1},\ldots,\bx_{n}$ are independent and identically distributed (i.i.d.) according to $F$. The setup in this work is slightly different since we posit that the latent positions are deterministic. To distinguish between an arbitrary element $\bX\in\calX^n$ and the ground truth, we denote $\bX_0$ the true latent position matrix that generates the observed adjacency matrix $\bA$ according to the sampling model $\bA\sim\mathrm{RDPG}(\bX_0)$.
We modify the limit theorem of the ASE originally presented in \cite{athreya2016limit} to accommodate the deterministic setup for $\bx_{01},\ldots,\bx_{0n}$ in the current framework and summarize the results in the following theorem. 
In the case of deterministic latent position matrix $\bX_0$, the proof technique for the asymptotic normality of the rows of the ASE is very different from that presented in \cite{athreya2016limit} and \cite{tang2018}. The proof of Theorem \ref{thm:ASE_limit_theorem} is deferred to Appendix. 


\begin{theorem}
\label{thm:ASE_limit_theorem}
Let $\bA\sim\mathrm{RDPG}(\bX_0)$ with a sparsity factor $\rho_n$ and condition \eqref{eqn:strong_convergence_measure} hold for some $\bX_0 = [\bx_{01},\ldots,\bx_{0n}]\transpose\in\calX^n$. Suppose either $\rho_n\equiv 1$ for all $n$ or $\rho_n\to 0$ but $(\log n)^4/(n\rho_n)\to 0$ as $n\to\infty$, and denote $\rho = \lim_{n\to\infty}\rho_n$. Let $\widehat\bX^{(\mathrm{ASE})} = [\widehat\bx_1^{(\mathrm{ASE})},\ldots,\widehat\bx_n^{(\mathrm{ASE})}]\transpose$ be the ASE defined by \eqref{eqn:ASE_least_squared_problem}. 
Denote
\[
\bDelta = \int_\calX \bx\bx\transpose F(\mathrm{d}\bx),
\quad
\bSigma(\bx) = \bDelta^{-1}\left[\int_\calX\left\{\bx_1\transpose\bx\left(1 - \rho\bx_1\transpose\bx\right)\right\}\bx_1\bx_1\transpose F(\mathrm{d}\bx_1)\right]\bDelta^{-1},
\]
and assume that $\bDelta$ and $\bSigma(\bx)$ are strictly positive definite for all $\bx\in\calX$. 
Then there exists a sequence of orthogonal matrices $(\bW)_{n = 1}^\infty = (\bW_n)_{n = 1}^\infty\subset\mathbb{O}(d)$, such that
\begin{align}
\label{eqn:ASE_convergence}
&\|\widehat\bX^{(\mathrm{ASE})}\bW - \rho_n^{1/2}\bX_0\|_{\mathrm{F}}^2\overset{a.s.}{\to} 
\int_\calX \mathrm{tr}\{\bSigma(\bx)\}F(\mathrm{d}\bx),
\end{align} 
and for any fixed index $i\in[n]$, 
\begin{align}
\label{eqn:ASE_normality}
&\sqrt{n}(\bW\transpose\widehat\bx_i^{(\mathrm{ASE})} - \rho_n^{1/2}\bx_{0i})\overset{\calL}{\to}\mathrm{N}(\zero, \bSigma(\bx_{0i})).
\end{align}
\end{theorem}
\noindent
In the rest of the paper, we drop the subscript $n$ in $\bW_n$ for notational simplicity and make the convention that the orthogonal alignment matrix $\bW$  implicitly depends on $n$.  


\subsection{Motivation: Efficiency in estimating a single latent position}  
\label{sub:a_motivating_one_dimensional_example}


Theorem \ref{thm:ASE_limit_theorem} suggests the following two properties of the ASE:  
Globally for all vertices,   $\widehat\bX^{(\mathrm{ASE})}$ is a consistent estimator for $\rho_n^{1/2}\bX_0$ as there exists a sequence of orthogonal matrices $(\bW)_{n = 1}^\infty\subset\mathbb{R}^{d\times d}$ such that the sum of squares error $\|\widehat\bX^{(\mathrm{ASE})}\bW - \bX_0\|_{\mathrm{F}}^2$ can be fully characterized by \eqref{eqn:ASE_convergence} as $n\to\infty$; 
Locally, for each fixed vertex $i\in[n]$, the distribution of the $i$th row $\widehat\bx_i^{(\mathrm{ASE})}$ of $\widehat\bX^{(\mathrm{ASE})}$ after $\sqrt{n}$-scaling and centering at $\rho_n^{1/2}\bx_{0i}$, converges to a mean-zero multivariate normal distribution with covariance matrix $\bSigma(\bx_{0i})$, up to a sequence of orthogonal transformations. 
Nevertheless, it remains open whether the results of Theorem \ref{thm:ASE_limit_theorem} are optimal. In this work, we will propose an estimator $\widehat\bX$ for the latent positions that dominates the ASE asymptotically in the following sense: Globally for all vertices, it yields a smaller asymptotic sum of squares error $\|\widehat\bX\bW - \rho_n^{1/2}\bX_0\|_{\mathrm{F}}^2$ than \eqref{eqn:ASE_convergence} for a sequence of orthogonal alignment matrices $(\bW)_{n = 1}^\infty$; Locally for each fixed vertex $i\in[n]$, the corresponding row of $\widehat\bX$, after $\sqrt{n}$-scaling and centering at $\rho_n^{1/2}\bx_{0i}$, also converges to a mean-zero multivariate normal distribution, up to a sequence of orthogonal transformations, but the asymptotic covariance matrix is no greater than $\bSigma(\bx_{0i})$ in spectra. 

Before elaborating on the estimator for the entire latent position matrix $\bX_0$, we begin with the problem of estimating a single latent position $\bx_{0i}$ when the rest of the latent positions are known. The theory established herein motivates the development of the proposed efficient estimation procedure. Specifically, for a fixed $i\in[n]$, we estimate $\bx_{0i}$ via the maximum likelihood estimator, assuming that the rest of the latent positions $\{\bx_{0j}:j\in[n],j\neq i\}$ are known. For simplicity, we assume that the sparsity factor $\rho_n\equiv 1$ for all $n$ in this subsection. The result is summarized in the following theorem. 
\begin{theorem}\label{thm:one_dimensional_MLE}
Let $\bA\sim\mathrm{RDPG}(\bX_0)$ for some $\bX_0 = [\bx_{01},\ldots,\bx_{0n}]\transpose\in\calX^n$ with $\rho_n\equiv 1$ for all $n$, 
and condition \eqref{eqn:strong_convergence_measure} hold. Suppose that there exists some constant $\delta > 0$ such that $(\bx_{0j})_{j = 1}^n\subset\calX(\delta)$. Let $i\in[n]$ be fixed and consider the problem of estimating $\bx_{0i}$ where $\{\bx_{0j}:j\in[n],j\neq i\}$ are known. Further assume that $\bx_{0i}$ is in the interior of $\calX(\delta)$, and for any $\bx\in\calX(\delta)$, define $\bG$ the following matrix-valued function $\bG:\calX(\delta)\to\mathbb{R}^{d\times d}$:
\[
\bG(\bx) = \int_\calX\left\{\frac{\bx_1\bx_1\transpose}{\bx\transpose\bx_1(1 - \bx\transpose\bx_1)}\right\} F(\mathrm{d}\bx_1).
\]
Then the maximum likelihood estimator $\widehat\bx_i^{(\mathrm{MLE})} = \argmax_{\bx\in\calX(\delta)}\ell_\bA(\bx)$ is consistent for $\bx_{0i}$, where
\[
\ell_\bA(\bx) = \sum_{j \neq i}\{A_{ij}\log(\bx\transpose\bx_{0j}) + (1 - A_{ij})\log(1 - \bx\transpose\bx_{0j})\}.
\]
Furthermore,
\begin{align}
\label{eqn:MLE_asymptotic_normality}
\sqrt{n}(\widehat\bx_i^{(\mathrm{MLE})} - \bx_{0i})\overset{\calL}{\to}\mathrm{N}(\zero, \bG(\bx_{0i})^{-1}).
\end{align} 
\end{theorem}
\begin{remark}
Recall that the cumulative distribution function $F$ is defined on $\calX$. Note that under the conditions of Theorem \ref{thm:one_dimensional_MLE}, $(\bx_{0j})_{j = 1}^n\subset\calX(\delta)$ for a constant $\delta$ that does not depend on $n$. Therefore the cumulative distribution function $F$ can be further restricted to the subset $\calX(\delta)$ of $\calX$, which is compact, and the precision matrix $\bG(\bx)$ can be alternatively written as
\[
\bG(\bx) = \int_{\calX(\delta)}\left\{\frac{\bx_1\bx_1\transpose}{\bx\transpose\bx_1(1 - \bx\transpose\bx_1)}\right\} F(\mathrm{d}\bx_1).
\]
\end{remark}
\begin{remark}
Although the definition of $\bG(\bx)$ given in Theorem \ref{thm:one_dimensional_MLE} is with regard to the case where $\rho_n\equiv 1$ for all $n$, we remark that it can also be generalized to the case where the sparsity factor $\rho_n\to 0$ as $n\to\infty$ (see equation \eqref{eqn:efficient_covariance} in Section \ref{sec:an_efficient_one_step_estimator}).
\end{remark}

Recall that for the $i$th row $\widehat\bx_i^{(\mathrm{ASE})}$ of the ASE, $\sqrt{n}(\bW\transpose\widehat\bx_i^{(\mathrm{ASE})} - \bx_{0i})\overset{\calL}{\to}\mathrm{N}(\zero, \bSigma(\bx_{0i}))$ for a sequence of orthogonal matrices $(\bW)_{n = 1}^\infty = (\bW_n)_{n = 1}^\infty\subset\mathbb{O}(d)$ by Theorem \ref{thm:ASE_limit_theorem}. 
 We now claim that $\bSigma(\bx_{0i}) - \bG(\bx_{0i})^{-1}$ is positive semidefinite. In fact, since $F_n(\cdot) = (1/n)\sum_{i = 1}^n\mathbbm{1}\{\bx_i\leq \cdot\}$ converges to $F$ strongly according to condition \eqref{eqn:strong_convergence_measure}, it follows that for any $\bx\in\calX(\delta)$,
\begin{align*}
\bDelta_n &:= \int_\calX \bx\bx\transpose F_n(\mathrm{d}\bx) = \frac{1}{n}\bX_0\transpose\bX_0\to \bDelta,\\
\bSigma_n(\bx) &:= \bDelta_n^{-1}\left[\int_\calX\{\bx\transpose\bx_1(1 - \bx\transpose\bx_1)\}\bx_1\bx_1\transpose F_n(\mathrm{d}\bx_1)\right]\bDelta_n^{-1}\\
& = \left(\frac{1}{n}\bX_0\transpose\bX_0\right)^{-1}\left(\frac{1}{n}\bX_0\transpose\bD_n(\bx)\bX_0\right)\left(\frac{1}{n}\bX_0\transpose\bX_0\right)^{-1}\to\bSigma(\bx_{0i}),\\
\bG_n(\bx) &:= \int_\calX\left\{\frac{\bx_1\bx_1\transpose}{\bx\transpose\bx_1(1 - \bx\transpose\bx_1)}\right\} F_n(\mathrm{d}\bx_1) = \frac{1}{n}\bX_0\transpose\bD_n(\bx)^{-1}\bX_0\to \bG(\bx),
\end{align*}
where $\bD_n(\bx) = \mathrm{diag}\{\bx\transpose\bx_{01}(1 - \bx\transpose\bx_{01}),\ldots,\bx\transpose\bx_{0n}(1 - \bx\transpose\bx_{0n})\}$. Now let $\bX_0$ yield singular value decomposition $\bX_0 = \bU_0\bS_0^{1/2}\bV_0\transpose$ with $\bU_0\in\mathbb{O}(n,d)$, $\bS_0^{1/2}$ being diagonal, and $\bV_0\in\mathbb{O}(d)$. We see immediately that
\begin{align*}
\bSigma_n(\bx) &= \left(\frac{1}{n}\bX_0\transpose\bX_0\right)^{-1}\left(\frac{1}{n}\bX_0\transpose\bD_n(\bx)\bX_0\right)\left(\frac{1}{n}\bX_0\transpose\bX_0\right)^{-1}
\\&
= n (\bV_0\bS_0^{-1}\bV\transpose_0)(\bV_0\bS_0^{1/2}\bU_0\transpose\bD_n(\bx)\bU_0\bS_0^{1/2}\bV_0\transpose)(\bV_0\bS_0^{-1}\bV_0\transpose)\\
& = n\bV_0\bS_0^{-1/2}(\bU_0\transpose\bD_n(\bx)\bU_0)\bS_0^{-1/2}\bV_0\transpose,\\
\bG_n(\bx)^{-1}& = n(\bX_0\transpose\bD_n(\bx)^{-1}\bX_0)^{-1} = n(\bV_0\bS_0^{1/2}\bU_0\transpose\bD_n(\bx)^{-1}\bU_0\bS_0^{1/2}\bV_0\transpose)^{-1}\\
& = n\bV_0\bS_0^{-1/2}(\bU_0\transpose\bD_n(\bx)^{-1}\bU_0)^{-1}\bS_0^{-1/2}\bV_0\transpose.
\end{align*}
Since $\bU_0\transpose\bU_0 = \eye_d$, it follows that $\bU_0\transpose\bD_n(\bx)\bU_0 - (\bU_0\transpose\bD_n(\bx)^{-1}\bU_0)^{-1}$ is positive semidefinite \citep{Marshall1990}, and hence, $\bSigma(\bx) - \bG(\bx)^{-1} = \lim_{n\to\infty}\{\bSigma_n(\bx_{0i}) - \bG_n(\bx_{0i})^{-1}\}$ is positive semidefinite for any $\bx\in\calX(\delta)$. 
Although the resulting inequality is not strict, we will present an example where there exists at least one negative eigenvalue of $\bG(\bx_{0i})^{-1} - \bSigma_n(\bx_{0i})$ in Section \ref{sec:an_efficient_one_step_estimator}.
The conclusion of this example is that the ASE is \emph{inefficient} for estimating the latent position $\bx_{0i}$ for vertex $i$ when the rest of the latent positions are known,
  in contrast to the efficiency of the maximum likelihood estimator. 
The notion of efficiency in estimating a single latent position of a random dot product graph model is slightly subtle, as this special case does not belong to the classical (i.i.d.) parametric models. Here we make the convention that the notion of efficiency is taken in analogy to the case of parametric models. Namely, we say an estimator $\widehat{\bx}_i^{(\mathrm{Eff})}$ is asymptotically efficient for estimating a single latent position vector $\bx_{0i}$, if it is asymptotically equivalent to the maximum likelihood estimator in the following sense:
\[
\sqrt{n}(\widehat{\bx}_i^{(\mathrm{Eff})} - \bx_{0i})\overset{\calL}{\to}\mathrm{N}(0, \bG(\bx_{0i})^{-1}).
\]
We will see in Section \ref{sec:an_efficient_one_step_estimator} that when all the latent positions are unknown, we can still construct an estimator $\widehat\bX = [\widehat\bx_1,\ldots,\widehat\bx_n]\transpose$, such that for each vertex $i$, $\sqrt{n}(\bW\transpose\widehat\bx_i - \rho_n^{1/2}\bx_{0i})$ converges to a multivariate normal distribution up to a sequence of orthogonal alignment matrices $(\bW)_{n = 1}^\infty = (\bW_n)_{n = 1}^\infty\subset\mathbb{O}(d)$, but the covariance matrix is the same as that of the maximum likelihood estimator as if the rest of the latent positions are known.


\section{Efficient Estimation via a One-step Procedure} 
\label{sec:an_efficient_one_step_estimator}



The inefficiency of the ASE, indicated by $\bSigma(\bx_{0i})\succeq \bG(\bx_{0i})^{-1}$, is due to the fact that the ASE 
is a least squares estimator not depending on the likelihood function of the sampling model. In contrast, the maximum likelihood estimator $\widehat\bx_i^{(\mathrm{MLE})}$ utilizes the Bernoulli likelihood function $\ell_\bA(\bx) = \sum_{j\neq i}\{A_{ij}\log(\bx\transpose\bx_{0j}) + (1 - A_{ij})\log(1 - \bx\transpose\bx_{0j})\}$, and this is a main factor for the asymptotic efficiency. For estimating the entire latent position matrix $\bX$, one strategy that takes advantage of the likelihood information is the maximum likelihood method as a competitor against the ASE. 
Unfortunately, 
when all latent positions are unknown, random dot product graphs belong to a curved exponential family rather than a canonical exponential family, and 
neither the existence nor the uniqueness of the maximum likelihood estimator of random dot product graphs has been established. As pointed out  
in \cite{bickel2015mathematical}, properties of the maximum likelihood estimator in curved exponential families are harder to develop than canonical ones.
Therefore, we seek another approach to find an estimator that is asymptotically equivalent to the maximum likelihood estimator. Recall that when $\{\bx_{0j}:j\in[n],j\neq i\}$ are known, the maximum likelihood estimator for $\bx_{0i}$ is a solution to the estimating equation
\[
\bPsi_n(\bx): = \frac{1}{n}\sum_{j\neq i}^n\frac{(A_{ij} - \bx\transpose\bx_{0j})\bx_{0j}}{\bx\transpose\bx_{0j}(1 - \bx\transpose\bx_{0j})} = \zero. 
\]
Then, given an ``appropriate'' initial guess of the solution $\widetilde\bx_i$, we can perform a one-step Newton-Raphson update to obtain another estimator $\widehat\bx_i^{(\mathrm{OS})}$  that is closer to the zero of the estimating equation $\bPsi_n$ (see, for example, Section 5.7 of \citealp{van2000asymptotic}):
\begin{equation}
\label{eqn:One_step_onedimensional}
\begin{aligned}
\widehat\bx_i^{(\mathrm{OS})}
& = \widetilde\bx_i + \left\{\frac{1}{n}\sum_{j\neq i}^n\frac{\bx_{0j}\bx_{0j}\transpose}{\widetilde\bx_i\transpose\bx_{0j}(1 - \widetilde\bx_i\transpose\bx_{0j})}\right\}^{-1}\left\{\frac{1}{n}\sum_{j\neq i}^n\frac{(A_{ij} - \widetilde\bx_i\transpose\bx_{0j})\bx_{0j}}{\widetilde\bx_i\transpose\bx_{0j}(1 - \widetilde\bx_i\transpose\bx_{0j})}\right\}.
\end{aligned}
\end{equation}
In the case of estimating $\bx_{0i}$ with the rest of the latent positions being known, the requirement for $\widetilde\bx_i$ is that it is $\sqrt{n}$-consistent for $\bx_{0i}$, and the resulting one-step estimator $\widehat\bx_i^{(\mathrm{OS})}$ is as efficient as the maximum likelihood estimator $\widehat\bx_i^{(\mathrm{MLE})}$. 
This result is summarized in the following theorem, which is a variation of Theorem 5.45 of \cite{van2000asymptotic}. 
\begin{theorem}\label{thm:OSE_single_vertex}
Let $\bA\sim\mathrm{RDPG}(\bX_0)$ for some $\bX_0 = [\bx_{01},\ldots,\bx_{0n}]\transpose\in\calX^n$ with $\rho_n\equiv 1$ for all $n$, and assume that the conditions of Theorem \ref{thm:one_dimensional_MLE} hold. Consider the problem of estimating $\bx_{0i}$ with $\{\bx_{0j}:j\in[n],j\neq i\}$ being known. Let $\widetilde\bx_i$ be a $\sqrt{n}$-consistent estimator of $\bx_{0i}$, \emph{i.e.}, $\sqrt{n}(\widetilde\bx_i - \bx_{0i}) = O_{\prob_0}(1)$. Then the one-step estimator $\widehat{\bx}_i^{\mathrm{(OS)}}$ satisfies
\[
\sqrt{n}(\widehat\bx_i^{\mathrm{(OS)}} - \bx_{0i})\overset{\calL}{\to}\mathrm{N}\left(0, \bG(\bx_{0i})^{-1}\right).
\]
\end{theorem}

The above result motivates us to generalize the one-step estimator \eqref{eqn:One_step_onedimensional} to the case where the latent positions $\bx_{01},\ldots,\bx_{0n}$ are all unknown. Let $\widetilde\bX = [\widetilde\bx_1,\ldots,\widetilde\bx_n]\transpose\in\mathbb{R}^{n\times d}$ be an initial estimator $\widetilde\bX$ for $\bX_0$. An intuitive choice for generalizing the one-step updating scheme \eqref{eqn:One_step_onedimensional} to the case of unknown $(\bx_{0j})_{j\neq i}$ is to substitute the unknown $\bx_{0j}$ by the initial estimator $\widetilde\bx_j$ for all $j\neq i$ in \eqref{eqn:One_step_onedimensional}.
We thus define the following one-step estimator $\widehat\bX = [\widehat\bx_1,\ldots,\widehat\bx_n]\transpose$ for $\bX_0$:
\begin{align}\label{eqn:one_step_estimator}
\widehat\bx_i = \widetilde\bx_i + \left\{\frac{1}{n}\sum_{j = 1}^n\frac{\widetilde\bx_{j}\widetilde\bx_{j}\transpose}{\widetilde\bx_{i}\transpose\widetilde\bx_{j}(1 - \widetilde\bx_i\transpose\widetilde\bx_j)}\right\}^{-1}\left\{\frac{1}{n}\sum_{j = 1}^n\frac{(A_{ij} - \widetilde\bx_i\transpose\widetilde\bx_j)\widetilde\bx_j}{\widetilde\bx_{i}\transpose\widetilde\bx_{j}(1 - \widetilde\bx_i\transpose\widetilde\bx_j)}\right\},\quad i = 1,2,\ldots,n.
\end{align}
Unlike the coarse $\sqrt{n}$-consistency requirement for the initial estimator in the case of estimating a single latent position with the rest of the latent positions being known, we need to require that the initial estimator $\widetilde\bX = [\widetilde\bx_1,\ldots,\widetilde\bx_n]$ for the entire latent position matrix $\bX_0$ satisfies a finer condition, referred to as the \emph{approximate linearization property}.
\begin{definition}[Approximate linearization property]
Given $\bA\sim\mathrm{RDPG}(\bX_0)$ with a sparsity factor $\rho_n$, where $\bX_0 = [\bx_{01},\ldots,\bx_{0n}]\transpose\in\calX^n$, an estimator $\widetilde\bX = [\widetilde\bx_1,\ldots,\widetilde\bx_n]\transpose$ is said to satisfy the approximate linearization property, if for all $n$, 
there exists an orthogonal matrix $\bW = \bW_n\in\mathbb{O}(d)$ and an $n\times d$ matrix $\widetilde\bR = [\widetilde\bR_1,\ldots,\widetilde\bR_n]\transpose$ with $\|\widetilde\bR\|_{\mathrm{F}}^2 = O_{\prob_0}((n\rho_n)^{-1}(\log n)^{\omega})$ for some $\omega\geq 0$, such that
\begin{align}\label{eqn:linearization_property}
\bW\transpose\widetilde\bx_{i} - \rho_n^{1/2}\bx_{0i} = \rho_n^{-1/2}\sum_{j = 1}^n(A_{ij} - \rho_n\bx_{0i}\transpose\bx_{0j})\bzeta_{ij} + \widetilde\bR_i,\quad i = 1,2,\ldots,n,
\end{align}
where $\{\bzeta_{ij}:i,j\in[n]\}$ is a collection of vectors in $\mathbb{R}^d$ with $\sup_{i,j\in[n]}\|\bzeta_{ij}\|\lesssim 1/n$. 
\end{definition} 
The approximate linearization property describes that the deviation of the estimator $\widetilde\bX$ after an appropriate orthogonal alignment $\bW$
from the true value $\bX_0$ can be approximately controlled by a linear combination of the centered Bernoulli random variables $(A_{ij} - \rho_n\bx_{0i}\transpose\bx_{0j})_{i < j}$. It has been shown in \cite{athreya2016limit} and \cite{tang2018} that for the ASE $\widehat\bX^{(\mathrm{ASE})} = [\widehat{\bx}_1^{(\mathrm{ASE})},\ldots,\widehat{\bx}^{(\mathrm{ASE})}_n]\transpose$, for all sufficiently large $n$, there exists an orthogonal $\bW = \bW_n\in\mathbb{O}(d)$ such that
\[
\bW\transpose\widehat\bx^{(\mathrm{ASE})}_i - \rho_n^{1/2}\bx_{0i} = \rho_n^{-1/2}\sum_{j = 1}^n(A_{ij} - \rho_n\bx_{0i}\transpose\bx_{0j})[\bX_0(\bX_0\transpose\bX_0)^{-1}]_{j\cdot} + \widehat\bR^{(\mathrm{ASE})}_i,
\]
where $[\bX_0(\bX_0\transpose\bX_0)^{-1}]_{j\cdot}$ denotes the vector formed by transposing the $j$th row of $\bX_0(\bX_0\transpose\bX_0)^{-1}$, and $(\sum_{i = 1}^n\|\widehat\bR_i^{(\mathrm{ASE})}\|_{\mathrm{F}}^2)^{1/2} = O_{\prob_0}((n\rho_n)^{-1})$. Thus, the ASE satisfies the approximate linearization property \eqref{eqn:linearization_property} with $\omega = 0$ and $\bzeta_{ij} = [\bX_0(\bX_0\transpose\bX_0)^{-1}]_{j\cdot}$, and hence, $\widehat\bX^{(\mathrm{ASE})}$ can be chosen to be an initial estimator for the one-step procedure in practice. Note that the ASE also satisfies the approximate linearization property when $\bx_{01},\ldots,\bx_{0n}$ are deterministic satisfying \eqref{eqn:strong_convergence_measure} (see, for example, Theorem A.5 in \citealp{doi:10.1080/10618600.2016.1193505}). Another initial estimator satisfying the approximate linearization property will be given in Theorem \ref{thm:LSE_transform_initial_condition} using the Laplacian spectral embedding. 

We present the complete procedure for obtaining the one-step estimator \eqref{eqn:one_step_estimator} initialized at the ASE in Algorithm \ref{alg:OSE}.
\begin{algorithm}[h] 
  \renewcommand{\algorithmicrequire}{\textbf{Input:}}
  \renewcommand{\algorithmicensure}{\textbf{Output:} }
  \caption{One-step procedure initialized with the ASE} 
  \label{alg:OSE} 
  \begin{algorithmic}[1] 
    \State{\textbf{Input: }
      The adjacency matrix $\bA = [A_{ij}]_{n\times n}$ and the embedding dimension $d$. }
    \State{\textbf{Step 1:} Compute the eigen-decomposition of the adjacency matrix: 
    \[
    \bA = \sum_{i = 1}^n\widehat{\lambda}_i\widehat{\bu}_i\widehat{\bu}_i\transpose,
    \]
    where $|\widehat\lambda_1|\geq|\widehat\lambda_2|\geq\ldots\geq|\widehat\lambda_n|$, and $\widehat{\bu}_i\transpose\widehat{\bu}_j = \mathbbm{1}(i = j)$ for all $i,j\in[n]$.}
    \State{\textbf{Step 2:} Compute the ASE
        \[
        \widetilde\bX = \widehat\bX^{(\mathrm{ASE})} = \sum_{k = 1}^d|\widehat\lambda_k|^{1/2}\widehat{\bu}_k
        \]
        and write $\widetilde\bX = [\widetilde\bx_1,\ldots,\widetilde\bx_n]\transpose\in\mathbb{R}^{n\times d}$.}
    \State{\textbf{Step 3:} For $i = 1,2,\ldots,n$, compute
        \[
        \widehat\bx_i = \widetilde\bx_i + \left\{\frac{1}{n}\sum_{j = 1}^n\frac{\widetilde\bx_{j}\widetilde\bx_{j}\transpose}{\widetilde\bx_{i}\transpose\widetilde\bx_{j}(1 - \widetilde\bx_i\transpose\widetilde\bx_j)}\right\}^{-1}\left\{\frac{1}{n}\sum_{j = 1}^n\frac{(A_{ij} - \widetilde\bx_i\transpose\widetilde\bx_j)\widetilde\bx_j}{\widetilde\bx_{i}\transpose\widetilde\bx_{j}(1 - \widetilde\bx_i\transpose\widetilde\bx_j)}\right\}.
        \]
        }
    \State{\textbf{Output: }The one-step estimator $\widehat\bX = [\widehat\bx_1,\ldots,\widehat\bx_n]\transpose$. }
  \end{algorithmic}
\end{algorithm}

The notion of efficiency for random dot product graphs becomes less clear when the number of unknown latent positions grows with the number of vertices. In a classical i.i.d. parametric model, the dimension of the parameter space does not change with the sample size. However, in random dot product graphs,  when  all latent positions are unknown and estimating the entire latent position matrix $\bX_0 = [\bx_{01},\ldots,\bx_{0n}]\transpose$ is of interest, the dimension of the parameter space $\calX^n$  grows with the number of vertices. 
Therefore, the definition of the efficiency for   classical i.i.d. parametric models does not apply. To this end, we introduce the notion of local efficiency for random dot product graphs. The idea is that any row of the estimator $\widehat{\bX}$ has the same asymptotic covariance matrix with that of the maximum likelihood estimator as if the rest of the latent positions are known, up to appropriate centering, rescaling, and an orthogonal alignment. 
\begin{definition}[Local efficiency]
Let $\bA\sim\mathrm{RDPG}(\bX_0)$ with a sparsity factor $\rho_n$ for some $\bX_0 = [\bx_{01},\ldots,\bx_{0n}]\transpose\in\calX^n$, $\bx_{01},\ldots,\bx_{0n}\in\calX(\delta)$ for some $\delta > 0$ that does not depend on $n$, and either $\rho_n\equiv 1$ or $\rho_n\to 0$. Denote $\rho = \lim_{n\to\infty}\rho_n$. 
Assume the condition \eqref{eqn:strong_convergence_measure} holds. An estimator $\widehat{\bX}^{(\mathrm{Eff})} = [\widehat{\bx}_1^{(\mathrm{Eff})},\ldots,\widehat{\bx}_n^{(\mathrm{Eff})}]\transpose$ is said to be a locally efficient estimator for $\bX_0$, if there exists a sequence of orthogonal alignment matrices $(\bW)_{n = 1}^\infty = (\bW_n)_{n = 1}^\infty$, such that for all $i\in [n]$,
\[
\sqrt{n}(\bW\transpose\widehat{\bx}_i^{(\mathrm{Eff})} - \rho_n^{1/2}\bx_{0i})\overset{\calL}{\to}\mathrm{N}(\zero, \bG(\bx_{0i})^{-1}),
\]
where $\bG$ is a matrix-valued function $\bG:\calX(\delta)\to\mathbb{R}^{d\times d}$ defined by
\begin{align}
\label{eqn:efficient_covariance}
\bG(\bx) = \int_\calX\frac{\bx_1\bx_1\transpose}{\bx\transpose\bx_1(1 - \rho\bx\transpose\bx_1)}F(\mathrm{d}\bx_1).
\end{align}
\end{definition}

Theorem \ref{thm:asymptotic_normality_OS} and Theorem \ref{thm:convergence_OS} below, which are the main technical results of this paper, establish the asymptotic behavior of the one-step estimator \eqref{eqn:one_step_estimator}. In particular, Theorem \ref{thm:convergence_OS} shows that the one-step estimator $\widehat{\bX}$ is locally efficient. 
\begin{theorem}\label{thm:asymptotic_normality_OS}
Let $\bA\sim\mathrm{RDPG}(\bX_0)$ with a sparsity factor $\rho_n$ for some $\bX_0 = [\bx_{01},\ldots,\bx_{0n}]\transpose\in\calX^n$. Assume that condition \eqref{eqn:strong_convergence_measure} holds, and there exists some constant $\delta > 0$ that is independent of $n$ such that 
$(\bx_{0i})_{i = 1}^n\subset\calX(\delta)$. 
Denote $\widehat\bX = [\widehat\bx_1,\ldots,\widehat\bx_n]\transpose$ the one-step estimator defined by \eqref{eqn:one_step_estimator} initialized at an estimator $\widetilde\bX = [\widetilde\bx_1,\ldots,\widetilde\bx_n]\transpose$ that satisfies the approximate linearization property \eqref{eqn:linearization_property}. Denote
\[
\bG_n(\bx) = \frac{1}{n}\sum_{j = 1}^n\frac{\bx_{0j}\bx_{0j}\transpose}{\bx\transpose\bx_{0j}(1 - \rho_n\bx\transpose\bx_{0j})}
\]  
for any $\bx\in\calX(\delta)$.
If either $\rho_n\equiv 1$ for all $n$ or $\rho_n\to 0$ but $(\log n)^{2(1\vee\omega)}/(n\rho_n^5)\to 0$ as $n\to\infty$,
then there exists a sequence of orthogonal matrices $(\bW)_{n = 1}^\infty = (\bW_n)_{n = 1}^\infty\subset\mathbb{O}(d)$ such that
\begin{align}\label{eqn:OS_expansion}
\bW\transpose\widehat\bx_i - \rho_n^{1/2}\bx_{0i} = \frac{1}{n\sqrt{\rho_n}}\sum_{j = 1}^n\frac{(A_{ij} - \rho_n\bx_{0i}\transpose\bx_{0j})}{\bx_{0i}\transpose\bx_{0j}(1 - \rho_n\bx_{0i}\transpose\bx_{0j})}\bG_n(\bx_{0i})^{-1}\bx_{0j} + \widehat\bR_i,\quad i = 1,\ldots,n,
\end{align}
where $\|\widehat\bR_i\| = O_{\prob_0}(n^{-1}\rho_n^{-5/2}(\log n)^{(1\vee\omega)})$
 and $\sum_{i = 1}^n\|\widehat\bR_i\|^2 = O_{\prob_0}((n\rho_n^5)^{-1}(\log n)^{2(1\vee\omega)})$. 
\end{theorem}
\begin{theorem}\label{thm:convergence_OS}
Let $\bA\sim\mathrm{RDPG}(\bX_0)$ with a sparsity factor $\rho_n$ for some $\bX_0 = [\bx_{01},\ldots,\bx_{0n}]\transpose\in\calX^n$. Assume that the conditions of Theorem \ref{thm:asymptotic_normality_OS} hold, and denote $\rho = \lim_{n\to\infty}\rho_n$. Let $\widehat\bX = [\widehat{\bx}_1,\ldots,\widehat{\bx}_n]\transpose$ be the one-step estimator \eqref{eqn:one_step_estimator} based on an initial estimator $\widetilde{\bX}$ that satisfies the approximate linearization property. Then there exists a sequence of orthogonal matrices $(\bW)_{n = 1}^\infty = (\bW_n)_{n = 1}^\infty\subset\mathbb{O}(d)$ such that as $n\to\infty$,
\begin{align}
\label{eqn:OS_convergence}
\left\|\widehat\bX\bW - \rho_n^{1/2}\bX_0\right\|_{\mathrm{F}}^2 \overset{\prob_0}{\to} \int_{\calX}\mathrm{tr}\left\{\bG(\bx)^{-1}\right\}F(\mathrm{d}\bx),
\end{align}
and for each fixed $i\in[n]$,
\begin{align}
\label{eqn:OS_normality}
\sqrt{n}(\bW\transpose\widehat\bx_i - \rho_n^{1/2}\bx_{0i})\overset{\calL}{\to}\mathrm{N}(\zero, \bG(\bx_{0i})^{-1}),
\end{align} 
where $\bG(\bx)$ is given by equation \eqref{eqn:efficient_covariance}. 
\end{theorem}

Since we have already shown that $\bSigma(\bx_{0i})\succeq \bG(\bx_{0i})^{-1}$ for all $i\in[n]$, it follows that 
\begin{align*}
&\|\widehat\bX\bW - \rho_n^{1/2}\bX_0\|_{\mathrm{F}}^2 - \|\widehat\bX^{(\mathrm{ASE})}\bW - \rho_n^{1/2}\bX_0\|_{\mathrm{F}}^2\\
&\quad\overset{\prob_0}{\to} \int_\calX \mathrm{tr}\{\bSigma(\bx) - \bG(\bx)^{-1}\}F(\mathrm{d}\bx)
 = \lim_{n\to\infty}\frac{1}{n}\sum_{i = 1}^n\mathrm{tr}\{\bSigma(\bx_{0i}) - \bG(\bx_{0i})^{-1}\}
 \geq 0,
\end{align*}
and hence we conclude that the one-step estimator $\widehat\bX$ improves the ASE $\widehat\bX^{(\mathrm{ASE})}$ globally for all vertices asymptotically. Furthermore, for every fixed vertex $i\in[n]$, the $i$th row of the one-step estimator $\widehat\bx_i$ is locally efficient by definition, and the corresponding asymptotic covariance matrix is no greater than that of the corresponding row of the ASE in spectra. 

\begin{remark}
 Theorem \ref{thm:asymptotic_normality_OS}  implies that the one-step estimator $\widehat\bX$  initialized at an estimator that satisfies the approximate linearization property \eqref{eqn:linearization_property} also satisfies \eqref{eqn:linearization_property} by itself when the graph is dense ($\rho_n \equiv 1$ for all $n$).
In this case, one can apply the one-step procedure multiple times, and the resulting estimator still has the same asymptotic behavior as given by Theorem \ref{thm:asymptotic_normality_OS}. This multi-step updating strategy is of practical interest for more accurate estimation  when the sample size is insufficient for asymptotic approximation. 
\end{remark}

\vspace*{1ex}
\noindent\textbf{Proofs sketch for Theorem \ref{thm:asymptotic_normality_OS} and Theorem \ref{thm:convergence_OS}.} 
The key to the proofs of Theorem \ref{thm:asymptotic_normality_OS} and Theorem \ref{thm:convergence_OS} is formula \eqref{eqn:OS_expansion}. From here, we can apply the logarithmic Sobolev concentration inequality to \eqref{eqn:OS_expansion} (see, for example, Section 6.4 in \citealp{boucheron2013concentration}) to show that $\|\widehat\bX\bW - \rho_n^{1/2}\bX_0\|_{\mathrm{F}}^2$ converges in probability to its expectation, which is exactly the quantity on the right-hand side of \eqref{eqn:OS_convergence}. 
The asymptotic normality \eqref{eqn:OS_normality} of $\widehat\bx_i$ can be obtained by directly applying the Lyapunov's central limit theorem to
\[
\frac{1}{\sqrt{n\rho_n}}\sum_{j = 1}^n\frac{(A_{ij} - \rho_n\bx_{0i}\transpose\bx_{0j})}{\bx_{0i}\transpose\bx_{0j}(1 - \rho_n\bx_{0i}\transpose\bx_{0j})}\bG_n(\bx_{0i})^{-1}\bx_{0j},
\]
which is a sum of independent random variables. 
We now sketch the derivation for \eqref{eqn:OS_expansion}. By construction of the one-step estimator \eqref{eqn:one_step_estimator}, we have, 
\begin{align*}
\bW\transpose\widehat\bx_i - \rho_n^{1/2}\bx_{0i}& = \frac{1}{n\sqrt{\rho_n}}\sum_{j = 1}^n\frac{(A_{ij} - \rho_n\bx_{0i}\transpose\bx_{0j})}{\bx_{0i}\transpose\bx_{0j}(1 - \rho_n\bx_{0i}\transpose\bx_{0j})}\bG_n(\bx_{0i})^{-1}\bx_{0j} + 
(\bW\transpose\widetilde \bx_i - \rho_n^{1/2}\bx_{0i})\\
&\quad + \bG_n(\bx_{0i})^{-1}\bR_{i1} + \bR_{i2}\bR_{i1} + \bR_{i3},
\end{align*}
where
\begin{align*}
&\bR_{i1} = \frac{1}{n\sqrt{\rho_n}}\sum_{j = 1}^n\left\{
\bphi_{ij}(\rho_n^{-1/2}\bW\transpose\widetilde\bx_i, \rho_n^{-1/2}\bW\transpose\widetilde\bx_j) - \bphi_{ij}(\bx_{0i}, \bx_{0j})\right\},\\
&\bphi_{ij}(\bu, \bv) = \frac{(A_{ij} - \rho_n\bu\transpose\bv)\bv}{\bu\transpose\bv(1 - \rho_n\bu\transpose\bv)},\\
&\bR_{i2} = \bW\transpose\left\{\frac{1}{n}\sum_{j = 1}^n\frac{\widetilde\bx_{j}\widetilde\bx_{j}\transpose}{\widetilde\bx_{i}\transpose\widetilde\bx_{j}(1 - \widetilde\bx_i\transpose\widetilde\bx_j)}\right\}^{-1}\bW - \bG_n(\bx_{0i})^{-1},\\
&\bR_{i3}  = \frac{1}{n\sqrt{\rho_n}}\sum_{j = 1}^n\frac{(A_{ij} - \rho_n\bx_{0i}\transpose\bx_{0j})}{\bx_{0i}\transpose\bx_{0j}(1 - \rho_n\bx_{0i}\transpose\bx_{0j})}\bR_{i2}\bx_{0j}.
\end{align*}
For $\bR_{i1}$, we apply Taylor's expansion  to $\bphi_{ij}$
together with the result 
\[
\|\widetilde\bX\bW - \rho_n^{1/2}\bX_0\|_{2\to\infty} = O_{\prob_0}\left(\frac{(\log n)^{(1\vee\omega)/2}}{\sqrt{n\rho_n}}\right),
\]
which is a variation of Lemma 2.1 in \cite{lyzinski2014}, to obtain
\begin{align*}
\bR_{i1} 
&= \frac{1}{n\sqrt{\rho_n}}\sum_{j = 1}^n\expect_0\left\{\frac{\partial\bphi_{ij}}{\partial\bu\transpose}(\bx_{0i},\bx_{0j})\right\}(\rho_n^{-1/2}\bW\transpose\widetilde\bx_i - \bx_{0i}) + O_{\prob_0}\left(\frac{(\log n)^{1\vee\omega}}{n\rho_n^{5/2}}\right)\\
&= - \bG_n(\bx_{0i})(\bW\transpose\widetilde\bx_i - \rho_n^{1/2}\bx_{0i}) + O_{\prob_0}\left(\frac{(\log n)^{1\vee\omega}}{n\rho_n^{5/2}}\right).
\end{align*}
For $\bR_{i2}$, we directly obtain from $\|\widetilde\bX\bW - \rho_n^{1/2}\bX_0\|_{2\to\infty} = O_{\prob_0}((n\rho_n)^{-1/2}(\log n)^{(1\vee\omega)/2})$ and the uniform Lipschitz continuity of the function $(\bu,\bv)\mapsto \{\bu\transpose\bv(1 - \rho_n\bu\transpose\bv)\}^{-1}\bv\bv\transpose$ to conclude that $\|\bR_{i2}\| = O_{\prob_0}(\rho_n^{-1}n^{-1/2}(\log n)^{(1\vee\omega)/2})$. Finally, an application of Hoeffding's inequality in conjunction with the union bound yields $\|\bR_{i3}\| = O_{\prob_0}(\rho_n^{-3/2}n^{-1}(\log n)^{1\vee\omega})$. Thus we obtain that
\[
\bW\transpose\widehat\bx_i - \rho_n^{1/2}\bx_{0i} = \frac{1}{n\sqrt{\rho_n}}\sum_{j = 1}^n\frac{(A_{ij} - \rho_n\bx_{0i}\transpose\bx_{0j})}{\bx_{0i}\transpose\bx_{0j}(1 - \rho_n\bx_{0i}\transpose\bx_{0j})}\bG_n(\bx_{0i})^{-1}\bx_{0j} + O_{\prob_0}\left(\frac{(\log n)^{1\vee\omega}}{n\rho_n^{5/2}}\right).
\]
The result $\sum_{i = 1}^n\|\widehat\bR_i\|^2 = O_{\prob_0}((n\rho_n^5)^{-1}(\log n)^{2(1\vee\omega)})$ follows a similar but more technical argument. The detailed proof is provided in Section \ref{sec:proof_of_theorem_OS} of Appendix. 

\begin{remark}
Theorem \ref{thm:convergence_OS} asserts that the one-step estimator $\widehat\bX$ dominates the ASE
 under the
density condition $(n\rho_n^5)^{-1}(\log n)^{2(1\vee\omega)}\to 0$ as $n\to\infty$. When the graph is dense, \emph{i.e.}, $\rho_n\equiv 1$ for all $n$, it is easy to show that this condition holds.  When $\rho_n^{-1}$ is a polynomial of $\log n$, indicating that the graph is moderately sparse, 
this condition still holds.  This condition starts to fail when 
 the graph becomes very sparse, \emph{e.g.}, $\rho_n^{-1} \asymp n^{t}$ for some $t \geq 1/5$, in which case a broad range of statistical inference tasks become challenging due to the weak signal. 
\end{remark}
\begin{remark}\label{remark:strong_density_condition}
Theorem \ref{thm:asymptotic_normality_OS} requires that the sparsity factor $\rho_n$ is lower bounded by $n^{-1/5}$ times a polynomial factor of $\log n$. This causes the average expected degree to grow at a polynomial rate of $n$, and the resulting graph is considered as moderately sparse. In contrast, Theorem \ref{thm:ASE_limit_theorem} only requires $\rho_n$ to be lower bounded by $n^{-1}$ times a polynomial factor of $\log n$, and this results in the average expected degree to grow at a polynomial rate of $\log n$, which is a sparser regime than that required by Theorem \ref{thm:asymptotic_normality_OS}. The stronger density assumption that the average expected degree is a polynomial factor of $n$ is essential for the proof strategy employed in this work. Nevertheless, we remark that the proof strategy is standard (see, for example, Section 5.7 of \citealp{van2000asymptotic}). In fact, the stronger density assumption stems from the Lipschitz continuity of the Hessian of the average log-likelihood function, which is guaranteed by the continuity of the third derivatives. This is referred to as the \emph{classical conditions} for M-estimators (see, for example, Section 5.6 of \citealp{van2000asymptotic}). Further discussion of the sparsity condition for the one-step estimator \eqref{eqn:one_step_estimator} is provided in Section \ref{sec:further_discussion_of_sparse_graphs} of Appendix.
\end{remark}
Theorem \ref{thm:convergence_OS} claims that the asymptotic covariance matrix of any fixed row of the one-step estimator \eqref{eqn:one_step_estimator} is no greater than that of the ASE in spectra. The following example shows that the inequality can be strict, namely, there exist situations where $\bG(\bx_{0i})^{-1} - \bSigma(\bx_{0i})$ contains at least one strictly negative eigenvalue. This implies that the one-step estimator dominates the ASE asymptotically. 
\begin{example} (Two-block stochastic block model)
\label{example:strict_dominance_OSEA}
Consider the following two-block stochastic block model, which has also been considered in \cite{tang2018}. Let $F = \pi_1\delta_p + \pi_2\delta_q$ be the distribution on $(0, 1)$ giving rise to the latent positions $x_{01},\ldots,x_{0n}$ via \eqref{eqn:strong_convergence_measure}, where $p,q\in(0, 1)$ and $p\neq q$. This results in an $n\times n$ adjacency matrix $\bA$ drawn from $\mathrm{RDPG}(\bX_0)$ with $\bX_0 = [x_{01},\ldots,x_{0n}]\transpose\in\mathbb{R}^{n\times 1}$. Let $\tau:[n]\to \{1,2\}$ be a cluster assignment function such that $\tau(i) = 1$ if $x_{0i} = p$, $\tau(i) = 2$ if $x_{0i} = q$, and denote
\[
\bB = 
\begin{bmatrix*}
p^2 & pq \\ pq & q^2
\end{bmatrix*}.
\]
Then the distribution of $\bA$ can be also regarded as a stochastic block model with a block probability matrix $\bB$ and  a cluster assignment function $\tau$. Let $\widehat\bX^{(\mathrm{ASE})} = [\widehat\bx_1^{(\mathrm{ASE})},\ldots,\widehat\bx_n^{(\mathrm{ASE})}]\transpose$ be the ASE and $\widehat\bX = [\widehat\bx_1,\ldots,\widehat\bx_n]\transpose$ be the one-step estimator satisfying the conditions of Theorem \ref{thm:asymptotic_normality_OS}. Using formulas \eqref{eqn:ASE_normality} and \eqref{eqn:OS_normality}, we obtain:
\begin{align*}
\sqrt{n}(\widehat\bx_i^{\mathrm{(ASE)}} - p)\overset{\calL}{\to}\mathrm{N}\left(0, \Sigma(p)\right),\quad\text{if }x_{0i} = p,\\
\sqrt{n}(\widehat\bx_i^{\mathrm{(ASE)}} - q)\overset{\calL}{\to}\mathrm{N}\left(0, \Sigma(q)\right),\quad\text{if }x_{0i} = q,
\end{align*}
where
\[
\Sigma(p) = \frac{\pi_1p^4(1 - p^2) + \pi_2pq^3(1 - pq)}{(\pi_1p^2 + \pi_2q^2)^2},\quad
\Sigma(q) = \frac{\pi_1p^3q(1 - pq) + \pi_2q^4(1 - q^2)}{(\pi_1p^2 + \pi_2q^2)^2},
\]
and
\begin{align*}
\sqrt{n}(\widehat\bx_i - p)\overset{\calL}{\to}\mathrm{N}\left(0, G(p)^{-1}\right),\quad\text{if }x_{0i} = p,\\
\sqrt{n}(\widehat\bx_i - q)\overset{\calL}{\to}\mathrm{N}\left(0, G(q)^{-1}\right),\quad\text{if }x_{0i} = q,
\end{align*}
where
\[
G(p) = \frac{\pi_1p^2}{p^2(1 - p^2)} + \frac{\pi_2q^2}{pq(1 - pq)},\quad
G(q) = \frac{\pi_1p^2}{pq(1 - pq)} + \frac{\pi_2q^2}{q^2(1 - q^2)}.
\]
By Cauchy-Schwartz inequality, we see that $G(p)^{-1}\leq \Sigma(p)$ and $G(q)^{-1}\leq \Sigma(q)$ for all $p, q\in(0, 1)$, and in particular, $G(p)^{-1} = \Sigma(p)$ if and only if $q = (1 - p^2)/p$, and $G(q)^{-1} = \Sigma(q)$ if and only if $q = (1/2)(\sqrt{p^2 + 4} - p)$ (recall that $p\neq q$). Namely, the asymptotic variance of the one-step estimator is strictly smaller than that of the ASE for all but finitely many $(p, q)$ pairs in $(0, 1)^2\backslash\{(p, q):p = q\}$. The comparison of variances for different estimators is further visualized in Figure \ref{fig:SBM_variance_comparison}. 
\begin{figure}[htbp]
  \centerline{\includegraphics[width=1\textwidth]{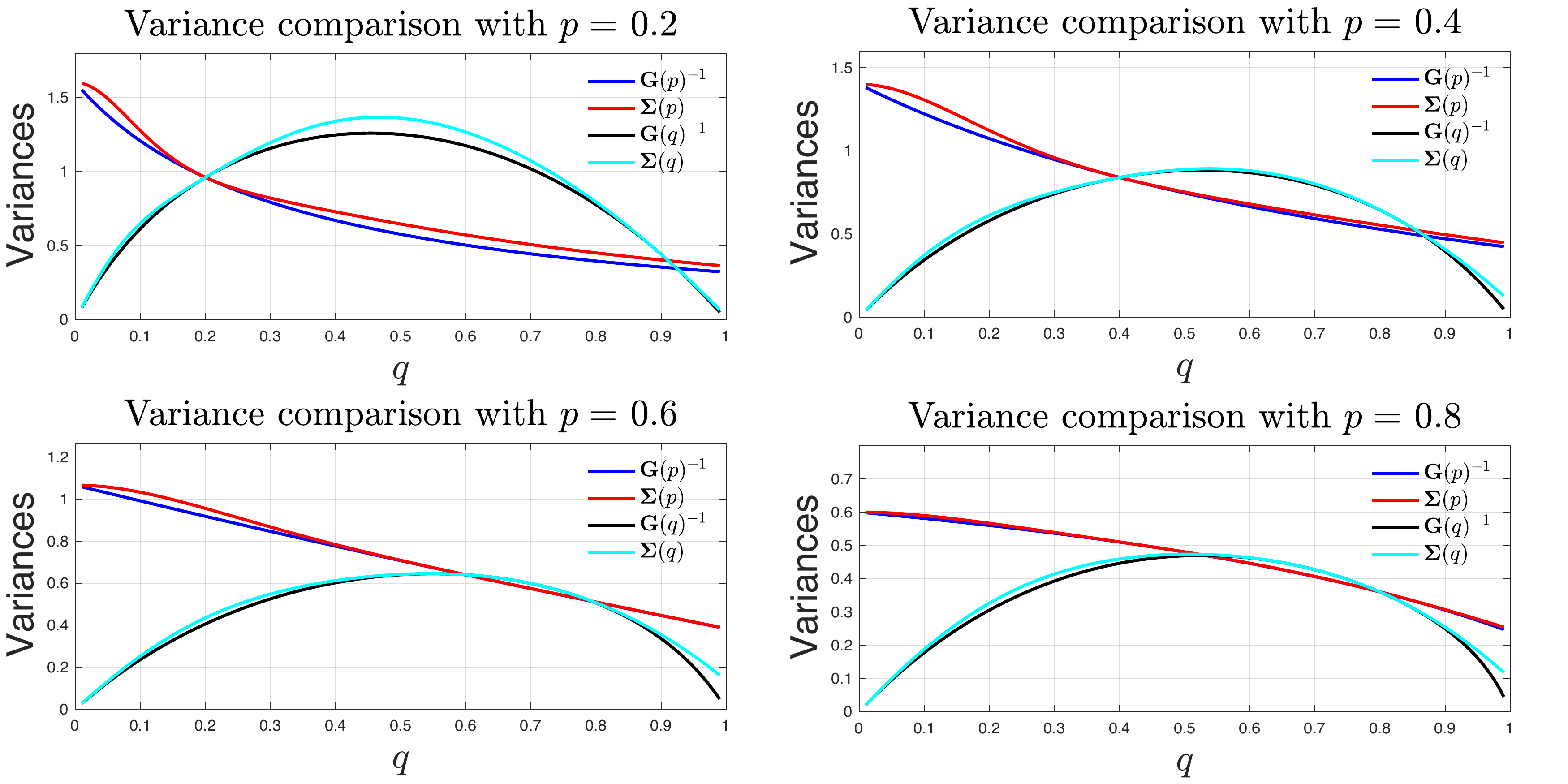}}
  \caption{Comparison of the variances $G(p)^{-1}, G(q)^{-1}, \Sigma(p), \Sigma(q)$ for different values of $p,q\in(0, 1)$ in Example \ref{example:strict_dominance_OSEA}. The cluster assignment probabilities are set to $\pi_1 = 0.6$ and $\pi_2 = 0.4$. Note that all the variances $G(p)^{-1}, G(q)^{-1},\Sigma(p),\Sigma(q)$ depend on both $p$ and $q$. }
  \label{fig:SBM_variance_comparison}
\end{figure}
\end{example}

\section{Application to Estimating the Laplacian Matrix} 
\label{sub:a_plug_in_estimator_for_the_normalized_laplacian_matrix}
Instead of directly analyzing the adjacency matrix $\bA$, 
 another widely adopted technique for statistical analysis on random graphs is based on the normalized Laplacian of  $\bA$, which is particularly useful for vertex clustering in stochastic block models \citep{rohe2011,sarkar2015}. Formally, the normalized Laplacian of a matrix $\bM$ with non-negative entries and positive row sums, denoted by $\calL(\bM)$, is defined by
\[
\calL(\bM) = (\mathrm{diag}(\bM\one))^{-1/2}\bM(\mathrm{diag}(\bM\one))^{-1/2},
\]
where, given $\bz = [z_1,\ldots,z_n]\transpose\in\mathbb{R}^n$, $\mathrm{diag}(\bz)$ is the $n\times n$ diagonal matrix  with $z_1,\ldots,z_n$ being its diagonal entries. Here we follow the definition of the normalized Laplacian adopted in \cite{tang2018} in contrast to the combinatorial Laplacian $\mathrm{diag}(\bM\one) - \bM$ that has been applied to graph theory (see, for example, \citealp{merris1994laplacian}). For the adjacency matrix $\bA$, the $(i,j)$ entry of the normalized Laplacian matrix $\calL(\bA)$ can be interpreted as the connection between vertices $i$ and $j$ normalized by the square roots of the degrees of the two vertices. 

Recall that the edge probability matrix $\rho_n\bX\bX\transpose$ is positive semidefinite low-rank when $\bA\sim\mathrm{RDPG}(\bX)$ with a sparsity factor $\rho_n$. Similarly, the normalized Laplacian of   $\rho_n\bX\bX\transpose$ is also a positive semidefinite low-rank matrix: 
\[
\calL(\rho_n\bX\bX\transpose) = (\mathrm{diag}(\bX\bX\transpose\one))^{-1/2}\bX\bX\transpose(\mathrm{diag}(\bX\bX\transpose\one))^{-1/2} = \bY\bY\transpose, 
\]
where $\bY = [\by_{1},\ldots,\by_{n}]\transpose\in\mathbb{R}^{n\times d}$, and $\by_{i} = \bx_{i}(\sum_{j = 1}^n\bx_{i}\transpose\bx_{j})^{-1/2}$. Following the same spirit of the formulation of the ASE through \eqref{eqn:ASE_least_squared_problem}, one can analogously define the Laplacian spectral embedding (LSE) $\breve\bX$ of $\bA$ into $\mathbb{R}^{d}$ by solving the least squares problem \citep{rohe2011}
\begin{align}\label{eqn:LSE_least_squared_problem}
\breve\bX = \argmin_{\bY\in\mathbb{R}^{n\times d}}\|\calL(\bA) - \bY\bY\transpose\|_{\mathrm{F}}^2.
\end{align}
Since the LSE $\breve{\bX}$ is an estimator for $\bY$, we refer to the $n\times d$ matrix $\bY$ as the population LSE. 
The random random matrix $\breve{\bX}$, which is the LSE of $\bA$ into $\mathbb{R}^d$, is also referred to as the sample LSE in contrast to the population LSE $\bY$.  
Alternatively, the population LSE can be viewed as a transformation $\bY = \bY(\bX)$ of the latent position matrix $\bX$ defined by
\begin{align}
\label{eqn:LSE_transformation}
\bY(\bX) = [\by_1(\bX),\ldots,\by_n(\bX)]\transpose,\quad \by_i = \frac{\bx_i}{\sqrt{\sum_{j = 1}^n\bx_i\transpose\bx_j}},\quad i = 1,\ldots,n.
\end{align}

\cite{tang2018} established the consistency and asymptotic distribution results for the (sample) LSE in random dot product graphs with independent and identically distributed latent positions $\bx_{01},\ldots,\bx_{0n}$. 
 In the context of the deterministic latent positions framework adopted in this work, we provide the analogous results for the  LSE in Theorem \ref{thm:LSE_limit_theorem}, which is a modification of Theorems 3.1 and 3.2 in \cite{tang2018}. The proof  is deferred to Appendix. 


\begin{theorem}
\label{thm:LSE_limit_theorem}
Let $\bA\sim\mathrm{RDPG}(\bX_0)$ with a sparsity factor $\rho_n$ for some $\bX_0 = [\bx_{01},\ldots,\bx_{0n}]\transpose\in\calX^n\subset\mathbb{R}^{n\times d}$, where $\bx_{01},\ldots,\bx_{0n}$ satisfy \eqref{eqn:strong_convergence_measure}. 
Suppose either $\rho_n\equiv 1$ for all $n$ or $\rho_n\to 0$ but $(\log n)^4/(n\rho_n)\to 0$ as $n\to\infty$, and denote $\rho = \lim_{n\to\infty}\rho_n$. Let $\breve\bX = [\breve\bx_1,\ldots,\breve\bx_n]\transpose$ be the LSE of $\bA$ into $\mathbb{R}^d$ defined by \eqref{eqn:LSE_least_squared_problem}. Define the following quantities:
\begin{align*}
&\bY_0 = \bY(\bX_0),\quad \bmu = \int_\calX\bx F(\mathrm{d}\bx), \quad
\widetilde\bDelta = \int_\calX\frac{\bx\bx\transpose}{\bx\transpose\bmu}F(\mathrm{d}\bx),\\
& \widetilde\bSigma(\bx) = \left(\widetilde\bDelta^{-1} - \frac{\bx\bmu\transpose}{2\bmu\transpose\bx}\right)\left[\int_\calX\left\{\frac{\bx\transpose\bx_1(1 - \rho\bx\transpose\bx_1)}{\bmu\transpose\bx(\bmu\transpose\bx_1)^2}\bx_1\bx_1\transpose\right\}F(\mathrm{d}\bx_1)\right]\left(\widetilde\bDelta^{-1} - \frac{\bx\bmu\transpose}{2\bmu\transpose\bx}\right)\transpose.
\end{align*}
Then there exists a sequence of orthogonal $(\bW)_{n = 1}^\infty = (\bW_n)_{n = 1}^\infty\subset\mathbb{R}^{d\times d}$ such that as $n\to\infty$,
\begin{align}
\label{eqn:LSE_convergence}
n\rho_n\|\breve\bX\bW - \bY_0\|_{\mathrm{F}}^2\overset{a.s.}{\to} \int \mathrm{tr}\{\widetilde\bSigma(\bx)\}F(\mathrm{d}\bx).
\end{align} 
Furthermore, assume the graph model falls into one of the following two regimes:
\begin{itemize}
   \item[(i)] \textbf{Dense regime:} $\rho_n\equiv 1$ for all $n$;
   \item[(ii)] \textbf{Sparse stochastic block model regime:} $\rho_n\to 0$ with $(\log n)^4/(n\rho_n)\to 0$ as $n\to\infty$, and $F$ is a mixture of $K$ linearly independent point masses $\bnu_1,\ldots,\bnu_K\in\calX$ with $K \geq d$:
    \[
    F(\mathrm{d}\bx) = \sum_{k = 1}^K\pi_k\delta_{\bnu_k}(\mathrm{d}\bx),\quad \pi_1,\ldots,\pi_K > 0,\quad\sum_{k = 1}^K\pi_k = 1.
    \]
    Namely, the underlying random dot product graph coincides with a stochastic block model.
 \end{itemize}
Then for any fixed $i\in[n]$, 
\begin{align}
\label{eqn:LSE_normality}
n\rho_n^{1/2}(\bW\transpose\breve\bx_i - \by_{0i})\overset{\calL}{\to}\mathrm{N}(\zero, \widetilde\bSigma(\bx_{0i})).
\end{align}
\end{theorem}

The LSE can be applied to construct another initial estimator that satisfies the approximate linearization property. This is given in the following theorem.
\begin{theorem}\label{thm:LSE_transform_initial_condition}
Let $\bA\sim\mathrm{RDPG}(\bX_0)$ with a sparsity factor $\rho_n$ for some $\bX_0 = [\bx_{01},\ldots,\bx_{0n}]\transpose\in\calX^n\subset\mathbb{R}^{n\times d}$, where $\bx_{01},\ldots,\bx_{0n}$ satisfy \eqref{eqn:strong_convergence_measure}. Suppose either $\rho_n\equiv 1$ for all $n$ or $\rho_n\to 0$ but $(\log n)^4/(n\rho_n)\to 0$ as $n\to\infty$, and denote $\rho = \lim_{n\to\infty}\rho_n$. Let $\breve\bX$ be the LSE of $\bA$ into $\mathbb{R}^d$ defined by \eqref{eqn:LSE_least_squared_problem}. Then the estimator $\widetilde\bX = \mathrm{diag}(\sum_{j = 1}^nA_{1j},\ldots,\sum_{j = 1}^nA_{nj})^{1/2}\breve{\bX}$ satisfies the approximate linearization property. 
\end{theorem}

Similar to the ASE, the LSE is also a least squares type estimator and does not involve the likelihood function. 
Therefore, to 
estimate the population LSE $\bY_0 = \bY(\bX_0)$ using the likelihood information of the sampling model, we propose the following one-step estimator $\widehat\bY$ for $\bY_0$ based on the one-step estimator $\widehat\bX = [\widehat\bx_1,\ldots,\widehat\bx_n]\transpose$ defined in \eqref{eqn:one_step_estimator} and an initial estimator $\widetilde\bX = [\widetilde\bx_1,\ldots,\widetilde\bx_n]\transpose$ that satisfies the approximate linearization property \eqref{eqn:linearization_property}:
\begin{align}\label{eqn:one_step_estimator_Laplacian}
\widehat\bY = [\widehat\by_1,\ldots,\widehat\by_n]\transpose,\quad
\widehat\by_i = \frac{\widehat\bx_i}{\sqrt{\sum_{j = 1}^n\widehat\bx_i\transpose\widetilde\bx_j}},\quad i = 1,2,\ldots,n.
\end{align}
In matrix form, we can write $\widehat\bY = \{\mathrm{diag}(\widehat\bX\widetilde\bX\one)\}^{-1/2}\widehat\bX$. 
The likelihood information is thus absorbed into $\widehat\bY$ through the one-step estimator $\widehat\bX$.
We characterize the global and local behavior of the one-step estimator $\widehat\bY$ for the population LSE via the following two theorems. 
\begin{theorem}\label{thm:asymptotic_normality_OS_Laplacian}
Let $\bA\sim\mathrm{RDPG}(\bX_0)$ with a sparsity factor $\rho_n$ for some $\bX_0 = [\bx_{01},\ldots,\bx_{0n}]\transpose\in\calX^n$. Assume that the conditions of Theorem \ref{thm:asymptotic_normality_OS} hold. 
Denote 
$\widehat\bY = [\widehat\by_1,\ldots,\widehat\by_n]\transpose$ the one-step estimator for the population LSE defined by \eqref{eqn:one_step_estimator_Laplacian}, and $\bmu_n = (1/n)\sum_{i = 1}^n\bx_{0i}$. 
Then there exists a sequence of orthogonal matrices $(\bW)_{n = 1}^\infty = (\bW_n)_{n = 1}^\infty\subset\mathbb{O}(d)$ such that
\begin{align*}
\sqrt{n}(\bW\transpose\widehat\by_i - \by_{0i}) = \rho_n^{-1/2}\frac{1}{\sqrt{\bmu_n\transpose\bx_{0i}}}\left(\eye_d - \frac{\bx_{0i}\bmu_n\transpose}{2\bmu_n\transpose\bx_{0i}}\right)(\bW\transpose\widehat\bx_i - \rho_n^{1/2}\bx_{0i}) + \bR_i^{(\mathrm{L})},\quad i = 1,2,\ldots,n,
\end{align*}
where $\|\bR_i^{(\mathrm{L})}\| = O_{\prob_0}((n\rho_n^2)^{-1}(\log n)^{1\vee\omega})$ and $\sum_{i = 1}^n\|\bR_i^{(\mathrm{L})}\|^2 = O_{\prob_0}\left((n\rho_n^4)^{-1}{(\log n)^{2(1\vee\omega)}}\right)$. 
\end{theorem}

\begin{theorem}\label{thm:convergence_OS_Laplacian}
Let $\bA\sim\mathrm{RDPG}(\bX_0)$ with a sparsity factor $\rho_n$ for some $\bX_0 = [\bx_{01},\ldots,\bx_{0n}]\transpose\in\calX^n$.  
Assume the conditions of Theorem \ref{thm:asymptotic_normality_OS_Laplacian} hold. Denote 
$\widehat\bY = [\widehat\by_1,\ldots,\widehat\by_n]\transpose$ the one-step estimator for the population LSE defined by \eqref{eqn:one_step_estimator_Laplacian}, and
\[
\widetilde\bG(\bx) = \frac{1}{(\bmu\transpose\bx)}\left(\eye_d - \frac{\bx\bmu\transpose}{2\bmu\transpose\bx}\right)\bG(\bx)^{-1}\left(\eye_d - \frac{\bx\bmu\transpose}{2\bmu\transpose\bx}\right)\transpose
\]
for any $\bx\in\calX(\delta)$, 
 where $\bmu = \int_\calX\bx F(\mathrm{d}\bx)$ and $\bG(\cdot)$ is defined in equation \eqref{eqn:efficient_covariance}. 
Then there exists a sequence of orthogonal matrices $(\bW)_{n = 1}^\infty = (\bW_n)_{n = 1}^\infty\subset\mathbb{O}(d)$ such that
\begin{align}
\label{eqn:OSL_convergence}
n\rho_n\left\|\widehat\bY\bW - \bY_0\right\|_{\mathrm{F}}^2 \overset{\prob_0}{\to}\int_\calX\mathrm{tr}\left\{\widetilde\bG(\bx)\right\}F(\mathrm{d}\bx),
\end{align}
and for each fixed $i\in[n]$,
\begin{align}
\label{eqn:OLS_asymptotic_normality}
n\rho_n^{1/2}(\bW\transpose\widehat\by_i - \by_{0i})\overset{\calL}{\to}\mathrm{N}(\zero, \widetilde\bG(\bx_{0i})).
\end{align}
\end{theorem}
\begin{remark}
The key difference between the assumption of Theorem \ref{thm:asymptotic_normality_OS_Laplacian} for the one-step estimator for the population LSE and that of Theorem \ref{thm:LSE_limit_theorem} for the (sample) LSE is that, under the sparse regime (ii), we drop the requirement that $F$ is a finite mixture of point masses and $F$ is allowed to be a general distribution function on $\calX^n$, at the cost of a stronger density assumption $(\log n)^{2(1\vee\omega)}/(n\rho_n^4)\to 0$. The same density condition for $\rho_n$ is also required in Theorem \ref{thm:asymptotic_normality_OS}.
\end{remark}
In Section \ref{sec:an_efficient_one_step_estimator}, it is shown that the one-step estimator $\widehat\bX$ dominates the ASE $\widehat\bX^{(\mathrm{ASE})}$ for estimating $\bX_0$ asymptotically. Here we argue that $\widehat\bY$ dominates the LSE $\breve{\bX}$ as well. 
Denote 
\[
\bLambda = \mathrm{diag}\{(\bmu_n\transpose\bx_{01})^{-1},\ldots,(\bmu_n\transpose\bx_{0n})^{-1}\}\]
and
\[
\widetilde\bDelta_n = \int_\calX\left(\frac{\bx_{1}\bx_{1}\transpose}{\bmu_n\transpose\bx_{1}}\right)F_n(\mathrm{d}\bx_1) = \frac{1}{n}\sum_{j = 1}^n\frac{\bx_{0j}\bx_{0j}\transpose}{\bmu_n\transpose\bx_{0j}} = \frac{1}{n}\bX_0\transpose\bLambda\bX_0. 
\] 
Clearly, $\widetilde\bDelta_n\to \widetilde\bDelta$ by \eqref{eqn:strong_convergence_measure}.
Suppose $\bX_0$ yields the singular value decomposition $\bX_0 = \bU_0\bS_0^{1/2}\bV_0\transpose$ with $\bU_0\in\mathbb{O}(n, d)$, $\bS_0^{1/2}$ being diagonal, and $\bV_0\in\mathbb{O}(d)$. By Corollary 2.1 in \cite{PECARIC1996455}, we have
\begin{align*}
(\bU_0\transpose\bLambda\bU_0)(\bU_0\transpose\bD_n(\bx)\bU_0)(\bU_0\transpose\bLambda\bU_0)\succeq \bU_0\transpose\bLambda\bD_n(\bx)\bLambda\bU_0,
\end{align*}
implying that
\begin{align*}
&
\widetilde\bDelta^{-1}_n\left(\frac{1}{n}\bX_0\transpose\bLambda\bD_n(\bx)\bLambda\bX_0\right)\widetilde\bDelta_n^{-1}
\\
&\quad = n(\bX_0\transpose\bLambda\bX_0)^{-1}\left(\bX_0\transpose\bLambda\bD_n(\bx)\bLambda\bX_0\right)(\bX_0\transpose\bLambda\bX_0)^{-1}\\
&\quad = n\bV_0\bS_0^{-1/2}(\bU_0\transpose\bLambda\bU_0)^{-1}(\bU_0\transpose\bLambda\bD_n(\bx)\bLambda\bU_0)(\bU_0\transpose\bLambda\bU_0)^{-1}\bS_0^{-1/2}\bV_0\transpose\\
&\quad\succeq n\bV_0\bS_0^{-1/2}(\bU_0\transpose\bD_n(\bx)^{-1}\bU_0)^{-1}\bS_0^{-1/2}\bV_0\transpose\\
&\quad = n(\bV_0\bS_0^{1/2}\bU_0\transpose\bD_n(\bx)^{-1}\bU_0\bS_0^{1/2}\bV_0\transpose)^{-1}\\
&\quad = \left(\frac{1}{n}\bX_0\transpose\bD_n(\bx)^{-1}\bX_0\right)^{-1}.
\end{align*}
Since $\widetilde\bDelta_n\bmu_n = \bmu_n$, it follows that 
\begin{align*}
\widetilde\bSigma_n(\bx)&:= \frac{1}{\bmu\transpose\bx}\left(\widetilde\bDelta^{-1}_n - \frac{\bx\bmu_n\transpose}{2\bmu_n\transpose\bx}\right)\left\{\frac{1}{n}\sum_{j = 1}^n\frac{\bx\transpose\bx_{0j}(1 - \rho_n\bx\transpose\bx_{0j})}{(\bmu\transpose\bx_{0j})^2}\bx_{0j}\bx_{0j}\transpose\right\}\left(\widetilde\bDelta^{-1}_n - \frac{\bx\bmu_n\transpose}{2\bmu_n\transpose\bx}\right)\transpose\\
& = \frac{1}{\bmu\transpose\bx}\left(\widetilde\bDelta^{-1}_n - \frac{\bx\bmu_n\transpose\widetilde\bDelta_n^{-1}}{2\bmu_n\transpose\bx}\right)
\left(\frac{1}{n}\bX_0\transpose\bLambda\bD_n(\bx)\bLambda\bX_0\right)
\left(\widetilde\bDelta^{-1}_n - \frac{\bx\bmu_n\transpose\widetilde\bDelta_n^{-1}}{2\bmu_n\transpose\bx}\right)\transpose\\
& = \frac{1}{\bmu\transpose\bx}\left(\eye_d - \frac{\bx\bmu_n\transpose}{2\bmu_n\transpose\bx}\right)
\widetilde\bDelta_n^{-1}
\left(\frac{1}{n}\bX_0\transpose\bLambda\bD_n(\bx)\bLambda\bX_0\right)\widetilde\bDelta_n^{-1}
\left(\eye_d - \frac{\bx\bmu_n\transpose}{2\bmu_n\transpose\bx}\right)\transpose\\
&\succeq \frac{1}{\bmu\transpose\bx}\left(\eye_d - \frac{\bx\bmu_n\transpose}{2\bmu_n\transpose\bx}\right)
\left(\frac{1}{n}\bX_0\transpose\bD_n(\bx)^{-1}\bX_0\right)^{-1}
\left(\eye_d - \frac{\bx\bmu_n\transpose}{2\bmu_n\transpose\bx}\right)\transpose\to \widetilde\bG(\bx)
\end{align*}
as $n\to\infty$. Therefore, $\widetilde\bSigma(\bx) = \lim_{n\to\infty}\widetilde\bSigma_n(\bx)\succeq\widetilde\bG(\bx)$. This shows that locally for a fixed vertex $i$, the one-step estimator $\widehat\bY$ improves the LSE $\breve{\bX}$ asymptotically in terms of a smaller asymptotic covariance matrix in spectra. In addition,  
\begin{align*}
&n\rho_n\|\breve{\bX}\bW - \bY_0\|_{\mathrm{F}}^2 - n\rho_n\|\widehat\bY\bW - \bY_0\|_{\mathrm{F}}^2
\overset{\prob_0}{\to} \int_\calX \mathrm{tr}\{\widetilde\bSigma(\bx) - \widetilde\bG(\bx)\}F(\mathrm{d}\bx)\geq 0.
\end{align*}
Namely, the one-step estimator $\widehat\bY$ also improves the LSE $\breve{\bX}$ globally for all vertices in terms of the sum of squares error $\|\widehat\bY\bW - \bY_0\|_{\mathrm{F}}^2$.


\section{Numerical Examples} 
\label{sec:numerical_examples}

\subsection{Clustering performance in stochastic block models via Chernoff information} 
\label{sub:comparison_of_different_estimators_via_chernoff_information}

In this section, we consider the stochastic block models with positive semidefinite block probability matrices in the context of random dot product graphs and focus on vertex clustering as a subsequent inference task of interest after obtaining estimates of the latent positions or the population LSE. In particular, the following four estimates are considered: the ASE \eqref{eqn:ASE_least_squared_problem}, the one-step estimator \eqref{eqn:one_step_estimator} initialized at the ASE, abbreviated as OSE-A, the LSE \eqref{eqn:LSE_least_squared_problem}, and the one-step estimator \eqref{eqn:one_step_estimator_Laplacian} for the population LSE, abbreviated as OSE-L. These estimates are then used as input features for vertex clustering. 

Our goal is to compare the  vertex clustering using these four estimates rather than the performance of specific clustering algorithms. Hence, we  need a criterion that is independent of the choice of the clustering algorithm, but focuses on     distributions of the input features.
To this end, we introduce the concept of \emph{minimum pairwise Chernoff distance}. Generically, let $\bx_1,\ldots,\bx_n$ be i.i.d. random variables following a distribution $F\in\{F_1,\ldots,F_K\}$, where $F_k(\mathrm{d}\bx) = f_k(\bx)\mathrm{d}\bx$, $k\in [K]$, and suppose the task is to determine whether $F = F_k$ for $k\in[K]$. 
Assume that $F = F_k$ with prior probability $\pi_k$, $k\in[K]$. Then for any decision rule $u$, the risk of $u$ is $r(u) = \sum_{k = 1}^K\pi_k\sum_{l\neq k}p_{kl}(u)$, where $p_{kl}(u)$ is the probability that the decision rule $u$ assigns $F = F_l$ when the underlying true distribution is $F = F_k$. 
In the context of vertex clustering, the decision rule $u$ plays the role of a clustering algorithm, and $\bx_i$'s are treated as the rows of one of the aforementioned four estimates. Since we are interested in a criterion that does not depend on $u$, it is natural to investigate the behavior of the risk when the optimal decision rule (clustering algorithm) is applied. 
The following result characterized the optimal error rate \citep{567705}:
\[
\inf_u\lim_{n\to\infty}\frac{1}{n}r(u) = -\min_{k\neq l}C(F_k, F_l),
\]
where $C(F_k, F_l)$ is the \emph{Chernoff information} between $F_k$ and $F_l$ defined by \citep{chernoff1952,10.2307/2236974}
\begin{align}\label{eqn:Chernoff_information}
C(F_k, F_l) = \sup_{t\in(0, 1)}\left\{-\log\int f_k(\bx)^tf_l(\bx)^{1 - t}\mathrm{d}\bx\right\},
\end{align}
and $\min_{k\neq l}C(F_k, F_l)$ is the minimum pairwise Chernoff distance. This quantity describes the asymptotic decaying rate of the error for the optimal decision rule, with larger values indicating smaller optimal error rate. 
In our  context, since the asymptotic distributions of the rows of the four estimators are multivariate normal,  it is useful to derive the Chernoff information for two multivariate normal distributions: 
\[
C(F_k, F_l) = \sup_{t\in(0, 1)}\left\{\frac{t(1 - t)}{2}(\bmu_k - \bmu_l)\transpose\bV_t^{-1}(\bmu_k - \bmu_l) + \frac{1}{2}\log\frac{|\bV_t|}{|\bV_k|^t|\bV_l|^{1 - t}}\right\},
\]
where $F_k = \mathrm{N}(\bmu_k, \bV_k)$ and $F_l = \mathrm{N}(\bmu_l, \bV_l)$, and $\bV_t = t\bV_k + (1 - t)\bV_l$. 


For a $K$-block stochastic block model with a positive semidefinite block probability matrix $\bB = (\bX^*_0)(\bX^*_0)\transpose$, where $\bX^*_0 = [\bnu_1,\ldots,\bnu_K]\transpose\in\mathbb{R}^{K\times d}$, $d\leq K$, and a cluster assignment function $\tau:[n]\to[K]$ satisfying $(1/n)\sum_{i = 1}^n\mathbbm{1}\{\tau(i) = k\}\to \pi_k$ for $k\in[K]$ and $\sum_{k = 1}^K\pi_k = 1$, we define the following quantities for the ASE, the LSE, the OSE-A, and the OSE-L, respectively:
\begin{align*}
\rho_{\mathrm{ASE}}^* &= \min_{k\neq l}\sup_{t\in(0, 1)}
\frac{nt(1-t)}{2}(\bnu_{k} - \bnu_{0l})\transpose\bSigma_{kl}^{-1}(t)(\bnu_k - \bnu_l),\\
\rho_{\mathrm{LSE}}^* &= \min_{k\neq l}\sup_{t\in(0, 1)}
\frac{n^2t(1-t)}{2}(\by_{0k}^* - \by_{0l}^*)\transpose\widetilde\bSigma_{kl}^{-1}(t)(\by_{0k}^* - \by_{0l}^*),\\
\rho_{\mathrm{OSE-A}}^* &= \min_{k\neq l}\sup_{t\in(0, 1)}
\frac{nt(1-t)}{2}(\bnu_{k} - \bnu_{0l})\transpose\bG_{kl}^{-1}(t)(\bnu_{k} - \bnu_{0l}),\\
\rho_{\mathrm{OSE-L}}^* &= \min_{k\neq l}\sup_{t\in(0, 1)}
\frac{n^2t(1-t)}{2}(\by_{0k}^* - \by_{0l}^*)\transpose\widetilde\bG_{kl}^{-1}(t)(\by_{0k}^* - \by_{0l}^*),
\end{align*}
where
\[
\begin{array}{ll}
\bSigma_{kl}(t) = t\bSigma(\bnu_k) + (1 - t)\bSigma(\bnu_{l}),&\widetilde\bSigma_{kl}(t) = t\widetilde\bSigma(\bnu_k) + (1 - t)\widetilde\bSigma(\bnu_{l}),\\
\bG_{kl}(t) = t\bG(\bnu_k)^{-1} + (1 - t)\bG(\bnu_{l})^{-1},& \widetilde\bG_{kl}(t) = t\widetilde\bG_k(\bnu_k) + (1 - t)\widetilde\bG(\bnu_{l}),
\end{array}
\]
and $\by_{0k}^* = \bnu_k(\sum_{l = 1}^Kn\pi_l\bnu_k\transpose\bnu_{l})^{-1/2}$. These quantities are motivated by the use of the minimum pairwise Chernoff distance for measuring clustering performance.
Note that for all $t\in(0, 1)$, we have seen in Section \ref{sec:an_efficient_one_step_estimator} and Section \ref{sub:a_plug_in_estimator_for_the_normalized_laplacian_matrix} that
\begin{align*}
\bSigma_{kl}(t) &= t\bSigma(\bnu_{k}) + (1 - t)\bSigma(\bnu_{l})\succeq t\bG(\bnu_{k})^{-1} + (1 - t)\bG(\bnu_{l})^{-1} = \bG_{kl}(t),\\
\widetilde\bSigma_{kl}(t) &= t\widetilde\bSigma(\bnu_{k}) + (1 - t)\widetilde\bSigma(\bnu_{l})\succeq t\widetilde\bG(\bnu_{k}) + (1 - t)\widetilde\bG(\bnu_{l}) = \widetilde\bG_{kl}(t).
\end{align*}
It follows that $\rho_{\mathrm{ASE}}^*\leq \rho_{\mathrm{OSE-A}}^*$ and $\rho_{\mathrm{LSE}}^*\leq \rho_{\mathrm{OSE-L}}^*$ regardless of the choice of the underlying true latent positions. Namely, the decaying rate of the optimal decision error using the OSE-A is always smaller than that using the ASE, and the same conclusion holds for the comparison between the OSE-L and the LSE. We also note that the above criteria are independent of the choice of the clustering algorithm and only depend on the distribution of the input features. 

\begin{example}\label{example:1D2BlockSBM}
We revisit Example \ref{example:strict_dominance_OSEA} in Section \ref{sec:an_efficient_one_step_estimator}. Consider the following rank-one stochastic block model example with two communities on $n$ vertices. The block probability matrix is
\[
\bB = \begin{bmatrix*}
p^2 & pq \\ pq & q^2
\end{bmatrix*},
\]
where $p,q\in(0, 1)$, and the cluster assignment function $\tau:[n]\to [2]$ satisfies 
\[
\lim_{n\to\infty}\frac{1}{n}\sum_{i = 1}^n\mathbbm{1}\{\tau(i) = 1\} = \pi_1,\quad \lim_{n\to\infty}\frac{1}{n}\sum_{i = 1}^n\mathbbm{1}\{\tau(i) = 2\} = \pi_2,\quad\text{where }\pi_1 + \pi_2 = 1.
\]
The distribution $F$ satisfying condition \eqref{eqn:strong_convergence_measure} can be explicitly computed: $F(\mathrm{d}x) = \pi_1\delta_{p}(\mathrm{d}x) + \pi_2\delta_q(\mathrm{d}x)$ with $\pi_1 + \pi_2 = 1$, $p,q\in(0, 1)$. 
We obtain by simple algebra that 
\begin{align*}
\rho_{\mathrm{OSE-A}}^* &= \frac{n(p - q)^2}{2}\{G(p)^{-1/2} + G(q)^{-1/2}\}^{-2},\\
\rho_{\mathrm{OSE-L}}^* &= \frac{n(p - q)^2}{2}\left\{\frac{\sqrt{p} + \sqrt{q}}{2\sqrt{p}}G(p)^{-1/2} + \frac{\sqrt{p} + \sqrt{q}}{2\sqrt{q}}G(q)^{-1/2}\right\}^{-2}.
\end{align*}
where
\begin{align*}
&G(p) = \frac{\pi_1p^2}{p^2(1 - p^2)} + \frac{\pi_2q^2}{pq(1 - pq)},\quad
G(q) = \frac{\pi_1p^2}{pq(1 - pq)} + \frac{\pi_2q^2}{q^2(1 - q^2)}.
\end{align*}

We have already shown that $\rho^*_{\mathrm{ASE}}\leq \rho^*_{\mathrm{OSE-A}}$ and $\rho^*_{\mathrm{LSE}}\leq \rho^*_{\mathrm{OSE-L}}$ always hold for stochastic block models. 
Furthermore, in this specific example, one can show that $\rho^*_{\mathrm{OSE-A}}\geq \rho_{\mathrm{OSE-L}}$ always holds regardless of the choice of $p$, $q$, $\pi_1$, and $\pi_2$. This means that the OSE-A dominates the OSE-L in terms of the optimal error rate in this specific rank-one stochastic block model example. Note that OSE-A does not necessarily dominate the OSE-L in general, as will be seen in Example \ref{example:2D3BlockSBM} below.
To visualize these findings, we 
 fix $\pi_1 = 0.6, \pi_2  = 0.4$, let $p$ range over $[0.2, 0.8]$, $r = q - p$ range over in $[-0.15, 0.15]\backslash\{0\}$, compute the ratios  $\rho_{\mathrm{ASE}}^*/\rho_{\mathrm{OSE-A}}^*$,  $\rho_{\mathrm{LSE}}^*/\rho_{\mathrm{OSE-L}}^*$,  $\rho_{\mathrm{ASE}}^*/\rho_{\mathrm{LSE}}^*$,  $\rho_{\mathrm{OSE-A}}^*/\rho_{\mathrm{OSE-L}}^*$, and plot the numerical results in Figure \ref{fig:SBM_chernoff_ratio_K2}. 
\begin{figure}[t]
  \centerline{\includegraphics[width=1\textwidth]{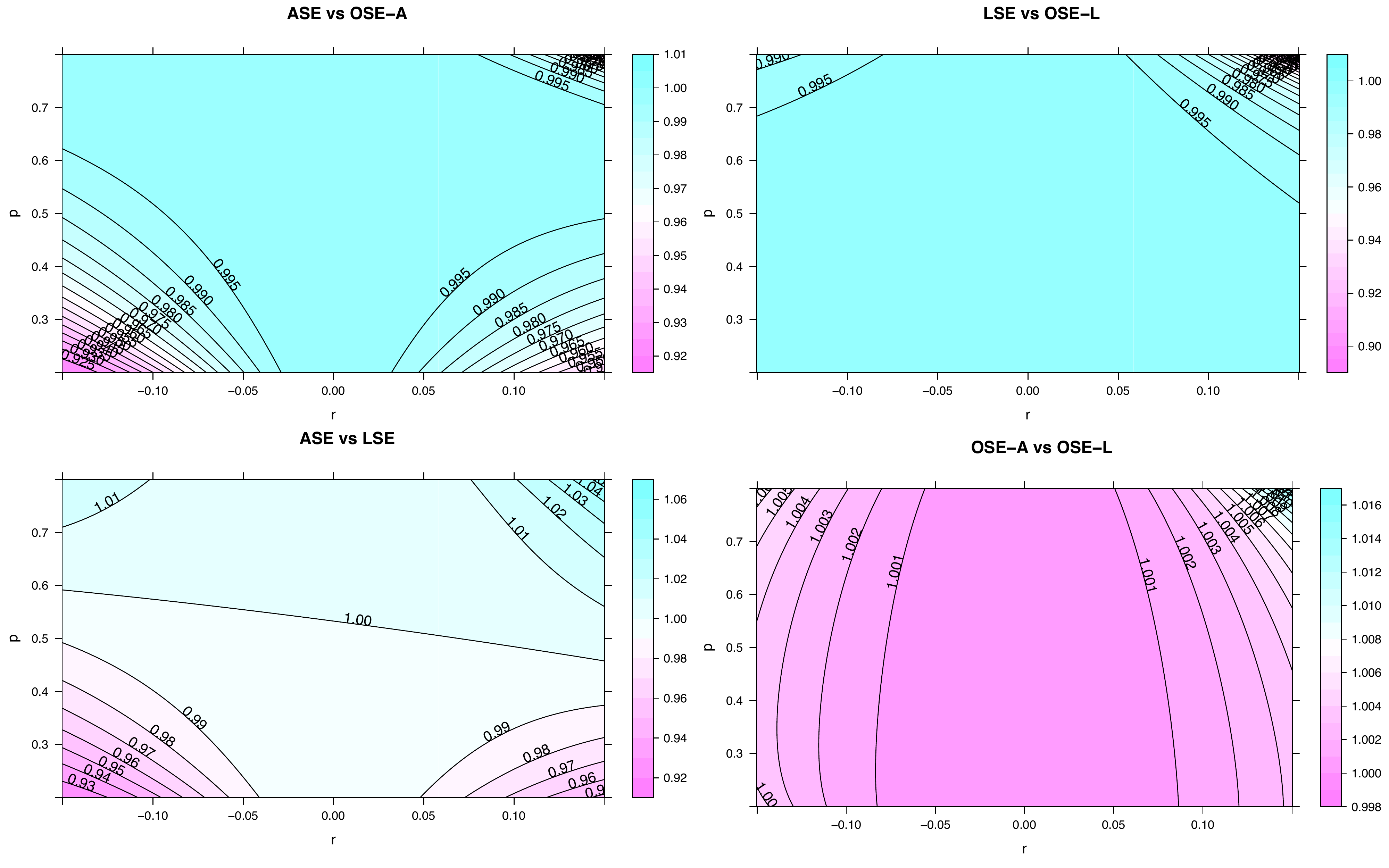}}
  \caption{Heatmap and level curves of the ratio $\rho_{\mathrm{OSE-A}}^*/\rho_{\mathrm{OSE-L}}^*$ for $p \in [0.2, 0.8]$ and $r\in[-0.15,0.15]\backslash\{0\}$ for Example \ref{example:1D2BlockSBM}. }
  \label{fig:SBM_chernoff_ratio_K2}
\end{figure}

Besides the aforementioned large sample conclusion, we perform two finite-sample experiments as well. We first compute the four estimates with $n = 200$, $p = 0.6, q = 0.4$ and $n = 200$, $p = 0.45$, $q = 0.6$, respectively. 
In each of the scenarios, we choose the Gaussian-mixture-model-based (GMM-based) clustering algorithm \citep{fraley2012mclust,fraley2002model}, which is recommended in \cite{tang2018}, for the subsequent vertex clustering task.
To evaluate the finite-sample clustering results, we adopt the Rand index \citep{doi:10.1080/01621459.1971.10482356} as a measurement of the clustering accuracy. 
Formally, given two partitions $\calC_1 = \{c_{11},\ldots,c_{1r}\}$ and $\calC_2 = \{c_{21},\ldots,c_{2s}\}$ of $[n]$, let $a$ be the number of pairs in $[n]$ that are both in the same block in partition $\calC_1$ and in the same block in partition $\calC_2$, and $b$ the number of pairs in $[n]$ that are neither in the same block in $\calC_1$ nor in the same block in $\calC_2$. Then the Rand index (between $\calC_1$ and $\calC_2$) is defined by $\mathrm{RI}(\calC_1,\calC_2) = 2(a + b)/\{n(n - 1)\}$. The Rand index ranges between $0$ and $1$, with higher value suggesting a better agreement between $\calC_1$ and $\calC_2$. 
Table \ref{table:SBM_simulation_RI_GMM_1} reports the average Rand indices of the four cluster assignment estimates in comparison with the underlying true cluster assignment based on $1000$ Monte Carlo replicates, together with the corresponding standard errors. The differences in the Rand indices are statistically significant at level $\alpha = 0.01$, and the results are in accordance with the aforementioned large sample conclusion. 
\begin{table}[htbp]
\caption{Rand indices of the GMM-based clustering algorithm using different estimates for Example \ref{example:1D2BlockSBM}. For each setup $p = 0.6,q = 0.4$ and $p =0.45$, $q = 0.6$, the Rand indices are averaged over $1000$ Monte Carlo replicates of adjacency matrices. The standard errors are included in the parentheses. }
\centering{%
	\begin{tabular}{c | c c }
	    \hline\hline
		Estimates & ASE & OSE-A \\
		\hline
		$p = 0.6$, $q = 0.4$   & 0.8985 ($1.13\times 10^{-3}$) & {\bf 0.9022} ($1.12\times 10^{-3}$)  \\
		$p  = 0.45$, $q = 0.6$ & 0.7635 ($2.06\times 10^{-3}$)& {\bf 0.7899} ($1.70\times 10^{-3}$)	 \\
    \hline\hline
    Estimates & LSE & OSE-L \\
    \hline
    $p = 0.6$, $q = 0.4$ & 0.8966 ($1.26\times 10^{-3}$) & 0.8972 ($1.25\times 10^{-3}$)     \\
    $p  = 0.45$, $q = 0.6$ & 0.7742 ($1.64\times 10^{-3}$) & 0.7863 ($1.37\times 10^{-3}$) \\
	    \hline\hline
	  \end{tabular}%
	}
\label{table:SBM_simulation_RI_GMM_1}
\end{table}

\end{example}

\begin{example}\label{example:2D3BlockSBM}
We next consider the following rank-two stochastic block model with three communities. The block probability matrix $\bB$ is given by $\bB = (\bX_0^\star)(\bX_0^\star)\transpose$, where 
\[
\bX_0^\star = \begin{bmatrix*}
q & q & p\\
q & p & q
\end{bmatrix*}.
\] 
Let $\tau:[n]\to\{1,2,3\}$ be the corresponding cluster assignment function satisfying
\[
\lim_{n\to\infty}\frac{1}{n}\sum_{i = 1}^n\mathbbm{1}\{\tau(i) = 1\} = 0.8,\quad 
\lim_{n\to\infty}\frac{1}{n}\sum_{i = 1}^n\mathbbm{1}\{\tau(i) = 2\} = 0.1,\quad
\lim_{n\to\infty}\frac{1}{n}\sum_{i = 1}^n\mathbbm{1}\{\tau(i) = 2\} = 0.1.
\]
We let $p$ range over $[0.3, 0.6]$ and $q = p - r$ with $r\in[-0.2, -0.01]$. We then explore the minimum pairwise Chernoff distance via the computation of the ratios $\rho^*_{\mathrm{ASE}}/\rho^*_{\mathrm{OSE-A}}$, $\rho^*_{\mathrm{LSE}}/\rho^*_{\mathrm{OSE-L}}$, $\rho^*_{\mathrm{ASE}}/\rho^*_{\mathrm{LSE}}$, $\rho^*_{\mathrm{OSE-A}}/\rho^*_{\mathrm{OSE-L}}$, and plot the ratios in Figure \ref{fig:SBM_chernoff_ratio_K3}. Panels (a) and (b) in Figure \ref{fig:SBM_chernoff_ratio_K3} show that the OSE-A outperforms the ASE and the OSE-L outperforms the LSE, respectively, in terms of the optimal clustering error rates. Panel (c) of Figure \ref{fig:SBM_chernoff_ratio_K3} indicate that the LSE outperforms the ASE for sparser stochastic block models corresponding to the lower-left region of the heatmap, and a similar comparison conclusion between the OSE-A and OSE-L can be drawn from panel (d) of Figure \ref{fig:SBM_chernoff_ratio_K3} as well. 
\begin{figure}[ht!]
  \centerline{\includegraphics[width=1\textwidth]{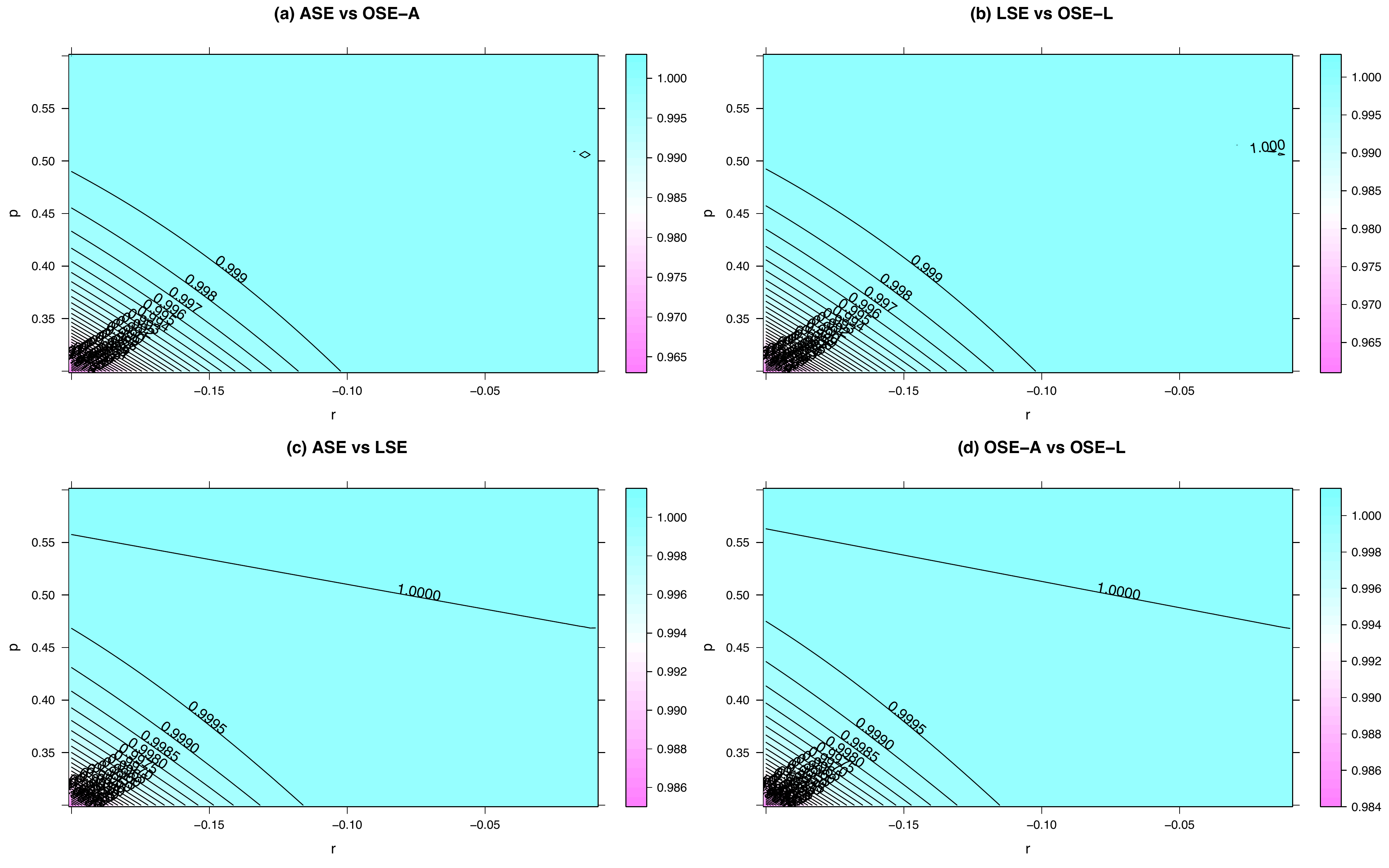}}
  \caption{Heatmap and level curves of the ratio $\rho_{\mathrm{OSE-A}}^*/\rho_{\mathrm{OSE-L}}^*$ for $p \in [0.2, 0.8]$ and $r\in[-0.15,0.15]\backslash\{0\}$ for Example \ref{example:2D3BlockSBM}. }
  \label{fig:SBM_chernoff_ratio_K3}
\end{figure}
\end{example}

\begin{remark}
Unlike the  minimum pairwise Chernoff distance, which is an asymptotic criterion for comparing the performance of different estimators in terms of the subsequent optimal clustering rates and does not depend on the clustering algorithm, the Rand index can only reflect the behavior of the clustering result in a finite-sample experiment and may depend on the clustering method we choose. 
\end{remark}


\subsection{A three-block stochastic block model example} 
\label{sub:a_simulated_example}
We next consider the following three-block stochastic block model on $n$ vertices with the block probability matrix $\bB = (\bX_0^*)(\bX_0^*)\transpose{}$, where
\[
(\bX_0^*)\transpose{} = \begin{bmatrix*}
0.3 & 0.3 & 0.6\\
0.3 & 0.6 & 0.3\\
\end{bmatrix*},
\]
and a cluster assignment function $\tau:[n]\to [3]$, such that as $n\to\infty$,
\[
\frac{1}{n}\sum_{i = 1}^n\mathbbm{1}\{\tau(i) = 1\}\to 0.3,\quad
\frac{1}{n}\sum_{i = 1}^n\mathbbm{1}\{\tau(i) = 2\}\to 0.3,\quad
\frac{1}{n}\sum_{i = 1}^n\mathbbm{1}\{\tau(i) = 3\}\to 0.4.
\]
The corresponding distribution $F$ satisfying condition \eqref{eqn:strong_convergence_measure} is 
$F(\mathrm{d}\bx) = \sum_{k = 1}^3\pi_k\delta_{\bnu_{k}}(\mathrm{d}\bx)$, 
where $\pi_1 = \pi_2 = 0.3$, $\pi_3 = 0.4$, $\bnu_{1} = [0.3, 0.3]\transpose$, $\bnu_{2} = [0.3, 0.6]\transpose$, and $\bnu_{3} = [0.6, 0.3]\transpose$. For each $n\in\{500, 600, \ldots, 1200\}$, we generate $10000$ replicates of the simulated adjacency matrices from the above sampling model, and then compute the following four estimates: the ASE \eqref{eqn:ASE_least_squared_problem}, the one-step estimate \eqref{eqn:one_step_estimator} initialized at the ASE (OSE-A), the LSE \eqref{eqn:LSE_least_squared_problem}, and the one-step estimate \eqref{eqn:one_step_estimator_Laplacian} for the population LSE (OSE-L). The goal is to compare the performance of vertex clustering by applying the GMM-based clustering algorithm to these estimates.


 Figure \ref{fig:RI_SBM_K3} and Table \ref{table:SBM_simulation_RI_GMM} present the Rand indices of clustering results obtained by applying the GMM-based clustering algorithm to the four estimates against the underlying true cluster assignment, and these Rand indices are averaged based on $10000$ Monte Carlo replicates. The standard errors corresponding to the Monte Carlo replicates are tabulated in the parentheses of Table \ref{table:SBM_simulation_RI_GMM}. 
 When the number of vertices $n\in\{500,600,700,800\}$, the clustering results based on the ASE outperform the rest of the competitors. 
 However, as $n$ increases with $n\geq 900$, the best result is given by either the OSE-A or the OSE-L, and the differences in the Rand indices are statistically significant at level $\alpha = 0.01$. In particular, when $n \in \{1100, 1200\}$, the OSE-A and the OSE-L yield better results than the 
ASE and the LSE, respectively.  These numerical results are in accordance with the fact that asymptotically, the ASE and the LSE are dominated by the OSE-A and OSE-L, respectively. 
\begin{figure}[ht!]
  \centerline{\includegraphics[width=.9\textwidth]{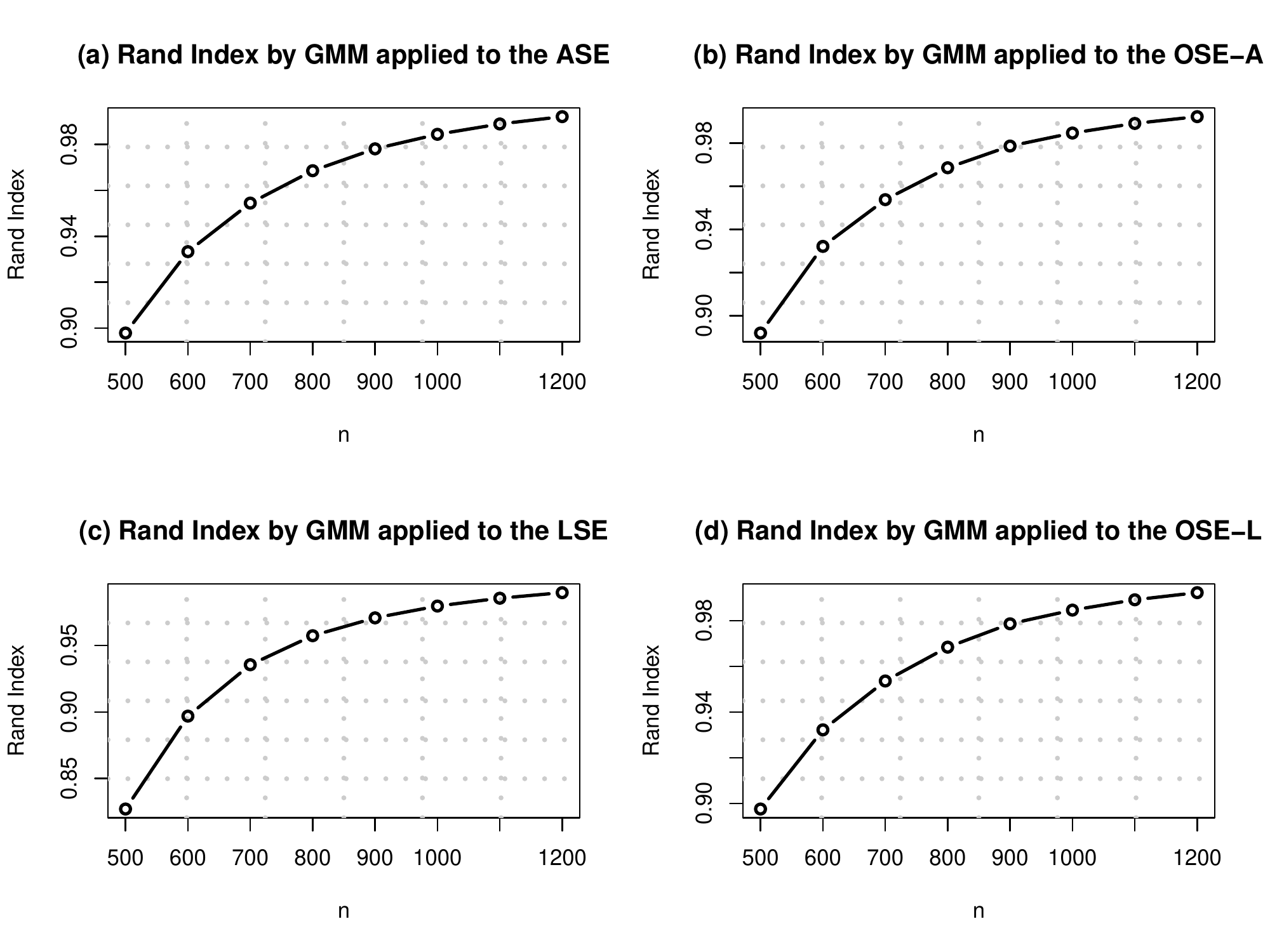}}
  \caption{The three-block stochastic block model example in Section \ref{sub:a_simulated_example}: The Rand indices of the GMM-model-based clustering method applied to different estimates (the ASE, the OSE-A, the LSE, and the OSE-L) when the number of vertices $n$ ranges in $\{500,600,\ldots,1200\}$. The results are averaged based on $10000$ Monte Carlo replicates. }
  \label{fig:RI_SBM_K3}
\end{figure}
\begin{table}[ht!]
\centering%
\caption{The three-block stochastic block model example in Section \ref{sub:a_simulated_example}: Rand indices of the GMM-based clustering algorithm using different estimates. The number of vertices $n$ ranges over $\{500, 600, \ldots, 1200\}$, and for each $n$, the Rand indices are averaged over $10000$ Monte Carlo replicates of adjacency matrices, with the standard errors included in parentheses. }
\vspace*{1ex}
\footnotesize
\begin{tabular}{c c c c c}
    \hline\hline
	Estimates & ASE & OSE-A & LSE & OSE-L \\
	\hline
	$n = 500$  & {\bf 0.89783} ($3.3\times10^{-4}$) & 0.89199 ($4.0\times10^{-4}$)	   & 0.82708 ($5.9\times 10^{-4}$) & 0.89761 ($3.4\times 10^{-4}$)	 	 \\
	$n = 600$  & {\bf 0.93329} ($1.9\times10^{-4}$) & 0.93212 ($2.2\times10^{-4}$)	   & 0.89691 ($4.2\times 10^{-4}$) & 0.93224 ($2.0\times 10^{-4}$) 	 \\
	$n = 700$  & {\bf 0.95445} ($1.3\times10^{-4}$) & 0.95378 ($1.4\times10^{-4}$)	   & 0.93555 ($2.1\times 10^{-4}$) & 0.95361 ($1.4\times 10^{-4}$)      \\
	$n = 800$  & {\bf 0.96857} ($8.6\times10^{-5}$) & 0.96856 ($8.6\times10^{-5}$)	   & 0.95740 ($1.2\times 10^{-4}$) & 0.96841 ($8.8\times 10^{-5}$)      \\
	$n = 900$  & 0.97811 ($6.4\times10^{-5}$)	   & { 0.97861} ($6.1\times10^{-5}$)& 
  0.97081 ($8.2\times 10^{-5}$) & {\bf 0.97862} ($6.0\times 10^{-5}$)		 \\
	$n = 1000$ & 0.98443 ($5.0\times10^{-5}$)	   & {\bf 0.98465} ($4.9\times10^{-5}$)& 
  0.97968 ($6.1\times 10^{-5}$) & { 0.98460} ($4.8\times 10^{-5}$)\\
	$n = 1100$ & 0.98894 ($4.0\times10^{-5}$)	   & { 0.98909} ($3.9\times10^{-5}$)& 
  0.98559 ($4.6\times 10^{-5}$) & {\bf 0.98911} ($3.9\times 10^{-5}$)      \\
	$n = 1200$ & 0.99213 ($3.1\times10^{-5}$)	   & { 0.99224} ($3.0\times10^{-5}$)& 
  0.98978 ($3.6\times 10^{-5}$) & {\bf 0.99225}	($3.0\times 10^{-5}$)	 \\
    \hline\hline
  \end{tabular}%
\label{table:SBM_simulation_RI_GMM}
\end{table}

For each $n\in\{600, 900, 1200\}$, we also compute the OSE-A $\widehat\bX$ and the OSE-L $\widehat\bY$ for each block, as well as the corresponding cluster-specific sample covariance matrices after applying the appropriate orthogonal transformation towards the underlying true $\bX_0$ and $\bY_0$, for one randomly selected instance among the $10000$ replicated adjacency matrices. The results are tabulated in Table \ref{table:SBM_covariance_matrix_K3_OSEA} and Table \ref{table:SBM_covariance_matrix_K3_OSEL}, respectively, in comparison with the limit covariance matrices given by Theorem \ref{thm:ASE_limit_theorem} (for the ASE), Theorem \ref{thm:convergence_OS} (for the OSE-A), Theorem \ref{thm:LSE_limit_theorem} (for the LSE), and Theorem \ref{thm:convergence_OS_Laplacian} (for the OSE-L). 
It can be seen that as $n$ increases, the sample covariance matrices converge to their corresponding cluster-specific limit covariance matrices. 
The scatter points of $\widehat\bX$ and $\widehat\bY$ after applying the orthogonal alignment matrix $\bW$ towards $\bX_0$ and $\bY_0$ are visualized in Figure \ref{fig:SBM_K3_OSEA_scatter}, along with the cluster-specific $95\%$ empirical and asymptotic confidence ellipses in dashed lines and solid lines, respectively. These figures also validate the aforementioned limit results empirically. 
\begin{table}[ht!]
\footnotesize
  \centering
  \caption{Three-block stochastic block model example in Section \ref{sub:a_simulated_example}: the cluster-specific sample covariance matrices for the OSE-A with the number of vertices $n\in\{600, 900, 1200\}$, in comparison with the limit covariance matrix of the OSE-A and the ASE. }
  \vspace*{1ex}
  \begin{tabular}{c | c c c }
    \hline\hline
     & $n = 600$& $n = 900$& $n = 1200$ \\
    \hline\hline\\[-0.3cm]
    $\bG_n(\bnu_1)^{-1}$ &
    $\begin{bmatrix*}
    3.762811 & -3.651978\\ -3.651978 & 4.621907
    \end{bmatrix*}$
    & $\begin{bmatrix*}
    3.647588 & -3.669011\\ -3.669011 & 4.756269
    \end{bmatrix*}$ 
    & $\begin{bmatrix*}
    3.550210 & -3.169485\\ -3.169485 & 4.032291
    \end{bmatrix*}$\\[0.3cm]
    \hline\\[-0.3cm]
    & $\bG(\bnu_1)^{-1}$ & $\bSigma(\bnu_1)$ \\[0.1cm]
    \hline\\[-0.3cm]
    Limit Covariances & $\begin{bmatrix*}
    3.220559 & -2.898602\\ -2.898602 & 3.703496
    \end{bmatrix*}$
    & $\begin{bmatrix*}
    3.221615 & -2.895962\\
    -2.895962 & 3.710096
    \end{bmatrix*}$\\[0.3cm]
    \hline\hline
    & $n = 600$& $n = 900$& $n = 1200$ \\
        \hline\hline\\[-0.3cm]
    $\bG_n(\bnu_2)^{-1}$&
    $\begin{bmatrix*}
    3.765898 & -3.803725\\ -3.803725 & 5.213684
    \end{bmatrix*}$
    & $\begin{bmatrix*}
    4.846434 & -4.701554\\ -4.701554 & 6.290449
    \end{bmatrix*}$ 
    & $\begin{bmatrix*}
    4.403550 & -4.438183\\ -4.438183 & 5.918178
    \end{bmatrix*}$
    \\[0.3cm]
    \hline\\[-0.3cm]
    & $\bG(\bnu_2)^{-1}$ & $\bSigma(\bnu_2)$ \\[0.1cm]
    \hline\\[-0.3cm]
    Limit Covariances & $\begin{bmatrix*}
    3.844914 & -3.518540\\ -3.518540 & 4.590484
    \end{bmatrix*}$
    & $\begin{bmatrix*}
        3.844943 & -3.519037\\
        -3.519037 & 4.598917
        \end{bmatrix*}$
    \\[0.3cm]
        \hline\hline
    & $n = 600$& $n = 900$& $n = 1200$ \\
    \hline\hline\\[-0.3cm]
    $\bG_n(\bnu_3)^{-1}$ &
    $\begin{bmatrix*}
    5.812994 & -4.747638\\ -4.747638 & 5.186311
    \end{bmatrix*}$
    & $\begin{bmatrix*}
    5.322016 & -4.583552\\ -4.583552 & 5.301519
    \end{bmatrix*}$ 
    & $\begin{bmatrix*}
    4.664418 & -4.140929\\ -4.140929 & 5.045003
    \end{bmatrix*}$
    \\[0.3cm]
    \hline\\[-0.3cm]
    & $\bG(\bnu_3)^{-1}$ & $\bSigma(\bnu_3)$ \\[0.1cm]
    \hline\\[-0.3cm]
    Limit Covariances & $\begin{bmatrix*}
    3.969281 & -3.495403\\ -3.495403 & 4.414981
    \end{bmatrix*}$
    & $\begin{bmatrix*}
    3.966424 & -3.496907\\
   -3.496907 & 4.414189
    \end{bmatrix*}$
    \\[0.3cm]
    \hline\hline
  \end{tabular}%
  \label{table:SBM_covariance_matrix_K3_OSEA}
\end{table}%

\begin{table}[ht!]
\footnotesize
  \centering
  \caption{Three-block stochastic block model example in Section \ref{sub:a_simulated_example}: the cluster-specific sample covariance matrices for the OSE-L with the number of vertices $n\in\{600, 900, 1200\}$, in comparison with the limit covariance matrix of the OSE-L and the LSE.}
  \vspace*{1ex}
  \begin{tabular}{c | c c c}
    \hline\hline\\[-0.3cm]
    $k$ & $n = 600$& $n = 900$& $n = 1200$ \\[0.1cm]
    \hline\hline\\[-0.3cm]
    $\widetilde{\bG}_n(\bnu_1)$ &
    $\begin{bmatrix*}
    14.90763 & -15.75143\\ -15.75143 & 17.78524
    \end{bmatrix*}$
    & $\begin{bmatrix*}
    14.41713 & -15.60931\\ -15.60931 & 18.02098
    \end{bmatrix*}$ 
    & $\begin{bmatrix*}
    13.55155 & -14.03450\\ -14.03450 & 15.8845
    \end{bmatrix*}$\\[0.3cm]
    \hline\\[-0.3cm]
    & $\widetilde{\bG}(\bnu_1)$ & $\widetilde{\bSigma}(\bnu_1)$  \\[0.1cm]
    \hline\\[-0.3cm]
    Limit Covariances & $\begin{bmatrix*}
    12.40965 & -12.78420\\ -12.78420 & 14.37281
    \end{bmatrix*}$
    & $\begin{bmatrix*}
    12.41030 & -12.78353\\ -12.78353 & 14.37349
    \end{bmatrix*}$
    \\[0.3cm]
    \hline\hline
    & $n = 600$& $n = 900$& $n = 1200$ \\    
    \hline\hline\\[-0.3cm]
    $\widetilde{\bG}_n(\bnu_2)$ &
    $\begin{bmatrix*}
    10.29810 & -11.05184\\ -11.05184 & 12.82275
    \end{bmatrix*}$
    & $\begin{bmatrix*}
    12.95708 & -13.73205\\ -13.73205 & 15.79435
    \end{bmatrix*}$ 
    & $\begin{bmatrix*}
    12.12355 & -13.08517\\ -13.08517 & 15.15464
    \end{bmatrix*}$\\[0.3cm]
    \hline\\[-0.3cm]
    & $\widetilde{\bG}(\bnu_2)$ & $\widetilde{\bSigma}(\bnu_2)$  \\[0.1cm]
    \hline\\[-0.3cm]
    Limit Covariances & $\begin{bmatrix*}
    10.22658 & -10.48123\\ -10.48123 & 11.73625
    \end{bmatrix*}$
    & $\begin{bmatrix*}
    10.23471 & -10.48190\\ -10.48190 & 11.73631
    \end{bmatrix*}$
    \\[0.3cm]
    \hline\hline
    & $n = 600$& $n = 900$& $n = 1200$ \\
    \hline\hline\\[-0.3cm]
    $\widetilde{\bG}_n(\bnu_3)$ &
    $\begin{bmatrix*}
    13.64657 & -13.33241\\ -13.33241 & 14.05080
    \end{bmatrix*}$
    & $\begin{bmatrix*}
    12.79163 & -12.98354\\ -12.98354 & 14.20550
    \end{bmatrix*}$ 
    & $\begin{bmatrix*}
    11.36609 & -11.81191\\ -11.81191 & 13.30482
    \end{bmatrix*}$
    \\
    \hline\\[-0.3cm]
    & $\widetilde\bG(\bnu_3)$ & $\widetilde{\bSigma}(\bnu_3)$ \\[0.1cm]
    \hline\\[-0.3cm]
    Limit Covariances & $\begin{bmatrix*}
    9.821792 & -10.16463\\ -10.16463 & 11.50649
    \end{bmatrix*}$
    & $\begin{bmatrix*}
    9.823044 & -10.16911\\ -10.16911 & 11.52254
    \end{bmatrix*}$
    \\[0.3cm]
    \hline\hline
  \end{tabular}%
  \label{table:SBM_covariance_matrix_K3_OSEL}
\end{table}%

\begin{figure}[htbp]
	\begin{center}
		\begin{tabular}{cc}
			\includegraphics[width=.43\textwidth]{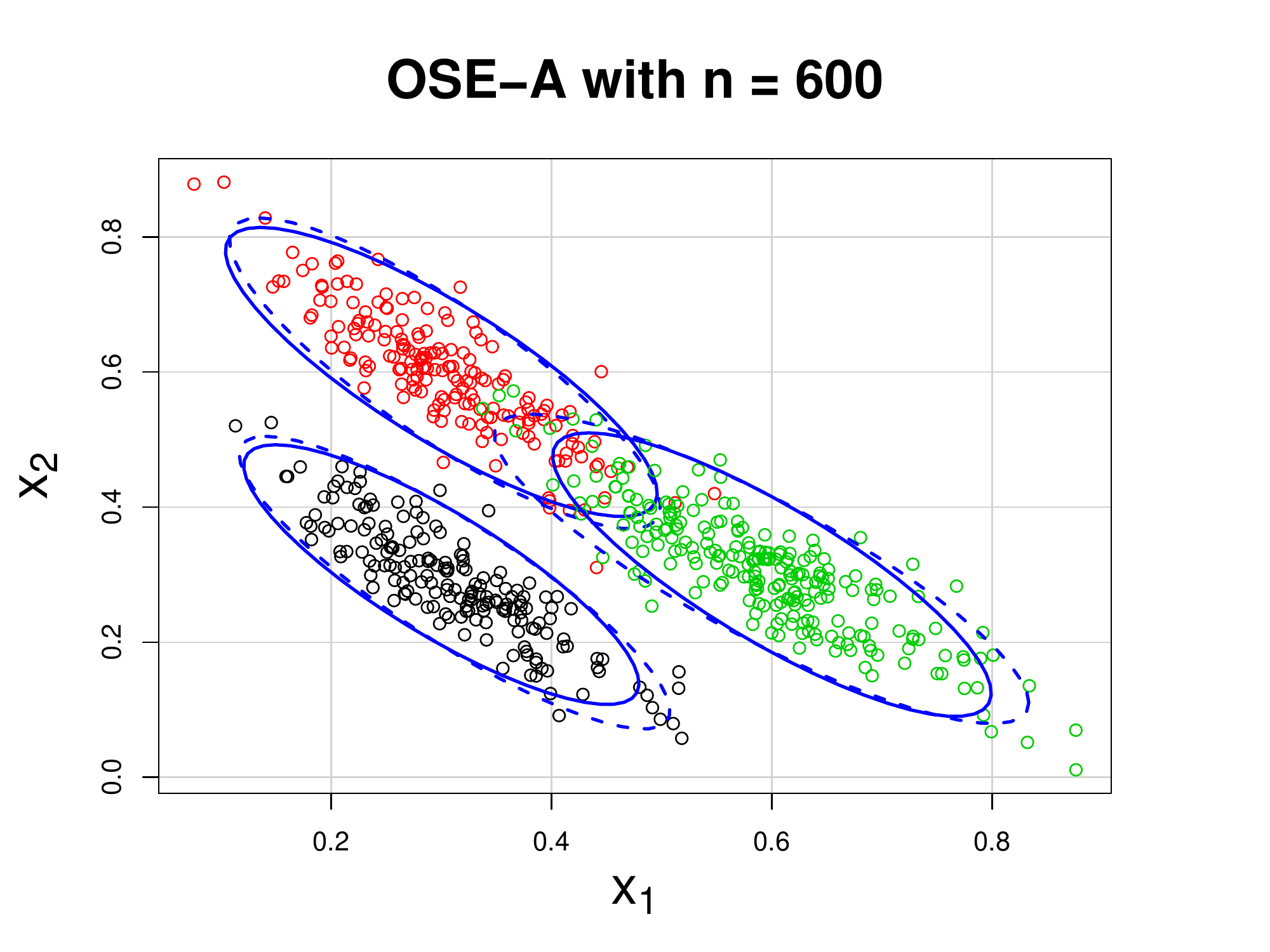}
			&
			\includegraphics[width=.43\textwidth]{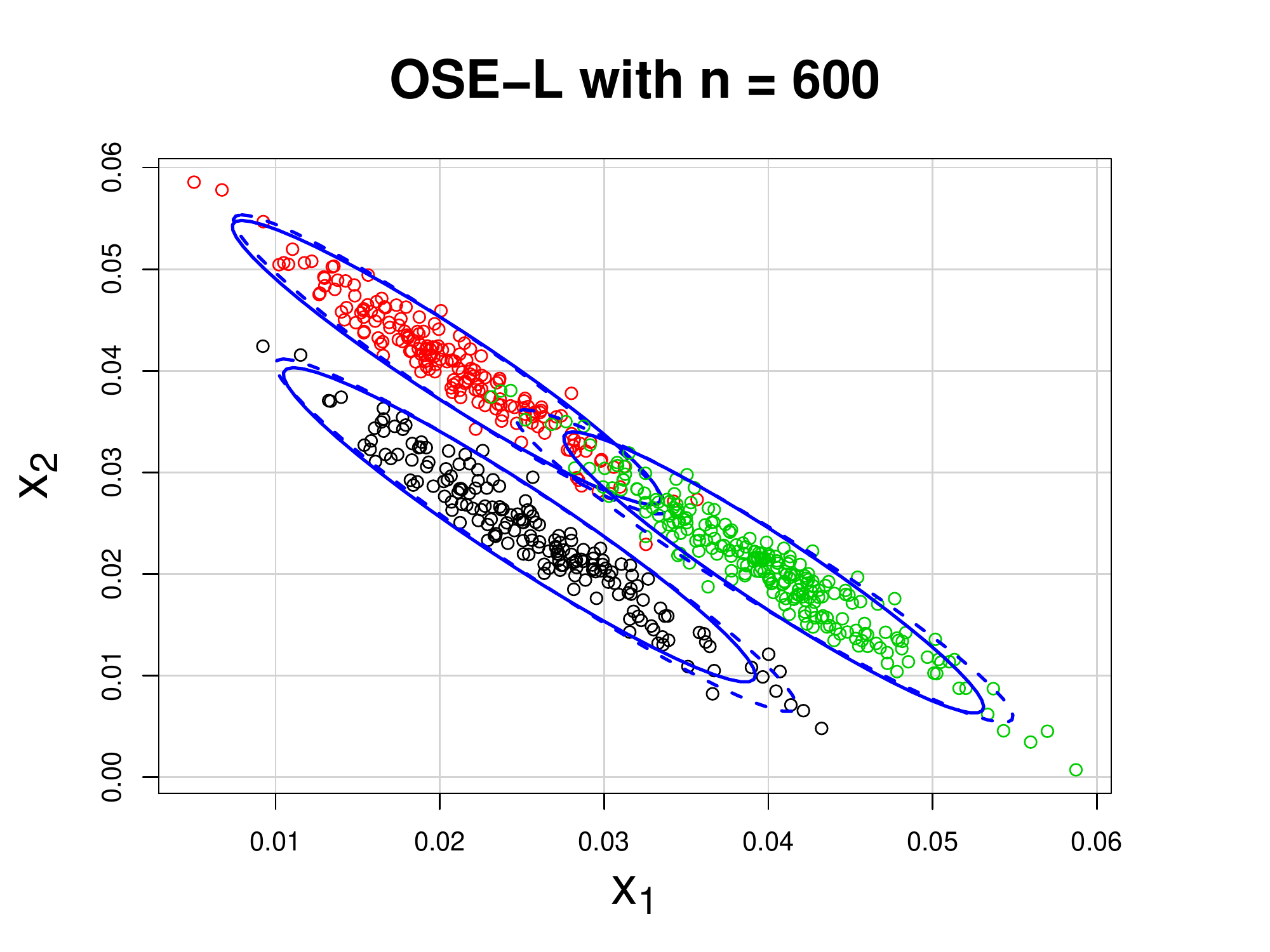}\\
			OSE-A with $n = 600$ & OSE-L with $n = 600$ \\
			\includegraphics[width=.43\textwidth]{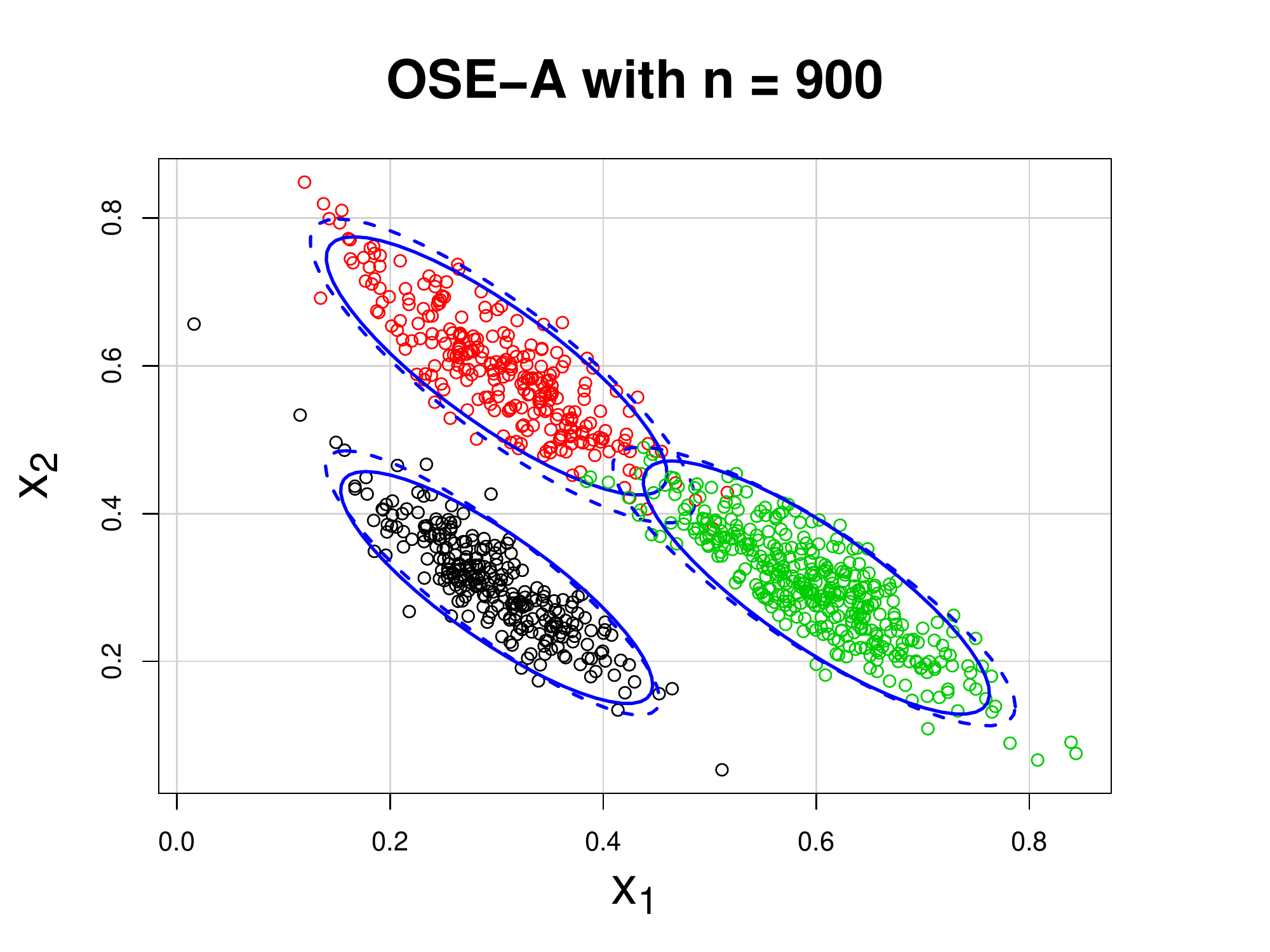}
			&
			\includegraphics[width=.43\textwidth]{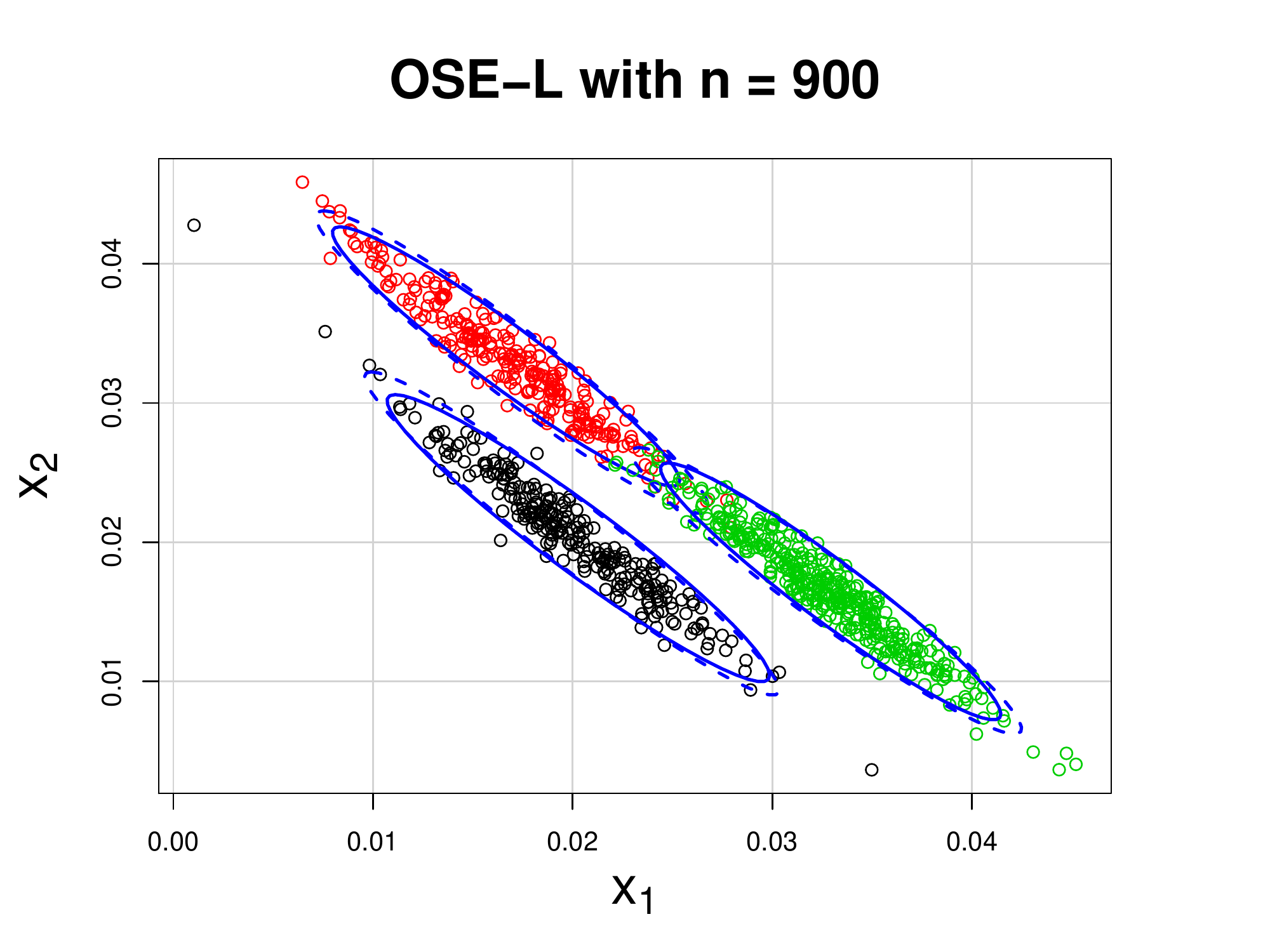}
			\\
			OSE-A with $n = 900$ & OSE-L with $n = 900$ \\
			\includegraphics[width=.43\textwidth]{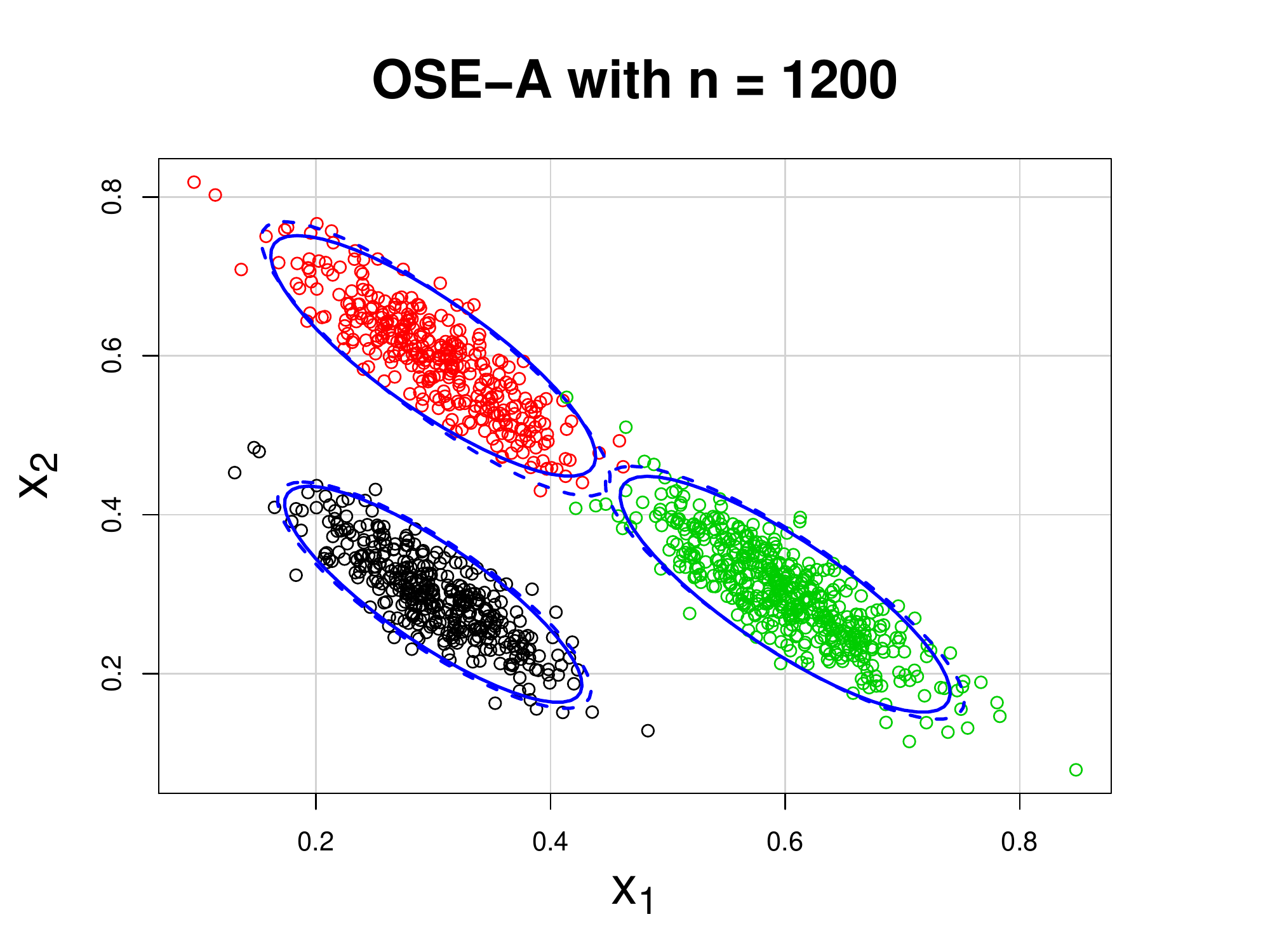}
			&
			\includegraphics[width=.43\textwidth]{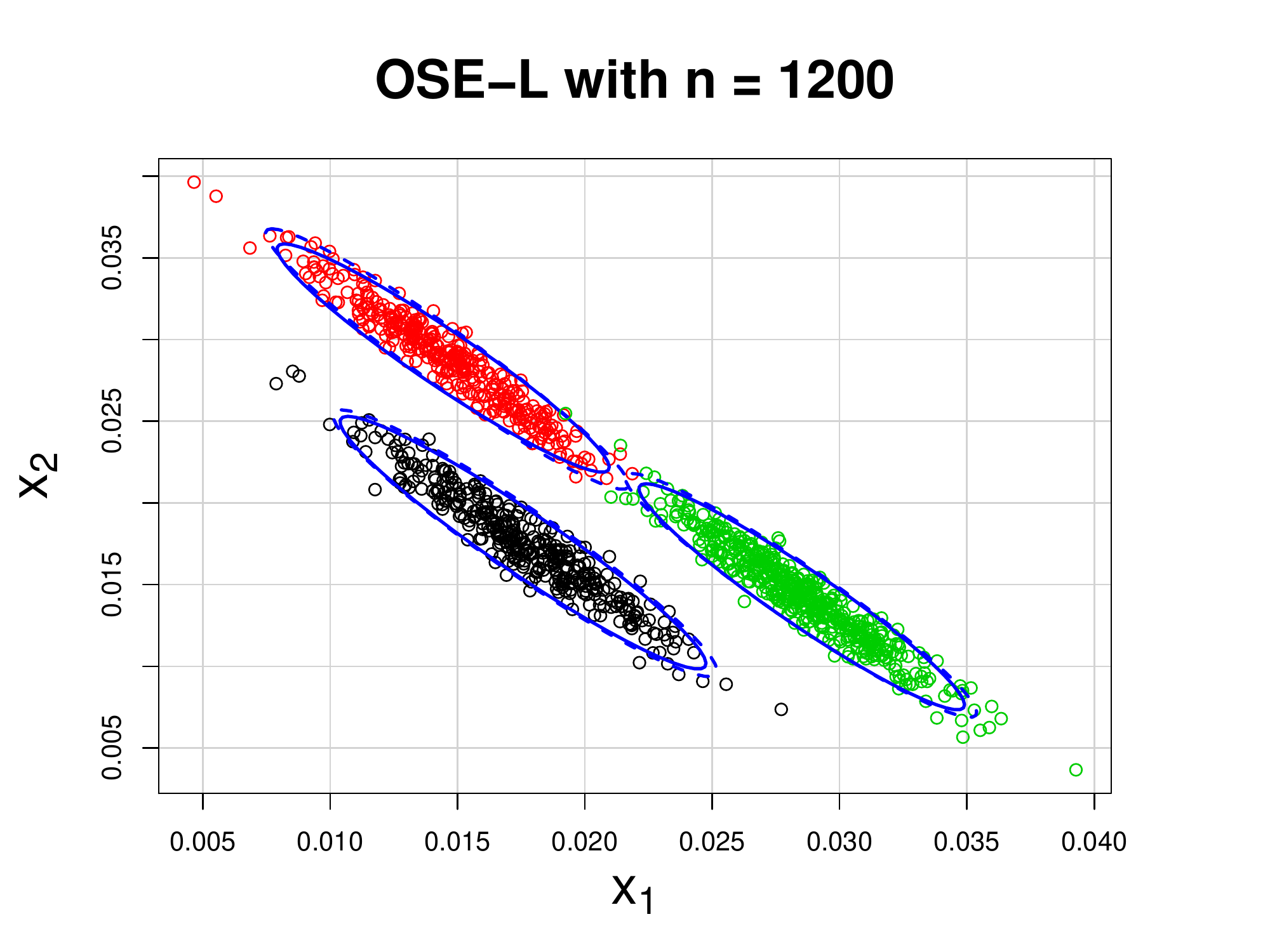}
			\\
			OSE-A with $n = 1200$ & OSE-L with $n = 1200$ \\
		\end{tabular}
	\end{center}
	\caption{Scatter plots of the OSE-A and OSE-L in the three-block stochastic block model example with $n$ vertices in Section \ref{sub:a_simulated_example}, with $n \in\{600, 900, 1200\}$. The scatter points are colored according to the cluster assignment of the corresponding vertices. For each specific cluster, the $95\%$ empirical confidence ellipses are displayed by the dashed lines, along with the $95\%$ asymptotic confidence ellipses drawn using the solid lines, as  provided by Theorem \ref{thm:convergence_OS} and Theorem \ref{thm:convergence_OS_Laplacian}. }
	\label{fig:SBM_K3_OSEA_scatter}
\end{figure}


\subsection{Wikipedia Graph data} 
\label{sub:wikipedia_graph_data}

We finally investigate the performance of the proposed one-step procedure to a real-world Wikipedia graph dataset, which is available at \url{http://www.cis.jhu.edu/~parky/Data/data.html}. The Wikipedia graph dataset consists of an adjacency matrix among $n = 1382$ Wikipedia articles that are within two hyperlinks of the article ``Algebraic Geometry'', and these articles are further manually labeled according to one of the following $6$ descriptions: People, Places, Dates, Things, Math, and Category. To determine a suitable embedding dimension $d$ for the random dot product graph  model, we follow the ad-hoc approach of \cite{ZHU2006918} and computes
\[
\widehat d = \argmax_{d = 1,2,\ldots,q}\left\{\sum_{k = 1}^d\log f(\sigma_k(\bA);\widehat\mu_1,\widehat\sigma^2) + \sum_{k = d + 1}^q\log f(\sigma_k(\bA);\widehat\mu_2,\widehat\sigma^2)\right\},
\]
where $f(x;\mu,\sigma^2) = (2\pi\sigma^2)^{-1/2}\exp\left\{-(x-\mu)^2/(2\sigma^2)\right\}$ is the normal density with mean $\mu$ and variance $\sigma^2$, 
\[
\mu_1 = \frac{1}{d}\sum_{k = 1}^d\sigma_k(\bA), \quad\mu_2 = \frac{1}{p - d}\sum_{k = d + 1}^p\sigma_k(\bA),\quad
\widehat\sigma^2 = \frac{(d - 1)s_1^2 + (p - d - 1)s_2^2}{p - 2},
\]
$s_1^2$, $s_2^2$ are the sample variances of $\{\sigma_k(\bA)\}_{k =1}^d$ and $\{\sigma_k(\bA)\}_{k = d + 1}^q$, respectively, and $q$ is an upper bound for the embedding dimension. Here we select $q = 50$ as a conservative upper bound, resulting in $\widehat d = 11$. 

We next compute the ASE, the LSE, the OSE-A, and the OSE-L, with the embedding dimension $d = 11$, and then  apply the GMM-based clustering algorithm to these estimates, with the number of clusters being $6$. 
We next compare the similarity between the manually assigned $6$ class labels and these clustering results by computing the respective Rand indices, which are tabulated in Table \ref{table:Wikidata_RI}. The results show that 
 the one-step procedure for the population LSE outperforms the rest of the competitors, as it provides the clustering result that is most similar to the original class label assignment among the four methods. 
\begin{table}[htbp]
  \caption{Wikipedia Graph Data: Rand indices of the GMM-based clustering algorithm applied to the ASE, the LSE, the OSE-A, and the OSE-L, respectively, with the number of clusters being $6$, in comparison with the corresponding manual labels.}
  \centering{%
  \begin{tabular}{c c c c c}
    \hline\hline
    Method & ASE & LSE & OSE-A & OSE-L\\
    \hline
    Rand Index & 0.7429 & 0.7350 & 0.7413 & {\bf 0.7538}\\
    \hline\hline
  \end{tabular}%
  }
  \label{table:Wikidata_RI}
\end{table}%
\begin{table}[htbp]
  \caption{Wikipedia Graph Data: Rand indices of the GMM-based clustering algorithm applied to the ASE, the LSE, the OSE-A, and the OSE-L, respectively, with the number of clusters being $2$, in comparison with the corresponding one-versus-all manual labels for the class ``Dates''.}
  \centering{%
  \begin{tabular}{c c c c c}
    \hline\hline
    Method & ASE & LSE & OSE-A & OSE-L\\
    \hline
    Rand Index & 0.5289 & 0.5097 & {\bf 0.5432} & 0.5313 \\
    \hline\hline
  \end{tabular}%
  }
  \label{table:Wikidata_RI_Date}
\end{table}%
\begin{figure}[htbp]
  \centerline{\includegraphics[width=1\textwidth]{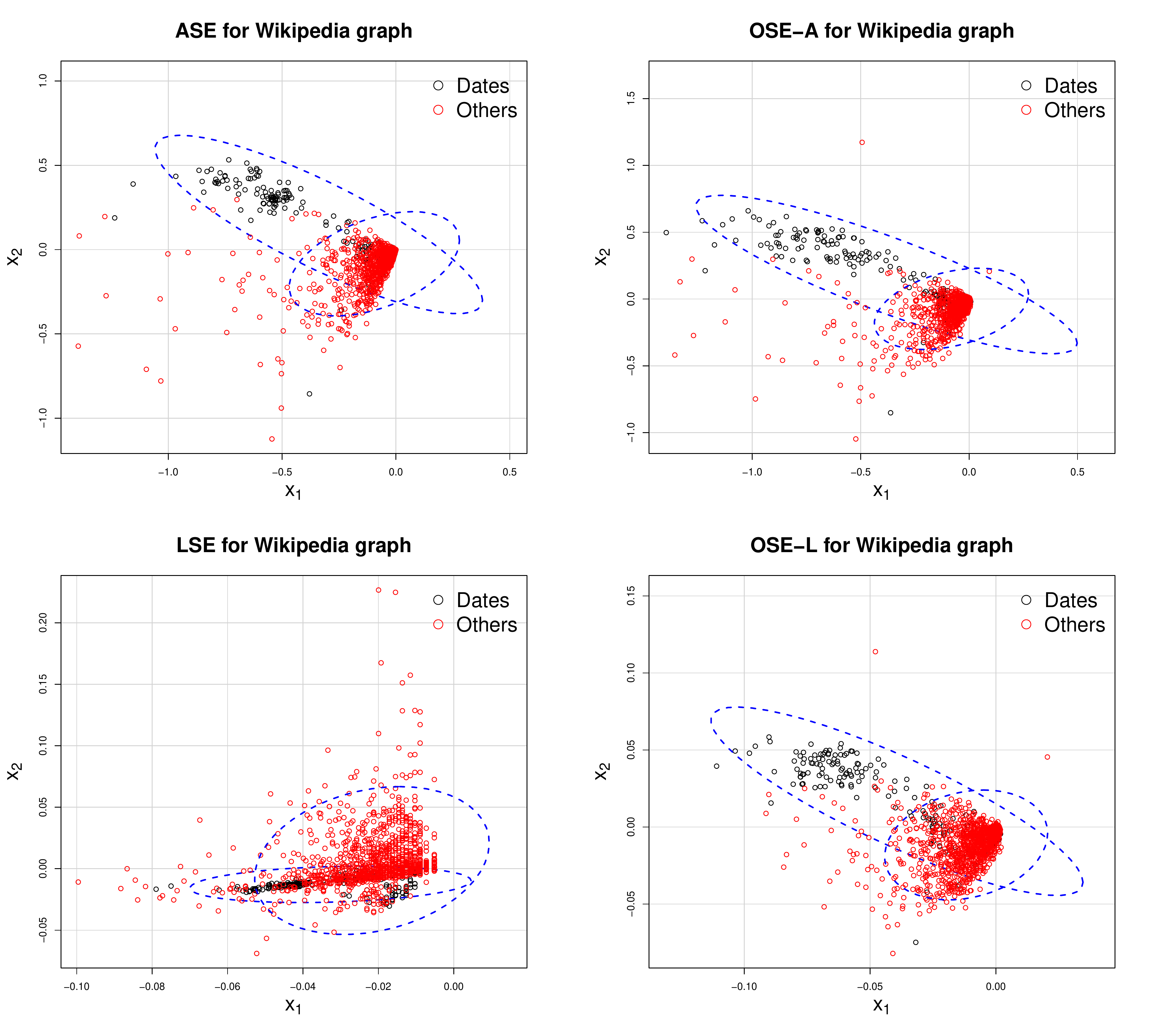}}
  \caption{Wikipedia graph data: The scatter plots of the first-versus-second dimension of the four estimates. The scatter points are colored according to whether the articles are in the class ``Dates'' or the others. The $95\%$ empirical cluster-specific confidence ellipses are displayed by the dashed lines. }
  \label{fig:Wikidata_embedding_plot}
\end{figure}

Besides evaluating the performance of the overall clustering for the $6$ manually-assigned labels, we also focus on the comparison of the article class ``Dates'' against the rest of the articles specifically. We apply the GMM-based clustering algorithm to the aforementioned four estimates again, but with the number of clusters being $2$, and tabulate the Rand indices in Table \ref{table:Wikidata_RI_Date}. We can see that the proposed  one-step procedure improves the clustering accuracy 
as well when we focus on the comparison between the article class ``Dates'' against the rest of the labels. The scatter plots of the first-versus-second dimension of the four estimates are visualized in Figure \ref{fig:Wikidata_embedding_plot}, along with the cluster-specific $95\%$ empirical confidence ellipses in dashed lines. 



\section{Discussion} 
\label{sec:discussion}


In the context of stochastic block models, \cite{10.5555/3122009.3153016} proposed a vertex clustering approach that improves the solution provided by the ASE and/or the LSE. The algorithm in \cite{10.5555/3122009.3153016} starts from the clustering solution of the ASE/LSE and then refines the cluster assignment of each vertex through the maximization of a penalized Bernoulli likelihood function, where the cluster memberships of the rest of the vertices are fixed at their most recent values. This approach is similar to our one-step procedure for estimating the latent positions in spirit, as both methods are implemented in a vertex-by-vertex optimization fashion with a warm start solution (\emph{i.e.}, the ASE/LSE or the cluster assignment given by them). Our method differs from the method of \cite{10.5555/3122009.3153016} in that the proposed one-step procedure aims at maximizing the Bernoulli likelihood function with regard to the continuous-valued latent positions and takes the gradient information of the likelihood function into account, whereas \cite{10.5555/3122009.3153016} focus on estimating  cluster memberships of vertices, and no gradient information is available due to the discrete nature of the variables of interest. 

We assume that the embedding dimension $d$ for the random dot product graph is known throughout the paper. 
The proposed one-step procedure is also valid when the true dimension $d$ for the underlying sampling model is unknown. In this case, the method proceeds by first finding the ASE into $\mathbb{R}^{d'}$ for some $d'\geq1$ and $d' < d$ (\emph{i.e.}, when the dimension is under-estimated)
and then computing the one-step estimator based on $d'$.  Our Theorem \ref{thm:convergence_OS} and Theorem \ref{thm:convergence_OS_Laplacian} still hold and can be easily proved as suggested by \cite{tang2018}. 
On the other hand, leveraging Bayesian methods when the dimension $d$ is unknown is a promising future direction in light of the recent progress in Bayesian theory and methods for low-rank matrix models with undetermined rank \citep{doi:10.1093/biomet/asr013,doi:10.1080/01621459.2015.1100620} and network models \citep{doi:10.1111/rssb.12233,xie2019optimal,doi:10.1080/01621459.2018.1458618}.

We have shown that the one-step procedure produces an estimator enjoying fascinating asymptotic properties both globally for all vertices and locally for each vertex. Nevertheless, for problems with relatively small sample sizes, we found in simulation examples that the one-step estimators do not necessarily provide us with better numerical results compared to the classical adjacency/Laplacian spectral embedding. 
Since the one-step procedure only implements a single iteration of the Newton-Raphson algorithm with the observed Hessian matrix replaced by the negative Fisher information matrix, we  hope to develop an iterative algorithm for finding a local maximum likelihood estimator by repeating the one-step procedure multiple times until convergence. Such an iterative algorithm can be implemented in conjunction with the regularization of the Fisher information matrix and backtracking procedure for finding suitable step sizes to achieve faster convergence \citep{nocedal2006numerical}. Furthermore, developing a scalable version of such an algorithm will be highly desirable in the presence of big data and extremely large networks. It will also be useful to explore the statistical properties of the estimator obtained by the iterative algorithm, and to establish its theoretical guarantee. We defer these research topics to future work.


\section*{Acknowledgements}
The work of Xu was supported by NSF 1918854 and NSF 1940107.


\bigskip

\counterwithin{lemma}{section}
\counterwithin{theorem}{section}


\begin{appendices}

\section{List of Notations} 
\label{sec:list_of_notations}
In Table \ref{table:notations_Supplement} we prepare a comprehensive lists of notations that are repeatedly used throughout the paper. 
\begin{longtable}[t]{ c  l }
\caption{Table of Notations} \\
\hline\hline\\[-0.2cm]
\endfirsthead
\caption[]{Table of Notations (continued)} \\
\hline\hline\\[-0.2cm]
\endhead
\hline\hline
\endfoot
  \label{table:notations_Supplement} 

    $\eye_d$ & $d\times d$ identity matrix \\[0.2cm]

      $\zero$ & The Eulidean vector of all zeros in its coordinates\\[0.2cm]

      $\one$ & The Eulidean vector of all ones in its coordinates\\[0.2cm]

      $[n]$ & $[n] = \{1,2,\ldots,n\}$; In this work $[n]$ may denote the set of all vertices. \\[0.2cm]

      $a\lesssim b$ & $a\lesssim b$ if $a\leq Cb$ for some constant $C > 0$ \\[0.2cm]

      $a\gtrsim b$ & $a\gtrsim b$ if $a\geq Cb$ for some constant $C > 0$\\[0.2cm]

      $a\vee b$ & $a\vee b = \max(a, b)$ \\[0.2cm]

      $a\wedge b$ & $a\wedge b = \min(a, b)$ \\[0.2cm]

      $a\asymp b$ & $a\asymp b$ if $a\lesssim b$ and $a\gtrsim b$ \\[0.2cm]

      $\mathbb{O}(n, d)$ & $\mathbb{O}(n, d) = \{\bU\in\mathbb{R}^{n\times d}:\bU\transpose{}\bU = \eye_d\}$, where $n\geq d$\\[0.2cm]

      $\mathbb{O}(d)$ & $\mathbb{O}(d) = \mathbb{O}(d, d)$\\[0.2cm]

      \multirow{2}{*}{$\bx\leq \by$} & For vectors $\bx = [x_1,\ldots,x_d]\transpose$ and $\by = [y_1,\ldots,y_d]\transpose$ in $\mathbb{R}^d$,\\
      &the inequality $\bx\leq \by$ means that $x_k \leq y_k$ for all $k = 1,2,\ldots,d$\\[0.2cm]

      $\|\bx\|$ & For a vector $\bx = [x_1,\ldots,x_d]\transpose\in\mathbb{R}^d$, $\|\bx\|_2 = (\sum_kx_k^2)^{1/2}$\\[0.2cm]

      $\bSigma_1\preceq \bSigma_2$ & $\bSigma_2 - \bSigma_1$ is positive semidefinite\\[0.2cm]

      $\bSigma_1\succeq \bSigma_2$ & $\bSigma_1 - \bSigma_2$ is positive semidefinite\\[0.2cm]

      $\bSigma_1\prec   \bSigma_2$ & $\bSigma_2 - \bSigma_1$ is positive definite\\[0.2cm]

      $\bSigma_1\succ   \bSigma_2$ & $\bSigma_1 - \bSigma_2$ is positive definite\\[0.2cm]

      $\sigma_k(\bZ)$ & $k$th largest singular value of a matrix $\bZ$\\[0.2cm]

      $\|\bZ\|_2$ & Spectral norm of a matrix $\bZ$: $\|\bZ\|_2 = \sigma_1(\bZ)$\\[0.2cm]

      $\|\bZ\|_{\mathrm{F}}$ & Frobenius norm of a matrix $\bZ = [Z_{ik}]_{n\times d}$: $\|\bZ\|_{\mathrm{F}} = (\sum_{i,k}Z_{ik}^2)^{1/2}$\\[0.2cm]

      $\|\bZ\|_{2\to\infty}$ & Two-to-infinity norm of a matrix $\bZ = [Z_{ik}]_{n\times d}$: $\|\bZ\|_{2\to\infty} = \max_i(\sum_kz_{ik}^2)^{1/2}$\\[0.2cm]

      \multirow{2}{*}{$\mathrm{diag}(\bz)$} & For a vector $\bz = [x_1,\ldots,x_n]\transpose\in\mathbb{R}^n$, the matrix $\mathrm{diag}(\bx)$ is a diagonal matrix\\
      & with the $i$th diagonal entries being $z_i$, $i = 1,\ldots,n$\\[0.2cm]

      $\calX$ & Latent space $\calX = \{\bx = [x_1,\ldots,x_d]\in\mathbb{R}^d: x_1,\ldots,x_d > 0, \|\bx\| < 1\}$\\[0.2cm]

      $\calX(\delta)$ & A subset of $\calX$ such that $\bx,\bu\in\calX(\delta)$ implies $\bx\transpose\bu\in[\delta, 1 - \delta]$ for some $\delta > 0$\\[0.3cm]

      \multirow{2}{*}{$\calX^n$}
                & $n$-fold cartesion product of $\calX$: \\
                & $\calX^n = \{\bX = [\bx_1,\ldots,\bx_n]\transpose\in\mathbb{R}^{n\times d}:\bx_1,\ldots,\bx_n\in\calX\}$\\[0.2cm]

      $n$ & Number of vertices \\[0.2cm]

      $d$ & Dimension of underlying latent positions \\[0.2cm]

      $\bX$ & Latent position matrix $\bX = [\bx_1,\ldots,\bx_n]\transpose\in\calX^n$\\[0.2cm]

      $\bx_i$ & The $i$th row of $\bX$ as a column vector in $\mathbb{R}^d$\\[0.2cm]

      $\bA$ & Adjacency matrix drawn from a random dot product graph\\[0.2cm]

      $\rho_n$ & Sparsity factor $\rho_n\in(0, 1]$ such that $A_{ij}\sim \mathrm{Bernoulli}(\bx_i\transpose\bx_j)$\\[0.2cm]

      $\bX_0$ & True latent position matrix $\bX_0 = [\bx_{01},\ldots,\bx_{0n}]\transpose\in \calX^n$\\[0.2cm]

      $\bx_{0i}$ & The $i$th row of $\bX_0$ as a column vector in $\mathbb{R}^d$\\[0.2cm]
      
      $F_n$ & Empirical distribution function of $(\bx_{0i})_{i = 1}^n$: $F_n(\bx) = (1/n)\sum_{i = 1}^n\mathbbm{\bx_i\leq \bx}$\\[0.2cm]

      $F$ & The cumulative distribution function such that \eqref{eqn:strong_convergence_measure} holds
      \\[0.2cm]

      $\bW$ & A generic orthogonal matrix in $\mathbb{O}(d)$\\[0.2cm]

      $\bDelta$ & Second moment matrix of $F$: $\bDelta = \int_\calX \bx\bx\transpose F(\mathrm{d}\bx)$\\[0.2cm]

      \multirow{2}{*}{$\tau$}
             & Cluster assignment function $\tau:[n]\to [K]$ in a $K$-block \\
             & stochastic block model\\[0.2cm]

      \multirow{2}{*}{$\bB$}
             & A symmetric $K\times K$ block probability matrix\\
             & for a $K$-block stochastic block model\\[0.2cm]
      $\widehat{\bX}^{(\mathrm{ASE})}$ & The adjacency spectral embedding of $\bA$ into $\mathbb{R}^d$ \\[0.2cm]

      $\bnu_k$ & A unique latent position in a positive semidefinite stochastic block model \\[0.2cm]

      $\widehat{\bx}_i^{\mathrm{(ASE)}}$ & The $i$th row of $\widehat{\bX}^{(\mathrm{ASE})}$ as a column vector in $\mathbb{R}^d$\\[0.2cm]

      $\rho$ & $\rho = \lim_{n\to\infty}\rho_n$; In this work we assume $\rho\in\{0, 1\}$. \\[0.2cm]


      $\bSigma(\bx)$ & 
      $\displaystyle\bDelta^{-1}\int_\calX \{\bx\transpose\bx_1(1 - \rho\bx\transpose\bx_1)\}\bx_1\bx_1\transpose F(\mathrm{d}\bx_1)\bDelta^{-1}$, where $\bx\in\calX(\delta)$
      \\[0.4cm]

            $\bG(\bx)$ & 
      $\displaystyle \int_\calX \left\{\frac{\bx_1\bx_1\transpose}{\bx\transpose\bx_1(1 - \rho\bx\transpose\bx_1)}\right\} F(\mathrm{d}\bx_1)$, where $\bx\in\calX(\delta)$
      \\[0.4cm]

      $\bDelta_n$ & $\displaystyle \int_\calX \bx_1\bx_1\transpose F_n(\mathrm{d}\bx)$, and can be alternatively expressed as $\displaystyle (1/n)\bX_0\transpose\bX_0$.
      \\[0.4cm]

      $\bSigma_n(\bx)$ & $\displaystyle \int_\calX \{\bx\transpose\bx_1(1 - \rho\bx\transpose\bx_1)\}\bx_1\bx_1\transpose F_n(\mathrm{d}\bx)$, where $\bx\in\calX$\\[0.4cm]

      $\bG_n(\bx)$ & $\displaystyle \int_\calX \left\{\frac{\bx_1\bx_1\transpose}{\bx\transpose\bx_1(1 - \rho\bx\transpose\bx_1)}\right\} F_n(\mathrm{d}\bx)$, where $\bx\in\calX(\delta)$.\\[0.4cm]

      $\bPsi_n(\bx)$ & Score function
      $\displaystyle 
      \bPsi_n(\bx) = \frac{1}{n}\sum_{j\neq i}^n\frac{(A_{ij} - \bx\transpose\bx_{0j})\bx_{0j}}{\bx\transpose\bx_{0j}(1 - \bx\transpose\bx_{0j})}
      $\\[0.2cm]

      $\widetilde{\bx}_i$ & Initial estimator of $\bx_i$ for the one-step procedures \eqref{eqn:One_step_onedimensional} and \eqref{eqn:one_step_estimator}\\[0.2cm]

      $\widetilde\bX$ & Initial estimator of $\bX$ for the one-step procedure in \eqref{eqn:one_step_estimator}\\[0.2cm]

     $\widehat\bX$ & The one-step estimator for $\bX$ defined by \eqref{eqn:one_step_estimator}\\[0.2cm]

     $\widehat{\bx}_i$ & The $i$th row of $\widehat{\bX}$ as a column vector in $\mathbb{R}^d$ \\[0.2cm]

     \multirow{2}{*}{$\calL(\bM)$} & The normalized Laplacian of a square matrix $\bM$ defined by \\
     &$\calL(\bM) = (\mathrm{diag}(\bM\one))^{-1/2}\bM(\mathrm{diag}(\bM\one))^{-1/2}$\\[0.2cm]

     $\bY$ & The population Laplacian spectral embedding:
     $\bY = \bY(\bX) = (\mathrm{diag}(\bX\one))^{-1/2}\bX$\\[0.2cm]

     $\by_i$ & The $i$th row of $\bY$ as a column vector in $\mathbb{R}^d$\\[0.2cm]

     $\bY_0$ & The true value of $\bY$ given by $\bY_0 = (\mathrm{diag}(\bX_0\bX_0\one))^{-1/2}\bX_0$\\[0.2cm]

     $\by_{0i}$ & The $i$th row of $\bY_0$ as a column vector in $\mathbb{R}^d$\\[0.2cm]

     $\breve{\bX}$ & The (sample) Laplacian spectral embedding of $\bA$ into $\mathbb{R}^d$\\[0.2cm]

     $\bmu$ & $\displaystyle \int_\calX \bx F(\mathrm{d}\bx)$\\[0.4cm]

     $\widetilde\bDelta$ & $\displaystyle\int_\calX \frac{\bx\bx\transpose}{\bx\transpose\bmu}F(\mathrm{d}\bx)$\\[0.4cm]

     $\widetilde\bSigma(\bx)$ & $\displaystyle
     \left(\widetilde{\bDelta}^{-1} - \frac{\bx\bmu\transpose}{2\bx\transpose\bmu}\right)
     \int_\calX\frac{\bx\transpose\bx_1(1  - \rho\bx\transpose\bx_1)\bx_1\bx_1\transpose F(\mathrm{d}\bx_1)}{\bmu\transpose\bx(\bmu\transpose\bx_1)^2}
     \left(\widetilde{\bDelta}^{-1} - \frac{\bx\bmu\transpose}{2\bx\transpose\bmu}\right)\transpose
     $, $\bx\in\calX(\delta)$
      \\[0.4cm]

    \multirow{2}{*}{$\widehat{\bY}$} & The one-step estimator for the population LSE: $\widehat{\bY} = \{\mathrm{diag}(\widehat\bX\transpose\widetilde\bX\one)\}^{-1/2}\widehat\bX$, \\
     &where $\widehat{\bX}$ is the one-step estimator for $\bX$, and $\widetilde\bX$ is an initial estimator for $\bX$\\[0.2cm]

     $\widehat{\by}_i$ & The $i$th row of $\widehat{\bY}$ as a column vector in
      $\mathbb{R}^d$\\[0.2cm]   

            $\bmu_n$ & $\displaystyle \int_\calX \bx F_n(\mathrm{d}\bx)$, can be equivalently expressed as $(1/n)\sum_{i = 1}^n\bx_{0i}$\\[0.4cm]   

      $\widetilde{\bG}(\bx)$ & $\displaystyle
      \frac{1}{\bmu\transpose\bx}
      \left(\eye_d - \frac{\bx\bmu\transpose}{2\bx\transpose\bmu}\right)\bG^{-1}(\bx)
      \left(\eye_d - \frac{\bx\bmu\transpose}{2\bx\transpose\bmu}\right)\transpose
      $, $\bx\in\calX(\delta)$\\[0.4cm]

      $\widetilde\bDelta_n$ & $\displaystyle\int_\calX\left(\frac{\bx_1\bx_1\transpose}{\bmu_n\transpose\bx_1}\right)F_n(\mathrm{d}\bx_1)$, can be equivalently expressed as
      $\displaystyle \frac{1}{n}\sum_{j = 1}^n\frac{\bx_{0j}\bx_{0j}\transpose}{\bmu_n\transpose\bx_{0j}}$\\[0.4cm]

      $\widetilde\bSigma_n(\bx)$
      & $\displaystyle
      \left(\widetilde{\bDelta}_n^{-1} - \frac{\bx\bmu_n\transpose}{2\bmu_n\transpose\bx}\right)
      \frac{1}{n}\sum_{j = 1}^n\frac{\bx\transpose\bx_{0j}(1 - \rho_n\bx\transpose\bx_{0j})\bx_{0j}\bx_{0j}\transpose}{(\bmu\transpose\bx)(\bmu\transpose\bx_{0j})^2}
      \left(\widetilde{\bDelta}_n^{-1} - \frac{\bx\bmu_n\transpose}{2\bmu_n\transpose\bx}\right)\transpose,\bx\in\calX(\delta)
      $\\[0.6cm]

      \multirow{2}{*}{$C(F_k, F_l)$}
      & Chernoff information between distributions $F_k$ and $F_l$, defined by\\
      &       $\displaystyle 
      \sup_{t\in(0, 1)}\left\{-\log\int f_k(\bx)^t f_l(\bx)^{1 - t}\mathrm{d}\bx\right\}
      $, where $F_k(\bx) = f_k(\bx)\mathrm{d}\bx$, $F_l(\bx) = f_l(\bx)\mathrm{d}\bx$\\[0.4cm]

      $\lambda_k(\mathbf{C})$ & The $k$th largest eigenvalue of a (square) positive semidefinite matrix $\mathbf{C}$ \\[0.2cm]
      $[\bz]_k$ & The $k$th coordinate of a Euclidean vector $\bz$\\[0.2cm]

      \multirow{2}{*}{$C, C_1,C_2,C_c, c, c_1, c_2,...$} & Generic constants that may change from line to line but is independent of the\\
      & quantities of interest\\[0.2cm]

      $\bV_n(\bx)$ & $(1/n)\sum_{j\neq i}\{\bx\transpose\bx_{0j}(1 - \rho_n\bx\transpose\bx_{0j})\}\bx_{0j}\bx_{0j}\transpose$\\[0.2cm]

      $\bV(\bx)$ & $\displaystyle
      \int_\calX\{\bx\transpose{}\bx_1(1 - \rho\bx\transpose\bx_1)\}\bx_1\bx_1\transpose F(\mathrm{d}\bx_1)
      $\\[0.4cm]

      $\bP_0$ & $\rho_n\bX_0\bX_0\transpose$\\[0.2cm]

      $\bE$ & $\bA - \rho_n\bX_0\bX_0\transpose = \bA - \bP_0$\\[0.2cm]

      $\bH$ & $(n\rho_n)^{-1/2}(\bA - \rho_n\bX_0\bX_0)\transpose = (n\rho_n)^{-1/2}\bE$\\[0.2cm]

      $H_{ij}$ & The $(i,j)$th element of $\bH$\\[0.2cm]

      \multirow{2}{*}{$\be_i$} & The standard basis vector in $\mathbb{R}^n$ with all $0$ in coordinates except the $i$th\\
      & coordinate being $1$\\[0.2cm]

      \multirow{2}{*}{$\bU_\bP\bS_\bP\bU_\bP\transpose$} & The spectral decomposition of $\bP_0$, where $\bS_\bP = \mathrm{diag}\{\lambda_1(\bP_0),\ldots,\lambda_d(\bP_0)\}$,\\
      & and $\bU_\bP\in\mathbb{O}(n, d)$\\[0.2cm]

      $\bu_{0k}$ & The $k$th column of $\bU_\bP$, where $k\in[d]$\\[0.2cm]

      \multirow{2}{*}{$\sum_{i = 1}^n\widehat{\lambda}_i\widehat{\bu}_i\widehat{\bu}_i\transpose$}
      & The spectral decomposition of $\bA$, where $|\widehat\lambda_1|\geq\ldots\geq|\widehat\lambda_n|\geq 0$, and\\
      & $\widehat{\bu}_i\transpose\widehat{\bu}_j = \mathbbm{1}(i = j)$\\[0.2cm]

      $\bU_\bA$ & $[\widehat{\bu}_1,\ldots,\widehat{\bu}_d]$\\[0.2cm]

      $\bS_\bA$ & $\mathrm{diag}\{|\widehat\lambda_1|,\ldots,|\widehat\lambda_d|\}$\\[0.2cm]

      \multirow{2}{*}{$\nu,\nu_1,\nu_2$} & A generic positive constant that may change from line to line but does not \\
      & depend on $n$\\[0.2cm]

      \multirow{2}{*}{$\bW^*$} & The product of the left and right singular factor of $\bU_\bP\transpose\bU_\bA$: 
      Let $\bW_1\bS_{\bP\bA}\bW_2\transpose$\\
      & be the singular value decomposition of $\bU_\bP\transpose\bU_\bA$. 
      Then $\bW^* = \bW_1\bW_2\transpose$\\[0.2cm]

      $\bW_\bX$ & The orthogonal matrix $\bW_\bX\in\mathbb{O}(d)$ such that $\bX_0 = \bU_\bP\bS_\bP^{1/2}\bW_\bX$\\[0.2cm]

      $M_n(\bx)$ & $\displaystyle
      \frac{1}{n}\sum_{j\neq i}^n\{A_{ij}\log(\bx\transpose\bx_{0j}) + (1 - A_{ij})\log(1 - \bx\transpose\bx_{0j})\}     
      $, $\bx\in\calX(\delta)$\\[0.4cm]

            $\widetilde{\bV}_n(\bx)$ & $\displaystyle
      \frac{1}{n}\sum_{i = 1}^n\frac{\bx_{0i}\transpose\bx(1 - \rho_n\bx_{0i}\transpose\bx)\bx_{0i}\bx_{0i}\transpose}{(\bx_{0i}\transpose\bmu_n)^2}
      $\\[0.4cm]

    $\widetilde{\bV}(\bx)$ & $\displaystyle
      \int_\calX\frac{\bx_{1}\transpose\bx(1 - \rho_n\bx_{1}\transpose\bx)\bx_{1}\bx_{1}\transpose}{(\bx_{1}\transpose\bmu_n)^2}F(\mathrm{d}\bx_1)
      $\\[0.4cm]

      $\bPi$ & A generic permutation matrix \\[0.2cm]

      $\bnu_k$ & A unique latent position in a positive semidefinite stochastic block model \\[0.2cm]

      $M(\bx)$ & $\expect_0\{M_n(\bx)\}$ \\[0.2cm]

      $\bPsi_n(\bx)$ & $\displaystyle \frac{\partial M_n}{\partial\bx}(\bx)$ \\[0.4cm]

      $\bPsi(\bx)$ & $\expect_0\{\bPsi_n(\bx)\}$ \\[0.2cm]

      $\bar{M}(\bx)$ & $\lim_{n\to\infty} M(\bx)$ (note that $M(\bx)$ depends on $n$ implicitly)\\[0.2cm]

      $\dot{\bPsi}_0$ & $\displaystyle \frac{1}{n}\sum_{j\neq i}^n\frac{\bx_{0j}\bx_{0j}\transpose}{\bx_{0i}\transpose\bx_{0j}(1 - \bx_{0i}\transpose\bx_{0j})}$\\[0.4cm]

      $\dot{\bPsi}_{n,0}$ & $\displaystyle \frac{1}{n}\sum_{j\neq i}^n\frac{\bx_{0j}\bx_{0j}\transpose}{\widetilde{\bx}_{i}\transpose\bx_{0j}(1 - \widetilde{\bx}_{i}\transpose\bx_{0j})}$\\[0.2cm]

      \multirow{2}{*}{$\vect(\bSigma)$} & the vectorization of the matrix $\bSigma$ defined to be the vector formed by stacking\\
      & the columns of $\bSigma$ consecutively\\[0.2cm]

      $H_i(\bX)$ & $\displaystyle \frac{1}{n}\sum_{j = 1}^n\frac{\bx_j\bx_j\transpose}{\bx_i\transpose\bx_j(1 - \bx_i\transpose\bx_j)}$ \\[0.4cm]

      $\bT$ & $\mathrm{diag}(\rho_n\bP_0\one) = \mathrm{diag}(\rho_n\bX_0\bX_0\transpose\one)$\\[0.2cm]

      $\bD$ & $\mathrm{diag}(\bA\one)$ \\[0.2cm]

      $\widetilde{\bE}$ & $n\rho_n\{\calL(\bA) - \calL(\bP_0)\}$, where $\calL(\bM) = (\mathrm{diag}(\bM\one))^{-1/2}\bM(\mathrm{diag}(\bM\one))^{-1/2}$\\[0.2cm]

      \multirow{2}{*}{$\sum_{i = 1}^n\widetilde{\lambda}_i(\widetilde{\bu}_\bA)_i(\widetilde{\bu}_\bA)_i\transpose$} 
      & The spectral decomposition of $\calL(\bA)$, where $|\widetilde{\lambda}_1|\geq\ldots\geq|\widetilde{\lambda}_n|$, and\\
      &$(\widetilde{\bu}_\bA)_i\transpose(\widetilde{\bu}_\bA)_j = \mathbbm{1}(i = j)$\\[0.2cm]

      \multirow{2}{*}{$\widetilde\bU_\bP\widetilde\bS_\bP\widetilde\bU_\bP$}
      & The spectral decomposition of $\calL(\bP_0)$, where $\widetilde\bU_\bP\in\mathbb{O}(n, d)$, and\\
      &$\widetilde\bS_\bP = \mathrm{diag}[\lambda_1\{\calL(\bP_0)\},\ldots,\lambda_d\{\calL(\bP_0)\}]$ \\[0.2cm]

      $\widetilde{\bu}_{0k}$ & The $k$th column of $\widetilde{\bU}_\bP$ as a vector in $\mathbb{R}^n$\\[0.2cm]

      $\widetilde{u}_{0jk}$ & The $j$th element of the $k$th column of $\widetilde\bU_\bP$\\[0.2cm]

      $\widetilde{\bU}_\bA$ & $[(\widetilde\bu_\bA)_1,\ldots,(\widetilde\bu_\bA)_2]$\\[0.2cm]

      $\widetilde{\bS}_\bA$ & $\mathrm{diag}(|\widetilde\lambda_1|,\ldots,|\widetilde\lambda_d|)$ \\[0.2cm]

      $\|\bZ\|_{\infty}$ & Infinity norm of a matrix $\bZ = [Z_{ik}]_{n\times d}$: $\|\bZ\|_{\infty} = \max_i\sum_k|z_{ik}|$\\[0.2cm]

      $\|\bz\|_\infty$ & Infinity norm of a vector $\bz = [z_1,\ldots,z_n]\transpose$: $\|\bz\|_\infty = \max_i|z_i|$\\[0.2cm]

      $[\bZ]_{*k}$ & The $k$th column of a matrix $\bZ$ as a column vector\\[0.2cm]

    $[\bZ]_{i*}$ & The $i$th row of a matrix $\bZ$ as a column vector\\[0.2cm]      

\end{longtable}

\section{Proof of Theorem \ref{thm:ASE_limit_theorem} (Limit Theorem for the ASE)} 
\label{sec:proof_of_theorem_thm:ase_limit_theorem}
Theorem \ref{thm:ASE_limit_theorem} is slightly different than those presented in \cite{tang2018} and \cite{athreya2016limit}, as the latent positions $\bx_{01},\ldots,\bx_{0n}$ are deterministic in the current setup. For comparison, we first state the limit theorem for the ASE when the latent positions are independent and identically distributed random variables in Theorem \ref{thm:Avanti_CLT} below, which is originally due to \cite{athreya2016limit}:
\begin{theorem}[Theorems 2.1 and 2.2 of \citealp{tang2018}]
\label{thm:Avanti_CLT}
Let $\bx_1,\ldots,\bx_n\iidsim F$ for some distribution $F$ supported on $\calX$, and suppose $\bA\mid \bX\sim\mathrm{RDPG}(\bX)$ with a sparsity factor $\rho_n$. Suppose either $\rho_n\equiv 1$ for all $n$ or $\rho_n\to 0$ but $(\log n)^4/(n\rho_n)\to 0$ as $n\to\infty$, and denote $\rho = \lim_{n\to\infty}\rho_n$. Let $\widehat\bX^{(\mathrm{ASE})} = [\widehat\bx_1^{(\mathrm{ASE})},\ldots,\widehat\bx_n^{(\mathrm{ASE})}]\transpose$ be the ASE defined by \eqref{eqn:ASE_least_squared_problem}. Let $\bDelta$ and $\bSigma(\bx)$ be given as in Theorem \ref{thm:ASE_limit_theorem}
and assume that $\bDelta$ and $\bSigma(\bx)$ are strictly positive definite for all $\bx\in\calX$. Then there exists a sequence of orthogonal matrices $(\bW_n)_{n = 1}^\infty\in\mathbb{R}^{d\times d}$ such that
\begin{align*}
&\|\widehat\bX^{(\mathrm{ASE})}\bW_n - \rho_n^{1/2}\bX_0\|_{\mathrm{F}}^2\overset{a.s.}{\to} 
\int_\calX \mathrm{tr}\{\bSigma(\bx)\}F(\mathrm{d}\bx),
\end{align*} 
and for any fixed index $i\in[n]$, $\sqrt{n}(\bW\transpose_n\widehat\bx_i^{(\mathrm{ASE})} - \rho_n^{1/2}\bx_{0i})$ converges to a normal mixture distribution with density
\begin{align*}
&
\int_{\calX} \frac{1}{\sqrt{\det\{2\pi\bSigma(\bx_1)\}}}\exp\left\{-\frac{1}{2}\bx\transpose\bSigma(\bx_1)^{-1}\bx\right\} F(\mathrm{d}\bx_1).
\end{align*}
\end{theorem}
We breakdown the proof into the proof of the limit\eqref{eqn:ASE_convergence} and the proof of the asymptotic normality of the rows of the ASE \eqref{eqn:ASE_normality}.

\subsection{Proof of the Limit \eqref{eqn:ASE_convergence}} 
\label{sub:proof_of_the_limit_eqn:ase_convergence}


\begin{proof}[\bf Proof of the limit \eqref{eqn:ASE_convergence}]
The proof of \eqref{eqn:ASE_convergence} is very similar to the proof given in Appendix A of \cite{tang2018} and here we only present the sketch. In the case where $\bA\sim\mathrm{RDPG}(\bX_0)$ where $\bX_0 = [\bx_{01},\ldots,\bx_{0n}]\transpose$, Theorem A.5 in \cite{doi:10.1080/10618600.2016.1193505} yields
\[
\|\widehat{\bX}^{\mathrm{(ASE)}}\bW - \rho_n^{1/2}\bX_0\|_{\mathrm{F}} = \rho_n^{-1/2}\|(\bA - \rho_n\bX_0\bX_0\transpose)\bX_0(\bX_0\transpose\bX_0)^{-1}\|_{\mathrm{F}} + O_{\prob_0}((n\rho_n)^{-1/2}).
\]
Denote $\zeta = \rho_n^{-1}\|(\bA - \rho_n\bX_0\bX_0\transpose)\bX_0(\bX_0\transpose\bX_0)^{-1}\|_{\mathrm{F}}^2$. Then Lemma A.5 in \cite{doi:10.1080/10618600.2016.1193505} further shows that $\zeta - \expect_0(\zeta)$ is asymptotically negligible. Apppendix A of \cite{tang2018} also shows that
\[
\expect_0(\zeta) = \mathrm{tr}\left[
n(\bX_0\transpose\bX_0)^{-1}\{n^{-2}\rho_n^{-1}\bX_0\transpose\expect_0\{(\bA - \rho_n\bX_0\bX_0)^2\}\bX_0\}n(\bX_0\transpose\bX_0)^{-1}
\right],
\]
where
\[
\expect_0\{(\bA - \rho_n\bX_0\bX_0)^2\} = \mathrm{diag}\left\{\sum_{j\neq i}\rho_n\bx_{01}\transpose\bx_{0j}(1 - \rho_n\bx_{01}\transpose\bx_{0j}),\ldots,\sum_{j\neq i}\rho_n\bx_{0n}\transpose\bx_{0j}(1 - \rho_n\bx_{0n}\transpose\bx_{0j})\right\}.
\]
In what follows we make use of the condition \eqref{eqn:strong_convergence_measure} to derive $\expect_0(\zeta)$. Condition \eqref{eqn:strong_convergence_measure} immediately implies that $n(\bX_0\transpose\bX_0)^{-1}\to \bDelta^{-1}$. Furthermore, 
\begin{align*}
n^{-2}\rho_n^{-1}\bX_0\transpose\expect_0\{(\bA - \rho_n\bX_0\bX_0)^2\}\bX_0
& = \frac{1}{n}\sum_{i = 1}^n\bx_{0i}\left\{\frac{1}{n}\sum_{j\neq i}^n\bx_{0i}\transpose\bx_{0j}(1 - \rho_n\bx_{0i}\transpose\bx_{0j})\right\}\bx_{0i}\transpose.
\end{align*}
Define a matrix-valued fucntion $\bV_n(\bx) = (1/n)\sum_{j\neq i}^n\{\bx\transpose\bx_{0j}(1 - \rho_n\bx\transpose\bx_{0j})\}\bx_{0j}\bx_{0j}\transpose$. Using the argument for proving Lemma \ref{lemma:uniform_convergence_G}, one can also show that $\bV_n(\bx)\to\bV(\bx)$ uniformly for all $\bx\in \bar{\calX}$ as well, where $\bar{\calX}$ is the closure of $\calX$ (which is compact), and 
\[
\bV(\bx) = \int_\calX\{\bx\transpose\bx_{1}(1 - \rho\bx\transpose\bx_{1})\}\bx_{1}\bx_{1}\transpose F(\mathrm{d}\bx_1).
\]
Since $F_n = (1/n)\sum_{j = 1}^n\delta_{\bx_{0i}}$ converges strongly to $F$, it follows that (see, for example, Exercise 3 in Section 4.4 of \citealp{chung2001course}) that
\begin{align*}
&\frac{1}{n}\sum_{i = 1}^n\bx_{0i}\left\{\frac{1}{n}\sum_{j\neq i}^n\bx_{0i}\transpose\bx_{0j}(1 - \rho_n\bx_{0i}\transpose\bx_{0j})\right\}\bx_{0i}\transpose
 = \frac{1}{n}\sum_{i = 1}^n \bV_n(\bx_i) = \int_\calX \bV_n(\bx_1)F_n(\mathrm{d}\bx_1) \to \int_\calX \bV(\bx_1)F(\mathrm{d}\bx).
\end{align*}
Hence we conclude that
\begin{align*}
\expect_0(\zeta)
 & = \mathrm{tr}\left[
n(\bX_0\transpose\bX_0)^{-1}\{n^{-2}\rho_n^{-1}\bX_0\transpose\expect_0\{(\bA - \rho_n\bX_0\bX_0)^2\}\bX_0\}n(\bX_0\transpose\bX_0)^{-1}
\right]\\
& \to \mathrm{tr}\left\{\bDelta^{-1}\int_\calX \bV(\bx_1)F(\mathrm{d}\bx_1)\bDelta^{-1}\right\}\\
& = \int_\calX\mathrm{tr}\left\{\bDelta^{-1}\bV(\bx)\bDelta^{-1}\right\}F(\mathrm{d}\bx)\\
& = \int_\calX\mathrm{tr}\left[\bDelta^{-1}\int_\calX\{\bx\transpose\bx_{1}(1 - \rho\bx\transpose\bx_{1})\}\bx_{1}\bx_{1}\transpose F(\mathrm{d}\bx_1)\bDelta^{-1}\right]F(\mathrm{d}\bx)\\
& = \int_\calX\mathrm{tr}\{\bSigma(\bx)\}F(\mathrm{d}\bx).
\end{align*}
\end{proof}

\subsection{Proof of the Asymptotic Normality \eqref{eqn:ASE_normality}} 
\label{sub:proof_of_the_asymptotic_normality_eqn:ase_normality}


The proof of the asymptotic normality \eqref{eqn:ASE_normality} is enormously different than that presented in Appendix A in \cite{tang2018} and \cite{athreya2016limit}. Here we follow the proof strategy adopted in \cite{cape2019signal}, which addresses the asymptotic normality of the rows of eigenvectors of random matrices. To this end, we need the following technical lemma.
\begin{lemma}\label{lemma:moment_bound_ASE}
Let $\bP_0 = \rho_n\bX_0\bX_0\transpose$ and $\bE = \bA - \bP_0$. Then for deterministic vector $\bv\in\mathbb{R}$, any $k = 1,\ldots,\lceil\log n\rceil$, and $p\leq \lceil(\log n)^2\rceil$, there exists a constant $C_\bE > 0$ such that
\[
\expect_0\left(|\be_i\transpose\bE^k\bv|^p\right)\leq (n\rho_n)^{kp/2} (2kp)^{kp}\|\bv\|_\infty^p.
\]
\end{lemma}

\begin{proof}[\bf Proof]
The proof is very similar to that of Lemma 5.4 in \cite{doi:10.1080/01621459.2020.1751645}, which originates from the proof of Lemma 7.10 in \cite{erdos2013}. Denote $\bH = (n\rho_n)^{-1/2}\bE = (n\rho_n)^{-1/2}(\bA - \bP_0)$, and let $H_{ij}$ be the $(i, j)$th entry of $\bH$. The proof is based on the following two observations:
\begin{itemize}
  \item $\expect_0(|H_{ij}|^m) \leq 1/n$ for all $m\geq 2$. This is because $n\rho_n\to \infty$, implying that 
  \[
  |H_{ij}| = \frac{|A_{ij} - \rho_n\bx_{0i}\transpose\bx_{0j}|}{\sqrt{n\rho_n}}\leq \frac{1}{\sqrt{n\rho_n}}\leq 1.
  \]
  Therefore, for any $m \geq 2$
  \[
  \expect_0(|H_{ij}|^m)\leq \expect_0(H_{ij}^2) = \frac{1}{n\rho_n}\expect_0\{(A_{ij} - \rho_n\bx_{0i}\transpose\bx_{0j})^2\} \leq \frac{1}{n}.
  \]
  \item To compute an upper bound for the $p$th moment of $|\be_i\transpose\bH^k\bv|$, the authors of \cite{erdos2013} use a multigraph construction technique by partition the summation indices into equivalent classes according to whether the indexed variables in the same equivalent class are the same or not. Then the number of summand is further upper bounded by the use of a spanning tree of the associated multigraph. Since the upper bound there uses the moment of the absolute values of $H_{ij}$'s and the vector $\bv$ is deterministic, we directly apply the proof there to the upper bound yields
  \[ 
  E(|\be_i\transpose\bH^k\bv|^p)\leq (2kp)^{kp}\|\bv\|_\infty^p.
  \]
  The factor $2$ comes from the fact that $\bH$ is diagonal and only the upper triangular part of $\bH$ are independent random variables. 
\end{itemize}
Since $\bH = (n\rho_n)^{-1/2}\bE$, it follows that
\[
\expect_0\left(|\be_i\transpose\bE^k\bv|^p\right)\leq (n\rho_n)^{kp/2} (2kp)^{kp}\|\bv\|_\infty^p.
\]
The proof is thus completed.
\end{proof}

\begin{proof}[\bf Proof of the asymptotic normality \eqref{eqn:ASE_normality}]
Denote $\bE = \bA - \bP_0 = \bA - \rho_n\bX_0\bX_0\transpose$. Let $\bA = \sum_{i = 1}^n\widehat\lambda_i\widehat\bu_i\widehat\bu_i\transpose$ be the spectral decomposition of $\bA$ with $|\widehat\lambda_1|\geq\ldots\geq|\widehat\lambda_n|$. Denote $\bS_\bA = \mathrm{diag}(|\widehat\lambda_1|,\ldots,|\widehat\lambda_d|)$ and $\bU_\bA = [\widehat\bu_1,\ldots,\widehat\bu_d]$. Similarly, let $\bP_0 = \bU_\bP\bS_\bP\bU_\bP\transpose$ be the spectral decomposition of $\bP_0$, where $\bU_\bP\in\mathbb{O}(n, d)$ and $\bS_\bP = \mathrm{diag}\{\lambda_1(\bP_0),\ldots,\lambda_d(\bP_0)\}$. Let $\bU_\bP = [\bu_{01},\ldots,\bu_{0d}]$.
In order to prove \eqref{eqn:ASE_normality}, we first list several useful facts:
\begin{itemize}
  \item[(a)] For any $c > 0$, there exists $C_1 > 0$ such that $\|\bA - \bP\|_2 \leq C_1(n\rho_n)^{1/2}$ with probability at least $1 - n^c$, and there exists constant $C_2 > 0$ such that $\|\bA - \bP\|_2 \leq C_2(n\rho_n)^{1/2}\log n$ with probability at least $1 - \exp\{-c(\log n)^2\}$. The first result is obtained from \cite{lei2015}, and the second result is a consequence of the matrix Bernstein's inequality.
  \item[(b)] For any $c > 0$, there exists $C > 0$ such that $\|\bS_\bA\|_2\leq Cn\rho_n$ and $\|\bS_\bA^{-1}\|_2\leq C(n\rho_n)^{-1}$ with probability at least $1 - \exp\{-c(\log n)^2\}$. This fact can be implied by fact (a) and Weyl's inequality. Furthermore, deterministically, $\|\bS_\bP\|_2\lesssim  n\rho_n$ and $\|\bS_\bP^{-1}\|_2\lesssim (n\rho_n)^{-1}$. This result follows from the fact that the smallest eigenvalue of $\bS_\bP$ is the same as the smallest eigenvalue of $\rho_n(\bX_0\transpose\bX_0)$, and $(1/n)(\bX_0\transpose\bX_0)\to \bDelta$ for some deterministic positive definite matrix $\bDelta$.
  \item[(c)] For any $c > 0$, there exists $C > 0$ such that $\|\bU_\bP\transpose\bE\bU_\bP\|_{\mathrm{F}}\leq C\log n$ with probability at least $1 - n^c$, and $\|\bU_\bP\transpose\bE\bU_\bP\|_{\mathrm{F}} = O_{\prob_0}(1)$. This follows from the union bound and Hoeffding's inequality (see equation (50) in \citealp{JMLR:v18:17-448}).
  \item[(d)] There exists some constant $c > 0$, such that for any fixed $i\in[n]$,
  \begin{align*}
  \prob\left(\bigcup_{k = 1}^d\left\{|\be_i\transpose\bE^2\bu_{0k}| > (n\rho_n)(\log n)^2\|\bu_{0k}\|_\infty\right\}\right)\leq \exp\left(-c\log n\right).
  \end{align*}
  This result can be derived as follows: By Lemma \ref{lemma:moment_bound_ASE} and Markov's inequality, with $p = \lfloor(\log n)/8\rfloor$, for every $k = 1,\ldots,d$, we have
  \begin{align*}
  \prob_0\left(|\be_i\transpose\bE^2\bu_{0k}| > (n\rho_n)(\log n)^{2}\|\bu_{0k}\|_\infty\right)
  &\leq \frac{(4p)^{2p}(n\rho_n)^p\|\bu_{0k}\|_\infty^p}{(n\rho_n)^p(\log n)^{2p}\|\bu_{0k}\|^p_\infty}\\
  & = \left(\frac{4p}{\log n}\right)^{2p} = \exp\left(-c\log n\right)
  \end{align*}
  for some constant $c > 0$. 
  The result then follows from the union bound applied to $k \in \{1,\ldots,d\}$.
  \item[(e)] There exists a constant $\nu > 0$, such that for all $\xi \in (1, 2]$, 
  \begin{align*}
  \prob_0\left(\bigcup_{t = 1}^{\lceil\log n\rceil}\bigcup_{i = 1}^n\bigcup_{k = 1}^d\left\{|\be_i\transpose\bE^t\bu_{0k}| > (n\rho_n)^{t/2}(\log n)^{\xi t}\|\bu_{0k}\|_\infty\right\}\right)
  \leq \exp\{-\nu(\log n)^\xi\}.
  \end{align*}
  The derivation is similar to (d), but for each fixed $i\in[n]$, $k\in[d]$, and $t = 1,\ldots,\lceil\log n\rceil$, we replace the choice of $p = \lfloor(\log n)/8\rfloor$ by $\lfloor(\log n)^{\xi}/(4t)\rfloor$, and invoke Lemma \ref{lemma:moment_bound_ASE} to derive
  \begin{align*}
  \prob_0\left(|\be_i\transpose\bE^t\bu_{0k}| > (n\rho_n)^{t/2}(\log n)^{\xi t}\|\bu_{0k}\|_\infty\right)
  &\leq \frac{(n\rho_n)^{pt/2}(2pt)^{pt}\|\bu_{0k}\|_{\infty}^p}{(n\rho_n)^{pt/2}(\log n)^{\xi tp}\|\bu_{0k}\|^p_\infty}
  \\
  & = \left\{\frac{2pt}{(\log n)^{\xi}}\right\}^{pt} = \exp\left\{-c(\log n)^{\xi}\right\}.
  \end{align*}
  Since $\xi > 1$ and $(\log n)^{\xi} \gg \log n$, taking the union bound over $i\in[n]$, $t = 1,\ldots,\lceil\log n\rceil$, and $k\in[d]$ leads to the desired result. 
  \item[(f)] $\|\bU_\bP\|_{2\to\infty}\lesssim n^{-1/2}$. This is because
  $\|\bU_\bP\|_{2\to\infty} = \|\bX_0\bS_\bP^{-1/2}\|_{2\to\infty}\leq \|\bX_0\|_{2\to\infty}\|\bS_\bP^{-1/2}\| \lesssim n^{-1/2}$.
\end{itemize}
Let $\bU_\bP\transpose\bU_\bA$ have singular value decomposition $\bU_\bP\transpose\bU_\bA = \bW_1\bS_{\bP\bA}\bW_2\transpose$, and let $\bW^* = \bW_1\bW_2\transpose$. 
We begin the proof by observing the following decomposition of $\bU_\bA - \bU_\bP\bW^*$ originating from \cite{cape2019signal}:
\begin{align*}\label{eqn:ASE_eigenvector_decomposition}
\bU_\bA - \bU_\bP\bW^* = \bE\bU_\bP\bS_\bP^{-1}\bW^* + \bR,
\end{align*}
where $\bR = \bR^{(1)} + \bR^{(2)} + \bR_2^{(1)} + \bR_2^{(2)} + \bR_2^{(\infty)}$, 
\begin{align*}
\bR^{(1)} & = \bU_\bP\bS_\bP\bR^{(3)},\quad \bR^{(3)} = \bU_\bP\transpose\bU_\bA\bS_\bA^{-1} - \bS_\bP^{-1}\bU_\bP\transpose\bU_\bA,\\
\bR^{(2)} & = \bU_\bP(\bU_\bP\transpose\bU_\bA - \bW^*),\quad
\bR_2^{(1)} = \bE\bU_\bP\bS_\bP(\bU_\bP\transpose\bU_\bA\bS_\bA^{-2} - \bS_\bP^{-2}\bU_\bP\transpose\bU_\bA),\\
\bR_2^{(2)} & = \bE\bU_\bP\bS_\bP^{-1}(\bU_\bP\transpose\bU_\bA - \bW^*),
\quad
\bR_2^{(\infty)}  = \sum_{t = 2}^\infty \bE^t\bU_\bP\bS_\bP\bU_\bP\transpose\bU_\bA\bS_\bA^{-(t + 1)}.
\end{align*}
The above derivation is directly obtained from the proof of Theorem 2 in \cite{cape2019signal}, and the key idea is that the spectra of $\bS_\bA$ and $\bE$ are disjoint with high probability, such that one can apply the von Neumann trick to write
\[
\bU_\bA = \sum_{t = 0}^\infty\bE^t\bU_\bP\bS_\bP\bU_\bP\transpose\bU_\bA\bS_\bA^{-(t + 1)}
\]
Also see (S3) in the Supplementary Material of \cite{cape2019signal}. 

Let $\xi\in (1, 2]$ be a constant to be specified later. We now fix the row index $i\in[n]$ and show that
\begin{align}\label{eqn:ASE_leading_remainder}
\|\be_i\transpose\bE(\bU_\bA - \bU_\bP\bW^*)\|_2 = O_{\prob_0}\left\{\frac{(\log n)^2}{\sqrt{n}}\right\}.
\end{align}
This can be done by the establishing the following results:
\begin{itemize}
  \item[(1)] $\|\bE\bR^{(1)}\|_{2\to\infty} = o_{\prob_0}(n^{-1/2})$. In fact, following \cite{cape2019signal}, we see that
  \[
  \|\bR^{(3)}\|_2\leq d\|\bS_\bP^{-1}\|_2\|\bS_\bA^{-1}\|_2\|\bU_\bP\transpose\bU_\bA\bS_\bA - \bS_\bP\bU_\bP\transpose\bU_\bA\|_2 = O_{\prob_0}\{(n\rho_n)^{-2}(\|\bU_\bP\transpose\bE\bU_\bP\|_2 + 1)\}.
  \]
  By result (e) with $t = 1$, we have $\|\bE\bU_\bP\|_{2\to\infty} = O_{\prob_0}((n\rho)^{1/2}(\log n)^{\xi}\|\bU_\bP\|_{2\to\infty})$. Therefore, by results (b), (c), and (f), we obtain
  \begin{align*}
  \|\bE\bR^{(1)}\|_{2\to\infty}
  &\leq \|\bE\bU_\bP\|_{2\to\infty}\|\bS_\bP\|_2\|\bR^{(3)}\|_2 = O_{\prob_0}\left\{\frac{(\log n)^\xi(\|\bU_\bP\transpose\bE\bU_\bP\|_2 + 1)\|\bU_\bP\|_{2\to\infty}}{(n\rho_n)^{1/2}}\right\}\\
  &= \frac{1}{\sqrt{n}}O_{\prob_0}\left\{\frac{(\log n)^\xi}{(n\rho_n)^{1/2}}\right\} = o_{\prob_0}(n^{-1/2})
  \end{align*}
  \item[(2)] $\|\bE\bR^{(2)}\|_{2\to\infty} = o_{\prob_0}(n^{-1/2})$. Since $\|\bU_\bP\transpose\bU_\bA - \bW^*\|_2 = O_{\prob_0}((n\rho_n)^{-1})$, which is a simple consequence of the Davis-Kahan theorem \citep{doi:10.1093/biomet/asv008}, it follows from result (e) with $t = 1$ and result (f) that
  \begin{align*}
  \|\bE\bR^{(2)}\|_{2\to\infty}
  & \leq \|\bE\bU_\bP\|_{2\to\infty}\|\bU_\bP\transpose\bU_\bA - \bW^*\|_2
  = O_{\prob_0}\left( (n\rho_n)^{1/2}(\log n)^{\xi}n^{-1/2}(n\rho_n)^{-1} \right)
  = o_{\prob_0}(n^{-1/2}).
  \end{align*}
  \item[(3)] $\be_i\transpose\bE\bR_2^{(1)} = o_{\prob_0}(n^{-1/2})$. By the results (d) and (f), we have
  \[
  \|\be_i\transpose\bE^2\bU_\bP\|_2 = O_{\prob_0}\{n^{-1/2}(n\rho_n)(\log n)^2\}.
  \]
  From \cite{cape2019signal}, we have that
  \[
  \|\bU_\bP\transpose\bU_\bA\bS_\bA^{-2} - \bS_\bP^{-2}\bU_\bP\transpose\bU_\bA\|_2 = O_{\prob_0}((n\rho_n)^{-5/2}).
  \]
  Therefore,
  \begin{align*}
  \|\be_i\transpose\bE\bR_2^{(1)}\|_2
  & \leq \|\be_i\transpose\bE^2\bU_\bP\|_2\|\bS_\bP\|_2\|\bU_\bP\transpose\bU_\bA\bS_\bA^{-2} - \bS_\bP^{-2}\bU_\bP\transpose\bU_\bA\|_2\\
  & =  O_{\prob_0}\{n^{-1/2}(n\rho_n)(\log n)^2\} O_{\prob_0}((n\rho_n)^{-3/2})
  = O_{\prob_0}\left\{\frac{(\log n)^2}{\sqrt{n}(n\rho_n)}\right\} = o_{\prob_0}(n^{-1/2}).
  \end{align*}
  \item[(4)] $\|\bE\bR_2^{(2)}\|_{2\to\infty} = o_{\prob_0}(1)$. By result (e) with $t = 2$, the fact that $\|\bU_\bP\transpose\bU_\bA - \bW^*\|_2 = O_{\prob_0}((n\rho_n)^{-1})$, the results (b) and (f), we obtain
  \begin{align*}
  \|\bE\bR_2^{(2)}\|_{2\to\infty}
  &\leq 
  \|\bE^2\bU_\bP\|_{2\to\infty}\|\bS_\bP^{-1}\|_2\|\bU_\bP\transpose\bU_\bA - \bW^*\|_2\\
  &= O_{\prob_0}\left\{\|\bU_\bP\|_{2\to\infty}(\log n)^{2\xi}(n\rho_n) (n\rho_n)^{-1}(n\rho_n)^{-1}\right\}\\
  &= O_{\prob_0}\left\{n^{-1/2}\frac{(\log n)^{2\xi}}{n\rho_n}\right\} = o_{\prob_0}(n^{-1/2}).
  \end{align*}
  \item[(5)] $\|\bE\bR_2^{(\infty)}\|_{2\to\infty} = O_{\prob_0}\{n^{-1/2}(\log n)^2\}$. Denote $t(n) = 2\lceil(\log n)/(\log n\rho_n)\rceil$. Clearly, $t(n) \ll \log n$ since $n\rho_n\to\infty$. Write
  \begin{align*}
  \|\bE\bR_2^{(\infty)}\|_{2\to\infty}
  & = \left\|\sum_{t = 2}^\infty\bE^{t + 1}\bU_\bP\bS_\bP\bU_\bP\transpose\bU_\bA\bS_\bA^{-(t + 1)}\right\|_{2\to\infty}\\
  & \leq \sum_{t = 2}^{t(n) + 1}\|\bE^{t + 1}\bU_\bP\|_{2\to\infty}\|\bS_\bP\|_2\|\bS_\bA^{-1}\|_2^{t + 1}
   + \sum_{t = t(n) + 2}^\infty\|\bE\|_2^{t + 1}\|\bS_\bP\|_2\|\bS_\bA^{-1}\|_2^{t + 1}.
  \end{align*}
  By the results (b) and (e), for any $c > 0$, there exist constant $C_\bE > 0$, such that with probability at least $1 - n^c$, we have
  \begin{align*}
  \sum_{t = 2}^{t(n) + 1}\|\bE^{t + 1}\bU_\bP\|_{2\to\infty}\|\bS_\bP\|_2\|\bS_\bA^{-1}\|_2^{t + 1}
  &\leq \sum_{t = 3}^{t(n) + 2}\|\bE^{t}\bU_\bP\|_{2\to\infty}\|\bS_\bP\|_2\|\bS_\bA^{-1}\|_2^{t}\\
  &\leq \sum_{t = 3}^\infty C_\bE^t(n\rho_n)^{t/2}(\log n)^{t\xi}\|\bU_\bP\|_{2\to\infty}\|\bS_\bP\|_2(n\rho_n)^{-t}\\
  &\lesssim \frac{(n\rho_n)}{\sqrt{n}}\sum_{t = 3}^\infty \{C_\bE(n\rho_n)^{-1/2}(\log n)^\xi\}^t\\
  &\lesssim \frac{n\rho_n}{\sqrt{n}}\frac{(\log n)^{3\xi}C_\bE^{3}}{(n\rho_n)^{3/2}} \asymp \frac{(\log n)^{3\xi}}{\sqrt{n}(n\rho_n)^{1/2}}.
  \end{align*}
  For the second infinite sum, by the results (a) and (b), with probability at least $1 - n^c$, we directly compute
  \begin{align*}
  \sum_{t = t(n) + 2}^\infty\|\bE\|_2^{t + 1}\|\bS_\bP\|_2\|\bS_\bA^{-1}\|_2^{t + 1}
  &\lesssim \sum_{t = t(n) + 2}^\infty n\rho_n\{C(n\rho_n)^{1/2}(n\rho_n)^{-1}\}^{t + 1}\\
  & = n\rho_n\sum_{t = t(n) + 2}^\infty \{C(n\rho_n)^{-1/2}\}^{t + 1}\\
  & = \exp\left[ \{t(n) + 3\} \log C - \frac{t(n) + 3}{2}\log (n\rho_n) + \log(n\rho_n)\right]\\
  &\leq \exp\left[ \left\{2\left\lceil\frac{\log n}{\log (n\rho_n)}\right\rceil+ 3\right\} \log C - \frac{1}{2}\log (n\rho_n) - \log n \right]\\
  &\lesssim \exp\left\{-\frac{1}{2}\log (n\rho_n)-\frac{1}{2}\log n\right\} = o(n^{-1/2}).
  \end{align*}
  Therefore, with $\xi = 4/3$, we conclude that 
  \[
  \|\bE\bR_2^{(\infty)}\|_{2\to\infty} = O_{\prob_0}\left\{\frac{1}{\sqrt{n}}\frac{(\log n)^4}{(n\rho_n)^{1/2}}\right\}
  = O_{\prob_0}\left\{\frac{(\log n)^2}{\sqrt{n}}\frac{(\log n)^2}{(n\rho_n)^{1/2}}\right\} = O_{\prob_0}\left\{\frac{(\log n)^2}{\sqrt{n}}\right\}
  .
  \]
  \item[(6)] $\|\be_i\transpose\bE^2\bU_\bP\bS_\bP\bW^*\|_{2} = O_{\prob_0}\{n^{-1/2}(\log n)^2\}$. Since 
  \[
  \|\be_i\transpose\bE^2\bU_\bP\bS_\bP\bW^*\|_{2}\leq \|\be_i\transpose\bE^2\bU_\bP\|_{2}\|\bS_\bP^{-1}\|_{2}\lesssim (n\rho_n)^{-1}\|\be_i\transpose\bE^2\bU_\bP\|_{2},
  \]
  then by result (d), with probability going to $1$, and result (f),
  \[
  \|\be_i\transpose\bE^2\bU_\bP\bS_\bP\bW^*\|_{2}\lesssim (n\rho_n)^{-1} (n\rho_n) (\log n)^2\|\bU_\bP\|_{2\to\infty} = O_{\prob_0}\left\{\frac{(\log n)^2}{\sqrt{n}}\right\}.
  \]
\end{itemize}
We are finally in a position to prove the asymptotic normality \eqref{eqn:ASE_normality} of the $i$th row of the ASE.
Let $\bW_\bX$ be the orthogonal matrix such that $\bX_0 = \bU_\bP\bS_\bP^{1/2}\bW_\bX$, and let $\bW = (\bW^*)\transpose\bW_\bX$.
 Following the derivation in Appendix A.1 and A.2 in \cite{JMLR:v18:17-448}, we obtain the following decomposition:
\begin{align*}
\sqrt{n}\be_i\transpose\{\widehat\bX^{(\mathrm{ASE})}\bW - \bX_0\}
& = \sqrt{n}\be_i\transpose(\bA - \bP)\bU_\bP\bS_\bP^{-1/2}\bW_\bX
 + \sqrt{n}\be_i\transpose\bE\bU_\bP(\bW^*\bS_\bA^{-1/2} - \bS_\bP^{-1/2}\bW^*)\bW\\
&\quad - \sqrt{n}\be_i\transpose\bU_\bP(\bU_\bP\transpose\bE\bU_\bP)\bW^*\bS_\bA^{-1/2}\bW
 + \sqrt{n}\be_i\transpose\bE(\bU_\bA - \bU_\bP\bW^*)\bS_\bA^{-1/2}\bW\\
&\quad - \sqrt{n}\be_i\transpose\bU_\bP\bU_\bP\transpose\bE(\bU_\bA - \bU_\bP\bW^*)\bS_\bA^{-1/2}\bW
 + \sqrt{n}\be_i\transpose(\bR_1\bS_\bA^{1/2} + \bU_\bP\bR_2),
\end{align*}
where
\[
\bR_1 = \bU_\bP(\bU_\bP\transpose\bU_\bA - \bW^*),\quad \bR_2 = \bW^*\bS_\bA^{1/2} - \bS_\bP^{1/2}\bW^*.
\]
By Appendix A in \cite{tang2018}, 
\[
\{\sqrt{n}\be_i\transpose(\bA - \bP)\bU_\bP\bS_\bP^{-1/2}\bW_\bX\}\transpose\overset{\calL}{\to}\mathrm{N}(\zero, \bSigma(\bx_{0i})). 
\]
It suffices to argue that the remainders are asymptotically negligible. 
\begin{itemize}
  \item[(7)] $\sqrt{n}\|\bE\bU_\bP(\bW^*\bS_\bA^{-1/2} - \bS_\bP^{-1/2}\bW^*)\bW\|_{2\to\infty} = o_{\prob_0}(1)$. By a similar argument for proving Lemma 49 of \cite{JMLR:v18:17-448}, we have
  \[
  \|\bW^*\bS_\bA^{-1/2} - \bS_\bP^{-1/2}\bW^*\| = O_{\prob_0}\left\{(n\rho_n)^{-3/2}\log n\right\}.
  \]
  Furthermore, by result (e) with $t = 1$ and result (f),
  \[
  \|\bE\bU_\bP\|_{2\to\infty} = O_{\prob_0}\{n^{-1/2}(n\rho_n)^{1/2}(\log n)^{4/3}\}.
  \]
  Therefore,
  \begin{align*}
  \sqrt{n}\|\bE\bU_\bP(\bW^*\bS_\bA^{-1/2} - \bS_\bP^{-1/2}\bW^*)\bW\|_{2\to\infty}
  &\leq O_{\prob_0}\{(n\rho_n)^{1/2}(\log n)^{4/3}\}O_{\prob_0}\left\{(n\rho_n)^{-3/2}\log n\right\}
  \\
  & = O_{\prob_0}\{(n\rho_n)^{-1}(\log n)^{7/3}\} = o_{\prob_0}(1).
  \end{align*}

  \item[(8)] $\sqrt{n}\|\bU_\bP(\bU_\bP\transpose\bE\bU_\bP)\bW^*\bS_\bA^{-1/2}\bW\|_{2\to\infty} = o_{\prob_0}(1)$. By the results (b), (c), and (f), we have
  \begin{align*}
  \sqrt{n}\|\bU_\bP(\bU_\bP\transpose\bE\bU_\bP)\bW^*\bS_\bA^{-1/2}\bW\|_{2\to\infty}
  &\leq \sqrt{n}\|\bU_\bP\|_{2\to\infty}\|\bU_\bP\transpose\bE\bU_\bP\|_2\|\bS_\bA^{-1/2}\|_2
  \\
  &\lesssim \|\bU_\bP\transpose\bE\bU_\bP\|_2\|\bS_\bA^{-1/2}\|_2 = O_{\prob_0}(\log n)O_{\prob_0}((n\rho_n)^{-1/2}) = o_{\prob_0}(1).
  \end{align*}

  \item[(9)] $\sqrt{n}\be_i\transpose\bE(\bU_\bA - \bU_\bP\bW^*)\bS_\bA^{-1/2}\bW = o_{\prob_0}(1)$. By the previously claimed result \eqref{eqn:ASE_leading_remainder}, we have
  \[
  \|\be_i\transpose\bE(\bU_\bA - \bU_\bP\bW^*)\|_2 = O_{\prob_0}\left\{\frac{(\log n)^2}{\sqrt{n}}\right\}.
  \]
  Together with result (b), we obtain
  \begin{align*}
  \|\sqrt{n}\be_i\transpose\bE(\bU_\bA - \bU_\bP\bW^*)\bS_\bA^{-1/2}\bW\|_2
  &\leq \sqrt{n}\|\be_i\transpose\bE(\bU_\bA - \bU_\bP\bW^*)\|_2\|\bS_\bA^{-1/2}\|_2\\
  & = O_{\prob_0}\left\{\frac{(\log n)^2}{\sqrt{n\rho_n}}\right\} = o_{\prob_0}(1).
  \end{align*}

  \item[(10)] $\|\sqrt{n}\bU_\bP\bU_\bP\transpose\bE(\bU_\bA - \bU_\bP\bW^*)\bS_\bA^{-1/2}\bW\|_{2\to\infty} = o_{\prob_0}(1)$. By the Davis-Kahan theorem and results (a) and (b), $\|\bU_\bA - \bU_\bP\bW^*\|_2 = O_{\prob_0}((n\rho_n)^{-1/2})$. Hence,
  \begin{align*}
  \|\sqrt{n}\bU_\bP\bU_\bP\transpose\bE(\bU_\bA - \bU_\bP\bW^*)\bS_\bA^{-1/2}\bW\|_{2\to\infty}
  & \leq \sqrt{n}\|\bU_\bP\|_{2\to\infty}\|\bE\|_2\|\bU_\bA - \bU_\bP\bW^*\|_2\|\bS_\bA^{-1/2}\|_2\\
  & = O_{\prob_0}\{(n\rho_n)^{1/2}(n\rho_n)^{-1/2}(n\rho_n)^{-1/2}\} = o_{\prob_0}(1).
  \end{align*}

  \item[(11)] $\|\sqrt{n}(\bR_1\bS_\bA^{1/2} + \bU_\bP\bR_2)\|_{2\to\infty} = o_{\prob_0}(1)$. By Lemma 49 of \cite{JMLR:v18:17-448} (which also holds for deterministic $\bx_{01},\ldots,\bx_{0n}$), $\|\bR_2\|_{\mathrm{F}} = O_{\prob_0}((n\rho)^{-1/2}\log n)$. Hence, 
  \[
  \sqrt{n}\|\bU_\bP\bR_2\|_{2\to\infty}
  \leq \sqrt{n}\|\bU_\bP\|_{2\to\infty}\|\bR_2\|_2 = O_{\prob_0}((n\rho_n)^{-1/2}\log n) = o_{\prob_0}(1).
  \]
  In addition, we obtain
  \[
  \sqrt{n}\|\bR_1\bS_\bA^{1/2}\|_{2\to\infty}\leq\sqrt{n}\|\bU_\bP\|_{2\to\infty}\|\bU_\bP\transpose\bU_\bA - \bW^*\|_2\|\bS_\bA^{1/2}\|_2\lesssim \|\bU_\bP\transpose\bU_\bA - \bW^*\|_2\|\bS_\bA^{1/2}\|_2 = O_{\prob_0}((n\rho_n)^{-1})
  \]
  by result (b).
\end{itemize}
The proof is completed by summarizing the above results. 
\end{proof}


\section{Proof of Theorems \ref{thm:one_dimensional_MLE} and \ref{thm:OSE_single_vertex}
(Estimating One Latent Position)} 
\label{sec:proof_of_theorems_thm:one_dimensional_mle_and_thm:ose_single_vertex}


\subsection{Proofs of Theorem \ref{thm:one_dimensional_MLE}} 
\label{sec:proof_of_theorem_thm:one_dimensional_mle}
Before presenting the proof of Theorem \ref{thm:one_dimensional_MLE}, we need to establish the following real analysis result. We note that this result may not be completely original and can be proved using a standard real analysis argument. 
\begin{lemma}\label{lemma:real_analysis_lemma}
Let $(f_n)_{n = 1}^\infty\subset C^2(D)$ be a sequence of functions in $C^2(D)$, where $D$ is a convex compact subset of $\mathbb{R}^d$ with non-empty interior, and $C^2(D)$ is the class of twice continuously differentiable functions on $D$. Let $f\in C^2(D)$ as well. Assume the following conditions hold:
\begin{itemize}
  \item[(i)] $f_n$ converges uniformly to $f$ within $D$, \emph{i.e.}, $\sup_{\bx\in D}|f_n(\bx) - f(\bx)| \to 0$ as $n\to\infty$;
  \item[(ii)] $f_n$ and $f$ both have continuous Hessians 
  \[
  \frac{\partial^2 f_n}{\partial\bx\partial\bx\transpose}(\bx),\quad
  \frac{\partial^2 f}{\partial\bx\partial\bx\transpose}  (\bx)
  \]
  and the negative Hessians are positive definite for all $\bx\in D$. 
\end{itemize}
Let $\widehat{\bx}_n$ be the unique maximizer of $f_n$, and $\widehat{\bx}_0$ be the unique maximizer of $f$. 
Then:
\begin{itemize}
  \item[(a)] The sequence $(\widehat{\bx}_n)_{n = 1}^\infty$ converges to the unique maximizer $\widehat{\bx}_0$ of the limit function $f$;
  \item[(b)] For all $\eps > 0$, there exists a positive $\eta = \eta(\eps)$ that does not depend on $n$, and a positive integer $N = N(\eps)$ that only depends on $\eps$, such that for all $n\geq N$, and all $\bx$ satisfying $\|\bx - \widehat\bx_n\|_2 > \eps$, 
  \[
  f_n(\widehat{\bx}_n) - f_n(\bx) \geq \eta(\eps)
  \]
  holds. 
\end{itemize}
\begin{proof}[\bf Proof]
For conclusion (a) we take $(\widehat{\bx}_{n_k})_{k = 1}^\infty$ to be any subsequence of $(\widehat{\bx}_n)_{n = 1}^\infty$ that converges. Suppose $\widehat{\bx}_{n_k}
\to \widehat{\bx}_0^*$ as $n\to\infty$. It follows that
\begin{align*}
|f_{n_k}(\widehat{\bx}_{n_k}) - f(\widehat{\bx}_0^*)|
&\leq |f_{n_k}(\widehat{\bx}_{n_k}) - f(\widehat{\bx}_{n_k})| + |f(\widehat{\bx}_{n_k}) - f(\widehat{\bx}_0^*)|\\
&\leq \sup_{\bx\in D}|f_{n_k}(\bx) - f(\bx)| + |f(\widehat{\bx}_{n_k}) - f(\widehat{\bx}_0^*)|.
\end{align*}
The sequence $(\sup_{\bx\in D}|f_n(\bx) - f(\bx)|)_{n = 1}^\infty$ converges to $0$ as $n\to\infty$, and the function $f$ is continuous. Therefore, the two terms on the right-hand side of the previous display converges to $0$ as $k\to\infty$ because $(\sup_{\bx\in D}|f_{n_k}(\bx) - f(\bx)|)_{k = 1}^\infty$ is a subsequence of $(\sup_{\bx\in D}|f_n(\bx) - f(\bx)|)_{n = 1}^\infty$, and hence, $f_{n_k}(\widehat{\bx}_{n_k}) \to f(\widehat{\bx}_0^*)$ as $n\to\infty$. By definition, $\widehat{\bx}_{n_k}$ is the maximizer of $f_{n_k}$, implying that $f_{n_k}(\widehat{\bx}_{n_k})\geq f_{n_k}(\bx)$ for all $\bx\in D$. Therefore,
\[
f(\widehat{\bx}_0^*) = \lim_{n\to\infty}f_{n_k}(\widehat{\bx}_{n_k})\geq \lim_{n\to\infty}f_{n_k}(\bx) = f(\bx)
\]
for any $\bx\in D$, where we have used the fact that $f_{n_k}$ converges point-wise to $f$ as $k\to\infty$. This further shows that $\widehat{\bx}_0^*$ is the maximizer of $f$. Since the maximizer of $f$ is unique, we conclude that $\widehat{\bx}_{n_k}\to \widehat{\bx}_0^* = \widehat{\bx}_0$ as $k\to\infty$. Note that any converging subsequence $(\widehat{\bx}_{n_k})_{k = 1}^\infty$ of $(\widehat{\bx}_n)_{n = 1}^\infty$ converges to the same limit point $\widehat{\bx}_0$. Therefore we conclude that the entire sequence $\widehat{\bx}_n$ converges to $\widehat{\bx}_0$.

For part (b), we first observe that $\widehat{\bx}_0$ is a strict maximizer of $f$ because the negative Hessian of $f$ is positive definite, and $-f$ is convex because the Hessian of $-f$ is always positive definite. Therefore, there exists some $\delta > 0$, such that
$f(\widehat{\bx}_0) > f(\bx)$ for all $\bx\in\{\bx\in D:0 < \|\bx - \widehat\bx_0\| < \delta\}$. Now we claim that there exists some constant $\xi > 0$ that only depends on $\delta > 0$, such that 
\begin{align}
\label{eqn:real_analysis_result1}
f(\bx) + \xi\leq f(\widehat{\bx}_0)\quad
\text{for any }\bx\in\{\bx\in D:\|\bx - \widehat\bx_0\|_2 \geq \delta\}.
\end{align}
In fact, if one assumes otherwise, then for all $\xi > 0$, there exists some $\bx_\xi$ with $\|\bx_\xi - \widehat{\bx}_0\|_2\geq\delta$, such that $f(\bx_\xi) + \xi > f(\widehat{\bx}_0)$. Taking a sequence $(\xi_j)_{j = 1}^\infty = (1/j)_{j = 1}^\infty$ yields a sequence $(\bx_{j})_{j = 1}^\infty$ such that
\[
f(\bx_{j}) + \frac{1}{j} > f(\widehat{\bx}_0).
\]
Let $(\bx_{j_k})_{k = 1}^\infty$ be a converging subsequence of $(\bx_j)_{j = 1}^\infty$ that converges to some point $\by\in D$. Then the continuity of $f$ leads to
\[
f(\by) = \lim_{k\to\infty}\left\{f(\bx_{j_k}) + \frac{1}{j_k}\right\}\geq f(\widehat{\bx}_0)
.
\]
Since $(\bx_{j_k})_{k = 1}^\infty\subset \{\bx\in D:\|\bx - \widehat\bx_{0}\|_2 \geq \delta\}$, and the latter superset is compact, it follows that $\|\by - \widehat{\bx}_0\|_2 \geq \delta$, and hence, it must be the case that $f(\by) < f(\widehat{\bx}_0)$ due to the uniqueness of the maximizer of $f$. Therefore we conclude that there exists some constant $\xi > 0$ that only depends on $\delta > 0$, such that 
\[
f(\bx) + \xi\leq f(\widehat{\bx}_0)\quad\text{ for all }\bx\in\{\bx\in D:\|\bx - \widehat\bx_0\|_2 \geq \delta\}.
\]

Now we further claim that for any $\eps > 0$, there exists some $\eta(\eps) > 0$ such that
\begin{align}
\label{eqn:real_analysis_result2}
f(\widehat{\bx}_0) \geq f(\bx) + 4\eta(\eps)\quad\text{ for all }\bx\in\{\bx\in D:\|\bx - \widehat{\bx}_0\|_2 > \eps\}.
\end{align}
Now let $\eps > 0$ be arbitrarily given. We consider two cases: If $\eps \geq \delta$, then we can take $\eta(\eps) = \xi /4$ directly. Then the previous claim applies; If $\eps < \delta$, then we see that 
\[
\{\bx\in D:\|\bx - \widehat{\bx}_0\|_2 > \eps\}
 = \{\bx\in D:\eps < \|\bx - \widehat{\bx}_0\|_2 < \delta\}\cup
 \{\bx\in D:\|\bx - \widehat{\bx}_0\|_2 \geq \delta\}.
\]
For any $\bx\in \{\bx\in D:\eps < \|\bx - \widehat{\bx}_0\|_2 < \delta\}$, by the fact that $\widehat{\bx}_0$ is a strict maximizer of $f$, it must be the case that
\[
\sup_{\bx:\eps < \|\bx - \widehat{\bx}_0\|_2 < \delta}f(\bx) < f(\widehat{\bx}_0).
\]
Now take
\[
\eta(\eps) = \min\left[\frac{1}{4}\left\{f(\widehat{\bx}_0) - \sup_{\bx:\eps < \|\bx - \widehat{\bx}_0\|_2 < \delta}f(\bx)\right\}, \frac{\xi}{4}\right].
\]
It follows from \eqref{eqn:real_analysis_result1} that
\begin{align*}
&\sup_{\bx\in D:\|\bx - \widehat{\bx}_0\|_2 > \eps}f(\bx) + 4\eta(\eps)\\
&\quad = \max\left[
\sup_{\bx\in D:\eps < \|\bx - \widehat{\bx}_0\|_2 < \delta}f(\bx) + 4\eta(\eps),
\sup_{\bx\in D: \|\bx - \widehat{\bx}_0\|_2 \geq \delta}f(\bx) + 4\eta(\eps)
\right]
\\
&\quad\leq \max\left[
\sup_{\bx\in D:\eps < \|\bx - \widehat{\bx}_0\|_2 < \delta}f(\bx) + f(\widehat{\bx}_0) - 
\sup_{\bx\in D:\eps < \|\bx - \widehat{\bx}_0\|_2 < \delta}f(\bx),
\sup_{\bx\in D: \|\bx - \widehat{\bx}_0\|_2 \geq \delta}f(\bx) + \xi
\right]\\
&\quad = f(\widehat{\bx}_0),
\end{align*}
and this completes the proof of \eqref{eqn:real_analysis_result2}.

We finally prove the desired result. Since $f_n\to f$ uniformly as $n\to\infty$, then there exists some large positive constant $N_1 = N_1(\eps)$, such that
\[
\sup_{\bx\in D}|f_n(\bx) - f(\bx)| < \eta(\eps/2)\quad\text{ for all }n > N_1.
\]
Furthermore, by part (a) and the continuity of $f$, there exists some large positive constant $N_2 = N_2(\eps)$, such that 
\[
\|\widehat{\bx}_n - \widehat{\bx}_0\| < \frac{\eps}{2}\quad\text{and}\quad|f(\widehat{\bx}_n) - f(\widehat{\bx}_0)| < \eta(\eps/2)\quad\text{ for all }n > N_2.
\]
Note that 
\[
\{\bx\in D:\|\bx - \widehat{\bx}_n\|_2 > \eps\}
\subset
\{\bx\in D:\|\bx - \widehat{\bx}_0\|_2 > \eps/2\}
\]
for $n\geq N_2(\eps)$. 
Hence, for all $n\geq N_2(\eps)$, we obtain
\begin{align*}
\sup_{\bx\in D:\|\bx - \widehat{\bx}_n\|_2 > \eps}f_n(\bx)
& \leq \sup_{\bx\in D:\|\bx - \widehat{\bx}_n\|_2 > \eps}|f_n(\bx) - f(\bx)| + \sup_{\bx\in D:\|\bx - \widehat{\bx}_n\|_2 > \eps}f(\bx)\\
& \leq \sup_{\bx\in D}|f_n(\bx) - f(\bx)| + \sup_{\bx\in D:\|\bx - \widehat{\bx}_n\|_2 > \eps}f(\bx)\\
& \leq \sup_{\bx\in D}|f_n(\bx) - f(\bx)| + \sup_{\bx\in D:\|\bx - \widehat{\bx}_0\|_2 > \eps/2}f(\bx)\\
& \leq \sup_{\bx\in D}|f_n(\bx) - f(\bx)| + f(\widehat{\bx}_0) - 4\eta(\eps/2),
\end{align*}
where we have used \eqref{eqn:real_analysis_result2} at the end of the previous display. We further write
\begin{align*}
\sup_{\bx\in D:\|\bx - \widehat{\bx}_n\|_2 > \eps}f_n(\bx)
& \leq \sup_{\bx\in D}|f_n(\bx) - f(\bx)| + f(\widehat{\bx}_0)\\
& =    \sup_{\bx\in D}|f_n(\bx) - f(\bx)| + 
f(\widehat{\bx}_0) - f(\widehat{\bx}_n) + f(\widehat{\bx}_n) - f_n(\widehat{\bx}_n) + f_n(\widehat\bx_n)
 - 4\eta(\eps/2)\\
&\leq \sup_{\bx\in D}|f_n(\bx) - f(\bx)| + 
    |f(\widehat{\bx}_0) - f(\widehat{\bx}_n)| + 
    \sup_{\bx\in D}|f(\bx) - f_n(\bx)| + f_n(\widehat\bx_n)
 - 4\eta(\eps/2)
\end{align*}
for all $n\geq N_2(\eps)$. Hence for any $n \geq N = N(\eps) = \max\{N_1(\eps), N_2(\eps)\}$, 
\begin{align*}
\sup_{\bx\in D:\|\bx - \widehat{\bx}_n\|_2 > \eps}f_n(\bx)
&\leq 2\sup_{\bx\in D}|f_n(\bx) - f(\bx)| + 
    |f(\widehat{\bx}_0) - f(\widehat{\bx}_n)| + f_n(\widehat\bx_n) - 4\eta(\eps/2)\\
&\leq f_n(\widehat\bx_n) + 2\eta(\eps/2) + \eta(\eps/2) - 4\eta(\eps/2) = f_n(\widehat\bx_n) -\eta(\eps/2).
\end{align*}
The proof is completed by observing that both $N$ and $\eta(\eps/2)$ only depend on $\eps$ so that we can replace $\eta(\eps/2)$ by $\eta(\eps)$ without changing the proof, and these two quantities does not depend on $n$. 
\end{proof}
\end{lemma}
\begin{proof}[\bf Proof of Theorem \ref{thm:one_dimensional_MLE}]
We begin the proof with writing down the likelihood function for $\bx_i$:
\[
\ell_\bA(\bx_i) = \sum_{j\neq i}^n\{A_{ij}\log(\bx_i\transpose\bx_{0j}) + (1 - A_{ij})\log(1 - \bx_i\transpose\bx_{0j})\}.
\]
For convenience we denote the following functions:
\begin{align*}
M_n(\bx) &= \frac{1}{n}\ell_\bA(\bx) = \frac{1}{n}\sum_{j\neq i}^n\{A_{ij}\log(\bx\transpose\bx_{0j}) + (1 - A_{ij})\log(1 - \bx\transpose\bx_{0j})\},\\
M(\bx) &= \expect_0\{M_n(\bx)\} = \frac{1}{n}\sum_{j\neq i}^n\{\bx_{0i}\transpose\bx_{0j}\log(\bx\transpose\bx_{0j}) + (1 - \bx_{0i}\transpose\bx_{0j})\log(1 - \bx\transpose\bx_{0j})\},\\
\bPsi_n(\bx)&= \frac{\partial M_n}{\partial \bx}(\bx) = \frac{1}{n}\sum_{j\neq i}^n\left\{\frac{(A_{ij} - \bx\transpose\bx_{0j})}{\bx\transpose\bx_{0j}(1 - \bx\transpose\bx_{0j})}\right\}\bx_{0j},\\
\bPsi(\bx) & = \expect_0\{\bPsi_n(\bx)\} = \expect_0\left\{\frac{\partial M_n}{\partial \bx}(\bx)\right\} = \frac{1}{n}\sum_{j\neq i}^n\left\{\frac{(\bx_{0i} - \bx)\transpose\bx_{0j}}{\bx\transpose\bx_{0j}(1 - \bx\transpose\bx_{0j})}\right\}\bx_{0j}.
\end{align*}
Note that the function $M$ itself also depends on $n$ implicitly. Denote $\Psi_{nk}$ the $k$th component of $\bPsi_n$, \emph{i.e.}, 
\[
\Psi_{nk}(\bx) = \frac{1}{n}\sum_{j\neq i}^n\left\{\frac{(A_{ij} - \bx\transpose\bx_{0j})}{\bx\transpose\bx_{0j}(1 - \bx\transpose\bx_{0j})}\right\}x_{0jk},\quad k = 1,2,\ldots,d,
\]
where $\bx_{0j} = [x_{0j1},\ldots,x_{0jd}]\transpose\in\mathbb{R}^d$. 
Simple algebra shows that
\begin{align*}
\frac{\partial^2M_n}{\partial\bx\partial\bx\transpose}(\bx)
& = -\frac{1}{n}\sum_{j\neq i}^n\frac{\bx_{0j}\bx_{0j}\transpose}{\bx\transpose\bx_{0j}(1 - \bx\transpose\bx_{0j})} - \frac{1}{n}\sum_{j\neq i}^n\frac{(A_{ij} - \bx\transpose\bx_{0j})(1 - 2\bx\transpose\bx_{0j})\bx_{0j}\bx_{0j}\transpose}{\{\bx\transpose\bx_{0j}(1 - \bx\transpose\bx_{0j})\}^2},\\
\frac{\partial^2M}{\partial\bx\partial\bx\transpose}(\bx)
& = -\frac{1}{n}\sum_{j\neq i}^n\left\{\frac{(\bx\transpose\bx_{0j})^2 - 2\bx\transpose\bx_{0i}\bx\transpose\bx_{0j} + \bx\transpose\bx_{0i}}{\bx\transpose\bx_{0j}(1 - \bx\transpose\bx_{0j})}\right\}\bx_{0j}\bx_{0j}\transpose,
\\
\frac{\partial^2\Psi_{nk}}{\partial\bx\partial\bx\transpose}(\bx)
& = \frac{1}{n}\sum_{j\neq i}^n\frac{x_{0jk}(1 - 2\bx\transpose\bx_{0j})}{\{\bx\transpose\bx_{0j}(1 - \bx\transpose\bx_{0j})\}^2}\bx_{0j}\bx_{0j}\transpose
+ \frac{1}{n}\sum_{j\neq i}^n\frac{x_{0jk}\{(1 - 2\bx\transpose\bx_{0j}) + 2(A_{ij} - \bx\transpose\bx_{0j})\}}{\{\bx\transpose\bx_{0j}(1 - \bx\transpose\bx_{0j})\}^2}\bx_{0j}\bx_{0j}\transpose
\\&\quad
+ \frac{1}{n}\sum_{j\neq i}^n\frac{x_{0jk}\{2(A_{ij} - \bx\transpose\bx_{0j})(1 - 2\bx\transpose\bx_{0j})^2\}}{\{\bx\transpose\bx_{0j}(1 - \bx\transpose\bx_{0j})\}^3}\bx_{0j}\bx_{0j}\transpose.
\end{align*}
Clearly, $\calX(\delta)$ is compact and $M_n(\bx)$ is continuous. Therefore $\widehat\bx_i^{(\mathrm{MLE})} = \argmax_{\bx\in\calX(\delta)}M_n(\bx)$ exists with probability one. Furthermore, by Shannon's lemma (see, for example, Lemma 2.2.1 in \citealp{bickel2015mathematical}), we know that $M(\bx)$ is maximized at $\bx = \bx_{0i}$. 
Since $\bx\transpose\bx_{0j}\in[\delta, 1- \delta]$ for all $j\in[n]$, implying that
\begin{align*}
-\frac{\partial^2M}{\partial\bx\partial\bx\transpose}(\bx) &= 
\frac{1}{n}\sum_{j\neq i}^n\left\{\frac{(\bx\transpose\bx_{0j})^2 - 2\bx\transpose\bx_{0i}\bx\transpose\bx_{0j} + \bx\transpose\bx_{0i}}{(\bx\transpose\bx_{0j})^2(1 - \bx\transpose\bx_{0j})^2}\right\}\bx_{0j}\bx_{0j}\transpose\\
&\succ\frac{1}{n}\sum_{j\neq i}^n\left\{\frac{(\bx\transpose\bx_{0j})^2 - 2\bx\transpose\bx_{0i}\bx\transpose\bx_{0j} + (\bx\transpose\bx_{0i})^2}{(\bx\transpose\bx_{0j})^2(1 - \bx\transpose\bx_{0j})^2}\right\}\bx_{0j}\bx_{0j}\transpose\\
& = \frac{1}{n}\sum_{j\neq i}^n\left\{\frac{(\bx\transpose\bx_{0j} - \bx\transpose\bx_{0i})^2}{(\bx\transpose\bx_{0j})^2(1 - \bx\transpose\bx_{0j})^2}\right\}\bx_{0j}\bx_{0j}\transpose\succeq \mathbf{O},
\end{align*}
it follows that $\bx_{0i}$ is the unique maximizer of $M$ because $-M$ is strictly convex, and $\calX(\delta)$ is convex and compact. 
Observe that the function $M$ (implicitly) depends on $n$. Since the gradient of $M$ is
\begin{align*}
\sup_{n\in\mathbb{N}_+}\sup_{\bx\in\calX}\left\|\frac{\partial M}{\partial \bx}(\bx)\right\|_2 \leq \sup_{n\in\mathbb{N}_+}\frac{1}{n}\sum_{j\neq i}\sup_{\bx\in\calX}\left\| \frac{(\bx_{0i} - \bx)\transpose\bx_{0j}}{\bx\transpose\bx_{0j}(1 - \bx\transpose\bx_{0j})}\bx_{0j} \right\|_2\leq \frac{2}{\delta^2}<\infty.
\end{align*}
Therefore the function class $\{M(\bx)\}_n$ (each function $M$ depends on $n$ implicitly) is equicontinuous. We also observe that 
\begin{align*}
\lim_{n\to\infty}M(\bx)
& = \lim_{n\to\infty}\frac{1}{n}\sum_{j\neq i}^n\{\bx_{0i}\transpose\bx_{0j}\log(\bx\transpose\bx_{0j}) + (1 - \bx_{0i}\transpose\bx_{0j})\log(1 - \bx\transpose\bx_{0j})\}\\
& = \bar{M}(\bx):=\int_\calX\left\{\bx_{0i}\transpose\bx_1\log(\bx\transpose\bx_1) + (1 - \bx_{0i}\transpose\bx_1)\log(1 - \bx\transpose\bx_1)\right\}F(\mathrm{d}\bx_1)
\end{align*}
converges for all $\bx\in\calX(\delta)$, and $\calX(\delta)$ is a compact subset of $\mathbb{R}^d$. Hence, by the Arzela-Ascoli theorem, the convergence $M\to \bar{M}$ as $n\to\infty$ is also uniform. Note that the Hessian of $-M$ is strictly positive definite. Also note that
\begin{align*}
&\sup_{\bx_1,\bx\in \calX}\left\|\frac{\partial}{\partial\bx}\left\{\bx_{0i}\transpose\bx_1\log(\bx\transpose\bx_1) + (1 - \bx_{0i}\transpose\bx_1)\log(1 - \bx\transpose\bx_1)\right\}\right\|_2\\
&\quad = \sup_{\bx_1,\bx\in \calX}\left\| \frac{(\bx_{0i} - \bx)\transpose\bx_1}{\bx\transpose\bx_1(1 - \bx\transpose\bx_1)}\bx_1 \right\|_2\leq \frac{2}{\delta^2} < \infty, 
\\
&\sup_{\bx_1,\bx\in \calX}\left\|\frac{\partial^2}{\partial\bx\partial\bx}\left\{\bx_{0i}\transpose\bx_1\log(\bx\transpose\bx_1) + (1 - \bx_{0i}\transpose\bx_1)\log(1 - \bx\transpose\bx_1)\right\}\right\|_2\\
&\quad = \sup_{\bx_1,\bx\in \calX}\left\| \frac{(\bx\transpose\bx_{1})^2 - 2\bx\transpose\bx_{0i}\bx\transpose\bx_1 + \bx\transpose\bx_{0i}}{\{\bx\transpose\bx_1(1 - \bx\transpose\bx_1)\}^2 }\bx_1\bx_1\transpose \right\|_2\leq \frac{4}{\delta^4} < \infty, 
\end{align*}
and hence, by the Lebesgue dominating convergence theorem, the negative Hessian of $\bar{M}$ equals
\begin{align*}
-\frac{\partial^2\bar{M}(\bx)}{\partial\bx\partial\bx\transpose}(\bx)
& = \int_\calX\frac{(\bx\transpose\bx_{1})^2 - 2\bx\transpose\bx_{0i}\bx\transpose\bx_1 + \bx\transpose\bx_{0i}}{\{\bx\transpose\bx_1(1 - \bx\transpose\bx_1)\}^2 }\bx_1\bx_1\transpose F(\mathrm{d}\bx_1),
\end{align*}
which is also strictly positive definite. Hence $-\bar{M}$ is also strictly convex, and the maximizer for $\bar{M}$ is unique in $\calX(\delta)$ because $\calX(\delta)$ is convex and compact. Therefore, we apply Lemma \ref{lemma:real_analysis_lemma} to obtain the following conclusion: For any $\eps > 0$, there exists an $\eta(\eps) > 0$ that depends on $\eps > 0$ but does not depend on $n$, and a positive integer $N = N(\eps)$ that depends on $\eps > 0$, such that
\begin{align}\label{eqn:M_function_iso_maximum}
\sup_{\|\bx - \bx_{0i}\| > \eps} M(\bx) + \eta(\eps) \leq M(\bx_{0i})
\end{align}
for all $n > N(\eps)$.


\noindent
We first claim that 
\begin{align}\label{eqn:uniform_LLN_likelihood}
\sup_{\bx\in\calX(\delta)}|M_n(\bx) - M(\bx)|\overset{\prob_0}{\to} 0.
\end{align}
Define a stochastic process $\{J(\bx) = M_n(\bx) - M(\bx):\bx\in\calX(\delta)\}$. Since for any $\bx_1,\bx_2\in\calX(\delta)$, there exists a constant $K_\delta$ only depending on $\delta > 0$, such that
\begin{align*}
\left|\log\left(\frac{\bx_1\transpose\bx_{0j}}{1 - \bx_1\transpose\bx_{0j}}\right) - \log\left(\frac{\bx_2\transpose\bx_{0j}}{1 - \bx_2\transpose\bx_{0j}}\right)\right|
&\leq \sup_{\bx\in\calX(\delta),j\in[n]}\left\|\frac{\partial}{\partial\bx}\log\left(\frac{\bx\transpose\bx_{0j}}{1 - \bx\transpose\bx_{0j}}\right)\right\|\|\bx_1 - \bx_2\|
\\& 
\leq K_\delta\|\bx_1 - \bx_2\|,
\end{align*}
it follows from Hoeffding's inequality that
\[
\prob_0\left(|J(\bx_1) - J(\bx_2)| > t\right)\leq 2\exp\left(-\frac{2nt^2}{K_\delta^2\|\bx_1 - \bx_2\|^2}\right),
\]
implying that $J(\cdot)$ is a sub-Gaussian process with respect to $K_\delta n^{-1/2}\|\cdot\|$. Hence the packing entropy can also be bounded: There exists some large constant $C > 0$, such that
\[
\log \calD\left(\eps, \calX(\delta), \frac{K_\delta}{\sqrt{n}}\|\cdot\|\right)\leq d\log\left(\frac{C}{\eps\sqrt{n}}\right).
\]
Hence, by the fact that $\sup_{\bx_1,\bx_2\in\calX(\delta)}K_\delta n^{-1/2}\|\bx_1 - \bx_2\|\leq cn^{-1/2}$ for some constant $c\in (0, C)$, a maximum inequality for sub-Gaussian process (see, for example, Corollary 8.5 in \citealp{kosorok2008introduction}), and the change of variable $u = \log\{C/(\eps\sqrt{n})\}$, we have
\begin{align*}
\expect_0\left(\sup_{\bx\in\calX(\delta)}|J(\bx)|\right)
&\lesssim \expect_0(|J(\bx_{0i})|) + \int_0^{cn^{-1/2}}\sqrt{\log\calD\left(\eps, \calX(\delta), \frac{M}{\sqrt{n}}\|\cdot\|\right)}\mathrm{d}\eps\\
&\lesssim \sqrt{\var_0(J(\bx_{0i}))} + \int_0^{cn^{-1/2}}\sqrt{\log\frac{C}{\eps\sqrt{n}}}\mathrm{d}\eps\\
&= \left\{\frac{1}{n^2}\sum_{j\neq i}^n\bx_{0i}\transpose\bx_{0j}(1 - \bx_{0i}\transpose\bx_{0j})\log\frac{\bx_{0i}\transpose\bx_{0j}}{1 - \bx_{0i}\transpose\bx_{0j}}\right\}^{1/2} 
 + \frac{C}{\sqrt{n}}\int_{\log \frac{C}{c}}^\infty\sqrt{u}\mathrm{e}^{-u}\mathrm{d}u\to 0
\end{align*}
as $n\to\infty$. Therefore we conclude that $\sup_{\bx\in\calX(\delta)}|J(\bx)| = o_{\prob_0}(1)$. 

\noindent
In the proof below we shall drop the superscript ${(\mathrm{MLE})}$ from $\widehat\bx_i^{(\mathrm{MLE})}$ and write $\widehat\bx_i = \widehat\bx_i^{(\mathrm{MLE})}$ for short. 
We next use the claim \eqref{eqn:uniform_LLN_likelihood} to show that $\widehat\bx_i$ is consistent for $\bx_{0i}$. The proof here is quite similar to that of Theorem 5.7 in \cite{van2000asymptotic} and presented here for completeness. In fact, this implies that $M_n(\bx_{0i})- M(\bx_{0i})\overset{\prob_0}{\to}0$. Furthermore, $\widehat\bx_i$ is the maximizer of $M_n$, implying that
\begin{align*}
M(\bx_{0i}) - M(\widehat\bx_i) &=  M_n(\bx_{0i}) + o_{\prob_0}(1) - M(\widehat\bx_i)\leq M_n(\widehat\bx_i) - M(\widehat\bx_i) + o_{\prob_0}(1)\\
&\leq \sup_{\bx\in\calX(\delta)}|J(\bx)| + o_{\prob_0}(1) = o_{\prob_0}(1). 
\end{align*}
This shows that $\prob_0(M(\bx_{0i}) - M(\widehat\bx_i) \geq \eta)\to 0$ for all $\eta > 0$. Recall that by \eqref{eqn:M_function_iso_maximum} for all $\eps > 0$, there exists some $\eta(\eps) > 0$ not depending on $n$, such that $\|\widehat\bx - \bx_{0i}\| > \eps$ implies $ M(\widehat\bx_i) \leq M(\bx_{0i}) - \eta(\eps)$, although the function $M(\bx_{0i})$ implicitly depends on $n$. Namely, for all $\eps > 0$, there exists some $\eta = \eta(\eps) > 0$ such that 
\[
\prob_0\left(\|\widehat\bx_i - \bx_{0i}\| > \eps\right)\leq \prob_0\left\{M(\bx_{0i}) - M(\widehat\bx_i) \geq \eta(\eps)\right\} \to 0
\]
as $n\to\infty$. This completes the proof of consistency of $\widehat\bx_i$ for $\bx_{0i}$. 

\noindent
We finally show the asymptotic normality of $\widehat\bx_i$. 
Since $\widehat\bx_i$ is consistent for $\bx_{0i}$, it follows that with probability tending to one, $\widehat\bx_i$ is in the interior of $\calX(\delta)$ since $\bx_{0i}$ is. Assume this event occurs. By Taylor's expansion, we have, for $k = 1,2,\ldots,d$, that
\[
0 = \Psi_{nk}(\widehat\bx_i) = \Psi_{nk}(\bx_{0i}) + \frac{\partial\Psi_{nk}}{\partial\bx\transpose}(\bx_{0i})(\widehat\bx_i - \bx_{0i}) + \frac{1}{2}(\widehat\bx_i - \bx_{0i})\transpose\left\{\frac{\partial^2\Psi_{nk}}{\partial\bx\partial\bx\transpose}(\widetilde\bx_k)\right\}(\widehat\bx_i - \bx_{0i}),
\]
where $\widetilde\bx_k$ lies on the line segment linking $\bx_{0i}$ and $\widehat\bx_{0i}$. Since for any $\bx\in\calX(\delta)$, 
\begin{align*}
\left\|\frac{\partial^2\Psi_{nk}}{\partial\bx\partial\bx\transpose}(\bx)\right\|
&\leq \frac{1}{n}\sum_{j\neq i}^n\frac{\{1 + 2(1 - \delta)\}\|\bx_{0j}\|^2}{\delta^2(1 - \delta)^2}
 + \frac{1}{n}\sum_{j\neq i}^n\frac{\{(3 - 2\delta) + 2(2 - \delta)\}\|\bx_{0j}\|^2}{\delta^2(1 - \delta)^2}
 \\&\quad 
 + \frac{1}{n}\sum_{j\neq i}^n\frac{2(2 - \delta)(3 - 2\delta)^2\|\bx_{0j}\|^2}{\delta^3(1 - \delta)^3}
\lesssim \frac{1}{n}\|\bX_0\|_{\mathrm{F}}^2\leq 1,
\end{align*}
it follows that the Hessian of $\Psi_{nk}(\widetilde \bx)$ is bounded in probability. Observe that, 
\begin{align*}
\expect_0\left\{\frac{\partial\bPsi_n(\bx_{0i})}{\partial\bx\transpose}\right\} &= -\frac{1}{n}\sum_{j\neq i}^n\frac{\bx_{0j}\bx_{0j}\transpose}{\bx_{0i}\transpose\bx_{0j}(1 - \bx_{0i}\transpose\bx_{0j})},
\end{align*}
and for any $s,t\in[d]$,
\begin{align*}
\var_0\left\{\frac{\partial\Psi_{ns}(\bx_{0i})}{\partial x_t}\right\} &= \frac{1}{n^2}\sum_{j\neq i}^n\frac{(1 - 2\bx_{0i}\transpose\bx_{0j})^2(x_{0js}x_{0jt})^2}{\{\bx_{0i}\transpose\bx_{0j}(1 - \bx_{0i}\transpose\bx_{0j})\}^3}\to 0
\end{align*}
as $n\to\infty$, we obtain from the law of large numbers that
\[
\frac{\partial\bPsi_n(\bx_{0i})}{\partial\bx\transpose} = -\bG_n(\bx_{0i}) + o_{\prob_0}(1).
\]
Therefore, we conclude from the Taylor's expansion and $\widehat\bx_i - \bx_{0i} = o_{\prob_0}(1)$ that
\begin{align*}
-\bPsi_n(\bx_{0i}) 
& = \left\{-\bG_n(\bx_{0i}) + o_{\prob_0}(1) + \frac{1}{2}(\widehat\bx_i - \bx_{0i})\transpose O_{\prob_0}(1)\right\}(\widehat\bx_i - \bx_{0i})\\
&= \left\{-\bG_n(\bx_{0i}) + o_{\prob_0}(1)\right\}(\widehat\bx_i - \bx_{0i}). 
\end{align*}
Namely, 
\[
\sqrt{n}(\widehat\bx_i - \bx_{0i}) = \{\bG_n(\bx_{0i}) + o_{\prob_0}(1)\}^{-1}\left\{\frac{1}{\sqrt{n}}\sum_{j\neq i}^n\frac{(A_{ij} - \bx_{0i}\transpose\bx_{0j})\bx_{0j}}{\bx_{0i}\transpose\bx_{0j}(1 - \bx_{0i}\transpose\bx_{0j})}\right\}. 
\]
Observe that
\begin{align*}
\sum_{j\neq i}^n\expect_0\left\{\left\|\frac{1}{\sqrt{n}}\frac{(A_{ij} - \bx_{0i}\transpose\bx_{0j})\bx_{0j}}{\bx_{0i}\transpose\bx_{0j}(1 - \bx_{0i}\transpose\bx_{0j})}\right\|^3\right\}
&\leq \frac{1}{n^{3/2}}\sum_{j\neq i}^n\frac{(2 - \delta)^3\|\bx_{0j}\|^3}{\{\delta(1 - \delta)\}^3}\to 0,\\
\var_0\left(\frac{1}{\sqrt{n}}\sum_{j\neq i}^n\frac{(A_{ij} - \bx_{0i}\transpose\bx_{0j})\bx_{0j}}{\bx_{0i}\transpose\bx_{0j}(1 - \bx_{0i}\transpose\bx_{0j})}\right)
& = \bG_n(\bx_{0i}) \to \bG(\bx_{0i})\quad\text{as }n\to\infty,
\end{align*}
it follows from Lyapunov's central limit theorem that $\sqrt{n}(\widehat\bx_i - \bx_{0i})\overset{\calL}{\to}\mathrm{N}(\zero,\bG(\bx_{0i})^{-1})$. The proof is thus completed.
\end{proof}


\subsection{Proof of Theorem \ref{thm:OSE_single_vertex}} 
\label{sec:proof_of_theorem_thm:ose_single_vertex}

\begin{proof}[\bf Proof of Theorem \ref{thm:OSE_single_vertex}]
The idea of the proof is very similar to that of Theorem 5.45 in \cite{van2000asymptotic}. Using the notation there (which coincides with the notation in Section \ref{sec:proof_of_theorem_thm:one_dimensional_mle}), we denote
\begin{align*}
\bPsi_n(\bx) 
&= \frac{1}{n}\sum_{j\neq i}^n\frac{(A_{ij} - \bx\transpose\bx_{0j})\bx_{0j}}{\bx\transpose\bx_{0j}(1 - \bx\transpose\bx_{0j})},
\quad
\dot{\bPsi}_0 
= \frac{1}{n}\sum_{j\neq i}^n\frac{\bx_{0j}\bx_{0j}\transpose}{\bx_{0i}\transpose\bx_{0j}(1 - \bx_{0i}\transpose\bx_{0j})},
\quad
\dot{\bPsi}_{n,0} 
= \frac{1}{n}\sum_{j\neq i}^n\frac{\bx_{0j}\bx_{0j}\transpose}{\widetilde\bx_{i}\transpose\bx_{0j}(1 - \widetilde\bx_{i}\transpose\bx_{0j})}.
\end{align*}
The two main ingredients to prove the asymptotic normality of the one-step estimator are:
\begin{itemize}
  \item[(i)] $\dot{\bPsi}_{n,0} - \dot{\bPsi}_{0} \overset{\prob_0}{\to} 0$. To prove this convergence result, we denote
  \[
  \bGamma_j(\bx) = \frac{\bx_{0j}\bx_{0j}\transpose}{\bx\transpose\bx_{0j}(1 - \bx\transpose\bx_{0j})}.
  \]
  Denote the $(k,l)$-th element of $\bGamma_j$ by $\Gamma_{jkl}$. Clearly, 
  \[
  \frac{\partial\Gamma_{jkl}(\bx)}{\partial\bx} = \frac{(1 - 2\bx\transpose\bx_{0j})\bx_{0j}x_{0jk}x_{jl}}{\{\bx\transpose\bx_{0j}(1 - \bx\transpose\bx_{0j})\}^2},
  \]
  where $\bx_{0j} = [x_{0j1},\ldots,x_{0jd}]\transpose$. 
  This implies that there exists some $\eps > 0$, such that for all $\bx$ in a neighborhood of $\bx_{0i}$ within radius $\eps$, denoted by $B(\bx_{0i},\eps) = \{\bx\in\mathbb{R}^d:\|\bx - \bx_{0i}\| < \eps\}$, 
  \[
  \max_{j\in[n]}\sup_{\bx\in B(\bx_{0i},\eps)}\frac{1}{n}\sum_{k,l\in[d]}
  \left\|
  \frac{\partial\Gamma_{jkl}(\bx)}{\partial\bx}
  \right\|_{\mathrm{F}} \leq C
  \]
  for some constant $C$ not depending on $n$. 
  Therefore, we have, over the event $\{\widetilde\bx_i \in B(\bx_{0i}, \eps)\}$, that
  \begin{align*}
  \|\dot{\bPsi}_{n,0} - \dot{\bPsi}_0\|_{\mathrm{F}}
  &\leq \frac{1}{n}\sum_{j\neq i}\left\|\bGamma_j(\widetilde\bx_i) - \bGamma_j(\bx_{0i})\right\|_{\mathrm{F}}\\
  &\leq \frac{1}{n}\sum_{j\neq i}\left\{\max_{j\in[n]}\sup_{\bx\in B(\bx_{0i},\eps)}\sum_{k,l\in[d]}
  \left\|
  \frac{\partial\Gamma_{jkl}(\bx)}{\partial\bx}
  \right\|_{\mathrm{F}}\right\}\|\widetilde\bx_i - \bx_{0i}\|_2\\
  &\leq C\|\widetilde\bx_i - \bx_{0i}\|_2.
  \end{align*}
  Then, for any $t > 0$, we have
  \begin{align*}
  \prob_0\left(\|\dot{\bPsi}_{n,0} - \dot{\bPsi}_0\|_{\mathrm{F}} > t\right)
  & = \prob_0\left(\|\dot{\bPsi}_{n,0} - \dot{\bPsi}_0\|_{\mathrm{F}} > t, \widetilde\bx_i\in B(\bx_{0i}, \eps)\right)\\
  &\quad + 
  \prob_0\left(\|\dot{\bPsi}_{n,0} - \dot{\bPsi}_0\|_{\mathrm{F}} > t, \widetilde\bx_i\notin B(\bx_{0i}, \eps)\right)\\
  &\leq  \prob_0\left(C\|\widetilde\bx_i - \bx_{0i}\|_2 > t, \widetilde\bx_i\in B(\bx_{0i}, \eps)\right)
  + \prob_0(\|\widetilde\bx_i - \bx_{0i}\|_2 \geq \eps)\\
  &\leq \prob_0\left(\|\widetilde\bx_i - \bx_{0i}\|_2 \geq \frac{t}{C}\right) + \prob_0(\|\widetilde\bx_i - \bx_{0i}\|_2 \geq \eps)
  \to 0,
  \end{align*}
  where we have used the fact that $\widetilde\bx_{i} - \bx_{0i} = O_{\prob_0}(n^{-1/2})$ in the last line of the previous display. This completes the proof of (i).
  \item[(ii)] 
    $\sqrt{n}\|
    \bPsi_n(\widetilde{\bx}_{i}) - \bPsi_n(\bx_{0i}) - \dot{\bPsi}_0(\widetilde{\bx}_{i} - \bx_{0i})
    \| \overset{\prob_0}{\to} 0$. To prove this convergence result, we need to argue that $\partial{\bPsi}_n(\bx_{0i})/\partial\bx\transpose$ is sufficiently close to $\dot{\bPsi}_0$, and the Taylor expansion of $\bPsi_n$ at $\bx_{0i}$ is valid. We first argue that $\partial{\bPsi}_n(\bx_{0i})/\partial\bx\transpose$ is sufficiently close to $\dot{\bPsi}_0$. By the proof of Theorem \ref{thm:one_dimensional_MLE}, 
    \[
    \frac{\partial\bPsi_n\transpose}{\partial\bx}(\bx) = - \frac{1}{n}\sum_{j\neq i}^n\frac{\bx_{0j}\bx_{0j}\transpose}{\bx\transpose\bx_{0j}(1 - \bx\transpose\bx_{0j})} - \frac{1}{n}\sum_{j\neq i}^n\frac{(A_{ij} - \bx\transpose\bx_{0j})(1 - 2\bx\transpose\bx_{0j})\bx_{0j}\bx_{0j}\transpose}{\{\bx\transpose\bx_{0j}(1 - \bx\transpose\bx_{0j})\}^2}.
    \]
    Hence
    \[
    \frac{\partial\bPsi_n\transpose}{\partial\bx}(\bx_{0i}) = \dot{\bPsi}_0 - \frac{1}{n}\sum_{j\neq i}^n\frac{(A_{ij} - \bx_{0i}\transpose\bx_{0j})(1 - 2\bx_{0i}\transpose\bx_{0j})\bx_{0j}\bx_{0j}\transpose}{\{\bx_{0i}\transpose\bx_{0j}(1 - \bx_{0i}\transpose\bx_{0j})\}^2}.
    \]
    Note that
    \begin{align*}
    \var_0\left[\vect\left\{
    \frac{\partial\bPsi_n\transpose}{\partial\bx}(\bx_{0i}) - \dot{\bPsi}_0
    \right\}
    \right]
    & = \frac{1}{n^2}\sum_{j\neq i}^n\frac{(1 - 2\bx_{0i}\transpose\bx_{0j})^2}{\{\bx_{0i}\transpose\bx_{0j}(1 - \bx_{0i}\transpose\bx_{0j})\}^3}\vect(\bx_{0j}\bx_{0j}\transpose)\vect(\bx_{0j}\bx_{0j}\transpose)\transpose\\
    & \asymp \frac{1}{n},
    \end{align*}
    where $\vect(\bSigma)$ is the vectorization of the matrix $\bSigma$ defined to be the vector formed by stacking the columns of $\bSigma$ consecutively. Furthermore,
    \begin{align*}
    \sum_{j\neq i}^n\expect_0\left[
    \left\|\frac{1}{\sqrt{n}}
    \frac{(A_{ij} - \bx_{0i}\transpose\bx_{0j})(1 - 2\bx_{0i}\transpose\bx_{0j})\bx_{0j}\bx_{0j}\transpose}{\{\bx_{0i}\transpose\bx_{0j}(1 - \bx_{0i}\transpose\bx_{0j})\}^2}
    \right\|_{\mathrm{F}}^3
    \right] \asymp\frac{1}{n^{1/2}}\to 0.
    \end{align*}
    Therefore, by Lyapunov's central limit theorem, 
    \[
    \frac{\partial\bPsi_n\transpose}{\partial\bx}(\bx_{0i}) - \dot{\bPsi}_0 = O_{\prob_0}\left(\frac{1}{\sqrt{n}}\right).
    \]
    We now establish the Taylor expansion of $\bPsi_n$ at $\bx_{0i}$. Let $\Psi_{nk}$ be the $k$th coordinate function of $\bPsi_n$, $k = 1,\ldots,d$. By the proof of Theorem \ref{thm:one_dimensional_MLE},
    \begin{align*}
    \frac{\partial^2\Psi_{nk}}{\partial\bx\partial\bx\transpose}(\bx)
    & = \frac{1}{n}\sum_{j\neq i}^n\frac{x_{0jk}(1 - 2\bx\transpose\bx_{0j})}{\{\bx\transpose\bx_{0j}(1 - \bx\transpose\bx_{0j})\}^2}\bx_{0j}\bx_{0j}\transpose
    \\&\quad 
    + \frac{1}{n}\sum_{j\neq i}^n\frac{x_{0jk}\{(1 - 2\bx\transpose\bx_{0j}) + 2(A_{ij} - \bx\transpose\bx_{0j})\}}{\{\bx\transpose\bx_{0j}(1 - \bx\transpose\bx_{0j})\}^2}\bx_{0j}\bx_{0j}\transpose
    \\&\quad
    + \frac{1}{n}\sum_{j\neq i}^n\frac{x_{0jk}\{2(A_{ij} - \bx\transpose\bx_{0j})(1 - 2\bx\transpose\bx_{0j})^2\}}{\{\bx\transpose\bx_{0j}(1 - \bx\transpose\bx_{0j})\}^3}\bx_{0j}\bx_{0j}\transpose.
    \end{align*}
    Therefore, there exists some $\eps > 0$, such that for all $\bx\in B(\bx_{0i}, \eps)$, we have
    \begin{align*}
    \Psi_{nk}(\bx) - \Psi_{nk}(\bx_{0i}) = \frac{\partial\Psi_{nk}}{\partial\bx\transpose}(\bx_{0i})(\bx - \bx_{0i}) + \frac{1}{2}(\bx - \bx_{0i})\transpose\left\{\frac{\partial^2\Psi_{nk}}{\partial\bx\partial\bx\transpose}(\widetilde\btheta_k)\right\}(\bx - \bx_{0i}),
    \end{align*}
    where $\widetilde\btheta_k$ lies on the line segment linking $\bx_{0i}$ and $\bx$. Again, by the proof of Theorem \ref{thm:one_dimensional_MLE}, we see that
    \[
    \max_{k\in [d]}\sup_{\bx\in B(\bx_{0i}, \eps)}\left\|\frac{\partial^2\Psi_{nk}}{\partial\bx\partial\bx\transpose}(\bx)\right\|\leq C
    \]
    for some constant $C > 0$ not depending on $n$. Therefore, 
    \begin{align}\label{eqn:Taylor_expansion_one_dimensional_OSE}
    \left\|\bPsi_n(\bx) - \bPsi_n(\bx_{0i}) - \frac{\partial\bPsi_{n}}{\partial\bx\transpose}(\bx_{0i})(\bx - \bx_{0i})\right\|_{\mathrm{F}}
    \leq Cd\|\bx - \bx_{0i}\|_2^2
    \end{align}
    whenever $\bx\in B(\bx_{0i},  \eps)$. 
    Now we are able to prove (ii). Write
    \begin{align*}
    &
    \sqrt{n}\|
    \bPsi_n(\widetilde{\bx}_{i}) - \bPsi_n(\bx_{0i}) - \dot{\bPsi}_0(\widetilde{\bx}_{i} - \bx_{0i})\|_2\\
    &\quad \leq \sqrt{n}\left\|
    \bPsi_n(\widetilde{\bx}_{i}) - \bPsi_n(\bx_{0i}) - \frac{\partial\bPsi_n}{\partial\bx\transpose}(\bx_{0i})(\widetilde{\bx}_{i} - \bx_{0i})\right\|_{\mathrm{{2}}}
     + \sqrt{n}\left\|\frac{\partial\bPsi_n\transpose}{\partial\bx}(\bx_{0i}) - \dot{\bPsi}_0\right\|_2\|\widetilde\bx_i - \bx_{0i}\|_2\\
    &\quad = \sqrt{n}\left\|
    \bPsi_n(\widetilde{\bx}_{i}) - \bPsi_n(\bx_{0i}) - \frac{\partial\bPsi_n}{\partial\bx\transpose}(\bx_{0i})(\widetilde{\bx}_{i} - \bx_{0i})\right\|_{\mathrm{{2}}} + O_{\prob_0}(n^{-1/2}).
    \end{align*}
    Now over the event $\{\widetilde\bx_i\in B(\bx_{0i}, \eps)\}$, we can apply the Taylor expansion \eqref{eqn:Taylor_expansion_one_dimensional_OSE} and write
    \[
    \left\|
    \bPsi_n(\widetilde{\bx}_{i}) - \bPsi_n(\bx_{0i}) - \frac{\partial\bPsi_n}{\partial\bx\transpose}(\bx_{0i})(\widetilde{\bx}_{i} - \bx_{0i})\right\|_{\mathrm{{2}}}\leq Cd\|\widetilde\bx_{i} - \bx_{0i}\|_2^2.
    \]
    Therefore, for any $t > 0$,
    \begin{align*}
    &\prob_0\left(
    \sqrt{n}\left\|
    \bPsi_n(\widetilde{\bx}_{i}) - \bPsi_n(\bx_{0i}) - \frac{\partial\bPsi_n}{\partial\bx\transpose}(\bx_{0i})(\widetilde{\bx}_{i} - \bx_{0i})\right\|_{\mathrm{{2}}} > t
    \right)\\
    &\quad \leq \prob_0\left(
    \sqrt{n}\left\|
    \bPsi_n(\widetilde{\bx}_{i}) - \bPsi_n(\bx_{0i}) - \frac{\partial\bPsi_n}{\partial\bx\transpose}(\bx_{0i})(\widetilde{\bx}_{i} - \bx_{0i})\right\|_{\mathrm{{2}}} > t, \widetilde\bx_i\in B(\bx_{0i}, \eps)
    \right) \\
    &\quad\quad + \prob_0\left(\|\widetilde\bx_i - \bx_{0i}\|_2 > \eps\right)\\
    &\quad \leq  \prob_0\left(Cd\|\widetilde\bx_i - \bx_{0i}\|_2^2 > \frac{t}{\sqrt{n}}\right) + \prob_0\left(\|\widetilde\bx_i - \bx_{0i}\|_2 > \eps\right)\to 0,
    \end{align*}
    where we have applied the fact that $\widetilde\bx_i - \bx_{0i} = O_{\prob_0}(n^{-1/2})$ to the last line of the previous display. The claim (ii) is thus established. 
\end{itemize}
We are now in a position to prove Theorem \ref{thm:OSE_single_vertex}. By definition and (ii), we have
\begin{align*}
\bPsi_{n,0}\sqrt{n}(\widehat\bx_i^{\mathrm{(OS)}} - \bx_{0i})
& = \bPsi_{n,0}\sqrt{n}(\widetilde\bx_i - \bPsi_{n,0}^{-1}\bPsi_n(\widetilde\bx_i) - \bx_{0i})\\
& = \bPsi_{n,0}\sqrt{n}(\widetilde\bx_i - \bx_{0i}) - \sqrt{n}\{\bPsi_n(\widetilde\bx_i) - \bPsi_n(\bx_{0i})\} + \sqrt{n}\bPsi_n(\bx_{0i})\\
& = (\bPsi_{n,0}  - \dot{\bPsi}_0)\sqrt{n}(\widetilde\bx_i - \bx_{0i}) + \sqrt{n}\bPsi_n(\bx_{0i}) + o_{\prob_0}(1).
\end{align*}
By (i) and the fact that $\sqrt{n}(\widetilde\bx_i - \bx_{0i}) = O_{\prob_0}(1)$, we see that the first term on the right-hand side of the previous display is also $o_{\prob_0}(1)$. Therefore,
\begin{align*}
\bPsi_{n,0}\sqrt{n}(\widehat\bx_i^{\mathrm{(OS)}} - \bx_{0i})
& = \sqrt{n}\bPsi_n(\bx_{0i}) + o_{\prob_0}(1)
 = \frac{1}{\sqrt{n}}\sum_{j\neq i}^n\frac{(A_{ij} - \bx_{0i}\transpose\bx_{0j})\bx_{0j}}{\bx_{0i}\transpose\bx_{0j}(1 - \bx_{0i}\transpose\bx_{0j})} + o_{\prob_0}(1).
\end{align*}
Using the proof of Theorem \ref{thm:one_dimensional_MLE}, we see that the first term in the right-hand side of the above display converges to $\mathrm{N}(\zero, \bG(\bx_{0i}))$ in distribution. Therefore, by (i), the fact that $\dot{\bPsi}_0 \to \bG(\bx_{0i})$ as $n\to\infty$, and Slutsky's theorem, we conclude that 
\begin{align*}
\sqrt{n}(\widehat\bx_i^{\mathrm{(OS)}} - \bx_{0i}) & = 
\bPsi_{n,0}^{-1}\frac{1}{\sqrt{n}}\sum_{j\neq i}^n\frac{(A_{ij} - \bx_{0i}\transpose\bx_{0j})\bx_{0j}}{\bx_{0i}\transpose\bx_{0j}(1 - \bx_{0i}\transpose\bx_{0j})} + o_{\prob_0}(1)\overset{\calL}{\to} \mathrm{N}(\zero, \bG(\bx_{0i})^{-1}).
\end{align*}
\end{proof}


\section{Proofs of Theorems \ref{thm:asymptotic_normality_OS} and \ref{thm:convergence_OS}} 
\label{sec:proof_of_theorem_OS}

\subsection{Some Technical Lemmas for the One-Step Procedure} 
\label{sec:proofs_of_technical_lemmas}

Before proceeding to the proof of Theorem \ref{thm:asymptotic_normality_OS}, we first establish a collection of technical lemmas for bounding the remainder $\widehat\bR_i$ in \eqref{eqn:OS_expansion}. 


\begin{lemma}\label{lemma:two_to_infinity_error_ASE}
Let $\bA\sim\mathrm{RDPG}(\bX_0)$ and assume the conditions in Theorem \ref{thm:asymptotic_normality_OS} holds. Let an estimator $\widetilde\bX\in\mathbb{R}^{n\times d}$ satisfy the approximate linearization property \eqref{eqn:linearization_property} with an orthogonal matrix $\bW\in\mathbb{O}(d)$ (possibly depending on $n$). Then 
\begin{align*}
\|\widetilde\bX\bW - \rho_n^{1/2}\bX_0\|_{2\to\infty} = O_{\prob_0}\left(\frac{(\log n)^{(1\vee\omega)/2}}{\sqrt{n\rho_n}}\right).
\end{align*}

\begin{proof}[\bf Proof
]
The proof of this lemma is similar to that of Lemma 2 in \cite{tang2017}, except that we consider the case where a sparsity factor $\rho_n$ is taken into account, and the proof is presented here for the sake of completeness. 
Recall from \eqref{eqn:linearization_property} that
\[
\|\widetilde\bX\bW - \rho_n^{1/2}\bX_0\|_{2\to\infty}\leq \rho_n^{-1/2}\sqrt{d}\max_{i\in[n],k\in[d]}\left|\sum_{j = 1}^n(A_{ij} - \rho_n\bx_{0i}\transpose\bx_{0j})\zeta_{ijk}\right| +  \|\widetilde\bR\|_{\mathrm{F}}.
\]
where $\bzeta_{ij} = [\zeta_{ij1},\ldots,\zeta_{ijd}]\transpose\in\mathbb{R}^d$. 
By Hoeffding's inequality, the union bound, and the condition that $\sup_{i,j\in[n]}\|\bzeta_{ij}\|\lesssim 1/n$, for any $t > 0$, we have,
\begin{align*}
\prob_0\left(\max_{i\in[n],k\in[k]}\left|\sum_{j = 1}^n(A_{ij} - \rho_n\bx_{0i}\transpose\bx_{0j})\zeta_{ijk}\right| > t\right)
\leq 
2nd\exp\left(-\frac{2t^2}{\sum_{j = 1}^n\zeta_{ijk}^2}\right) = 2nd\exp\left\{-{Knt^2}\right\}
\end{align*}
for some constant $K > 0$. 
Therefore, for any $c > 0$, there exists some constant $C_c > 0$ and $n_c\in\mathbb{N}_+$, such that for all $n \geq n_c$.
\begin{align*}
\prob_0\left(\rho_n^{-1/2}\sqrt{d}\max_{i\in[n],k\in[d]}\left|\sum_{j = 1}^n(A_{ij} - \rho_n\bx_{0i}\transpose\bx_{0j})\zeta_{ijk}\right| > C_c\sqrt{\frac{\log n}{n\rho_n}}\right)\leq \frac{1}{n^c}. 
\end{align*}
This shows that 
\[
\rho_n^{-1/2}\sqrt{d}\max_{i\in[n],k\in[d]}\left|\sum_{j = 1}^n(A_{ij} - \rho_n\bx_{0i}\transpose\bx_{0j})\zeta_{ijk}\right| = O_{\prob_0}\left(\sqrt{\frac{\log n}{n\rho_n}}\right).
\]
The proof is completed by applying the condition that $\|\widetilde\bR\|_{\mathrm{F}} = O_{\prob_0}\left((n\rho_n)^{-1/2}(\log n)^{\omega/2}\right)$. 
\end{proof}

\end{lemma}

\begin{lemma}\label{lemma:R12k_analysis}
Let $\bA\sim\mathrm{RDPG}(\bX_0)$ with sparsity factor $\rho_n$, and assume the conditions of Theorem \ref{thm:asymptotic_normality_OS} hold. Let an estimator $\widetilde\bX\in\mathbb{R}^{n\times d}$ satisfy the approximate linearization property \eqref{eqn:linearization_property} with an orthogonal matrix $\bW\in\mathbb{O}(d)$ (possibly depending on $n$). Then
\[
\max_{i\in[n]}\left\|\frac{1}{n\sqrt{\rho_n}}\sum_{j = 1}^n\frac{\rho_n\bx_{0j}\bx_{0i}\transpose}{\bx_{0i}\transpose\bx_{0j}(1 - \rho_n\bx_{0i}\transpose\bx_{0j})}(\rho_n^{-1/2}\bW\transpose\widetilde\bx_j - \bx_{0j})\right\| = O_{\prob_0}\left(\frac{(\log n)^{(1\vee\omega)/2}}{n\rho_n^{1/2}}\right). 
\]
\begin{proof}
[\bf Proof]
Recall by condition \eqref{eqn:linearization_property} that for any $j\in[n]$,
\[
[\bW\transpose\widetilde\bx_j - \rho_n^{1/2}\bx_{0j}]_k = \rho_n^{-1/2}\sum_{a = 1}^n(A_{ja} - \rho_n\bx_{0j}\transpose\bx_{0a})\zeta_{iak} + \widetilde R_{jk}, \quad k = 1,2,\ldots,d,
\]
where $\bzeta_{ij} = [\zeta_{ij1},\ldots,\zeta_{ijd}]\transpose$. 
It follows that for $k = 1,2,\ldots,d$,
\begin{align*}
&\frac{1}{n\sqrt{\rho_n}}\sum_{j = 1}^n\sum_{s = 1}^d\frac{\rho_nx_{0jk}x_{0is}}{\bx_{0i}\transpose\bx_{0j}(1 - \rho_n\bx_{0i}\transpose\bx_{0j})}[\rho_n^{-1/2}\bW\transpose\widetilde\bx_j - \bx_{0j}]_s\\
&\quad = \frac{1}{n\sqrt{\rho_n}}\sum_{s = 1}^d\sum_{j = 1}^n\sum_{a = 1}^n\frac{x_{0jk}x_{0is}}{\bx_{0i}\transpose\bx_{0j}(1 - \rho_n\bx_{0i}\transpose\bx_{0j})}(A_{ja} - \rho_n\bx_{0j}\transpose\bx_{0a})\zeta_{ias}\\
&\quad\quad + \frac{1}{n}\sum_{j = 1}^n\sum_{s = 1}^d\frac{x_{0jk}x_{0is}}{\bx_{0i}\transpose\bx_{0j}(1 - \rho_n\bx_{0i}\transpose\bx_{0j})}\widetilde R_{js}\\
&\quad = \frac{1}{n\sqrt{\rho_n}}\sum_{s = 1}^d\left\{\sum_{j < a}Z_{iksja} + \sum_{j > a}Z_{iksja} + \sum_{j = 1}^nZ_{iksjj}\right\} + \frac{1}{n}\sum_{j = 1}^n\sum_{s = 1}^d\frac{x_{0jk}x_{0is}}{\bx_{0i}\transpose\bx_{0j}(1 - \rho_n\bx_{0i}\transpose\bx_{0j})}\widetilde R_{js},
\end{align*}
where 
\[
Z_{iksja} = \frac{x_{0jk}x_{0is}}{\bx_{0i}\transpose\bx_{0j}(1 - \rho_n\bx_{0i}\transpose\bx_{0j})}(A_{ja} - \rho_n\bx_{0j}\transpose\bx_{0a})\zeta_{ias}.
\]
Observe that by Hoeffding's inequality and the union bound,
\begin{align*}
\prob_0\left(\frac{1}{n\sqrt{\rho_n}}\max_{i\in[n]}\left|\sum_{j < a}Z_{iksja}\right| > t\right)
&\leq 2n\exp\left[-2n^2\rho_nt^2\left\{\sum_{j < a}\left(\frac{\zeta_{ias}x_{0jk}x_{0is}}{\bx_{0i}\transpose{}\bx_{0j}(1 - \rho_n\bx_{0i}\transpose{}\bx_{0j})}\right)^2\right\}^{-1}\right]\\
&\leq 2n\exp\left(-Kn^2\rho_nt^2\right).
\end{align*}
This shows that
\[
\frac{1}{n\sqrt{\rho_n}}\max_{i\in[n]}\left|\sum_{j < a}Z_{iksja}\right| = O_{\prob_0}\left(\sqrt{\frac{\log n}{n^2\rho_n}}\right),
\]
and hence, a similar argument yields that
\[
\frac{1}{n\sqrt{\rho_n}}\max_{i\in[n]}\left|\sum_{j = 1}^n\sum_{a = 1}^nZ_{iksja}\right| = O_{\prob_0}\left(\sqrt{\frac{\log n}{n^2\rho_n}}\right).
\]
In addition, by the fact that $\|\widetilde\bR\|_{\mathrm{F}}^2 = O_{\prob_0}((n\rho_n)^{-1}(\log n)^\omega)$ we have
\begin{align*}
&\max_{i\in[n]}\left|\frac{1}{n}\sum_{j = 1}^n\sum_{s = 1}^d
\frac{x_{0jk}x_{0is}}{\bx_{0i}\transpose\bx_{0j}(1 - \rho_n\bx_{0i}\transpose\bx_{0j})}\widetilde R_{js}
\right|\\
&\quad\leq \frac{1}{n}\left[\sum_{j = 1}^n\sum_{s = 1}^d\max_{i\in[n]}\left\{\frac{x_{0jk}x_{0is}}{\bx_{0i}\transpose\bx_{0j}(1 - \rho_n\bx_{0i}\transpose\bx_{0j})}\right\}^2\right]^{1/2}\|\widetilde\bR\|_{\mathrm{F}}
\lesssim \frac{1}{\sqrt{n}}\|\widetilde\bR\|_{\mathrm{F}} = O_{\prob_0}\left(\frac{(\log n)^{\omega/2}}{n\rho_n^{1/2}}\right).
\end{align*}

Therefore, we conclude that
\begin{align*}
&\max_{i\in[n]}\left\|\frac{1}{n\sqrt{\rho_n}}\sum_{j = 1}^n\frac{\rho_n\bx_{0j}\bx_{0i}\transpose}{\bx_{0i}\transpose\bx_{0j}(1 - \rho_n\bx_{0i}\transpose\bx_{0j})}(\rho_n^{-1/2}\bW\transpose\widetilde\bx_j - \bx_{0j})\right\|\\
&\quad\lesssim \sum_{k = 1}^d\sum_{s = 1}^d\max_{i\in[n]}\left|\frac{1}{n\sqrt{\rho_n}}\sum_{j = 1}^n\sum_{a = 1}^nZ_{iksja}\right|
 + \sum_{k = 1}^d\max_{i\in[n]}\left|\frac{1}{n}\sum_{j = 1}^n\sum_{s = 1}^d
\frac{x_{0jk}x_{0is}}{\bx_{0i}\transpose\bx_{0j}(1 - \rho_n\bx_{0i}\transpose\bx_{0j})}\widetilde R_{js}
\right|
\\&\quad
= O_{\prob_0}\left(\frac{(\log n)^{(1\vee\omega)/2}}{n\rho_n^{1/2}}\right),
\end{align*}
and the proof is thus completed. 
\end{proof}
\end{lemma}

\begin{lemma}\label{lemma:R13k_analysis}
Let $\bA\sim\mathrm{RDPG}(\bX_0)$ with sparsity factor $\rho_n$ and assume the conditions of Theorem \ref{thm:asymptotic_normality_OS} hold.Let an estimator $\widetilde\bX\in\mathbb{R}^{n\times d}$ satisfy the approximate linearization property \eqref{eqn:linearization_property} with an orthogonal matrix $\bW\in\mathbb{O}(d)$ (possibly depending on $n$). 
Suppose $\{\balpha_{ijk}:i,j\in[n],k\in[d]\}$ is a collection of deterministic vectors in $\mathbb{R}^d$ with $\sup_{i,j\in[n],k\in[d]}\|\balpha_{ijk}\| < \infty$. Then
\[
\max_{i\in[n],k\in[d]}\frac{1}{n\sqrt{\rho_n}}\left|\sum_{j = 1}^n(A_{ij} - \rho_n\bx_{0i}\transpose\bx_{0j})\balpha_{ijk}\transpose(\rho_n^{-1/2}\bW\transpose\widetilde\bx_i - \bx_{0i})\right|= O_{\prob_0}\left(\frac{(\log n)^{1/2 + (1\vee\omega)/2}}{n\rho_n}\right).
\]
\begin{proof}[\bf Proof]
First observe that by Cauchy-Schwarz inequality, 
\begin{align*}
&
\left|\frac{1}{n\sqrt{\rho_n}}\sum_{j = 1}^n(A_{ij} - \rho_n\bx_{0i}\transpose\bx_{0j})\balpha_{ijk}\transpose(\rho_n^{-1/2}\bW\transpose\widetilde\bx_i - \bx_{0i})\right|
\\&\quad
\leq \frac{1}{\sqrt{\rho_n}}\|\bW\transpose\widetilde\bx_i - \rho_n^{1/2}\bx_{0i}\|\left\|\frac{1}{n\sqrt{\rho_n}}\sum_{j = 1}^n\balpha_{ijk}(A_{ij} - \rho_n\bx_{0i}\transpose\bx_{0j})\right\|.
\end{align*}
By Hoeffding's inequality and the union bound, for all $r = 1,2,\ldots,d$,
\begin{align*}
\prob_0\left(\max_{i\in[n],k\in[d]}\left|\frac{1}{n\sqrt{\rho_n}}\sum_{j = 1}^n[\balpha_{ijk}]_r(A_{ij} - \rho_n\bx_{0i}\transpose\bx_{0j})\right| > t\right)
&\leq
2nd\exp\left(-Kn\rho_nt^2\right)
\end{align*}
for some constant $K > 0$. This shows that
\[
\max_{i\in[n],k\in[d]}\left|\frac{1}{n\sqrt{\rho_n}}\sum_{j = 1}^n[\balpha_{ijk}]_r(A_{ij} - \rho_n\bx_{0i}\transpose\bx_{0j})\right| = O_{\prob_0}\left(\sqrt{\frac{\log n}{n\rho_n}}\right)
\]
Hence, we conclude from Lemma \ref{lemma:two_to_infinity_error_ASE} that
\begin{align*}
&\max_{i\in[n]}\left|\frac{1}{n\sqrt{\rho_n}}\sum_{j = 1}^n(A_{ij} - \rho_n\bx_{0i}\transpose\bx_{0j})\balpha_{ijk}\transpose(\rho_n^{-1/2}\bW\transpose\widetilde\bx_i - \bx_{0i})\right|\\
&\quad\lesssim \frac{\|\widetilde\bX\bW - \rho_n^{1/2}\bX_0\|_{2\to\infty}}{\sqrt{\rho_n}} \sum_{r = 1}^d\left|\max_{i\in[n],k\in[d]}\frac{1}{n\sqrt{\rho_n}}\sum_{j = 1}^n[\balpha_{ijk}]_r(A_{ij} - \rho_n\bx_{0i}\transpose\bx_{0j})\right|
 = O_{\prob_0}\left(\frac{(\log n)^{(1/2) + (1\vee\omega)/2}}{n\rho_n}\right).
\end{align*}
The proof is thus completed. 
\end{proof}
\end{lemma}

\begin{lemma}\label{Lemma:R14k_analysis}
Let $\bA\sim\mathrm{RDPG}(\bX_0)$ with sparsity factor $\rho_n$, and assume the conditions of Theorem \ref{thm:asymptotic_normality_OS} hold. 
Let an estimator $\widetilde\bX\in\mathbb{R}^{n\times d}$ satisfy the approximate linearization property \eqref{eqn:linearization_property} with an orthogonal matrix $\bW\in\mathbb{O}(d)$ (possibly depending on $n$). 
Suppose $\{\bbeta_{ijk}:i,j\in[n],k\in[d]\}$ is a collection of deterministic vectors in $\mathbb{R}^d$ such that $\sup_{i,j\in[n],k\in[d]}\|\bbeta_{ijk}\| < \infty$. Then for each individual $i\in[n]$,
\[
\left|\frac{1}{n\sqrt{\rho_n}}\sum_{j = 1}^n(A_{ij} - \rho_n\bx_{0i}\transpose\bx_{0j})(\rho_n^{-1/2}\bW\transpose\widetilde\bx_j - \bx_{0j})\transpose\bbeta_{ijk}\right| = O_{\prob_0}\left(\frac{(\log n)^{\omega/2}}{n\rho_n^{3/2}}\right)
\]
and
\[
\sum_{i = 1}^n\left\{\frac{1}{n\sqrt{\rho_n}}\sum_{j = 1}^n(A_{ij} - \rho_n\bx_{0i}\transpose\bx_{0j})(\rho_n^{-1/2}\bW\transpose\widetilde\bx_j - \bx_{0j})\transpose\bbeta_{ijk}\right\}^2 = O_{\prob_0}\left(\frac{(\log n)^\omega}{n\rho_n^{3}}\right).
\] 
\begin{proof}[\bf Proof]
Denote $\bbeta_{ijk} = [\beta_{ijk1},\ldots,\beta_{ijkd}]\transpose$. 
Recall the approximate linearization property \eqref{eqn:linearization_property} that
\[
[\bW\transpose\widetilde\bx_j - \rho_n^{1/2}\bx_{0j}]_s = \rho_n^{-1/2}\sum_{a = 1}^n(A_{ja} - \rho_n\bx_{0j}\transpose\bx_{0a})\zeta_{ias} + \widetilde R_{js},\quad s = 1,2,\ldots,d,
\]
where $\bzeta_{ia} = [\zeta_{ia1},\ldots,\zeta_{iad}]\transpose$. 
It follows that
\begin{align*}
Q_{ik} &:= 
\frac{1}{n\sqrt{\rho_n}}\sum_{j = 1}^n(A_{ij} - \rho_n\bx_{0i}\transpose\bx_{0j})(\rho_n^{-1/2}\bW\transpose\widetilde\bx_j - \bx_{0j})\transpose\bbeta_{ijk}\\
&
= \frac{1}{n\rho_n^{3/2}}\sum_{s = 1}^d\sum_{j = 1}^n\sum_{a = 1}^nz_{iksja}
+ \frac{1}{n{\rho_n}}\sum_{s = 1}^d\sum_{j = 1}^n(A_{ij} - \rho_n\bx_{0i}\transpose\bx_{0j})\beta_{ijks}\widetilde R_{js},
\end{align*}
where 
$z_{iksja} = \zeta_{ias}\beta_{ijks}(A_{ij} - \rho_n\bx_{0i}\transpose\bx_{0j})(A_{ja} - \rho_n\bx_{0j}\transpose\bx_{0a})$.
Clearly, 
\begin{align*}
\frac{1}{n^2\rho_n^3}\max_{i\in[n]}\expect_0\left\{\left(\sum_{s = 1}^d\sum_{j = 1}^n\sum_{a = 1}^nz_{iksja}\right)^2\right\}
& \lesssim \frac{1}{n^2\rho_n^3}\sum_{s = 1}^d\sum_{j = 1}^n\sum_{a = 1}^n\sum_{h = 1}^n\sum_{b = 1}^n\max_{i\in[n]}\expect_0(z_{iksja}z_{ikshb}).
\end{align*}
We now argue that the summation $\sum_{j = 1}^n\sum_{a = 1}^n\sum_{h = 1}^n\sum_{b = 1}^n\max_{i\in[n]}\max_{i\in[n]}\expect_0(z_{iksja}z_{ikshb})$ is upper bounded by $\sup_{i,j}\|\bzeta_{ij}\|^2n^2\rho_n$ up to a multiplicative constant. Note that as the indices $j,a,h,b$ ranging over $[n]$, $\expect_0(z_{iksja}z_{ikshb})$ is nonzero only if the cardinality of the collection of random variables $\{A_{ij}, A_{ih}, A_{aj}, A_{bh}\}$ is $2$ or $1$. These cases occur only if either one of the following cases happens:
\begin{enumerate}
  \item $A_{ij}$ and $A_{ih}$ are the same random variable, and $A_{aj}, A_{bh}$ are the same random variable. This happens only if one the following cases occur:
  \begin{enumerate}
    \item $(i,j) = (i,h),(a,j) = (b,h)\Rightarrow j = h$, $a = b$, and the number of terms is $O(n^2)$;
    \item $(i,j) = (h,i),(a,j) = (b,h)\Rightarrow i = j = h$, $a = b$, and the number of terms is $O(n)$;
    \item $(i,j) = (i,h),(a,j) = (h,b)\Rightarrow j = h = a = b$, and the number of terms is $O(n)$;
    \item $(i,j) = (h,i),(a,j) = (h,b)\Rightarrow i = j = h = a = b$, and the number of terms is $1$;
  \end{enumerate}
  \item $A_{ij}$ and $A_{aj}$ are the same random variable, and $A_{ih}, A_{bh}$ are the same random variable. This happens only if one the following cases occur:
  \begin{enumerate}
    \item $(i,j) = (a,j),(i,h) = (b,h)\Rightarrow i = a = b$, and the number of terms is $O(n^2)$;
    \item $(i,j) = (j,a),(i,h) = (b,h)\Rightarrow i = j = a = b$, and the number of terms is $O(n)$;
    \item $(i,j) = (a,j),(i,h) = (h,b)\Rightarrow i = h = a = b$, and the number of terms is $O(n)$;
    \item $(i,j) = (j,a),(i,h) = (h,b)\Rightarrow i = j = h = a = b$, and the number of terms is $1$;
  \end{enumerate}
  \item $A_{ij}$ and $A_{bh}$ are the same random variable, and $A_{ih}, A_{aj}$ are the same random variable. This happens only if one the following cases occur:
  \begin{enumerate}
    \item $(i,j) = (b,h),(i,h) = (a,j)\Rightarrow i = b = a$, $h = j$, and the number of terms is $O(n)$;
    \item $(i,j) = (h,b),(i,h) = (a,j)\Rightarrow i = j = h = a = b$, and the number of terms is $1$;
    \item $(i,j) = (b,h),(i,h) = (j,a)\Rightarrow j = j = h = a = b$, and the number of terms is $1$;
    \item $(i,j) = (h,b),(i,h) = (j,a)\Rightarrow i = j = h = a = b$, and the number of terms is $1$.
  \end{enumerate}
\end{enumerate}
Therefore, the number of nonzero terms in the summation
\[
\sum_{j = 1}^n\sum_{a = 1}^n\sum_{h = 1}^n\sum_{b = 1}^n\max_{i\in[n]}\max_{i\in[n]}\expect_0(z_{iksja}z_{ikshb})
\]
is $O(n^2)$. Furthermore, the centered second and fourth moments of $\mathrm{Bernoulli}(\rho_n\bx_{0i}\transpose\bx_{0j})$ is upper bounded by $\rho_n$. Therefore, we obtain that
\begin{align*}
\frac{1}{n^2\rho_n^3}\max_{i\in[n]}\expect_0\left\{\left(\sum_{s = 1}^d\sum_{j = 1}^n\sum_{a = 1}^nz_{iksja}\right)^2\right\}
&
\lesssim\frac{1}{n^2\rho_n^3}\sup_{i,j\in[n]}\|\bzeta_{ij}\|^2n^2\rho_n
\lesssim \frac{1}{(n\rho_n)^2}.
\end{align*}
In addition, 
\begin{align*}
\max_{i\in[n]}\left|\frac{1}{n{\rho_n}}\sum_{s = 1}^d\sum_{j = 1}^n(A_{ij} - \rho_n\bx_{0i}\transpose\bx_{0j})\beta_{ijks}R_{js}\right|
&
\leq \frac{1}{n{\rho_n}}\max_{i\in[n]}\left\{\sum_{s = 1}^d\sum_{j = 1}^n(A_{ij} - \rho_n\bx_{0i}\transpose\bx_{0j})^2\beta_{jks}^2\right\}^{1/2}\|\widetilde\bR\|_{\mathrm{F}}\\
& = O_{\prob_0}\left(\frac{(\log n)^{\omega/2}}{n\rho_n^{3/2}}\right).
\end{align*}
Namely, this implies that for each individual $i\in[n]$,
\[
\left|\frac{1}{n\sqrt{\rho_n}}\sum_{j = 1}^n(A_{ij} - \rho_n\bx_{0i}\transpose\bx_{0j})(\rho_n^{-1/2}\bW\transpose\widetilde\bx_j - \bx_{0j})\transpose\bbeta_{ijk}\right| = O_{\prob_0}\left(\frac{(\log n)^{\omega/2}}{{n\rho_n^{3/2}}}\right)
\]
Furthermore, 
\begin{align*}
\sum_{i = 1}^n\expect_0\left\{\left(\frac{1}{n\rho_n^{3/2}}\sum_{s = 1}^d\sum_{j = 1}^n\sum_{a = 1}^nz_{iksja}\right)^2\right\}
\leq \frac{1}{n\rho_n^3}\max_{i\in[n]}\expect_0\left\{\left(\sum_{s = 1}^d\sum_{j = 1}^n\sum_{a = 1}^nz_{iksja}\right)^2\right\}\lesssim \frac{1}{n\rho_n^2},
\end{align*}
implying that
\[
\sum_{i = 1}^n\left\{\frac{1}{n\rho_n^{3/2}}\sum_{s = 1}^d\sum_{j = 1}^n\sum_{a = 1}^nz_{iksja}\right\}^2 = O_{\prob_0}\left(\frac{1}{n\rho_n^2}\right)
\]
by Markov's inequality. 
Therefore, we conclude that
\begin{align*}
&\sum_{i = 1}^n\left\{\frac{1}{n\sqrt{\rho_n}}\sum_{j = 1}^n(A_{ij} - \rho_n\bx_{0i}\transpose\bx_{0j})(\rho_n^{-1/2}\bW\transpose\widetilde\bx_j - \bx_{0j})\transpose\bbeta_{ijk}\right\}^2\\
&\quad\leq 2\sum_{i = 1}^n\left\{\frac{1}{n\rho_n^{3/2}}\sum_{s = 1}^d\sum_{j = 1}^n\sum_{a = 1}^nz_{iksja}\right\}^2 + 2n\max_{i\in[n]}\left|\frac{1}{n{\rho_n}}\sum_{s = 1}^d\sum_{j = 1}^n(A_{ij} - \rho_n\bx_{0i}\transpose\bx_{0j})\beta_{ijks}R_{js}\right|^2
 = O_{\prob_0}\left(\frac{(\log n)^\omega}{n\rho_n^3}\right).
\end{align*}
The proof is thus completed.
\end{proof}
\end{lemma}


\begin{lemma}\label{lemma:convergence_score_function}
Let $\bA\sim\mathrm{RDPG}(\bX_0)$ and assume the conditions in Theorem \ref{thm:convergence_OS} holds. Denote 
$Z = Z(\bA) = \sum_{i = 1}^n\|\sum_{j = 1}^n(A_{ij} - \rho_n\bx_{0i}\transpose\bx_{0j})\bgamma_{ij}\|^2$,
where $\{\bgamma_{ij}:i,j\in[n]\}$ is a collection of deterministic vectors in $\mathbb{R}^d$ such that $\sup_{i,j\in[n]}\|\bgamma_{ij}\|\lesssim{(n\sqrt{\rho_n})^{-1}}$. 
Then $Z = \expect_0(Z) + o_{\prob_0}(1)$. 
\begin{proof}[\bf Proof]
The proof of Lemma \ref{lemma:convergence_score_function} relies on the following logarithmic Sobolev concentration inequality:
\begin{lemma}[Theorem 6.7 in \citealp{boucheron2013concentration}]\label{lemma:log_Sobolev_inequality}
Let $\bA,\bA'\in\{0,1\}^{n\times n}$ be two symmetric hollow random adjacency matrices and $Z = Z(\bA)$ be a measurable function of of $\bA$. Denote by $\bA^{(kl)}$ the adjacency matrix obtained by replacing the $(k,l)$ and $(l,k)$ entries of $\bA$ by those of $\bA'$, and $Z_{kl} = Z(\bA^{(kl)})$. If there exists a constant $v > 0$ such that
\[
\prob\left(\sum_{k<l}(Z - Z_{kl})^2 > v\right)\leq \eta,
\]
then for all $\eps > 0$, $\prob(|Z - \expect(Z)| > t) \leq 2\exp\{-t^2/(2v)\} + \eta$. 
\end{lemma}
Let $\bA'$ be another symmetric hollow random adjacency matrix. Denote by $\bA^{(kl)}$ the adjacency matrix obtained by replacing the $(k,l)$ and $(l,k)$ entries of $\bA$ by those of $\bA'$, and $Z_{kl} = Z(\bA^{(kl)})$
Since that $\bA$ and $\bA^{(kl)}$ only differs by the $(k,l)$ and $(l,k)$ entries, and that the entries of $\bA$ and $\bA'$ are binary, we see that when $Z - Z_{kl}\neq 0$, 
$(A_{kl} - A'_{kl})(Z - Z_{kl}) = C_{1kl} + C_{2kl} + c_{kl}$,
where
\begin{align*}
C_{1kl} = 2\sum_{a = 1}^n(A_{ka} - \rho_n\bx_{0k}\transpose\bx_{0a})\bgamma_{kl}\transpose\bgamma_{ka},\quad
C_{2kl} = 2\sum_{a = 1}^n(A_{la} - \rho_n\bx_{0l}\transpose\bx_{0a})\bgamma_{lk}\transpose\bgamma_{la},
\end{align*}
and $c_{kl} = (1 - 2\rho_n\bx_{0k}\transpose\bx_{0l})(\|\bgamma_{kl}\|^2 + \|\bgamma_{lk}\|^2) - 2(A_{kl} - \rho_n\bx_{0k}\transpose\bx_{0l})\|\bgamma_{kl}\|^2 - 2(A_{lk} - \rho_n\bx_{0l}\transpose\bx_{0k})\|\bgamma_{lk}\|^2$. 
Since
\begin{align*}
\sum_{k < l}\expect_0(C_{1kl}^2) &= 4\sum_{k < l}\sum_{a = 1}^n\sum_{b = 1}^n\expect_0\{(A_{ka} - \rho_n\bx_{0k}\transpose\bx_{0a})(A_{kb} - \rho_n\bx_{0k}\transpose\bx_{0b})\}(\bgamma_{kl}\transpose\bgamma_{ka})(\bgamma_{kl}\transpose\bgamma_{kb})
\\
&= 4\sum_{k < l}\sum_{a = 1}^n\expect_0\{(A_{ka} - \rho_n\bx_{0k}\transpose\bx_{0a})^2\}(\bgamma_{kl}\transpose\bgamma_{ka})^2\\
&\leq 4\sum_{k < l}\sum_{a = 1}^n\rho_n\bx_{0k}\transpose\bx_{0a}(1 - \rho_n\bx_{0k}\transpose\bx_{0a})\|\bgamma_{kl}\|^2\|\bgamma_{ka}\|^2\lesssim \frac{1}{n\rho_n},
\end{align*}
\begin{align*}
\sum_{k < l}\expect_0(C_{2kl}^2) &= 4\sum_{k < l}\sum_{a = 1}^n\sum_{b = 1}^n\expect_0\{(A_{la} - \rho_n\bx_{0l}\transpose\bx_{0a})(A_{lb} - \rho_n\bx_{0l}\transpose\bx_{0b})\}(\bgamma_{lk}\transpose\bgamma_{la})(\bgamma_{lk}\transpose\bgamma_{lb})
\\
&= 4\sum_{k < l}\sum_{a = 1}^n\expect_0\{(A_{la} - \rho_n\bx_{0l}\transpose\bx_{0a})^2\}(\bgamma_{lk}\transpose\bgamma_{la})^2\lesssim \frac{1}{n\rho_n},
\\
\sum_{k < l}\expect_0(c_{kl}^2) &\leq 6\sum_{k < l}(1 - 2\rho_n\bx_{0k}\transpose\bx_{0l})(\|\bgamma_{kl}\|^4 + \|\bgamma_{lk}\|^4) 
 + 6\sum_{k < l}\expect_0\{(A_{kl} - \rho_n\bx_{0k}\transpose\bx_{0l})^2\}(\|\bgamma_{kl}\|^4 + \|\bgamma_{lk}\|^4)\\
&\lesssim \frac{1}{n^2\rho_n^2} + \frac{\rho_n}{n^2\rho_n^2}\lesssim \frac{1}{n^2\rho_n},
\end{align*}
we conclude that $\expect_0\{\sum_{k < l}(Z - Z_{kl})^2\}\leq C/(n\rho_n)$ for some constant $C > 0$. Therefore, by Markov's inequality, 
\[
\prob\left(\sum_{k < l}(Z - Z_{kl})^2 > \frac{1}{\log n}\right)\leq \frac{C\log n}{n\rho_n}\leq \frac{C(\log n)^2}{n\rho_n^5}\to 0.
\]
Invoking Lemma \ref{lemma:log_Sobolev_inequality}, we obtain that
\begin{align*}
\prob_0\left(\left|Z - \expect_0(Z)\right| > \eps\right) &= 
\prob_0\left[\left|Z - \frac{1}{n}\sum_{i = 1}^n\mathrm{tr}\{\bG_n(\bx_{0i})^{-1}\}\right| > \eps\right]
\leq 2\exp\left(-\frac{1}{2}\eps^2\log n\right) + \frac{C\log n}{n\rho_n}\to 0
\end{align*}
for all $\eps > 0$. The proof is thus completed. 
\end{proof}
\end{lemma}

\begin{lemma}\label{lemma:uniform_convergence_G}
Let $\bG_n(\bx)$ be defined as in Theorem \ref{thm:asymptotic_normality_OS}, $\bG(\bx)$ defined as in Theorem \ref{thm:one_dimensional_MLE}, $\widetilde\bG(\bx)$ be defined as in Theorem \ref{thm:convergence_OS_Laplacian}. Denote
\[
\widetilde\bG_n(\bx) = \frac{1}{\bmu_n\transpose\bx}\left(\eye_d - \frac{\bx\bmu_n\transpose}{2\bx\transpose\bmu_n}\right)\bG_n(\bx)^{-1}\left(\eye_d - \frac{\bx\bmu_n\transpose}{2\bx\transpose\bmu_n}\right),
\]
where $\bmu_n = (1/n)\sum_{i = 1}^n\bx_{0i}$. Let $\calX(\delta)$ be the set of all $\bx\in\calX$ such that any $\bx,\bu\in\calX(\delta)$ satisfy $\delta\leq \bx\transpose\bu\leq 1 - \delta$, where $\delta > 0$ is some small constant independent of $n$. Then
\[
\sup_{\bx\in\calX(\delta)}\|\bG_n(\bx)^{-1} - \bG(\bx)^{-1}\|_{\mathrm{F}}\to 0,\quad\text{and}\quad
\sup_{\bx\in\calX(\delta)}\|\widetilde\bG_n(\bx) - \widetilde\bG(\bx)\|_{\mathrm{F}}\to 0
\]
as $n\to\infty$.
\end{lemma}
\begin{proof}[\bf Proof]
We first show that $\bG_n(\bx) \to \bG(\bx)$ as $n\to\infty$ uniformly for all $\bx\in\calX(\delta)$. It suffices to show that for all $s,t\in[d]$, 
\[
\sup_{\bx\in\calX(\delta)}|\be_s\transpose\bG_n(\bx)\be_t - \be_s\transpose\bG(\bx)\be_s| \to 0
\]
as $n\to\infty$. Observe that $\be_s\transpose\bG_n(\bx)\be_t - \be_s\transpose\bG(\bx)\be_s\to 0$ as $n\to\infty$ for each $\bx\in\calX(\delta)$. Furthermore,
\begin{align*}
\sup_{n\geq 1,s,t,\in[d]}\left\|
\frac{\partial}{\partial\bx}\be_s\transpose\bG_n(\bx)\be_t
\right\|_2
& \leq \sup_{n\geq 1,s,t,\in[d]}\frac{1}{n}\sum_{j = 1}^n\left\|
\frac{[\bx_{0j}]_s[\bx_{0j}]_t(1 - 2\rho_n\bx\transpose\bx_{0j})\bx_{0j}}{\{\bx\transpose\bx_{0j}(1 - \rho_n\bx\transpose\bx_{0j})\}^2}
\right\|_2\\
&\leq \sup_{n\geq 1,s,t,\in[d]}\frac{1}{n}\sum_{j = 1}^n\frac{3}{\delta^4} = \frac{3}{\delta^4} <\infty.
\end{align*}
This means that the function class $(\be_s\transpose\bG_n(\bx)\be_t)_{n = 1}^\infty$ is a uniformly Lipschitz function class and hence is equicontinuous. Since $\calX(\delta)$ is compact, it follows from the Arzela-Ascoli theorem that the convergence $\be_s\transpose\bG_n(\bx)\be_t - \be_s\transpose\bG(\bx)\be_s\to 0$ is also uniform.

\vspace*{1ex}
\noindent
We next show that $\sup_{\bx\in\calX(\delta)}\|\bG_n(\bx)^{-1} - \bG(\bx)^{-1}\|_{\mathrm{F}}\to 0$. This immediately follows from the inequality
\[
\sup_{\bx\in\calX(\delta)}\|\bG_n(\bx) - \bG(\bx)\|_{\mathrm{F}}\leq \sup_{\bx\in\calX(\delta)}\|\bG_n(\bx)^{-1}\|_{\mathrm{F}}\|\bG_n(\bx) - \bG(\bx)\|_{\mathrm{F}}\|\bG(\bx)^{-1}\|_{\mathrm{F}},
\]
the fact that $\bG_n(\bx)\succeq(1/2)\bDelta$ and $\bG(\bx)\succeq \bDelta$ for sufficiently large $n$, and the uniform convergence of $\bG_n(\bx)\to\bG(\bx)$ for all $\bx\in\calX(\delta)$. 

\vspace*{1ex}
\noindent We finally show that $\widetilde\bG_n(\bx) - \widetilde\bG(\bx)$ uniformly for all $\bx\in\calX(\delta)$. This result also follows from the fact that
\begin{align*}
&\sup_{\bx\in\calX(\delta)}\left|\frac{1}{\bmu_n\transpose\bx} - \frac{1}{\bmu\transpose\bx}\right|\leq \frac{1}{\delta^2}\|\bmu_n - \bmu\|\to 0,\quad
\sup_{\bx\in\calX(\delta)}\left\|\frac{\bx\bmu_n\transpose}{(\bmu_n\transpose\bx)^2} - \frac{\bx\bmu\transpose}{(\bmu\transpose\bx)^2}\right\|_{\mathrm{F}}\leq \frac{3}{\delta^4}\|\bmu_n - \bmu\|\to 0,
\end{align*}
and the uniform convergence result $\sup_{\bx\in\calX(\delta)}\|\bG_n(\bx)^{-1} - \bG(\bx)^{-1}\|_{\mathrm{F}}\to 0$. The proof is thus completed.
\end{proof}


\subsection{Proof of Theorem \ref{thm:asymptotic_normality_OS}} 
\label{sub:proof_of_theorem_thm:asymptotic_normality_os}



\begin{proof}[\bf Proof of Theorem \ref{thm:asymptotic_normality_OS}]
Let $\bW$ be the matrix satisfying \eqref{eqn:linearization_property}. 
For any $\bX = [\bx_1,\ldots,\bx_n]\transpose\in\mathbb{R}^{n\times d}$, denote $\bH_i(\bX) = (1/n)\sum_{j = 1}^n\bx_{j}\{(\bx_{i}\transpose\bx_j)(1 - \bx_i\transpose\bx_j)\}^{-1}\bx_{j}\transpose$. By definition, 
\begin{align*}
\bW\transpose\widehat\bx_i - \rho_n^{1/2}\bx_{0i}
& = (\bW\transpose\widetilde \bx_i - \rho_n^{1/2}\bx_{0i})
\\&\quad 
 + \bW\transpose\bH_i(\widetilde\bX)^{-1}\bW\frac{1}{n\sqrt{\rho_n}}\sum_{j = 1}^n\left\{\frac{(A_{ij} - \widetilde\bx_i\transpose\widetilde\bx_j)(\rho_n^{-1/2}\bW\transpose\widetilde\bx_j)}{\rho_n^{-1}\widetilde\bx_{i}\transpose\widetilde\bx_{j}(1 - \widetilde\bx_i\transpose\widetilde\bx_j)} - \frac{(A_{ij} - \rho_n\bx_{0i}\transpose\bx_{0j})\bx_{0j}}{\bx_{0i}\transpose\bx_{0j}(1 - \rho_n\bx_{0i}\transpose\bx_{0j})}\right\}\\
&\quad
 + \frac{1}{n\sqrt{\rho_n}}\sum_{j = 1}^n\frac{(A_{ij} - \rho_n\bx_{0i}\transpose\bx_{0j})}{\bx_{0i}\transpose\bx_{0j}(1 - \rho_n\bx_{0i}\transpose\bx_{0j})}(\bW\transpose\bH_i(\widetilde\bX)^{-1}\bW)\bx_{0j}\\
& = (\bW\transpose\widetilde \bx_i - \rho_n^{1/2}\bx_{0i})
\\&\quad  + \bG_n(\bx_{0i})^{-1}\frac{1}{n\sqrt{\rho_n}}\sum_{j = 1}^n\left\{\frac{(A_{ij} - \widetilde\bx_i\transpose\widetilde\bx_j)(\rho_n^{-1/2}\bW\transpose\widetilde\bx_j)}{\rho_n^{-1}\widetilde\bx_{i}\transpose\widetilde\bx_{j}(1 - \widetilde\bx_i\transpose\widetilde\bx_j)} - \frac{(A_{ij} - \rho_n\bx_{0i}\transpose\bx_{0j})\bx_{0j}}{\bx_{0i}\transpose\bx_{0j}(1 - \rho_n\bx_{0i}\transpose\bx_{0j})}\right\}\\
&\quad + \{\bW\transpose\bH_i(\widetilde\bX)^{-1}\bW - \bG_n(\bx_{0i})^{-1}\}
\\&\quad \times\frac{1}{n\sqrt{\rho_n}}\sum_{j = 1}^n\left\{\frac{(A_{ij} - \widetilde\bx_i\transpose\widetilde\bx_j)(\rho_n^{-1/2}\bW\transpose\widetilde\bx_j)}{\rho_n^{-1}\widetilde\bx_{i}\transpose\widetilde\bx_{j}(1 - \widetilde\bx_i\transpose\widetilde\bx_j)} - \frac{(A_{ij} - \rho_n\bx_{0i}\transpose\bx_{0j})\bx_{0j}}{\bx_{0i}\transpose\bx_{0j}(1 - \rho_n\bx_{0i}\transpose\bx_{0j})}\right\}\\
&\quad + \frac{1}{n\sqrt{\rho_n}}\sum_{j = 1}^n\frac{(A_{ij} - \rho_n\bx_{0i}\transpose\bx_{0j})}{\bx_{0i}\transpose\bx_{0j}(1 - \rho_n\bx_{0i}\transpose\bx_{0j})}\{\bW\transpose\bH_i(\widetilde\bX)^{-1}\bW - \bG_n(\bx_{0i})^{-1}\}\bx_{0j}\\
&\quad + \frac{1}{n\sqrt{\rho_n}}\sum_{j = 1}^n\frac{(A_{ij} - \rho_n\bx_{0i}\transpose\bx_{0j})}{\bx_{0i}\transpose\bx_{0j}(1 - \rho_n\bx_{0i}\transpose\bx_{0j})}\bG_n(\bx_{0i})^{-1}\bx_{0j}\\
& = (\bW\transpose\widetilde \bx_i - \rho_n^{1/2}\bx_{0i}) + \bG_n(\bx_{0i})^{-1}\bR_{i1} + \bR_{i2}\bR_{i1} + \bR_{i3}
\\&\quad 
 + \frac{1}{n\sqrt{\rho_n}}\sum_{j = 1}^n\frac{(A_{ij} - \rho_n\bx_{0i}\transpose\bx_{0j})}{\bx_{0i}\transpose\bx_{0j}(1 - \rho_n\bx_{0i}\transpose\bx_{0j})}\bG_n(\bx_{0i})^{-1}\bx_{0j},
\end{align*}
where
\begin{align*}
\bR_{i1} &= \frac{1}{n\sqrt{\rho_n}}\sum_{j = 1}^n\left\{\frac{(A_{ij} - \widetilde\bx_i\transpose\widetilde\bx_j)(\rho_n^{-1/2}\bW\transpose\widetilde\bx_j)}{\rho_n^{-1}\widetilde\bx_{i}\transpose\widetilde\bx_{j}(1 - \widetilde\bx_i\transpose\widetilde\bx_j)} - \frac{(A_{ij} - \rho_n\bx_{0i}\transpose\bx_{0j})\bx_{0j}}{\bx_{0i}\transpose\bx_{0j}(1 - \rho_n\bx_{0i}\transpose\bx_{0j})}\right\},\\
\bR_{i2} &= \bW\transpose\bH_i(\widetilde\bX)^{-1}\bW - \bG_n(\bx_{0i})^{-1},\\
\bR_{i3} & = \frac{1}{n\sqrt{\rho_n}}\sum_{j = 1}^n\frac{(A_{ij} - \rho_n\bx_{0i}\transpose\bx_{0j})}{\bx_{0i}\transpose\bx_{0j}(1 - \rho_n\bx_{0i}\transpose\bx_{0j})}\bR_{i2}\bx_{0j}.
\end{align*}
We first analyze $\bR_{i1}$. Denote the function $\bphi_{ij}:\mathbb{R}^d\times\mathbb{R}^d\to\mathbb{R}$ by
\[
\bphi_{ij}(\bu,\bv) = \frac{(A_{ij} - \rho_n\bu\transpose\bv)\bv}{\bu\transpose\bv(1 - \rho_n\bu\transpose\bv)},\quad i,j\in[n],
\]
and let $\bphi_{ij} = [\phi_{ij1},\ldots,\phi_{ijd}]\transpose$. 
By Taylor's expansion, we have, if $\|\bu - \bx_{0i}\| < \eps$ and $\|\bv - \bx_{0j}\| < \eps$ for sufficiently small $\eps > 0$, and $\delta \leq \min_{i,j\in[n]}\bx_{0i}\transpose\bx_{0j} \leq \max_{i,j\in[n]}\bx_{0i}\transpose\bx_{0j}\leq 1 - \delta$ for some constant $\delta > 0$, then 
\begin{align*}
&\phi_{ijk}(\bu, \bv) - \phi_{ijk}(\bx_{0i}, \bx_{0j}) \\
&\quad = -\left\{\frac{\rho_n x_{0jk} \bx_{0j}}{\bx_{0i}\transpose\bx_{0j}(1 - \rho_n \bx_{0i}\transpose\bx_{0j})}\right\}\transpose(\bu - \bx_{0i}) - \left\{\frac{\rho_n \bx_{0i} x_{0jk}}{\bx_{0i}\transpose\bx_{0j}(1 - \rho_n \bx_{0i}\transpose\bx_{0j})}\right\}\transpose(\bv - \bx_{0j})\\
&\quad\quad - \left[\frac{x_{0jk}(A_{ij} - \rho_n \bx_{0i}\transpose\bx_{0j})(1 - 2\rho_n \bx_{0i}\transpose\bx_{0j})\bx_{0j}}{\{\bx_{0i}\transpose\bx_{0j}(1 - \rho_n \bx_{0i}\transpose\bx_{0j})\}^2}\right]\transpose(\bu - \bx_{0i})\\
&\quad\quad + \left[\frac{(A_{ij} - \rho_n \bx_{0i}\transpose\bx_{0j})\{\bx_{0i}\transpose\bx_{0j}(1 - \rho_n \bx_{0i}\transpose\bx_{0j})\be_k - x_{0jk}(1 - 2\rho_n \bx_{0i}\transpose\bx_{0j})\bx_{0i}\}}{\{\bx_{0i}\transpose\bx_{0j}(1 - \rho_n \bx_{0i}\transpose\bx_{0j})\}^2}\right]\transpose(\bv - \bx_{0j})
 + R_{ijk},
\end{align*}
where $\max_{i,j\in[n],k\in[d]}|R_{ijk}|\leq C_\delta \max(\|\bu - \bu_0\|^2,\|\bv - \bv_0\|^2)$ for some constant $C_\delta$ only depending on $\delta$. Applying the above fact to $\bR_{i1}$, we derive
\begin{align*}
\bR_{i1} 
& =  - \frac{1}{n\sqrt{\rho_n}}\sum_{j = 1}^n\frac{\rho_n\bx_{0j}\bx_{0j}\transpose}{\bx_{0i}\transpose\bx_{0j}(1 - \rho_n\bx_{0i}\transpose\bx_{0j})}(\rho_n^{-1/2}\bW\transpose\widetilde\bx_i - \bx_{0i}) \\
&\quad - \frac{1}{n\sqrt{\rho_n}}\sum_{j = 1}^n\frac{\rho_n \bx_{0j}\bx_{0i}\transpose}{\bx_{0i}\transpose\bx_{0j}(1 - \rho_n\bx_{0i}\transpose\bx_{0j})}(\rho_n^{-1/2}\bW\transpose\widetilde\bx_j - \bx_{0j})\\
& \quad - \frac{1}{n\sqrt{\rho_n}}\sum_{j = 1}^n\left[\frac{(A_{ij} - \rho_n\bx_{0i}\transpose\bx_{0j})(1 - 2\rho_n\bx_{0i}\transpose\bx_{0j})\bx_{0j}\bx_{0j}\transpose}{\{\bx_{0i}\transpose\bx_{0j}(1 - \rho_n\bx_{0i}\transpose\bx_{0j})\}^2}\right](\rho_n^{-1/2}\bW\transpose\widetilde\bx_i - \bx_{0i})\\
&\quad + \frac{1}{n\sqrt{\rho_n}}\sum_{j = 1}^n
\frac{(A_{ij} - \rho_n\bx_{0i}\transpose\bx_{ij})\{\bx_{0i}\transpose\bx_{0j}(1 - \rho_n\bx_{0i}\transpose\bx_{0j})\eye_d - (1 - 2\rho_n\bx_{0i}\transpose\bx_{0j})\bx_{0j}\bx_{0i}\transpose\}}
{\{\bx_{0i}\transpose\bx_{0j}(1 - \rho_n\bx_{0i}\transpose\bx_{0j})\}^2}
(\rho_n^{-1/2}\bW\transpose\widetilde\bx_j - \bx_{0j})\\
&\quad + \frac{1}{n\sqrt{\rho_n}}\sum_{j = 1}^n\bR_{ij}\\
& = - \bR_{i11} - \bR_{i12} - \bR_{i13} + \bR_{i14} + \frac{1}{n\sqrt{\rho_n}}\sum_{j = 1}^n\bR_{ij},
\end{align*}
where $\bR_{ij}$'s are such that $\max_{i,j\in[n]}\|\bR_{ij}\|\lesssim \|\rho_n^{-1/2}\widetilde\bX\bW - \bX_0\|_{2\to\infty}^2$ when $\|\rho_n^{-1/2}\widetilde\bX\bW - \bX_0\|_{2\to\infty}$ is sufficiently small. 
Clearly,  
\[
\bR_{i11} = \frac{1}{n}\sum_{j = 1}^n\frac{\bx_{0j}\bx_{0j}\transpose}{\bx_{0i}\transpose\bx_{0j}(1 - \rho_n\bx_{0i}\transpose\bx_{0j})}(\bW\transpose\widetilde\bx_i - \rho_n^{1/2}\bx_{0i}) = \bG_n(\bx_{0i})(\bW\transpose\widetilde\bx_i - \rho_n^{1/2}\bx_{0i}).
\]
Furthermore, Lemma \ref{lemma:R12k_analysis} shows that $\max_{i\in[n]}\|\bR_{i12}\| = O_{\prob_0}((n\sqrt{\rho_n})^{-1}(\log n)^{(1\vee\omega)/2})$. 
In addition, we have $\max_{i\in[n]}\|\bR_{i13}\| = O_{\prob_0}((n\rho_n)^{-1}(\log n)^{1/2 + (1\vee\omega)/2})$ by Lemma \ref{lemma:R13k_analysis}, $\|\bR_{i14}\| = O_{\prob_0}((n\rho_n^{3/2})^{-1}(\log n)^{\omega/2})$ by Lemma \ref{Lemma:R14k_analysis}, and 
\[
\max_{i\in[n]}\frac{1}{n\sqrt{\rho_n}}\sum_{j = 1}^n\max_{j\in[n]}\|\bR_{ij}\| = O_{\prob_0}\left(\frac{(\log n)^{1\vee\omega}}{n\rho_n^{5/2}}\right)
\]
by Lemma \ref{lemma:two_to_infinity_error_ASE}. 
This shows that
\begin{align*}
\|\bR_{i1} + \bR_{i11}\| &= \left\|\bR_{i1} + \bG_n(\bx_{0i})(\bW\transpose\widetilde\bx_i - \rho_n^{1/2}\bx_{0i})\right\|\\ 
&\leq \|\bR_{i12}\| + \|\bR_{i13}\| + \|\bR_{i14}\| + \frac{1}{n\sqrt{\rho_n}}\sum_{j = 1}^n\max_{i,j\in[n]}\|\bR_{ij}\|
= O_{\prob_0}\left(\frac{(\log n)^{1\vee\omega}}{n\rho_n^{5/2}}\right),
\end{align*}
and hence,
\[
\|\bR_{i1}\|\leq \|\bG_n(\bx_{0i})\|\|\widetilde\bX\bW - \rho_n^{1/2}\bX_0\|_{2\to\infty} + O_{\prob_0}\left(\frac{(\log n)^{1\vee\omega}}{n\rho_n^{5/2}}\right) = O_{\prob_0}\left(\sqrt{\frac{(\log n)^{1\vee\omega}}{n\rho_n}}\right)
\]
by Lemma \ref{lemma:two_to_infinity_error_ASE}.

Next we focus on $\bR_{i2}$. Since the function $(\bu, \bv)\mapsto \{\bu\transpose\bv(1 - \rho_n\bu\transpose\bv)\}^{-1}\bv\bv\transpose$ is Lipschitz continuous in a neighborhood of $(\bx_{0i},\bx_{0j})$, it follows immediately that 
\[
\max_{i\in[n]}\|\bW\transpose\bH_i(\widetilde\bX)\bW - \bG_n(\bx_{0i})\|_{\mathrm{F}}\lesssim \|\rho_n^{-1/2}\widetilde\bX\bW - \bX_0\|_{2\to\infty}
\]
when $\|\rho_n^{-1/2}\widetilde\bX\bW - \bX_0\|_{2\to\infty} \leq C_c\rho_n^{-1}\sqrt{n^{-1}}(\log n)^{(1\vee\omega)/2}$. Namely, 
\[
\max_{i\in[n]}\|\bW\transpose\bH_i(\widetilde\bX)\bW - \bG_n(\bx_{0i})\|_{\mathrm{F}} = O_{\prob_0}\left(\frac{(\log n)^{(1\vee\omega)/2}}{\rho_n\sqrt{n}}\right)
\] 
by Lemma \ref{lemma:two_to_infinity_error_ASE}. Furthermore, since $\bG_n(\bx_{0i})\to \bG(\bx_{0i})$ as $n\to\infty$, $\bG_n(\bx_{0i}) - \bDelta$ is positive definite for sufficiently large $n$, and
\[
|\lambda_d(\bW\transpose\bH_i(\widetilde\bX)\bW) - \lambda_d(\bG_n(\bx_{0i}))|\leq \|\bW\transpose\bH_i(\widetilde\bX)\bW - \bG_n(\bx_{0i})\|_{\mathrm{F}}^2,
\]
(see, for example, \citealp{hoffman2003variation}), we conclude that 
\[
\min_{i\in[n]}\lambda_d(\bW\transpose\bH_i(\widetilde\bX)\bW)\geq\min_{i\in[n]}\lambda_d(\bG_n(\bx_{0i})) - \max_{i\in[n]}\|\bW\transpose\bH_i(\widetilde\bX)\bW - \bG_n(\bx_{0i})\|_{\mathrm{F}}^2\geq\lambda_d(\bDelta) - o_{\prob_0}(1),
\]
namely, $\max_{i\in[n]}\lambda_d^{-1}(\bW\transpose\bH_i(\widetilde\bX)\bW) = O_{\prob_0}(1)$. Therefore,
\begin{align*}
\max_{i\in[n]}\|\bR_{i2}\|_{\mathrm{F}}
& = \max_{i\in[n]}\|\{\bW\transpose\bH_i(\widetilde\bX)\bW\}^{-1}\{\bW\transpose\bH_i(\widetilde\bX)\bW - \bG_n(\bx_{0i})\}\bG_n(\bx_{0i})^{-1}\|_{\mathrm{F}}\\
&\lesssim \max_{i\in[n]}\lambda_d^{-1}(\bW\transpose\bH_i(\widetilde\bX)\bW)\max_{i\in[n]}\|\bG_n(\bx_{0i})^{-1}\|_{\mathrm{F}}\max_{i\in[n]}\|\bW\transpose\bH_i(\widetilde\bX)\bW - \bG_n(\bx_{0i})\|_{\mathrm{F}}\\
&\leq O_{\prob_0}(1)\|\bDelta^{-1}\|_{\mathrm{F}}\max_{i\in[n]}\|\bW\transpose\bH_i(\widetilde\bX)\bW - \bG_n(\bx_{0i})\|_{\mathrm{F}} = O_{\prob_0}\left(\frac{(\log n)^{(1\vee\omega)/2}}{\rho_n\sqrt{n}}\right)
\end{align*}
We finally move forward to analyze $\bR_{i3}$. Since
\begin{align*}
\max_{i\in[n]}\|\bR_{i3}\|_{\mathrm{F}}\leq \max_{i\in[n]}\|\bR_{i2}\|_{\mathrm{F}}\sum_{k = 1}^d\frac{1}{n\sqrt{\rho_n}}\max_{i\in[n]}\left|\sum_{j = 1}^n\frac{(A_{ij} - \rho_n\bx_{0i}\transpose\bx_{0j})x_{0jk}}{\bx_{0i}\transpose\bx_{0j}(1 - \rho_n\bx_{0i}\transpose\bx_{0j})}\right|,
\end{align*}
and by Hoeffding's inequality and the union bound,
\begin{align*}
\prob_0\left(\max_{i\in[d]}\left|\sum_{j = 1}^n\frac{(A_{ij} - \rho_n\bx_{0i}\transpose\bx_{0j})x_{0jk}}{\bx_{0i}\transpose\bx_{0j}(1 - \rho_n\bx_{0i}\transpose\bx_{0j})}\right| > t\sqrt{n\log n}\right)
& = 2n\exp\left[-\frac{2t^2n\log n}{\sum_{j = 1}^nx_{0jk}^2\{\bx_{0i}\transpose\bx_{0j}(1 - \rho_n\bx_{0i}\transpose\bx_{0j})\}^{-2}}\right]\\
&
\leq 2\exp\left\{-(Mt^2 - 1)\log n\right\}
\end{align*}
for some constant $M > 0$. Hence, 
\[
\sum_{k = 1}^d\frac{1}{n\sqrt{\rho_n}}\max_{i\in[n]}\left|\sum_{j = 1}^n\frac{(A_{ij} - \rho_n\bx_{0i}\transpose\bx_{0j})x_{0jk}}{\bx_{0i}\transpose\bx_{0j}(1 - \rho_n\bx_{0i}\transpose\bx_{0j})}\right| = O_{\prob_0}\left(\sqrt{\frac{\log n}{n\rho_n}}\right),
\]
and hence, $\max_{i\in[n]}\|\bR_{i3}\|_{\mathrm{F}} = O_{\prob_0}(\rho_n^{-3/2}n^{-1}(\log n)^{1/2 + (1\vee\omega)/2})$. Therefore, we conclude that
\begin{align*}
\bW\transpose\widehat\bx_i - \rho_n^{1/2}\bx_{0i}
& = (\bW\transpose\widetilde \bx_i - \rho_n^{1/2}\bx_{0i}) + \bG_n(\bx_{0i})^{-1}\bR_1 + \bR_2\bR_1 + \bR_3 
\\&\quad
+ \frac{1}{n\sqrt{\rho_n}}\sum_{j = 1}^n\frac{(A_{ij} - \rho_n\bx_{0i}\transpose\bx_{0j})}{\bx_{0i}\transpose\bx_{0j}(1 - \rho_n\bx_{0i}\transpose\bx_{0j})}\bG_n(\bx_{0i})^{-1}\bx_{0j}\\
& = \frac{1}{n\sqrt{\rho_n}}\sum_{j = 1}^n\frac{(A_{ij} - \rho_n\bx_{0i}\transpose\bx_{0j})}{\bx_{0i}\transpose\bx_{0j}(1 - \rho_n\bx_{0i}\transpose\bx_{0j})}\bG_n(\bx_{0i})^{-1}\bx_{0j} + \widehat\bR_i,
\end{align*}
where
\begin{align*}
\|\widehat\bR_i\| &= \left\|\bG_n(\bx_{0i})^{-1}(\bR_{i1} + \bR_{i11}) + \bR_{i2}\bR_{i1} + \bR_{i3}\right\|
\leq \|\bDelta_n^{-1}\|_2\|\|\bR_{i1} + \bR_{i11}\| + \left(\|\bR_{i2}\|_2\|\bR_{i1}\| + \|\bR_{i3}\|\right)
\\& 
= O_{\prob_0}\left(\frac{(\log n)^{1\vee\omega}}{n\rho_n^{5/2}}\right) + O_{\prob_0}\left(\frac{(\log n)^{1\vee\omega}}{n\rho_n^{3/2}}\right) + O_{\prob_0}\left(\frac{(\log n)^{1/2 + (1\vee\omega)/2}}{n\rho_n^{3/2}}\right)
= O_{\prob_0}\left(\frac{(\log n)^{1\vee\omega}}{n\rho_n^{5/2}}\right). 
\end{align*}
We now proceed to prove that $\sum_{i = 1}^n\|\widehat\bR_i\|^2 = O_{\prob_0}((n\rho_n^5)^{-1}(\log n)^{2(1\vee\omega)})$. Observe that by Lemma \ref{Lemma:R14k_analysis} we have
\begin{align*}
\sum_{i = 1}^n\left\|\bG_n(\bx_{0i})^{-1}( - \bR_{i12} - \bR_{i13} + \bR_{i14})\right\|^2
&\leq 3n\left\|\bG_n(\bx_{0i})^{-1}\right\|_{\mathrm{F}}^2\left\{\max_{i\in[n]}\|\bR_{i12}\|^2 + \max_{i\in[n]}\|\bR_{i13}\|^2\right\} + 3\sum_{i = 1}^n\|\bR_{i14}\|^2\\
&\leq n\left\|\bDelta^{-1}\right\|_{\mathrm{F}}^2\left\{\max_{i\in[n]}\|\bR_{i12}\|^2 + \max_{i\in[n]}\|\bR_{i13}\|^2\right\} + \sum_{i = 1}^n\|\bR_{i14}\|^2\\ 
& = O_{\prob_0}\left(\frac{(\log n)^{2(1\vee\omega)}}{n\rho_n^{3 }}\right),
\end{align*}
and that
\begin{align*}
\sum_{i = 1}^n\left\|\bG_n(\bx_{0i})^{-1}\frac{1}{n\sqrt{\rho_n}}\sum_{j = 1}^n\bR_{ij}\right\|^2
&\leq n\|\bDelta^{-1}\|_{\mathrm{F}}^2\left(\frac{1}{n\sqrt{\rho_n}}\sum_{j = 1}^n\max_{i\in[n]}\|\bR_{ij}\|\right)^2 = O_{\prob_0}\left(\frac{(\log n)^{2(1\vee\omega)}}{n\rho_n^5}\right).
\end{align*}
Besides, by the above derivation we have
\begin{align*}
\sum_{i = 1}^n\|\bR_{i2}\bR_{i1}\|^2
&\leq \max_{i\in[n]}\|\bR_{i2}\|_{\mathrm{F}}^2\sum_{i = 1}^n\|\bR_{i1}\|^2\\
&\leq O_{\prob_0}\left(\frac{(\log n)^{1\vee\omega}}{n^2\rho_n}\right)\left\{n\max_{i\in[n]}\|\bG_n(\bx_{0i})\|_{\mathrm{F}}^2\|\widetilde\bX\bW - \rho_n^{1/2}\bX_0\|^2_{\mathrm{F}}\right\}\\
&\quad + O_{\prob_0}\left(\frac{(\log n)^{1\vee\omega}}{n^2\rho_n}\right)\left\{\sum_{i = 1}^n\left\|\bR_{i12} + \bR_{i13} - \bR_{i14} - \frac{1}{n\sqrt{\rho_n}}\sum_{j = 1}^n\bR_{ij}\right\|^2\right\}\\
& = O_{\prob_0}\left(\frac{(\log n)^{2(1\vee\omega)}}{n^2\rho_n^2}\right),
\end{align*}
and 
\begin{align*}
\sum_{i = 1}^n\|\bR_{i3}\|_{\mathrm{F}}^2 \leq n\max_{i\in[n]}\|\bR_{i3}\|_{\mathrm{F}}^2 = O_{\prob_0}\left(\frac{(\log n)^{2(1\vee\omega)}}{n\rho_n^3}\right).
\end{align*}
Therefore, we conclude that
\begin{align*}
\sum_{i = 1}^n\|\widehat\bR_i\|_{\mathrm{F}}^2
&\leq 4\sum_{i = 1}^n\left\|\bG_n(\bx_{0i})^{-1}( - \bR_{i12} - \bR_{i13} + \bR_{i14})\right\|^2 + 4\sum_{i = 1}^n\left\|\bG_n(\bx_{0i})^{-1}\frac{1}{n\sqrt{\rho_n}}\sum_{j = 1}^n\bR_{ij}\right\|^2\\
&\quad + 4\sum_{i = 1}^n\|\bR_{i2}\bR_{i1}\|^2 + 4\sum_{i = 1}^n\|\bR_{i3}\|_{\mathrm{F}}^2
= O_{\prob_0}\left(\frac{(\log n)^{2(1\vee\omega)}}{n\rho_n^5}\right).
\end{align*}
The proof is thus completed.
\end{proof}

\subsection{Proof of Theorem \ref{thm:convergence_OS}} 
\label{sub:proof_of_theorem_thm:convergence_os}


\begin{proof}[\bf Proof of Theorem \ref{thm:convergence_OS}]
Let $(\bW)_{n = 1}^\infty = (\bW_n)_{n = 1}^\infty\subset\mathbb{O}(d)$ be a sequence of orthogonal matrices satisfying \eqref{eqn:linearization_property}. 
Denote 
\[
\bgamma_{ij} = \frac{1}{n\sqrt{\rho_n}}\frac{\bG_n(\bx_{0i})^{-1}\bx_{0j}}{\bx_{0i}\transpose\bx_{0j}(1 - \rho_n\bx_{0i}\transpose\bx_{0j})}. 
\]
First note that $\bG_n(\bx_{0i})^{-1}\succeq\bDelta$ for sufficiently large $n$, and hence,
\begin{align}\label{eqn:gammaij_bound}
\sup_{i,j\in[n]}\|\bgamma_{ij}\|\leq\frac{1}{n\sqrt{\rho_n}}\sup_{i,j\in[n]}\frac{\|\bG_n(\bx_{0i})^{-1}\|\|\bx_{0j}\|}{\bx_{0i}\transpose\bx_{0j}(1 - \rho_n\bx_{0i}\transpose\bx_{0j})}\leq \frac{1}{n\sqrt{\rho_n}}\frac{\|\bDelta\|_2}{\delta(1 - \delta)}\lesssim\frac{1}{n\sqrt{\rho_n}}.
\end{align}
Also observe that
\begin{align*}
\expect_0\left(\sum_{i = 1}^n\left\|\sum_{j = 1}^n(A_{ij} - \rho_n\bx_{0i}\transpose\bx_{0j})\bgamma_{ij}\right\|^2\right)
&
= \sum_{i = 1}^n\sum_{a = 1}^n\sum_{b = 1}^n\expect_0\left\{(A_{ia} - \rho_n\bx_{0i}\transpose\bx_{0a})(A_{ib} - \rho_n\bx_{0i}\transpose\bx_{0b})\bgamma_{ia}\transpose\bgamma_{ib}\right\}\\
&
= \frac{1}{n^2\rho_n}\sum_{i = 1}^n\sum_{a = 1}^n\frac{\expect_0\{(A_{ia} - \rho_n\bx_{0i}\transpose\bx_{0a})^2\}}{\{\bx_{0i}\transpose\bx_{0a}(1 - \rho_n\bx_{0i}\transpose\bx_{0a})\}^2}\bx_{0a}\transpose\bG_n^{-2}(\bx_{0i})\bx_{0a}\\
&
= \frac{1}{n}\sum_{i = 1}^n\frac{1}{n}\sum_{a = 1}^n\frac{\mathrm{tr}\{\bG_n(\bx_{0i})^{-1}\bx_{0a}\bx_{0a}\transpose\bG_n^{-1}(\bx_{0i})\}}{\bx_{0i}\transpose\bx_{0a}(1 - \rho_n\bx_{0i}\transpose\bx_{0a})}\\
&
= \frac{1}{n}\sum_{i = 1}^n\mathrm{tr}\{\bG_n(\bx_{0i})^{-1}\}.
\end{align*}
By Theorem \ref{thm:asymptotic_normality_OS} and Lemma \ref{lemma:convergence_score_function}, we can write
\begin{align*}
\left\|\widehat\bX\bW - \bX_0\right\|_{\mathrm{F}}^2
& = \sum_{i = 1}^n\left\|\sum_{j = 1}^n(A_{ij} - \rho_n\bx_{0i}\transpose\bx_{0j})\bgamma_{ij}\right\|^2 + 2\sum_{i = 1}^n\widehat\bR_i\transpose\sum_{j = 1}^n(A_{ij} - \rho_n\bx_{0i}\transpose\bx_{0j})\bgamma_{ij} + \sum_{i = 1}^n\|\widehat\bR_i\|^2\\
 & = \frac{1}{n}\sum_{i = 1}^n\mathrm{tr}\{\bG_n(\bx_{0i})^{-1}\} + 2\sum_{i = 1}^n\widehat\bR_i\transpose\sum_{j = 1}^n(A_{ij} - \rho_n\bx_{0i}\transpose\bx_{0j})\bgamma_{ij} + o_{\prob_0}(1) + O_{\prob_0}\left(\frac{(\log n)^{2(1\vee\omega)}}{n\rho_n^5}\right)\\
 & = \frac{1}{n}\sum_{i = 1}^n\mathrm{tr}\{\bG_n(\bx_{0i})^{-1}\} + 2\sum_{i = 1}^n\widehat\bR_i\transpose\sum_{j = 1}^n(A_{ij} - \rho_n\bx_{0i}\transpose\bx_{0j})\bgamma_{ij} + o_{\prob_0}(1).
\end{align*}
By Cauchy-Schwarz inequality,
\begin{align*}
\left|\sum_{i = 1}^n\widehat\bR_i\transpose\sum_{j = 1}^n(A_{ij} - \rho_n\bx_{0i}\transpose\bx_{0j})\bgamma_{ij}\right|
&\leq \sum_{i = 1}^n\|\widehat\bR_i\|\left\|\sum_{j = 1}^n(A_{ij} - \rho_n\bx_{0i}\transpose\bx_{0j})\gamma_{ijk}\right\|
\\&
\leq \left(\sum_{i = 1}^n\|\widehat\bR_i\|^2\right)^{1/2}\left\{\sum_{i = 1}^n\left\|\sum_{j = 1}^n(A_{ij} - \rho_n\bx_{0i}\transpose\bx_{0j})\bgamma_{ij}\right\|^2\right\}^{1/2}\\
&= O_{\prob_0}\left(\frac{(\log n)^{1\vee\omega}}{\sqrt{n\rho_n^5}}\right)\left\{\frac{1}{n}\sum_{i = 1}^n\mathrm{tr}\{\bG_n(\bx_{0i})^{-1}\} + o_{\prob_0}(1)\right\}^{1/2}
 = o_{\prob_0}(1).
\end{align*}
Hence, by condition \eqref{eqn:strong_convergence_measure} and the uniform convergence of $\bG_n(\bx)^{-1}\to\bG(\bx)^{-1}$ for all $\bx$ (Lemma \ref{lemma:uniform_convergence_G}), we obtain (see, for example, Exercise 3 in Section 4.4 of \citealp{chung2001course})
\[
\frac{1}{n}\sum_{i = 1}^n\mathrm{tr}\{\bG_n(\bx_{0i})\} = \int \mathrm{tr}\{\bG_n(\bx)^{-1}\}F_n(\mathrm{d}\bx)
\to\int_\calX\mathrm{tr}\{\bG(\bx)^{-1}\}F(\mathrm{d}\bx). 
\]
This completes the first part of the theorem. For the second part, we observe that
\begin{align*}
&\sum_{j = 1}^n\expect_0\left\{\left\|\frac{1}{\sqrt{n\rho_n}}\frac{(A_{ij} - \rho_n\bx_{0i}\transpose\bx_{0j})}{\bx_{0i}\transpose\bx_{0j}(1 - \rho_n\bx_{0i}\transpose\bx_{0j})}\bG_n(\bx_{0i})^{-1}\bx_{0j}\right\|^3\right\}\\
&\quad\leq \frac{1}{(n\rho_n)^{3/2}}\sum_{j = 1}^n\frac{\expect_0\{|A_{ij} - \rho_n\bx_{0i}\transpose\bx_{0j}|^3\}}{\{\bx_{0i}\transpose\bx_{0j}(1 - \rho_n\bx_{0i}\transpose\bx_{0j})\}^3}\|\bG_n^{-1}(\bx_{0i})\|_2^3\|\bx_{0j}\|^3
\lesssim\frac{1}{\sqrt{n\rho_n}}\lesssim\frac{(\log n)^{1\vee\omega}}{\sqrt{n\rho_n^5}}\to 0,
\end{align*}
and
\begin{align*}
\sum_{j = 1}^n\var_0\left\{\frac{1}{\sqrt{n\rho_n}}\frac{(A_{ij} - \rho_n\bx_{0i}\transpose\bx_{0j})\bG_n(\bx_{0i})^{-1}\bx_{0j}}{\bx_{0i}\transpose\bx_{0j}(1 - \rho_n\bx_{0i}\transpose\bx_{0j})}\right\}
& = \sum_{j \neq i}^n\var_0\left\{\frac{1}{\sqrt{n\rho_n}}\frac{(A_{ij} - \rho_n\bx_{0i}\transpose\bx_{0j})}{\bx_{0i}\transpose\bx_{0j}(1 - \rho_n\bx_{0i}\transpose\bx_{0j})}\bG_n(\bx_{0i})^{-1}\bx_{0j}\right\}\\
& = \frac{1}{n\rho_n}\sum_{j \neq i}^n\frac{\rho_n\bG_n(\bx_{0i})^{-1}\bx_{0j}\bx_{0j}\transpose\bG_n(\bx_{0i})^{-1}}{\bx_{0i}\transpose\bx_{0j}(1 - \rho_n\bx_{0i}\transpose\bx_{0j})}\\
& = \bG_n(\bx_{0i})^{-1}\to \bG(\bx_{0i})^{-1}. 
\end{align*}
It follows from the Lyapunov's central limit theorem and Theorem \ref{thm:asymptotic_normality_OS} that
\[
\sqrt{n}(\bW\transpose\widehat\bx_i - \bx_{0i})
= \sum_{j = 1}^n\frac{1}{\sqrt{n\rho_n}}\frac{(A_{ij} - \rho_n\bx_{0i}\transpose\bx_{0j})}{\bx_{0i}\transpose\bx_{0j}(1 - \rho_n\bx_{0i}\transpose\bx_{0j})}\bG_n(\bx_{0i})^{-1}\bx_{0j} + o_{\prob_0}(1) \overset{\calL}{\to} \mathrm{N}(\zero, \bG(\bx_{0i})^{-1}). 
\]
The proof is thus completed.
\end{proof}

\section{Proof of Theorems \ref{thm:LSE_limit_theorem} and \ref{thm:LSE_transform_initial_condition} (Limit Theorems for the LSE)} 
\label{sec:proof_of_theorem_thm:lse_limit_theorem}

In this section of the proofs, we introduce following notations: 
Let $\one = [1,\ldots,1]\transpose\in\mathbb{R}^n$ be the $n$-dimensional vector with coordinates being ones, $\bP_0 = \expect_0(\bA) = \rho_n\bX_0\bX_0\transpose$, $\bT = \mathrm{diag}(\rho_n\bX_0\bX_0\transpose\one)$, $\bD = \mathrm{diag}(\bA\one)$, and $\widetilde\bE = n\rho_n\{\calL(\bA) - \calL(\bP_0)\}$. We further denote $\calL(\bA) = \sum_{i = 1}^n\widetilde{\lambda}_i(\widetilde{\bu}_{\bA})_i(\widetilde{\bu}_{\bA})_i\transpose$ the spectral decomposition of $\calL(\bA)$ with $|\widetilde{\lambda}_1|\geq\ldots\geq|\widetilde{\lambda}_n|$, $\calL(\bP_0) = \widetilde\bU_\bP\widetilde{\bS}_\bP\widetilde\bU_\bP\transpose$ the (compact) spectral decomposition of $\calL(\bP_0)$, where $\widetilde\bU_\bP\in\mathbb{O}(n, d)$, and $\widetilde{\bS}_\bP = \mathrm{diag}[\lambda_1\{\calL(\bP_0)\},\ldots, \lambda_d\{\calL(\bP_0)\}]$. We observe that the LSE can be written as $\breve{\bX} = \widetilde{\bU}_\bA\widetilde{\bS}_\bA^{1/2}$, where $\widetilde\bU_{\bA} = [(\widetilde{\bu}_{\bA})_1,\ldots,(\widetilde{\bu}_{\bA})_d]$, and $\widetilde{\bS}_\bA = \mathrm{diag}(|\widetilde\lambda_1|,\ldots,|\widetilde\lambda_d|)$. 

\subsection{Proof of the Limit \eqref{eqn:LSE_convergence}} 
\label{sub:proof_of_the_limit_eqn:lse_convergence}


\begin{proof}[\bf Proof of the limit \eqref{eqn:LSE_convergence}]
The proof is almost the same as the proof in \cite{tang2018} and we only sketch the argument and present the difference. Using the proof there, we obtain the following local expansion of $\breve\bX\bW - \bY_0$ for some orthogonal alignment matrix $\bW$ that may depend on $n$:
\begin{align}\label{eqn:LSE_expansion}
\breve{\bX}\bW - \bY_0 = \bT^{-1/2}(\bA - \bP_0)\bT^{-1/2}\bY_0(\bY_0\transpose\bY_0)^{-1} + \frac{1}{2}\bT^{-1}(\bT - \bD)\bY_0 + \breve\bR,
\end{align}
where $\|\breve{\bR}\|_{\mathrm{F}} = O((n\rho_n)^{-1})$ with high probability. Using argument developed in Appendix B.4 in \cite{tang2018} through the application of the concentration inequalities developed in \cite{boucheron2003} (Theorem 5 and Theorem 6 there), we obtain that $n\rho_n\|\breve{\bX}\bW - \bY_0\|_{\mathrm{F}}^2$ is concentrated around 
\begin{align*}
&n\rho_n\expect_0\left\{\left\|
\bT^{-1/2}(\bA - \bP_0)\bT^{-1/2}\bY_0(\bY_0\transpose\bY_0)^{-1} + \frac{1}{2}\bT^{-1}(\bT - \bD)\bY_0
\right\|_{\mathrm{F}}^2\right\}\\
&\quad
= n\rho_n\expect_0\left\{\|\bT^{-1/2}(\bA - \bP_0)\bT^{-1/2}\bY_0(\bY_0\transpose\bY_0)^{-1}\|_{\mathrm{F}}^2\right\}
+ \frac{n\rho_n}{4}\expect_0\left\{\|\bT^{-1}(\bT - \bD)\bY_0\|_{\mathrm{F}}^2\right\}\\
&\qquad + n\rho_n\mathrm{tr}\left[\expect_0\left\{\bY_0\transpose\bT^{-1}(\bT - \bD)\bT^{-1/2}(\bA - \bP_0)\bT^{-1/2}\bY_0(\bY_0\transpose\bY_0)^{-1}\right\}\right]
\end{align*}
almost surely.
The only difference is that the expected values in \cite{tang2018} is taken with regard to both $\bA$ and $\bP$, whereas the expected values here is taken with respect to $\bA$ only as $\bP$ is deterministic in the current setup.
Furthermore, using the argument developed in page 2407 of \cite{tang2018} where the expected value is taken with regard to $\bA$ conditioning on $\bP_0$, we obtain
\begin{align*}
&\expect_0\left\{\|\bT^{-1/2}(\bA - \bP_0)\bT^{-1/2}\bY_0(\bY_0\transpose\bY_0)^{-1}\|_{\mathrm{F}}^2\right\}
= \mathrm{tr}\left\{(\bY_0\transpose\bY_0)^{-1}\bY_0\transpose\widetilde\bM\bY_0(\bY_0\transpose\bY_0)^{-1}\right\},
\end{align*}
where
\[
\widetilde\bM = \mathrm{diag}\left\{\frac{1}{n^2\rho_n}
\sum_{j = 1}^n\frac{ \bx_{01}\transpose\bx_{0j}(1 - \rho_n\bx_{01}\transpose\bx_{0j}) }{{\bx_{01}\transpose\bmu_n\bx_{0j}\transpose\bmu_n}},\ldots, 
\frac{1}{n^2\rho_n}\sum_{j = 1}^n\frac{ \bx_{0n}\transpose\bx_{0j}(1 - \rho_n\bx_{0n}\transpose\bx_{0j}) }{{\bx_{0n}\transpose\bmu_n\bx_{0j}\transpose\bmu_n}}
\right\}.
\]
The key difference is that instead of using the strong law of large numbers applied to the random latent positions $\bx_{01},\ldots,\bx_{0n}$, we make use of the assumption \eqref{eqn:strong_convergence_measure} to show the desired convergence results. By construction of $\bY_0$, it is easy to see that $\bY_0\transpose\bY_0\to\widetilde\bDelta$ as $n\to\infty$. Furthermore, 
\begin{align*}
\bY_0\transpose\widetilde\bM\bY_0
& = \frac{1}{n^2\rho_n}\sum_{i = 1}^n\by_{0i}\by_{0i}\transpose\sum_{j = 1}^n\frac{ \bx_{0i}\transpose\bx_{0j}(1 - \rho_n\bx_{0i}\transpose\bx_{0j}) }{{\bx_{0i}\transpose\bmu_n\bx_{0j}\transpose\bmu_n}}\\
& = \frac{1}{n^3\rho_n}\sum_{i = 1}^n\frac{\bx_{0i}\bx_{0i}\transpose}{\bx_{0i}\transpose\bmu_n}\sum_{j = 1}^n\frac{ \bx_{0i}\transpose\bx_{0j}(1 - \rho_n\bx_{0i}\transpose\bx_{0j}) }{{\bx_{0i}\transpose\bmu_n\bx_{0j}\transpose\bmu_n}}\\
& = \frac{1}{n^3\rho_n}\sum_{j = 1}^n\sum_{i = 1}^n\frac{\bx_{0i}\bx_{0i}\transpose}{\bx_{0i}\transpose\bmu_n}\frac{ \bx_{0i}\transpose\bx_{0j}(1 - \rho_n\bx_{0i}\transpose\bx_{0j}) }{{\bx_{0i}\transpose\bmu_n\bx_{0j}\transpose\bmu_n}}\\
& = \frac{1}{n^2\rho_n}\sum_{j = 1}^n\frac{1}{\bx_{0j}\transpose\bmu_n}\frac{1}{n}\sum_{i = 1}^n\frac{ \bx_{0i}\transpose\bx_{0j}(1 - \rho_n\bx_{0i}\transpose\bx_{0j}) \bx_{0i}\bx_{0i}\transpose}{{(\bx_{0i}\transpose\bmu_n)^2}}\\
& = \frac{1}{n\rho_n}\left[\frac{1}{n}\sum_{j = 1}^n\widetilde\bV_n(\bx_{0j})\right],
\end{align*}
where
\[
\widetilde\bV_n(\bx) = \frac{1}{n}\sum_{i = 1}^n\frac{ \bx_{0i}\transpose\bx(1 - \rho_n\bx_{0i}\transpose\bx) \bx_{0i}\bx_{0i}\transpose}{{(\bx_{0i}\transpose\bmu_n)^2}} \to \widetilde\bV(\bx)
 = \int_\calX \frac{ \bx_{1}\transpose\bx(1 - \rho\bx_{1}\transpose\bx) \bx_{1}\bx_{1}\transpose}{{(\bx_{1}\transpose\bmu)^2}}F(\mathrm{d}\bx_1)
\]
uniformly over $\bx\in\calX$ (the argument for proving the uniform convergence is the same as that in the proof of Lemma \ref{lemma:uniform_convergence_G}). Hence we conclude that
\[
n\rho_n\mathrm{tr}\{(\bY_0\transpose\bY_0)^{-1}\bY_0\transpose\widetilde\bM\bY_0(\bY_0\transpose\bY_0)^{-1}\}
\to \mathrm{tr}\left\{\widetilde\bDelta^{-1}\int_\calX \bV(\bx)F(\mathrm{d}\bx)\widetilde\bDelta^{-1}\right\}.
\]
A similar argument can be applied to derive that
\[
\frac{n\rho_n}{4}\expect_0\left\{\|\bT^{-1}(\bT - \bD)\bY_0\|_{\mathrm{F}}^2\right\}\to
\frac{1}{4}\mathrm{tr}\left\{\int_\calX\frac{\bx_1\bx_1\transpose}{(\bx_1\transpose\bmu)^2}\left(1 - \frac{\rho\bx_1\transpose\bDelta\bx_1}{\bx_1\transpose\bmu}\right)F(\mathrm{d}\bx_1)\right\}
\]
and that
\begin{align*}
&n\rho_n\mathrm{tr}\left[\expect_0\left\{\bY_0\transpose\bT^{-1}(\bT - \bD)\bT^{-1/2}(\bA - \bP_0)\bT^{-1/2}\bY_0(\bY_0\transpose\bY_0)^{-1}\right\}\right]\\
&\quad\to \rho\mathrm{tr}\left[
\iint_{\calX^2}\frac{(\bx_1\transpose\bx_2\bx_2\transpose\bx_1)}{(\bx_1\transpose\bmu)^2(\bx_2\transpose\bmu)}\bx_1\bx_2\transpose{F}(\mathrm{d}\bx_1)F(\mathrm{d}\bx_2)\widetilde\bDelta^{-1}
\right] - \mathrm{tr}\left\{\int_\calX\frac{\bx_1\bx_1\transpose}{(\bx_1\transpose\bmu)^2}F(\mathrm{d}\bx_1)\right\}
\end{align*}
The proof is completed by combining the above derivations. 

\end{proof}

\subsection{Proof of the Asymptotic Normality \eqref{eqn:LSE_normality} Under the Sparse regime (ii)} 
\label{sub:proof_of_the_asymptotic_normality_eqn:lse_normality_under_the_sparse_regime_ii_}


\begin{proof}[\bf Proof of the asymptotic normality \eqref{eqn:LSE_normality} under the sparse regime (ii)]
We first prove the asymptotic normality of the rows of the LSE under the condition (ii), and this is a simple modification of Appendix B.1 in \cite{tang2018}. 

Let $\tau:[n]\to[K]$ be the cluster assignment function such that $\tau(i) = k$ when $\bx_{0i} = \bnu_k$, and let $n_k = \sum_{i = 1}^n\mathbbm{1}(\bx_{0i} = \bnu_k)$. Clearly, $n_k/n\to \pi_k$ for each $k\in[K]$. Since $F = \sum_{k = 1}^K\pi_k\delta_{\bnu_k}$, then there exists a permutation matrix $\bPi$ such that 
$
[\bnu_1,\ldots,\bnu_1,\bnu_2,\ldots,\bnu_2,\ldots,\bnu_K,\ldots,\bnu_K]\transpose
 = \bPi\transpose\bX_0 = \bPi\transpose
 [\bx_{01},\ldots,\bx_{0n}]\transpose.
$

Now suppose for a fixed $i\in[n]$, $\bx_{0i} = \bnu_k$ for a $k\in[K]$ that depends on $i$. Let $\bPi_k$ be the permutation matrix within the $k$th cluster. In the case where $1 < k < K$, $\bPi_k$ has the form
\[
\bPi_k = \bPi\begin{bmatrix*}
\eye_{n_1} & & & & & & \\
           & \ddots & & & & & \\
           & & \eye_{n_{k - 1}}\\
           & & & \bQ_k &  \\
           & & & & \eye_{n_{k + 1}}\\
           & & & & & \ddots\\
           & & & & & & \eye_{n_K}
\end{bmatrix*}\bPi\transpose
\]
for some $n_k\times n_k$ permutation matrix $\bQ_k$. A similar block diagonal $\bPi_k$ can be explicitly written down when $k = 1$ or $k = K$. Without loss of generality, we may assume that $k\neq 1$ and $k\neq K$. Furthermore, the adjacency matrix $\bA$ can be written as the following block matrix form
\[
\bPi\transpose\bA\bPi = \begin{bmatrix*}
\bA_{\star11} & \bA_{\star1k} & \bA_{\star12}\\
\bA_{\star1k}\transpose & \bA_{ kk}  & \bA_{\star 2k}\transpose\\
\bA_{\star21} & \bA_{\star 2k} & \bA_{\star22}
\end{bmatrix*},
\]
where $\bA_{kk} = [A_{mj}:\tau(m) = \tau(j) = k]$, $\bA_{\star 11} = [A_{mj}:\tau(m), \tau(j) < k]$, $\bA_{\star 22} = [A_{mj}:\tau(m), \tau(j) > k]$, $\bA_{\star 12} = [A_{mj}:\tau(m) < k, \tau(j) > k]$, and
\[
\begin{bmatrix*}
\bA_{\star1k}\\\bA_{\star2k}
\end{bmatrix*} = [\bA_{mj}:\tau(m) \neq k, \tau(j) = k].
\]
Clearly, the entries of $\bA_{kk}$ are independent $\mathrm{Bernoulli}(\rho_n\langle \bnu_k,\bnu_k\rangle)$ random variables, where $\langle\cdot,\cdot\rangle$ denotes the Euclidean inner product, and the entries of $\bA_{\star1k},\bA_{\star2k}$ are independent $\mathrm{Bernoulli}(\rho_n\langle\bnu_l,\bnu_k\rangle)$ random variables. Therefore,
\[
\bA_{\star kk}\overset{\calL}{=} \bQ_k\bA_{ kk}\bQ_k\transpose, \quad
\bA_{\star 1k}\overset{\calL}{=} \bA_{\star 1k}\bQ_k\transpose,\quad
\bA_{\star 2k}\overset{\calL}{=} \bA_{\star 2k}\bQ_k\transpose.
\]
Hence,
\begin{align*}
\bPi_k\bA\bPi_k\transpose = 
\bPi
\begin{bmatrix*}
\bA_{\star11} & \bA_{\star 1k}\bQ_k\transpose & \bA_{\star 12}\\
\bQ_k\bA_{\star1k}\transpose & \bQ_k\bA_{kk}\bQ_k\transpose & \bQ_k\bA_{\star 2k}\transpose\\
\bA_{\star 21} & \bA_{\star 2k}\bQ_k\transpose & \bA_{\star 22}
\end{bmatrix*}
\bPi\transpose
 \overset{\calL}{=} \bPi
\begin{bmatrix*}
\bA_{\star11} & \bA_{\star1k} & \bA_{\star12}\\
\bA_{\star1k}\transpose & \bA_{ kk}  & \bA_{\star 2k}\transpose\\
\bA_{\star21} & \bA_{\star 2k} & \bA_{\star22}
\end{bmatrix*}\bPi\transpose
= \bA. 
\end{align*}
Therefore, $\bPi_k\breve{\bX} \overset{\calL}{ = }\breve{\bX}$, $\bPi_k\bX_0 = \bX_0$, and $\bPi_k\bY_0 = \bY_0$. Observe that the orthogonal alignment matrix $\bW$ does not change with the permutation matrix $\bPi_k$. Hence, this further implies the exchangeability of the rows within the $k$th cluster for the remainder matrix $\bR$:
\[
\bPi_k\breve{\bR} = \bPi_k(\breve{\bX}\bW - \bY_0) = \bPi_k\breve{\bX}\bW - \bY_0 \overset{\calL}{ = } \breve{\bX}\bW - \bY_0 = \breve{\bR}.
\]
Let $\breve{\br}_j$ be the $j$th row of $\breve\bR$, $j\in [n]$. Clearly, by the exchangeability within clusters, for any $i,j\in[n]$ such that $\tau(j) = \tau(i) = k$, we see that $\breve{\br}_j = \breve{\br}_i$. Using the fact that $\bE_0(\|\breve\bR\|_{\mathrm{F}}^2) = O((n\rho_n)^{-2})$ (which can be easily derived using the fact that with probability at least $1 - n^{-3}$, $\|\breve\bR\|_{\mathrm{F}} \lesssim (n\rho_n)^{-1}$; also see Appendix B.1 in \citealp{tang2018}), we see that for any $i\in[n]$ with $\tau(i) = k$,
\begin{align*}
n^2\rho_n\expect_0(\|\breve{\br}_i\|_2^2) &= n^2\rho_n\frac{1}{n_k}\sum_{j:\tau(j) = \tau(i) = k}\expect_0(\|\breve{\br}_j\|_2^2) \leq 
  \frac{n^2\rho_n}{n_k}\expect(\|\breve\bR\|_{\mathrm{F}}^2)
\lesssim\frac{n\rho_n}{\pi_k + o(1)}(n\rho_n)^{-2}\to 0.
\end{align*}
This shows that $\breve{\br}_i = o_{\prob_0}(n\rho_n^{1/2})$. 

The rest of the proof is devoted to prove the asymtotic normality of the $i$th row of
\[
\bT^{-1/2}(\bA - \bP_0)\bT^{-1/2}\bY_0(\bY_0\transpose\bY_0)^{-1} + \frac{1}{2}\bT^{-1}(\bT - \bD)\bY_0,
\]
and is almost the same as that of Appendix B.1 in \cite{tang2018}. Let $\breve\bx_i$ be the $i$th row of $\breve\bX$. In essence, using the expansion \eqref{eqn:LSE_expansion}, we have
\begin{align*}
n\rho_n^{1/2}(\bW\transpose\breve\bx_i - \by_{0i})
& = \sum_{j\neq i}\left\{
\widetilde\bDelta_n^{-1}\frac{\bx_{0j}}{\sqrt{\bx_{0i}\transpose\bmu_n}\bx_{0j}\transpose\bmu_n} - \frac{\bx_{0i}}{2(\bx_{0i}\transpose\bmu_n)^{3/2}}
\right\}\frac{A_{ij} - \rho_n\bx_{0i}\transpose\bx_{0j}}{\sqrt{n\rho_n}}
+ o_{\prob_0}(1),
\end{align*}
where $\widetilde\bDelta_n = \bY_0\transpose\bY_0$. 
The sum of the third moment of the sum of independent random vectors in the first term of the right-hand side of the previous display is bounded above by
\begin{align*}
&\frac{1}{(n\rho_n)^{3/2}}\sum_{j\neq i}^n\left\|\widetilde\bDelta_n^{-1}\frac{1}{(\bx_{0i}\transpose\bmu_n)^{1/2}\bx_{0j}\transpose\bmu_n}
- \frac{\bx_{0i}}{2(\bx_{0i}\transpose\bmu_n)^{3/2}}
\right\|_2^3\expect_0(|A_{ij} - \rho_n\bx_{0i}\transpose\bx_{0j}|^3)
\lesssim \frac{n\rho_n}{(n\rho_n)^{3/2}}\to0.
\end{align*}
The variance of the sum of these independent random vectors is
\begin{align*}
&\frac{1}{n}\sum_{j\neq i}\left\{
\widetilde\bDelta_n^{-1}\frac{\bx_{0j}}{\sqrt{\bx_{0i}\transpose\bmu_n}\bx_{0j}\transpose\bmu_n} - \frac{\bx_{0i}}{2(\bx_{0i}\transpose\bmu_n)^{3/2}}
\right\} \bx_{0i}\transpose\bx_{0j}(1 - \rho_n\bx_{0i}\transpose\bx_{0j})
\left\{
\widetilde\bDelta_n^{-1}\frac{\bx_{0j}}{\sqrt{\bx_{0i}\transpose\bmu_n}\bx_{0j}\transpose\bmu_n} - \frac{\bx_{0i}}{2(\bx_{0i}\transpose\bmu_n)^{3/2}}
\right\}\transpose\\
&\quad
= \frac{1}{n}\sum_{j\neq i}\left\{
\frac{\widetilde\bDelta_n^{-1}(\bx_{0i}\transpose\bmu_n)\bx_{0j}}{({\bx_{0i}\transpose\bmu_n})^{3/2}} - \frac{\bx_{0i}\bmu_n\transpose\bx_{0j}}{2(\bx_{0i}\transpose\bmu_n)^{3/2}}
\right\} \frac{\bx_{0i}\transpose\bx_{0j}(1 - \rho_n\bx_{0i}\transpose\bx_{0j})}{(\bx_{0j}\transpose\bmu_n)^2}
\left\{
\frac{\widetilde\bDelta_n^{-1}(\bx_{0i}\transpose\bmu_n)\bx_{0j}}{({\bx_{0i}\transpose\bmu_n})^{3/2}} - \frac{\bx_{0i}\bmu_n\transpose\bx_{0j}}{2(\bx_{0i}\transpose\bmu_n)^{3/2}}
\right\}\transpose\\
&\quad = \frac{1}{n(\bx_{0i}\transpose\bmu_n)}
\sum_{j\neq i}
\left(\widetilde\bDelta_n^{-1} - \frac{\bx_{0i}\bmu_n\transpose}{2\bx_{0i}\transpose\bmu_n}\right) 
\frac{\bx_{0i}\transpose\bx_{0j}(1 - \rho_n\bx_{0i}\transpose\bx_{0j})\bx_{0j}\bx_{0j}\transpose}{(\bx_{0j}\transpose\bmu_n)^2}
\left(\widetilde\bDelta_n^{-1} - \frac{\bx_{0i}\bmu_n\transpose}{2\bx_{0i}\transpose\bmu_n}\right)\transpose\\
&\quad \to \widetilde\bSigma(\bx_{0i}).
\end{align*}
The proof of the asymptotic normality of $\bW\transpose\breve{\bx}_i - \by_{0i}$  is thus completed by the Lyapunov's central limit theorem. 
\end{proof}


\subsection{Some Technical Lemmas for the LSE} 
\label{sub:some_technical_lemmas_for_the_lse}


We next focus on the proof of the asymptotic normality \eqref{eqn:LSE_normality} of the rows of the LSE under the condition (i). 
The proof strategy is enormously different than that presented in \cite{tang2018}, where the authors make use of the exchangeability property of $\bA$ and $\bx_{01},\ldots,\bx_{0n}$ and such a exchangeability no longer holds in our current setup. 
Here we shall take a different approach and follow the framework of \cite{cape2019signal} to establish the asymptotic normality \eqref{eqn:LSE_normality}. In preparation for doing so, we need to establish a collection of preliminary results first. 
\begin{lemma}\label{lemma:Laplacian_moment_bound}
Assume the conditions of Theorem \ref{thm:LSE_limit_theorem} hold under the dense regime (i). Denote
\[
\widetilde{\bE}_1 = n\bT^{-1/2}(\bA - \bP_0)\bT^{-1/2} + \frac{n}{2}\bT^{-1/2}\bP_0\bT^{-3/2}(\bT - \bD) + \frac{n}{2}\bT^{-3/2}(\bT - \bD)\bP_0\bT^{-1/2}.
\]
Let $\widetilde{\bu}_{01},\ldots,\widetilde{\bu}_{0d}$ be the column vectors of $\widetilde\bU_\bP$. 
Then there exist constants $C_{\widetilde\bE} > 0$, $\nu > 0$, such that 
\begin{align*}
\prob_0\left[\bigcup_{t = 1}^2
\bigcup_{i = 1}^n\bigcup_{k = 1}^d\left\{|\be_i\transpose\widetilde\bE_1^t\widetilde{\bu}_{0k}| > C_{\widetilde\bE}^t(\log n)^{2t}n^{t/2}\|\widetilde\bu_{0k}\|_\infty\right\}
\right]\leq \exp\{-\nu(\log n)^2\}.
\end{align*}
\end{lemma}
\begin{proof}[\bf Proof]
Let $H_{ij} = n^{-1/2}(A_{ij} - \bx_{0i}\transpose\bx_{0j})$ and $\bH = [H_{ij}]_{n\times n}$. We first consider the case of $t = 1$. Let $\widetilde{\bu}_{0k} = [\widetilde{u}_{01k},\ldots,\widetilde{u}_{0nk}]\transpose$. Then for any $i\in[n]$ and $k\in [d]$, we have
\begin{align*}
|\be_i\transpose\widetilde\bE_1\widetilde{\bu}_{0k}|
&\leq |\be_i\transpose n\bT^{-1/2}(\bA - \bP_0)\bT^{-1/2}\widetilde{\bu}_{0k}| + \frac{n}{2}|\be_i\transpose\bT^{-1/2}\bP_0\bT^{-3/2}(\bT - \bD)\widetilde{\bu}_{0k}|
\\&\quad
 + \frac{n}{2}|\be_i\transpose\bT^{-3/2}(\bT - \bD)\bP_0\bT^{-1/2}\widetilde{\bu}_{0k}|\\
&\leq |\be_i\transpose n\bT^{-1/2}(\bA - \bP_0)\bT^{-1/2}\widetilde{\bu}_{0k}| + n\|\bT^{-1/2}\|_\infty\|\bP_0\|_\infty\|\bT^{-3/2}\|_\infty\|\bT - \bD\|_\infty \|\widetilde{\bu}_{0k}\|_\infty.
\end{align*}
Since $\|\bT^{-1}\|_\infty = \|\bT^{-1/2}\|_2 = O(n^{-1/2})$, $\|\bP_0\|_\infty = O(n)$, and $\|\widetilde{\bS}_\bP\|_\infty = \|\widetilde\bS_\bP\|_2 = O(1)$, the second term of the previous display is upper bounded by a constant multiple of $\|\widetilde{\bu}_{0k}\|_\infty\|\bT - \bD\|_\infty$.
By the Chernoff bound, we see that $\|\bT - \bD\|_\infty$ is upper bounded by a constant multiple of $\sqrt{n}\log n$ with probability at least $1 - \exp\{-c(\log n)^2\}$ for a constant $c > 0$. Hence we conclude
\[
n\|\bT^{-1/2}\|_\infty\|\bP_0\|_\infty\|\bT^{-3/2}\|_\infty\|\bT - \bD\|_\infty \|\widetilde{\bu}_{0k}\|_\infty\lesssim (\log n)n^{1/2}\|\widetilde{\bu}_{0k}\|_\infty
\]
with probability at least $1 - \exp\{-c(\log n)^2\}$. For the first term, it has the form
\[
\left|\sum_{j = 1}^n\frac{\{A_{ij} - \expect_0(A_{ij})\}\widetilde{u}_{0jk}}{(\bx_{0i}\transpose\bmu_n\bx_{0j}\transpose\bmu_n)^{1/2}}\right|,
\]
which can be upper bounded by a constant of multiple of $\log n\asymp  (\log n)n^{1/2}\|\widetilde{\bu}_{0k}\|_\infty$ with probability at least $1 - \exp\{-c(\log n)^2\}$ for a constant $c > 0$ by Hoeffding's inequality. Hence we conclude from the union bound over $i\in[n]$ and $k\in[d]$ that 
\begin{align*}
\prob_0\left[
\bigcup_{i = 1}^n\bigcup_{k = 1}^d\left\{|\be_i\transpose\widetilde\bE_1\widetilde{\bu}_{0k}| > C_{\widetilde\bE_1}(\log n)^{2}n^{1/2}\|\widetilde{\bu}_{0k}\|_\infty\right\}
\right]\leq \exp\{-\nu_1(\log n)^2\}
\end{align*}
for some constants $C_{\widetilde\bE_1} > 0$ and $\nu_1 > 0$ because $ne^{-c(\log n)^2}\leq e^{-c/2(\log n)^2}$ for sufficiently large $n$. 

\noindent
We next consider the case where $t = 2$, which is slightly more involved. Write
\begin{align*}
\be_i\transpose\widetilde\bE_1^2\widetilde{\bu}_{0k}
& = n^2\be_i\transpose\bT^{-1/2}(\bA - \bP_0)\bT^{-1}(\bA - \bP_0)\bT^{-1/2}\widetilde{\bu}_{0k}
 + \frac{n^2}{2}\be_i\transpose\bT^{-1/2}(\bA - \bP_0)\bT^{-1}\bP_0\bT^{-3/2}(\bT - \bD)\widetilde{\bu}_{0k}\\
&\quad + \frac{n^2}{2}\be_i\transpose\bT^{-1/2}(\bA - \bP_0)\bT^{-2}(\bT - \bD)\bP_0\bT^{-1/2}\widetilde{\bu}_{0k}\\ 
&\quad + \frac{n^2}{2}\be_i\transpose\bT^{-1/2}\bP_0\bT^{-3/2}(\bT - \bD)\bT^{-1/2}(\bA - \bP_0)\bT^{-1/2}\widetilde{\bu}_{0k}\\
&\quad + \frac{n^2}{2}\be_i\transpose\bT^{-3/2}(\bT - \bD)\bP_0\bT^{-1}(\bA - \bP_0)\bT^{-1/2}\widetilde{\bu}_{0k}
 + \be_i\transpose\bR_{\widetilde{\bE}_1}^{(2)}\widetilde{\bu}_{0k},
\end{align*}
where
\begin{align*}
\|\bR_{\widetilde{\bE}_1}^{(2)}\|_\infty
& \lesssim n^2\|\bT^{-1/2}\|_\infty^2\|\bP_0\|_\infty^2\|\bT^{-3/2}\|_\infty^2\|\bT - \bD\|_\infty^2
\lesssim \|\bT - \bD\|_\infty^2.
\end{align*}
Since Chernoff bound implies $\|\bT - \bD\|_\infty^2\lesssim n(\log n)^2$ with probability at least $1 - 2\exp\{-c(\log n)^2\}$, it follows that
\[
|\be_i\transpose\bR_{\widetilde{\bE}_1}^{(2)}\widetilde{\bu}_{0k}|\leq \|\bR_{\widetilde{\bE}_1}^{(2)}\|_\infty\|\widetilde\bu_{0k}\|_\infty\lesssim n(\log n)^2\|\widetilde\bu_{0k}\|_\infty
\]
with probability at least $1 - 2\exp\{-c(\log n)^2\}$. We then analyze the first five terms on the right-hand side of $\be_i\transpose\widetilde{\bE}_1^2\widetilde{\bu}_{0k}$ separately. Let $i\in[n]$ and $k\in[d]$ be fixed. 
\begin{itemize}
  \item[(1)] For the first term, we observe that
  \begin{align*}
  &n^2\be_i\transpose\bT^{-1/2}(\bA - \bP_0)\bT^{-1}(\bA - \bP_0)\bT^{-1/2}\widetilde{\bu}_{0k}\\
  &\quad = \sum_{i_1 = 1}^n\sum_{i_2 = 1}^n\frac{\{A_{ii_1} - \expect_0(A_{ii_1})\}}{(\bx_{0i}\transpose\bmu_n)^{1/2}(\bx_{0i_1}\transpose\bmu_n)^{1/2}}\frac{\{A_{i_1i_2} - \expect_0(A_{i_1i_2})\}}{(\bx_{0i_1}\transpose\bmu_n)^{1/2}(\bx_{0i_2}\transpose\bmu_n)^{1/2}}\widetilde{u}_{0i_2k}\\
  &\quad = 
  n\sum_{i_1 = 1}^n
  \sum_{i_2 = 1}^n H_{ii_1}H_{i_1i_2}
  \frac{\widetilde{u}_{0i_2k}}
  {(\bx_{0i}\transpose\bmu_n)^{1/2}
   (\bx_{0i_1}\transpose\bmu_n)
   (\bx_{0i_2}\transpose\bmu_n)^{1/2}}.
  \end{align*}
  By exploiting the proof of Lemma \ref{lemma:moment_bound_ASE}, we immediately conclude that
  \[
  \expect_0\{|n^2\be_i\transpose\bT^{-1/2}(\bA - \bP_0)\bT^{-1}(\bA - \bP_0)\bT^{-1/2}\widetilde{\bu}_{0k}|^p\}\leq n^p(2Cp)^{2p}\|\widetilde{\bu}_{0k}\|_\infty^p
  \]
  for any $p$ for a constant $C \geq 1$, and hence, with $p = \lfloor(\log n)^2/(4C)\rfloor$, 
  \begin{align*}
  &\prob_0\left\{|n^2\be_i\transpose\bT^{-1/2}(\bA - \bP_0)\bT^{-1}(\bA - \bP_0)\bT^{-1/2}\widetilde{\bu}_{0k}| > C^2(\log n)^4n\|\widetilde\bu_{0k}\|_\infty\right\}\\
  &\quad\leq \frac{n^p(2Cp)^{2p}\|\widetilde{\bu}_{0k}\|_\infty^p}{C^{2p}(\log n)^{4p}n^p\|\widetilde\bu_{0k}\|_\infty^p} = \left\{\frac{2p}{(\log n)^2}\right\}^{2p}
  = \exp\{-c(\log n)^2\}
  \end{align*}
  by Markov's inequality for a constant $c > 0$.
  \item[(2)] For the second term, write
  \begin{align*}
  &n^2\be_i\transpose \bT^{-1/2}(\bA - \bP_0)\bT^{-1}\bP_0\bT^{-3/2}(\bT - \bD)\widetilde{\bu}_{0k}\\
  &\quad
  =-\sum_{j = 1}^n\sum_{l = 1}^n\sum_{m = 1}^n
  \frac{\{A_{ij} - \expect_0(A_{ij})\}\{A_{lm} - \expect_0(A_{lm})\}\by_{0j}\transpose\by_{0l}\widetilde{u}_{0lk}}
  {(\bx_{0i}\transpose\bmu_n\bx_{0j}\transpose\bmu_n)^{1/2}\bx_{0l}\transpose\bmu_n} \\
  &\quad
  =-\left[\sum_{j = 1}^n\frac{\{A_{ij} - \expect_0(A_{ij})\}\by_{0j}}{(\bx_{0i}\transpose\bmu_n\bx_{0j}\transpose\bmu_n)^{1/2}}\right]\transpose\left[
  \sum_{l = 1}^n\sum_{m = 1}^n\frac{\{A_{lm} - \expect_0(A_{lm})\}\by_{0l}\widetilde{u}_{0lk}}{\bx_{0l}\transpose\bmu_n}
  \right]\\
  &\quad=-\left[\frac{1}{\sqrt{n}}\sum_{j = 1}^n\frac{\{A_{ij} - \expect_0(A_{ij})\}\bx_{0j}}{(\bx_{0i}\transpose\bmu_n)^{1/2}\bx_{0j}\transpose\bmu_n}\right]\transpose\left[
  \sum_{l = 1}^n\sum_{m = 1}^n\frac{\{A_{lm} - \expect_0(A_{lm})\}\bx_{0l}\widetilde{u}_{0lk}}{\sqrt{n}(\bx_{0l}\transpose\bmu_n)^{3/2}}
  \right]
  \end{align*}
  For the first factor on the right-hand side of the last display, we obtain directly from Hoeffding's inequality that
  \begin{align*}
  \left\|\frac{1}{\sqrt{n}}\sum_{j = 1}^n\frac{\{A_{ij} - \expect_0(A_{ij})\}\bx_{0j}}{(\bx_{0i}\transpose\bmu_n)^{1/2}\bx_{0j}\transpose\bmu_n}\right\|_2\leq C\log n
  \end{align*}
  with probability at least $1 - 2\exp\{-c(\log n)^2\}$ for constants $c,C > 0$. The second factor can be written as
  \begin{align*}
  &\sum_{l = 1}^n\sum_{m = 1}^n\frac{\{A_{lm} - \expect_0(A_{lm})\}\bx_{0l}\widetilde{u}_{0lk}}{\sqrt{n}(\bx_{0l}\transpose\bmu_n)^{3/2}}\\
  &\quad = \sum_{l < m}\left\{\frac{\bx_{0l}\widetilde{u}_{0lk}}{\sqrt{n}(\bx_{0l}\transpose\bmu_n)^{3/2}} + \frac{\bx_{0m}\widetilde{u}_{0mk}}{\sqrt{n}(\bx_{0m}\transpose\bmu_n)^{3/2}}\right\}\{A_{lm} - \expect_0(A_{lm})\}
   - \sum_{l = 1}^n\frac{(\bx_{0l}\transpose\bx_{0l})\bx_{0l}\widetilde{u}_{0lk}}{\sqrt{n}(\bx_{0l}\transpose\bmu_n)^{3/2}}\\
  &\quad = \sum_{l < m}\left\{\frac{\bx_{0l}\widetilde{u}_{0lk}}{\sqrt{n}(\bx_{0l}\transpose\bmu_n)^{3/2}} + \frac{\bx_{0m}\widetilde{u}_{0mk}}{\sqrt{n}(\bx_{0m}\transpose\bmu_n)^{3/2}}\right\}\{A_{lm} - \expect_0(A_{lm})\} + O(\sqrt{n}\|\widetilde{\bu}_{0k}\|_{\infty}).
  \end{align*}
  Similarly, the first sum on the right-hand side of the previous display is bounded by a constant multiple of $\sqrt{n}\|\widetilde{\bu}_{0k}\|_\infty\log n$ with probability at least $1 - 2\exp\{-c(\log n)^2\}$ for a constant $c > 0$ by Hoeffding's inequality. Hence we conclude that
  \begin{align*}
  &\prob_0\left\{|
  \be_i\transpose n^2\bT^{-1/2}(\bA - \bP_0)\bT^{-1}\bP_0\bT^{-3/2}(\bT - \bD)\widetilde{\bu}_{0k}| > Cn(\log n)^4\|\widetilde{\bu}_{0k}\|_\infty
  \right\}\\
  &\quad
  \leq 
  \prob_0\left\{|
  \be_i\transpose n^2\bT^{-1/2}(\bA - \bP_0)\bT^{-1}\bP_0\bT^{-3/2}(\bT - \bD)\widetilde{\bu}_{0k}| > C\sqrt{n}(\log n)^2\|\widetilde{\bu}_{0k}\|_\infty
  \right\}\\
  &\quad\leq 2\exp\{-c(\log n)^2\}
  \end{align*}
  for some constants $C, c> 0$.

  \item[(3)] The third term can be written as
  \begin{align*}
  &\frac{n^2}{2}\be_i\transpose\bT^{-1/2}(\bA - \bP_0)\bT^{-1/2}(\eye - \bT^{-1}\bD)(\bY_0\bY_0\transpose\widetilde{\bu}_{0k})\\
  &\quad
  = \frac{n}{2}\sum_{j = 1}^n
  \left\{
  \frac{A_{ij} - \expect_0(A_{ij})}{(\bx_{0i}\transpose\bmu_n)^{1/2}(\bx_{0j}\transpose\bmu_n)^{1/2}}\right\}
  \left\{1 - \frac{\sum_{l = 1}^nA_{jl}}{\sum_{l = 1}^n\expect_0(A_{jl})}\right\}\be_j\transpose\bY_0\bY_0\transpose\widetilde{\bu}_{0k}\\
  &\quad
  = -\frac{1}{2}\sum_{j = 1}^n\sum_{l = 1}^n
  \left\{
  \frac{A_{ij} - \expect_0(A_{ij})}{(\bx_{0i}\transpose\bmu_n)^{1/2}(\bx_{0j}\transpose\bmu_n)^{1/2}}
  \right\}\left\{
  \frac{A_{jl} - \expect_0(A_{jl})}{\bx_{0j}\transpose\bmu_n}\be_j\transpose\bY_0\bY_0\transpose\widetilde{\bu}_{0k}\right\}\\
  &\quad
  = -\frac{n}{2}\sum_{i_1 = 1}^n\sum_{i_2 = 1}^nH_{ii_1}H_{i_1i_2}\frac{\be_{i_1}\transpose\bY_0\bY_0\transpose\widetilde{\bu}_{0k}}{(\bx_{0i}\transpose\bmu_n)^{1/2}(\bx_{0i_1}\transpose\bmu_n)^{3/2}}.
  \end{align*}
  Since
  \begin{align*}
  \max_{i,j\in[n]}\left|\frac{\be_j\transpose\bY_0\bY_0\transpose\widetilde{\bu}_{0k}}{(\bx_{0i}\transpose\bmu_n)^{1/2}(\bx_{0j}\transpose\bmu_n)^{3/2}}\right|
  &\lesssim \max_{j\in [n]}\left|\by_{0j}\transpose\sum_{m = 1}^n\by_{0m}(\be_m\transpose\bu_{0k})\right|\lesssim
  \|\widetilde\bu_{0k}\|_\infty,
  \end{align*}
  we then apply the proof of Lemma \ref{lemma:moment_bound_ASE} and obtain that the $p$th moment on the right-hand side can be bounded as follows
  \[
  \expect_0\left\{\left|
  \frac{n^2}{2}\be_i\transpose\bT^{-1/2}(\bA - \bP_0)\bT^{-1/2}(\eye - \bT^{-1}\bD)(\bY_0\bY_0\transpose\widetilde{\bu}_{0k})
  \right|^p
  \right\}\leq
   n^p(2Cp)^{2p}\|\widetilde\bu_{0k}\|_\infty^p.
  \]
  Hence we conclude by Markov's inequality that
  \[
  \prob_0\left\{
  \left|
  \frac{n^2}{2}\be_i\transpose\bT^{-1/2}(\bA - \bP_0)\bT^{-1/2}(\eye - \bT^{-1}\bD)(\bY_0\bY_0\transpose\widetilde{\bu}_{0k})
  \right|
   > Cn(\log n)^4\|\widetilde\bu_{0k}\|_\infty
  \right\}\leq \exp\{-c(\log n)^2\}
  \]
  for some constants $C, c > 0$. The argument for the detailed derivation is exactly the same as (1).
  \item[(4)] For the fourth term, write
  \begin{align*}
  &n^2\be_i\transpose\bT^{-1/2}\bP_0\bT^{-3/2}(\bT - \bD)\bT^{-1/2}(\bA - \bP_0)\bT^{-1/2}\widetilde\bu_{0k}\\
  &\quad
  = \by_{0i}\transpose[\by_{01},\ldots,\by_{0n}]\mathrm{diag}\left\{\left(
  \frac{\sum_{l = 1}^n\expect_0(A_{jl}) - \sum_{l = 1}^nA_{jl}}{\bx_{0j}\transpose\bmu_n}
  \right)_{j = 1}^n\right\}
  \left[\frac{A_{jm} - \expect_0(A_{jm})}{(\bx_{0j}\transpose\bmu_n\bx_{0m}\transpose\bmu_n)^{1/2}}\right]_{n\times n}\widetilde{\bu}_{0k}\\
  &\quad
  = -\frac{1}{n}\sum_{j = 1}^n\sum_{l = 1}^n\sum_{m = 1}^n
  \frac{\{(A_{jl} - \expect_0(A_{jl})\}\{(A_{jm} - \expect_0(A_{jm})\}\bx_{0i}\transpose\bx_{0j}}{(\bx_{0j}\transpose\bmu_n)(\bx_{0i}\transpose\bmu_n\bx_{0j}\transpose\bmu_n)^{1/2}(\bx_{0j}\transpose\bmu_n\bx_{0m}\transpose\bmu_n)^{1/2}}\widetilde{u}_{0mk}
  \\
  &\quad
  = -\frac{1}{n}\sum_{l = 1}^n\sum_{j = 1}^n\sum_{m = 1}^n
  \frac{\{(A_{lj} - \expect_0(A_{lj})\}\{(A_{jm} - \expect_0(A_{jm})\}\bx_{0i}\transpose\bx_{0j}}{(\bx_{0j}\transpose\bmu_n)(\bx_{0i}\transpose\bmu_n\bx_{0j}\transpose\bmu_n)^{1/2}(\bx_{0j}\transpose\bmu_n\bx_{0m}\transpose\bmu_n)^{1/2}}\widetilde{u}_{0mk}
  = -\frac{1}{n}\sum_{l = 1}^nz_{ikl},
  \end{align*}
  where
  \begin{align*}
  z_{ikl}
  & = \sum_{j = 1}^n\sum_{m = 1}^n
  \frac{\{(A_{lj} - \expect_0(A_{lj})\}\{(A_{jm} - \expect_0(A_{jm})\}\bx_{0i}\transpose\bx_{0j}}{(\bx_{0j}\transpose\bmu_n)(\bx_{0i}\transpose\bmu_n\bx_{0j}\transpose\bmu_n)^{1/2}(\bx_{0j}\transpose\bmu_n\bx_{0m}\transpose\bmu_n)^{1/2}}\widetilde{u}_{0mk}\\
  & = n\sum_{i_1 = 1}^n\sum_{i_2 = 1}^nH_{li_1}H_{i_1i_2}\frac{\widetilde{u}_{0i_2k}}{(\bx_{0i_1}\transpose\bmu_n)(\bx_{0i}\transpose\bmu_n\bx_{0i_1}\transpose\bmu_n)^{1/2}(\bx_{0i_1}\transpose\bmu_n\bx_{0i_2}\transpose\bmu_n)^{1/2}}.
  \end{align*}
  An argument similar to the proof of Lemma \ref{lemma:moment_bound_ASE} yields
  \[
  \prob_0(|z_{ikl}| > Cn(\log n)^4\|\widetilde\bu_{0k}\|_\infty)\leq \exp\{-c(\log n)^2\}
  \]
  for some constants $C, c> 0$ that does not depend on $i,k,l$. 
  Since $\log n\ll(\log n)^2$, we apply the union bound to obtain
  \[
  \prob_0\left(\left|\frac{1}{n}\sum_{l = 1}^nz_{ikl}\right| > Cn(\log n)^2\|\widetilde\bu_{0k}\|_\infty\right)
  \leq \sum_{l = 1}^n\prob_0(|z_{ikl}| > Cn(\log n)^4\|\widetilde\bu_{0k}\|_\infty)\leq \exp\{-c(\log n)^2\},
  \]
  for some constant $c > 0$. 
  \item[(5)] We write the fifth term as
  \begin{align*}
  &\frac{n^2}{2}\be_i\transpose\bT^{-3/2}(\bT - \bD)\bP_0\bT^{-1}(\bA - \bP_0)\bT^{-1/2}\widetilde{\bu}_{0k}\\
  &\quad = \frac{1}{2}\left\{\sum_{j = 1}^n\frac{\expect_0(A_{ij}) - A_{ij}}{\bx_{0i}\transpose\bmu_n}\right\}\be_i\transpose\bY_0\bY_0\transpose\left[\frac{A_{lm} - \expect_0(A_{lm})}{(\bx_{0l}\transpose\bmu_n\bx_{0m}\transpose\bmu_n)^{1/2}}\right]_{n\times n}\widetilde{\bu}_{0k}\\
  &\quad = \frac{1}{2}\left\{\sum_{j = 1}^n\frac{\expect_0(A_{ij}) - A_{ij}}{\sqrt{n}\bx_{0i}\transpose\bmu_n}\right\}
  \frac{1}{\sqrt{n}}\sum_{l = 1}^n\sum_{m = 1}^n\frac{\bx_{0i}\transpose\bx_{0l}\{A_{lm} - \expect_0(A_{lm})\}\widetilde{u}_{0mk}}{(\bx_{0i}\transpose\bmu_n\bx_{0l}\transpose\bmu_n)^{1/2}(\bx_{0l}\transpose\bmu_n\bx_{0m}\transpose\bmu_n)^{1/2}}.
  \end{align*}
  By Hoeffding's inequality, the first factor can be bounded:
  \begin{align*}
  \prob_0\left(\left|\sum_{j = 1}^n\frac{\expect_0(A_{ij}) - A_{ij}}{\sqrt{n}\bx_{0i}\transpose\bmu_n}\right| > C\log n\right)\leq 2\exp\{-c(\log n)^2\}.
  \end{align*}
  We further write the second factor as
  \begin{align*}
  &\frac{1}{\sqrt{n}}\sum_{l = 1}^n\sum_{m = 1}^n\frac{\bx_{0i}\transpose\bx_{0l}\{A_{lm} - \expect_0(A_{lm})\}\widetilde{u}_{0mk}}{(\bx_{0i}\transpose\bmu_n\bx_{0l}\transpose\bmu_n)^{1/2}(\bx_{0l}\transpose\bmu_n\bx_{0m}\transpose\bmu_n)^{1/2}}\\
  &\quad = 
  \sum_{l < m}\frac{\zeta_{iklm}}{\sqrt{n}}\{A_{lm} - \expect_0(A_{lm})\}
  - \sum_{l = 1}^m\frac{\bx_{0l}\transpose\bx_{0l}\bx_{0i}\transpose\bx_{0l}\widetilde{u}_{0lk}}{\sqrt{n}(\bx_{0i}\transpose\bmu_n\bx_{0l}\transpose\bmu_n)^{1/2}\bx_{0l}\transpose\bmu_n}\\
  &\quad = 
  \sum_{l < m}\frac{\zeta_{iklm}}{\sqrt{n}}\{A_{lm} - \expect_0(A_{lm})\} + O(\sqrt{n}\|\widetilde\bu_{0k}\|_\infty),
  \end{align*}
  where
  \[
  \zeta_{iklm} = \left\{
  \frac{\bx_{0i}\transpose\bx_{0l}\widetilde{u}_{0mk}}{(\bx_{0i}\transpose\bmu_n\bx_{0l}\transpose\bmu_n)^{1/2}(\bx_{0l}\transpose\bmu_n\bx_{0m}\transpose\bmu_n)^{1/2}}
  + 
  \frac{\bx_{0i}\transpose\bx_{0m}\widetilde{u}_{0lk}}{(\bx_{0i}\transpose\bmu_n\bx_{0m}\transpose\bmu_n)^{1/2}(\bx_{0m}\transpose\bmu_n\bx_{0l}\transpose\bmu_n)^{1/2}}
  \right\}.
  \]
  Clearly, $\max_{i,k,l,m}|\zeta_{iklm}|\lesssim \|\widetilde{\bu}_{0k}\|_\infty$. It follows from Hoeffding's inequality that
  \begin{align*}
  \prob_0\left\{\left|
  \sum_{l < m}\frac{\zeta_{iklm}}{\sqrt{n}}\{A_{lm} - \expect_0(A_{lm})\}
  \right| > C\sqrt{n}\|\widetilde{\bu}_{0k}\|_\infty(\log n)\right\}\leq \exp\{-c(\log n)^2\}
  \end{align*}
  for some constants $C, c > 0$. We thus conclude that
  \begin{align*}
  &\prob_0\left\{\left|
  \frac{n^2}{2}\be_i\transpose\bT^{-3/2}(\bT - \bD)\bP_0\bT^{-1}(\bA - \bP_0)\bT^{-1/2}\widetilde{\bu}_{0k}
  \right|
  > C\sqrt{n}\|\widetilde{\bu}_{0k}\|_\infty(\log n)^2
  \right\}
  \leq 2\exp\{-c(\log n)^2\}
  \end{align*}
  for some constants $C, c > 0$. 
\end{itemize}
Now combining the analyses (1)-(5) together with the bound on $|\be_i\transpose\bR_{\widetilde{\bE}_1}^{(2)}\widetilde{\bu}_{0k}|$ yields
\begin{align*}
\prob_0\left\{|\be_i\transpose\widetilde{\bE}_1^2\widetilde\bu_{0k}| > Cn(\log n)^4\|\widetilde\bu_{0k}\|_\infty\right\}\leq \exp\{-c(\log n)^2\}
\end{align*}
for some constants $C, c > 0$. Hence, we can take the union bound over all $i\in[n]$ and $k\in [d]$ to conclude that
\begin{align*}
\prob_0\left[\bigcup_{i\in[n]}\bigcup_{k\in [d]}\left\{|\be_i\transpose\widetilde{\bE}_1^2\widetilde\bu_{0k}| > C_{\widetilde{\bE}_1}^2n(\log n)^4\|\widetilde\bu_{0k}\|_\infty\right\}\right]\leq \exp\{-\nu_2(\log n)^2\}
\end{align*}
The proof is completed by taking the union bound over $t = 1$ and $t = 2$. 
\end{proof}

\begin{lemma}\label{lemma:Laplacian_normality}
Assume the conditions of Theorem \ref{thm:LSE_limit_theorem} hold and assume that $\rho_n^{-1}\lesssim (\log n)^\eps$ for some $\eps > 0$. Then for any fixed $i\in[n]$, 
\[
  \left[\frac{1}{\sqrt{n\rho_n}}\be_i\transpose\widetilde{\bE}(\sqrt{n}\bY_0)\right]\transpose = \frac{1}{\sqrt{n\rho_n}}\sum_{j = 1}^n\widetilde{E}_{ij}\sqrt{n}\by_{0j}
  \overset{\calL}{\to}\mathrm{N}(\zero, \widetilde{\bDelta}\widetilde{\bSigma}(\bx_{0i})\widetilde{\bDelta})
  \]
\end{lemma}
\begin{proof}[\bf Proof]
We begin the proof by first remarking that Lemma B.4 of \cite{tang2018} also holds when $\bA\sim\mathrm{RDPG}(\bX_0)$ with a deterministic latent position matrix $\bX_0$ and recalling the decomposition of the random vector of interest in Appendix B of \cite{tang2018}:
\begin{equation}
\label{eqn:Laplacian_decomposition}
\begin{aligned}
\frac{1}{\sqrt{n\rho_n}}\be_i\transpose\widetilde{\bE}(\sqrt{n}\bY_0)
& = {\sqrt{n\rho_n}}\be_i\transpose\bT^{-1/2}(\bA - \bP_0)\bT^{-1/2}(\sqrt{n}\bY_0)
+ \frac{\sqrt{n\rho_n}}{2}\be_i\transpose\bT^{-3/2}(\bT - \bD)\bP_0\bT^{-1/2}(\sqrt{n}\bY_0)\\
&\quad + \frac{\sqrt{n\rho_n}}{2}\be_i\transpose\bT^{-1/2}\bP_0\bT^{-3/2}(\bT - \bD)(\sqrt{n}\bY_0)
\\&\quad
 + \frac{\sqrt{n\rho_n}}{2}\be_i\transpose\bT^{-1/2}(\bA - \bP_0)(\bD^{-1/2} - \bT^{-1/2})(\sqrt{n}\bY_0)\\
&\quad + \frac{\sqrt{n\rho_n}}{2}\be_i\transpose(\bD^{-1/2} - \bT^{-1/2})(\bA - \bP_0)\bD^{-1/2}(\sqrt{n}\bY_0)\\
&\quad + \frac{\sqrt{n\rho_n}}{2}\be_i\transpose\bT^{-3/2}(\bT - \bD)\bP_0(\bD^{-1/2} - \bT^{-1/2})(\sqrt{n}\bY_0)
\\&\quad
 + \frac{\sqrt{n\rho_n}}{2}\be_i\transpose\widetilde\bR^{(-1/2)}(\sqrt{n}\bY_0),
\end{aligned}
\end{equation}
where $\|\widetilde\bR^{(-1/2)}\|_\infty = O_{\prob_0}((n\rho_n)^{-1}\log n)$ by exploiting the proof of (B.11) and (B.13) in Appendix B.2 of \cite{tang2018}, together with the fact that $\|\bP_0\|_\infty = O(n\rho_n)$. We prove the desired asymptotic normality result by showing that the decomposition on the right-hand side of \eqref{eqn:Laplacian_decomposition} is dominated by the first two terms, whereas the remainders are asymptotically negligible.

\vspace*{2ex}\noindent
We next show that the remainders are $o_{\prob_0}(1)$ by establishing the following results.
\begin{itemize}
  \item[(1)] ${\sqrt{n\rho_n}}\|\be_i\transpose\bT^{-1/2}\bP_0\bT^{-3/2}(\bT - \bD)(\sqrt{n}\bY_0)\|_2 = o_{\prob_0}(1)$. Observing that $\|\widetilde\bS_\bP\|_2 = O(1)$ and $\|\widetilde\bU_\bP\|_{2\to\infty} = \|\bY_0\bS_\bP^{-1/2}\|_{2\to\infty}\lesssim n^{-1/2}$, we then write
  \begin{align*}
  {\sqrt{n\rho_n}}\|\be_i\transpose\bT^{-1/2}\bP_0\bT^{-3/2}(\bT - \bD)(\sqrt{n}\bY_0)\|_2 
  &\leq n\sqrt{\rho_n}\|\widetilde{\bU}_\bP\widetilde{\bS}_\bP\widetilde{\bU}_\bP\transpose(\bT - \bD)\bT^{-1}\widetilde{\bU}_\bP\widetilde{\bS}_\bP^{1/2}\|_{2\to\infty}\\
  &\lesssim n\sqrt{\rho_n}\|\widetilde{\bU}_\bP\|_{2\to\infty}\|\widetilde{\bU}_\bP\transpose(\bT - \bD)\bT^{-1}\widetilde{\bU}_\bP\|_2\\
  &\lesssim \|\sqrt{n\rho_n}\widetilde{\bU}_\bP\transpose(\bT - \bD)\bT^{-1}\widetilde{\bU}_\bP\|_2.
  \end{align*}
  Let $\widetilde{u}_{0jk}$ be the $(j, k)$ entry of $\widetilde\bU_{\bP}$. Clearly, $\max_{j,k}|\widetilde{u}_{0jk}|\leq\|\widetilde\bU_{\bP}\|_{2\to\infty}\lesssim n^{-1/2}$. We consider the $(k, l)$ entry of $\sqrt{n\rho_n}\widetilde{\bU}_\bP\transpose(\bT - \bD)\bT^{-1}\widetilde{\bU}_\bP$, which can be written as
  \begin{align*}
  \sum_{i = 1}^n\sum_{j = 1}^n\frac{\rho_{n}\bx_{0i}\transpose\bx_{0j} - A_{ij}}{\sqrt{n\rho_n(\bx_{0j}\transpose\bmu_n)^2}}\widetilde{u}_{0ik}\widetilde{u}_{0il}
  & = -2\sum_{i < j}\frac{A_{ij} - \rho_{n}\bx_{0i}\transpose\bx_{0j}}{\sqrt{n\rho_n}(\bx_{0j}\transpose\bmu_n)}\widetilde{u}_{0ik}\widetilde{u}_{0il} + \sum_{i = 1}^n\frac{\rho_n\bx_{0i}\transpose\bx_{0j}}{\sqrt{n\rho_n}(\bx_{0j}\transpose\bmu_n)}\widetilde{u}_{0ik}\widetilde{u}_{0il}\\
  & = -2\sum_{i < j}\frac{A_{ij} - \rho_{n}\bx_{0i}\transpose\bx_{0j}}{\sqrt{n\rho_n}(\bx_{0j}\transpose\bmu_n)}\widetilde{u}_{0ik}\widetilde{u}_{0il} + o\left(1\right).
  \end{align*}
  By Hoeffding's inequality, for any $t > 0$,
  \begin{align*}
  \prob_0\left(
  \left|\sum_{i < j}\frac{A_{ij} - \rho_{n}\bx_{0i}\transpose\bx_{0j}}{\sqrt{n\rho_n}(\bx_{0j}\transpose\bmu_n)}\widetilde{u}_{0ik}\widetilde{u}_{0il}\right| > t
  \right)
  &\leq 2\exp\left\{-t^2\left(\sum_{i < j}\frac{\widetilde{u}_{0ik}^2\widetilde{u}_{0il}^2}{n\rho_n}\right)^{-1}\right\} = 2\exp(-n\rho_nt^2)\to 0.
  \end{align*}
  Hence we conclude that
  \[
  {\sqrt{n\rho_n}}\|\be_i\transpose\bT^{-1/2}\bP_0\bT^{-3/2}(\bT - \bD)(\sqrt{n}\bY_0)\|_2 
  \lesssim
  \sum_{k,l\in[d]}\left|\sum_{i < j}\frac{A_{ij} - \rho_{n}\bx_{0i}\transpose\bx_{0j}}{\sqrt{n\rho_n}(\bx_{0j}\transpose\bmu_n)}\widetilde{u}_{0ik}\widetilde{u}_{0il}\right| + o\left(1\right) = o_{\prob_0}(1).
  \]

  \item[(2)] $\sqrt{n\rho_n}\|\be_i\transpose\bT^{-1/2}(\bA - \bP_0)(\bD^{-1/2} - \bT^{-1/2})(\sqrt{n}\bY_0)\|_2 = o_{\prob_0}(1)$.
  By Lemma B.4 of \cite{tang2018},
\[
\bD^{-1/2} - \bT^{-1/2} = \frac{1}{2}\bT^{-3/2}(\bT - \bD) + O_{\prob_0}((n\rho_n)^{-3/2}\log n). 
\]
Therefore,
\begin{align*}
&\sqrt{n\rho_n}\be_i\transpose\bT^{-1/2}(\bA - \bP_0)(\bD^{-1/2} - \bT^{-1/2})\sqrt{n}\bY_0\\
&\quad = \frac{1}{2}\sqrt{n\rho_n}\be_i\transpose\bT^{-1/2}(\bA - \bP_0)(\bT - \bD)\bT^{-3/2}\sqrt{n}\bY_0 + \sqrt{n}O_{\prob_0}\{\|\bA - \bP_0\|_2(n\rho_n)^{-3/2}\log n\}\\
&\quad = \frac{1}{2}\sqrt{n\rho_n}\be_i\transpose\bT^{-1/2}(\bA - \bP_0)(\bT - \bD)\bT^{-3/2}\sqrt{n}\bY_0 + O_{\prob_0}(\sqrt{n}(n\rho_n)^{-1}\log n)\\
&\quad = \frac{1}{2}\sqrt{n\rho_n}\be_i\transpose\bT^{-1/2}(\bA - \bP_0)(\bT - \bD)\bT^{-3/2}\sqrt{n}\bY_0 + o_{\prob_0}(1).
\end{align*}
Let $[\bY_{0}]_{*k}$ be the $k$th column of $\bY_0$, and $\widetilde\bbeta_k = \sqrt{n}\bT^{-3/2}[\bY_0]_{*k}$. Then
we consider
\begin{align*}
Z_{ik}
& = -\frac{1}{2}\sqrt{n\rho_n}\be_i\transpose\bT^{-1/2}(\bA - \bP_0)(\bT - \bD)\bT^{-3/2}\sqrt{n}[\bY_0]_{*k}
.
\end{align*}
Observe that
\[
\expect_0(Z_{ik}) = \frac{1}{2}\be_i\transpose\bT^{-1/2}(\bP_0 - \bP_0\circ\bP_0)\bT^{-3/2}\sqrt{n}[\bY_{0}]_{*k},
\]
by the computation of $\expect_0\{(\bA - \bP_0)(\bT - \bD)\}$, where $\circ$ is the Hadamard (entry-wise) matrix product operator, and that $\|\bY_0\|_\infty\lesssim n^{-1/2}$, we obtain
\begin{align*}
|\expect_0(Z_{ik})|
& \leq \frac{1}{2}\|\bT^{-1/2}\|_\infty\|\bP_0\|_\infty\|\bT^{-3/2}\|_\infty\sqrt{n}\|\bY_0\|_\infty
 = O\{(n\rho_n)^{-1/2}(n\rho_n)(n\rho_n)^{-3/2}\} = o(1).
\end{align*}
Hence $Z_{ik} = o_{\prob_0}(1)$ by Markov's inequality, and hence, 
\[
\sqrt{n\rho_n}\be_i\transpose\bT^{-1/2}(\bA - \bP_0)(\bD^{-1/2} - \bT^{-1/2})\sqrt{n}\bY_0 = o_{\prob_0}(1).\]

  \item[(3)] $\sqrt{n\rho_n}\|\be_i\transpose(\bD^{-1/2} - \bT^{-1/2})(\bA - \bP_0)\bD^{-1/2}(\sqrt{n}\bY_0)\|_2 = o_{\prob_0}(1)$. Again, by Lemma B.4 of \cite{tang2018}, we have
  $\bD^{-1/2} - \bT^{-1/2} = O_{\prob_0}((n\rho_n)^{-1}\log n)$. Therefore,
  \begin{align*}
  &\sqrt{n\rho_n}\|\be_i\transpose(\bD^{-1/2} - \bT^{-1/2})(\bA - \bP_0)\bD^{-1/2}(\sqrt{n}\bY_0)\|_2\\
  &\quad\leq \sqrt{n\rho_n}\|\be_i\transpose(\bD^{-1/2} - \bT^{-1/2})(\bA - \bP_0)\bT^{-1/2}(\sqrt{n}\bY_0)\|_2\\
  &\qquad + 
   \sqrt{n\rho_n}\|\be_i\transpose(\bD^{-1/2} - \bT^{-1/2})(\bA - \bP_0)(\bD^{-1/2} - \bT^{-1/2})(\sqrt{n}\bY_0)\|_2\\
  &\quad\leq \sqrt{n\rho_n}\|\bD^{-1/2} - \bT^{-1/2}\|_2
  \left\|
  \sum_{j = 1}^n\frac{A_{ij} - \rho_n\bx_{0i}\transpose\bx_{0j}}{\sqrt{n\rho_n\bx_{0j}\transpose\bmu_n}}\frac{\bx_{0j}}{\sqrt{\bx_{0j}\transpose\bmu_n}}
  \right\|_2\\
  &\qquad + \sqrt{n\rho_n}\|\bD^{-1/2} - \bT^{-1/2}\|_2^2\|\bA - \bP_0\|_2\|\sqrt{n}\bY_0\|_2\\
  &\quad = O_{\prob_0}\left((n\rho_n)^{-1/2}\log n\right)\left\|
  \sum_{j = 1}^n\frac{A_{ij} - \rho_n\bx_{0i}\transpose\bx_{0j}}{\sqrt{n\rho_n\bx_{0j}\transpose\bmu_n}}\frac{\bx_{0j}}{\sqrt{\bx_{0j}\transpose\bmu_n}}
  \right\|_2\\
  &\qquad + O_{\prob_0}\left\{(n\rho_n)^{1/2}(n\rho_n)^{-2}(\log n)^2(n\rho_n)^{1/2}\sqrt{n}\right\}\\
  &\quad = O_{\prob_0}\left((n\rho_n)^{-1/2}\log n\right)\left\|
  \sum_{j = 1}^n\frac{A_{ij} - \rho_n\bx_{0i}\transpose\bx_{0j}}{\sqrt{n\rho_n\bx_{0j}\transpose\bmu_n}}\frac{\bx_{0j}}{\sqrt{\bx_{0j}\transpose\bmu_n}}
  \right\|_2 + o_{\prob_0}(1).
  \end{align*}
  By the Lyapunov's central limit theorem,
  \[
  \sum_{j = 1}^n\frac{A_{ij} - \rho_n\bx_{0i}\transpose\bx_{0j}}{\sqrt{n\rho_n\bx_{0j}\transpose\bmu_n}}\frac{\bx_{0j}}{\sqrt{\bx_{0j}\transpose\bmu_n}} = O_{\prob_0}(1).
  \]
  Hence we conclude that $\sqrt{n\rho_n}\|\be_i\transpose(\bD^{-1/2} - \bT^{-1/2})(\bA - \bP_0)\bD^{-1/2}(\sqrt{n}\bY_0)\|_2 = o_{\prob_0}(1)$.

  \item[(4)] $\sqrt{n\rho_n}\|\be_i\transpose\bT^{-3/2}(\bT - \bD)\bP_0(\bD^{-1/2} - \bT^{-1/2})(\sqrt{n}\bY_0)\|_2 = o_{\prob_0}(1)$. Since $\bD$ and $\bT$ are diagonal matrices, we obtain 
  \[
  \bD^{-1/2} - \bT^{-1/2} = O_{\prob_0}((n\rho_n)^{-1}\log n)
  \quad\Longrightarrow\quad
  \|\bD^{-1/2} - \bT^{-1/2}\|_\infty = O_{\prob_0}((n\rho_n)^{-1}\log n).
  \]
  Similarly, $\|\bT - \bD\|_\infty = O_{\prob_0}((n\rho_n)^{1/2}\log n)$.
  It follows that
  \begin{align*}
  &\sqrt{n\rho_n}\|\be_i\transpose\bT^{-3/2}(\bT - \bD)\bP_0(\bD^{-1/2} - \bT^{-1/2})(\sqrt{n}\bY_0)\|_2\\
  &\quad\leq \sqrt{n\rho_n}\|\be_i\transpose\bT^{-3/2}(\bT - \bD)\bP_0(\bD^{-1/2} - \bT^{-1/2})(\sqrt{n}\bY_0)\|_\infty\\
  &\quad\leq \sqrt{n\rho_n}\|\bT^{-3/2}\|_\infty\|\bT - \bD\|_\infty\|\bP_0\|_\infty\|\bD^{-1/2} - \bT^{-1/2}\|_\infty\sqrt{n}\|\bY_0\|_\infty\\
  &\quad \lesssim (n\rho_n)^{-1}O_{\prob_0}((n\rho_n)^{1/2}\log n)(n\rho_n)O_{\prob_0}((n\rho_n)^{-1}\log n)\\
  &\quad = O_{\prob_0}\{(n\rho_n)^{-1/2}(\log n)^2\} = o_{\prob_0}(1).
  \end{align*}

  \item[(5)] $\sqrt{n\rho_n}\|\be_i\transpose\widetilde\bR^{(-1/2)}(\sqrt{n}\bY_0)\|_2 = o_{\prob_0}(1)$. This is a simple consequence of the result $\|\widetilde\bR^{(-1/2)}\|_\infty = O_{\prob_0}((n\rho_n)^{-1}\log n)$ and $\|\bY_0\|_\infty \lesssim n^{-1/2}$.
\end{itemize}

\vspace*{2ex}\noindent
We next show the asymptotic normality of the first two terms. Denote $\widetilde{\bDelta}_n = \bY_0\transpose\bY_0$. According to the aforementioned analysis, we write
\begin{align*}
\frac{1}{\sqrt{n\rho_n}}\be_i\transpose\widetilde\bE(\sqrt{n}\bY_0)
& = \sqrt{n\rho_n}\be_i\transpose\bT^{-1/2}(\bA - \bP_0)\bT^{-1/2}(\sqrt{n}\bY_0)
\\&\quad
 + \frac{\sqrt{n\rho_n}}{2}\be_i\transpose\bT^{-3/2}(\bT - \bD)\bP\bT^{-1/2}(\sqrt{n}\bY_0)
 + o_{\prob_0}(1)\\
& = \frac{1}{\sqrt{n\rho_n}}\sum_{j = 1}^n\frac{A_{ij} - \rho_n\bx_{0i}\transpose\bx_{0j}}{(\bx_{0i}\transpose\bmu_n)^{1/2}(\bx_{0j}\transpose\bmu_n)}\bx_{0j}\transpose \\
&\quad
 - \frac{1}{2\sqrt{n\rho_n}(\bx_{0i}\transpose\bmu_n)^{3/2}}\sum_{j = 1}^n(A_{ij} - \rho_n\bx_{0i}\transpose\bx_{0j})(\widetilde\bDelta_n\bx_{0i})\transpose + o_{\prob_0}(1).
\end{align*}
Equivalently, we have
\begin{align*}
\left[\frac{1}{\sqrt{n\rho_n}}\be_i\transpose\widetilde\bE(\sqrt{n}\bY_0)\right]\transpose
& = \frac{1}{\sqrt{n\rho_n}}\sum_{j = 1}^n\left\{\frac{\bx_{0j}}{(\bx_{0i}\transpose\bmu_n)^{1/2}(\bx_{0j}\transpose\bmu_n)} - \frac{1}{2}\frac{\widetilde\bDelta_n\bx_{0i}}{(\bx_{0i}\transpose\bmu_n)^{3/2}}\right\}(A_{ij} - \rho_n\bx_{0i}\transpose\bx_{0j})
+ o_{\prob_0}(1).
\end{align*}
The first term on the right-hand side of the previous display is a sum of independent mean-zero random vectors, the variance of which is
\begin{align*}
&\frac{1}{{n\rho_n}}\sum_{j = 1}^n\left\{\frac{\bx_{0j}}{(\bx_{0j}\transpose\bmu_n)} - \frac{1}{2}\frac{\widetilde\bDelta_n\bx_{0i}}{\bx_{0i}\transpose\bmu_n}\right\}\frac{\rho_n\bx_{0i}\transpose\bx_{0j}(1 - \rho_n\bx_{0i}\transpose\bx_{0j})}{(\bx_{0i}\transpose\bmu_n)}
\left\{\frac{\bx_{0j}}{(\bx_{0j}\transpose\bmu_n)} - \frac{1}{2}\frac{\widetilde\bDelta_n\bx_{0i}}{\bx_{0i}\transpose\bmu_n}\right\}\transpose\\
&\quad = \frac{1}{n}\sum_{j = 1}^n\left\{{\bx_{0j}} - \frac{1}{2}\frac{\widetilde\bDelta_n\bx_{0i}\bmu_n\transpose\bx_{0j}}{\bx_{0i}\transpose\bmu_n}\right\}
\frac{\bx_{0i}\transpose\bx_{0j}(1 - \rho_n\bx_{0i}\transpose\bx_{0j})}
    {(\bx_{0i}\transpose\bmu_n)(\bx_{0j}\transpose\bmu_n)^2}
\left\{{\bx_{0j}} - \frac{1}{2}\frac{\widetilde\bDelta_n\bx_{0i}\bmu_n\transpose\bx_{0j}}{\bx_{0i}\transpose\bmu_n}\right\}\transpose\\
&\quad = \widetilde{\bDelta}_n\frac{1}{n}\sum_{j = 1}^n\left\{\widetilde\bDelta_n^{-1} - \frac{1}{2}\frac{\bx_{0i}\bmu_n\transpose}{\bx_{0i}\transpose\bmu_n}\right\}
\frac{\bx_{0i}\transpose\bx_{0j}(1 - \rho_n\bx_{0i}\transpose\bx_{0j})\bx_{0j}\bx_{0j}\transpose}
    {(\bx_{0i}\transpose\bmu_n)(\bx_{0j}\transpose\bmu_n)^2}
\left\{\widetilde\bDelta_n^{-1} - \frac{1}{2}\frac{\bx_{0i}\bmu_n\transpose}{\bx_{0i}\transpose\bmu_n}\right\}\transpose\widetilde\bDelta_n\\
&\quad \to \widetilde\bDelta\widetilde{\bSigma}(\bx_{0i})\widetilde\bDelta.
\end{align*}
The proof is then completed by the Lyapunov's central limit theorem. 
\end{proof}

\subsection{Proof of the Asymptotic Normality \eqref{eqn:LSE_normality} Under the Dense Regime (i)} 
\label{sub:proof_of_the_asymptotic_normality_eqn:lse_normality_under_the_dense_regime_i_}

\begin{proof}[\bf Proof of the asymptotic normality \eqref{eqn:LSE_normality} under the dense regime (i)]
Let $\widetilde{E}_{ij}$ be the $(i, j)$th entry. In particular, we need to verify that the following conditions hold:
\begin{itemize}
  \item[A1.] $\rho_n\to 0$ with $n\rho_n\gtrsim (\log n)^{c_2}$ for some $c_2 > 0$. This condition automatically holds since $\rho\equiv 1$ under the dense regime (i).
  \item[A2.] $\lambda_d\{n\rho_n\calL(\bP_0)\}\gtrsim n\rho_n$ and $\lambda_d\{n\rho_n\calL(\bP_0)\}^{-1}\lambda_1\{n\rho_n\calL(\bP_0)\} = O(1)$. We see that 
  \[
  \lambda_k\{n\calL(\bP_0)\} = n\lambda_k(\bY_0\bY_0\transpose) = n\lambda_k(\bY_0\transpose\bY_0) \asymp n\lambda_k(\widetilde\bDelta) \asymp n
  \]
  because $\bY_0\transpose\bY_0\to \widetilde\bDelta$ as $n\to\infty$ and $\widetilde\bDelta$ is strictly positive definite.
  \item[A3.] There exists constants $C, c > 0$ such that $\|\widetilde{\bE}\|_2\leq C(n\rho_n)^{1/2}$ with probability at least $1 - n^c$ for all $n\geq n_0(C, c)$. This is the result of \cite{oliveira2009concentration} (also see Lemma B.1 of \citealp{tang2018}). 
  \item[A4.] There exists constants $C_{\widetilde\bE}, \nu > 0, \xi > 1$ such that for all $1\leq t\leq 2$, for each standard basis vector $\be_i$, and for each column vector $\widetilde\bu_{0k}$ of $\widetilde\bU_0$, 
  \[
  |\langle \be_i, \widetilde{\bE}^t\widetilde\bu_{0k}\rangle| \leq (C_\bE n\rho_n)^{t/2}(\log n)^{t\xi}\|\widetilde\bu_{0k}\|_\infty
  \]
  with probability at least $1 - \exp\{-\nu(\log n)^\xi\}$, provided that $n\geq n_0(C_{\widetilde\bE}, \nu, \xi)$. We show that this condition holds with $\xi = 2$ by applying Lemma \ref{lemma:Laplacian_moment_bound}. Following the proof of Lemma \ref{lemma:Laplacian_normality} (in particular, the derivation of equation \eqref{eqn:Laplacian_decomposition}), we obtain 
  \[
  \widetilde\bE = \widetilde\bE_1 + \widetilde\bE_2,
  \]
  where
  \begin{align*}
  \widetilde\bE_1& = n\bT^{-1/2}(\bA - \bP_0)\bT^{-1/2} + \frac{n}{2}\bT^{-1/2}\bP_0\bT^{-3/2}(\bT - \bD) + \frac{n}{2}\bT^{-3/2}(\bT - \bD)\bP_0\bT^{-1/2},
  \\
  \widetilde\bE_2
  & = \frac{n}{2}\bT^{-1/2}(\bA - \bP_0)(\bD^{-1/2} - \bT^{-1/2}) + \frac{n}{2}(\bD^{-1/2} - \bT^{-1/2})(\bA - \bP_0)\bT^{-1/2}\\
  &\quad + \frac{n}{2}(\bD^{-1/2} - \bT^{-1/2})(\bA - \bP_0)(\bD^{-1/2} - \bT^{-1/2})
  + \frac{n}{2}\bT^{-3/2}(\bT - \bD)\bP_0(\bD^{-1/2} - \bT^{-1/2})\\
  &\quad + \frac{n}{2}\widetilde\bR^{(-1/2)},
  \end{align*}
  and the remainder $\widetilde\bR^{(-1/2)}$ satisfies $\max\{\|\widetilde\bR^{(-1/2)}\|_2, \|\widetilde\bR^{(-1/2)}\|_\infty\}\leq C(\log n)/n$ with probability at least $1 - 2\exp\{-c(\log n)^2\}$ for some constants $C, c > 0$. 
  Since Chernoff bound together with the matrix Bernstein's inequality imply that with probability at least $1 - \exp\{-c(\log n)^2\}$, 
  \begin{align*}
  &\|\bD - \bT\|_\infty = \|\bD - \bT\|_2\leq C\sqrt{n}\log n\\
  &\|\bD^{-1/2} - \bT^{-1/2}\|_\infty = \|\bD^{-1/2} - \bT^{-1/2}\|_2 \leq Cn^{-1}\log n,\\
  &\|\bA - \bP_0\|_2\leq C\sqrt{n}\log n
  \end{align*}
  for some constants $C, c > 0$, then we immediately obtain that
  \[
  \|\widetilde\bE_1\|_2 \leq 3C\sqrt{n}\log n
  \]
  with probability at least $1 - \exp\{-c(\log n)^2\}$. 
  Furthermore, one exploits the analysis of Appendix B.2 in \cite{tang2018} and conclude that there exists constant $C_{\widetilde{\bE}_2} ,\nu > 0$, such that with probability at least $1 - \exp\{-\nu(\log n)^2\}$, 
  \[
  \|\widetilde\bE_2\|_2\leq C_{\widetilde{\bE}_2}(\log n)^2
  \]
  for sufficiently large $n$. 

  Therefore, by Lemma \ref{lemma:Laplacian_moment_bound} with $t = 2$, with probability at least $1 - \exp\{-\nu(\log n)^2\}$ for some constant $\nu > 0$, we have
  \begin{align*}
  |\be_i\transpose\widetilde\bE^2\widetilde{\bu}_{0k}|
  &\lesssim  |\be_i\transpose\widetilde\bE_1^2\widetilde{\bu}_{0k}| + \|\widetilde\bE_1\|_2\|\widetilde\bE_2\|_2 + \|\widetilde\bE_2\|_2^2\\
  &\lesssim (\log n)^4n\|\widetilde\bu_{0k}\|_\infty + (\log n)^3\sqrt{n} + (\log n)^4\\
  &\lesssim (\log n)^4n\|\widetilde\bu_{0k}\|_\infty
  \end{align*}
  for all $i\in [n]$ and $k\in [d]$, where we have used the fact that $\|\widetilde\bu_{0k}\|_\infty \geq n^{-1/2}$. 

  We now consider the case with $t = 1$. Then by Lemma \ref{lemma:Laplacian_moment_bound} again with $t = 1$, we obtain with probability at least $1 - \exp\{-\nu(\log n)^2\}$, for all $i\in [n]$ and $k\in [d]$, 
  \begin{align*}
  |\be_i\transpose\widetilde\bE\widetilde{\bu}_{0k}|
  &\leq |\be_i\transpose\widetilde\bE_1\widetilde{\bu}_{0k}| + 
  |\be_i\transpose\widetilde\bE_2\widetilde{\bu}_{0k}|
  \leq |\be_i\transpose\widetilde\bE_1\widetilde{\bu}_{0k}| + \|\widetilde{\bE}_2\|_2\\
  &\leq C_{\widetilde{\bE}_1}n^{1/2}(\log n)^2\sqrt{n}\|\widetilde\bu_{0k}\|_\infty + C_{\widetilde{\bE}_2}(\log n)^2\\
  &\leq (C_{\widetilde{\bE}_1} + C_{\widetilde{\bE}_2})(\log n)^2\sqrt{n}\|\widetilde\bu_{0k}\|_\infty.
  \end{align*}  
  Then we apply the union bound over $t = 1$ and $t = 2$ to conclude that there exists constants $\widetilde{C}_{\widetilde{\bE}}, \nu > 0$, such that for sufficiently large $n$, with probability at least $1 - \exp\{-\nu(\log n)^2\}$, 
  \[
  |\be_i\transpose\widetilde\bE^t\widetilde{\bu}_{0k}|\leq (C_{\widetilde{\bE}})^t (\log n)^{2t}n^{t/2}\|\widetilde\bu_{0k}\|_\infty
  \]
  for all $i\in [n], k\in [d]$, and $t = 1,2$.

  \item[A5.] $\bY_0\transpose\bY_0\to \widetilde\bDelta$, and for each fixed $i\in[n]$,
  \[
  \left[\frac{1}{\sqrt{n\rho_n}}\be_i\transpose\widetilde{\bE}(\sqrt{n}\bY_0)\right]\transpose = \frac{1}{\sqrt{n\rho_n}}\sum_{j = 1}^n\widetilde{E}_{ij}\sqrt{n}\by_{0j}
  \]
  converges in distribution to a centered multivariate normal distribution. This is exactly the result of Lemma \ref{lemma:Laplacian_normality}. In particular, the covariance matrix is given by $\widetilde\bDelta\widetilde\bSigma(\bx_{0i})\widetilde\bDelta$.

  \item[A6.] $(\log n)^{2\xi}/(n\rho_n)\to 0$ and 
  \[
  \rho_n^{-1/2}\max\{(\log n)^{2\xi},\|\widetilde\bU_0\transpose\widetilde\bE\widetilde\bU_0\|_2 + 1\}\|\widetilde\bU_0\|_{2\to\infty} = o_{\prob_0}(1).
  \]
  By (B.17) in Lemma B.4 of \cite{tang2018}, we see that $\|\widetilde\bU_0\transpose\widetilde\bE\widetilde\bU_0\|_2 = O_{\prob_0}(n^{-1})$. With $\xi = 2$ and $\rho_n\equiv 1$, the above statement is equivalent to 
  \[
  \max\{(\log n)^{2\xi},\|\widetilde\bU_0\transpose\widetilde\bE\widetilde\bU_0\|_2 + 1\}n^{-1/2} = o_{\prob_0}(1),
  \]
  which automatically holds. 
\end{itemize}
Hence, by Theorem 3 in \cite{cape2019signal}, there exists orthogonal matrices $\bW^*$, $\bW_\bX$ (possibly depending on $n$), such that for any fixed $i\in [n]$, 
\[
n\bW_\bX\transpose\{\bW^*(\widetilde\bU_\bA\transpose\be_i) - (\widetilde\bU_\bP\transpose\be_i)\}\overset{\calL}{\to}\mathrm{N}(\zero, \widetilde\bDelta^{-1/2}\widetilde\bSigma(\bx_{0i})\widetilde\bDelta^{-1/2}).
\]
In particular, the matrix $\bW_\bX$ can be taken such that $\bY_0 = \widetilde\bU_\bP\widetilde{\bS}_{\bP}^{1/2}\bW_\bX$, and $\bW^*$ can be taken as the product of the left and right singular vector matrices of $\widetilde\bU_\bP\transpose\widetilde\bU_\bA$. Then we proceed to compute
\begin{align*}
n\{\bW_\bX\transpose\bW^*(\breve{\bX}\transpose\be_i) - \by_{0i}\}
& = n\{\bW_\bX\transpose\bW^*(\breve{\bX}\transpose\be_i) - (\bY_0\transpose\be_i)\}\\
& = n\{\bW_\bX\transpose\bW^*(\widetilde\bS_\bA^{1/2}\widetilde{\bU}_{\bA}\transpose\be_i) - (\bW_\bX\transpose\widetilde{\bS}_{\bP}^{1/2}\widetilde\bU_\bP\transpose\be_i)\}\\
& = n\bW_\bX\transpose\{\bW^*(\widetilde\bS_\bA^{1/2}\widetilde{\bU}_{\bA}\transpose\be_i) - \widetilde{\bS}_{\bP}^{1/2}\widetilde\bU_\bP\transpose\be_i\}\\
& = n\bW_\bX\transpose\{(\bW^*\widetilde\bS_\bA^{1/2} - \widetilde\bS_\bP^{1/2}\bW^*)(\widetilde{\bU}_{\bA}\transpose\be_i) + \widetilde{\bS}_{\bP}^{1/2}(\bW^*\widetilde\bU_\bA\transpose\be_i - \widetilde\bU_\bP\transpose\be_i)\}\\
& = n\bW_\bX\transpose\{(\bW^*\widetilde\bS_\bA^{1/2} - \widetilde\bS_\bP^{1/2}\bW^*)(\widetilde{\bU}_{\bA}\transpose\be_i)\} + n\bW_\bX\transpose\widetilde\bS_\bP^{1/2}\{\bW^*(\bU_\bA\transpose\be_i) - (\widetilde\bU_\bP\transpose\be_i)\}.
\end{align*}
By Proposition B.2 and Lemma B.3 \cite{tang2018}, we have
\begin{align*}
\|\bW^*\widetilde\bS_\bA^{1/2} - \widetilde\bS_\bP^{1/2}\bW^*\|_2
&\leq \|(\bW^* - \widetilde\bU_\bP\transpose\widetilde\bU_\bA)\widetilde\bS_\bA^{1/2}\|_2
    + \|\widetilde\bU_\bP\transpose\widetilde\bU_\bA\bS_\bA^{1/2} - \widetilde\bS_\bP^{1/2}\widetilde\bU_\bP\transpose\widetilde\bU_\bA\|_2
    + \|\bS_\bP^{1/2}(\widetilde\bU_\bP\transpose\widetilde\bU_\bA - \bW^*)\|_2\\
&\leq \|\bW^* - \widetilde\bU_\bP\transpose\widetilde\bU_\bA\|_2(\|\widetilde\bS_\bA^{1/2}\|_2 + \|\widetilde\bS_\bP^{1/2}\|_2) + O_{\prob_0}(n^{-1})\\
& = O_{\prob_0}(n^{-1}).
\end{align*}
By Theorem 1 of \cite{cape2019signal}, we have
\begin{align*}
\|\widetilde\bU_\bA\|_{2\to\infty}&\leq \|\widetilde\bU_\bP\|_{2\to\infty} + \|\widetilde\bU_\bA - \widetilde\bU_\bP\bW^*\|_{2\to\infty}\\
&\leq \|\widetilde\bU_\bP\|_{2\to\infty} + O_{\prob_0}\left(\frac{1}{\sqrt{n}}\min\{\sqrt{d}(\log n)^2\|\bU_0\|_{2\to\infty}, 1\}\right)\\
&\lesssim \frac{1}{\sqrt{n}} + O_{\prob_0}\left(\frac{(\log n)^2}{n}\right) = O_{\prob_0}\left(\frac{1}{\sqrt{n}}\right).
\end{align*}
Therefore, we obtain
\begin{align*}
\|n\bW_\bX\transpose\{(\bW^*\widetilde\bS_\bA^{1/2} - \widetilde\bS_\bP^{1/2}\bW^*)(\widetilde{\bU}_{\bA}\transpose\be_i)\}\|_2
&\leq n\|\bW^*\widetilde\bS_\bA^{1/2} - \widetilde\bS_\bP^{1/2}\bW^*\|_2\|\widetilde{\bU}_{\bA}\transpose\be_i\|_2\\
&\leq n\|\bW^*\widetilde\bS_\bA^{1/2} - \widetilde\bS_\bP^{1/2}\bW^*\|_2\|\widetilde\bU_\bA\|_{2\to\infty}\\
& = o_{\prob_0}(1).
\end{align*}
This further implies that
\begin{align*}
n\{\bW_\bX\transpose\bW^*(\breve{\bX}\transpose\be_i) - \by_{0i}\}
& = n\bW_\bX\transpose\widetilde\bS_\bP^{1/2}\{\bW^*(\bU_\bA\transpose\be_i) - (\widetilde\bU_\bP\transpose\be_i)\} + o_{\prob_0}(1)\\
& = n\bW_\bX\transpose\widetilde\bS_\bP^{1/2}\bW_\bX[\bW_\bX\transpose\{\bW^*(\bU_\bA\transpose\be_i) - (\widetilde\bU_\bP\transpose\be_i)\}] + o_{\prob_0}(1).
\end{align*}
Using the asymptotic normality 
\[
n\bW_\bX\transpose\{\bW^*(\bU_\bA\transpose\be_i) - (\widetilde\bU_\bP\transpose\be_i)\}\overset{\cal}{\to}\mathrm{N}(\zero, \widetilde\bDelta^{-1/2}\widetilde\bSigma(\bx_{0i})\widetilde\bDelta^{-1/2})
\]
and the fact that $\bW_\bX\transpose\widetilde\bS^{1/2}_{\bP}\bW_\bX\to \widetilde\bDelta^{1/2}$, we apply Slutsky's theorem to conclude that
\[
n\{\bW\transpose(\breve{\bX}\transpose\be_i) - \by_{0i}\}\overset{\calL}{\to}\mathrm{N}(\zero, \widetilde\bSigma(\bx_{0i}))
\]
with the choice of $\bW = (\bW_\bX\transpose\bW^*)\transpose$.
\end{proof}



\subsection{Proof of Theorem \ref{thm:LSE_transform_initial_condition}} 
\label{sec:proof_of_theorem_thm:lse_transform_initial_condition}

\begin{proof}[\bf Proof of Theorem \ref{thm:LSE_transform_initial_condition}]
We adopt the notations used in Section \ref{sec:proof_of_theorem_thm:lse_limit_theorem}.
Following the decomposition \ref{eqn:LSE_expansion}, we immediately see that
\begin{align*}
\widetilde{\bX}\bW - \rho_n^{1/2}\bX_0 
& = \bD^{1/2}\breve\bX\bW - \bT^{1/2}\bY_0\\
& = (\bD^{1/2} - \bT^{1/2})\breve{\bX}\bW + \bT^{1/2}(
\breve{\bX}\bW - \bY_0)\\
& = (\bD^{1/2} - \bT^{1/2})\breve{\bX}\bW + (\bA - \bP_0)\bT^{-1/2}\bY_0(\bY_0\transpose\bY_0)^{-1}
+\frac{1}{2}\bT^{-1/2}(\bT - \bD)\bY_0 + \bT^{1/2}\breve{\bR},
\end{align*}
where the remainder $\breve{\bR}$ satisfies the following concentration property: For any constant $c > 0$, there exists a constant $C > 0$ such that $\|\breve{\bR}\|_{\mathrm{F}} \leq C(n\rho_n)^{-1}$ with probability at least $1 - n^c$. Let $t_i$ denote the $i$th diagonal element of $\bT$ and $d_i$ the $i$th diagonal element of $\bD$. Clearly,
\begin{align*}
\left|\sqrt{d_i} - \sqrt{t_i} - \frac{d_i - t_i}{\sqrt{2}t_i}\right|
& = \frac{(d_i - t_i)^2}{2\sqrt{t_i}(\sqrt{d}_i + \sqrt{t_i})^2}
\leq \frac{(d_i - t_i)^2}{2t_i^{3/2}}.
\end{align*}
Since $\min_{i\in [n]}t_i^{3/2} = O((n\rho_n)^{3/2})$, it follows from Hoeffding's inequality and the union bound that
\begin{align*}
\left\|\bD^{1/2} - \bT^{1/2} - \frac{1}{2}\bT^{-1/2}(\bD - \bT)\right\|_2 = O_{\prob_0}\left\{\frac{(\log n)^2}{(n\rho_n)^{1/2}}\right\}.
\end{align*}
Observe that $\widetilde\bDelta_n = \bY_0\transpose\bY_0$, that $\bY_0$ and $\breve{\bX}$ have spectra bounded away from $0$ and $\infty$, and that $\|\bT\|_2 = O(n\rho_n)$. Therefore,
\begin{align*}
\widetilde{\bX}\bW - \rho_n^{1/2}\bX_0 
& = (\bD^{1/2} - \bT^{1/2})\breve{\bX}\bW + (\bA - \bP_0)\bT^{-1/2}\bY_0(\bY_0\transpose\bY_0)^{-1}
+\frac{1}{2}\bT^{-1/2}(\bT - \bD)\bY_0 + \bT^{1/2}\breve{\bR}\\
& = \frac{1}{2}\bT^{-1/2}(\bD - \bT)\breve{\bX}\bW + O_{\prob_0}\left\{\frac{(\log n)^2}{(n\rho_n)^{1/2}}\right\} + (\bA - \bP_0)\bT^{-1/2}\bY_0\bDelta_n^{-1} \\
&\quad +\frac{1}{2}\bT^{-1/2}(\bT - \bD)\bY_0 + O_{\prob_0}\{(n\rho_n)^{-1/2}\}\\
& = (\bA - \bP_0)\bT^{-1/2}\bY_0\bDelta_n^{-1} + \frac{1}{2}\bT^{-1/2}(\bD - \bT)(\breve{\bX}\bW - \bY_0) + O_{\prob_0}\left\{\frac{(\log n)^2}{(n\rho_n)^{1/2}}\right\},
\end{align*}
where the remainders $O_{\prob_0}\{(n\rho_n)^{-1/2}(\log n)^2\}$ are with respect to $\|\cdot\|_{\mathrm{F}}$. Since we obtain from Appendix B.2 and Appendix B.3 in \cite{tang2018} that
\begin{align*}
\|\bA - \bP_0\|_2 = O_{\prob_0}\{(n\rho_n)^{1/2}\},\quad
\|\bT - \bD\|_2 = O_{\prob_0}\{(n\rho_n)^{1/2}\log n\},\quad
\|\bY_0\|_{\mathrm{F}} = O(1),\quad
\|\widetilde\bDelta_n^{-1}\|_{\mathrm{F}} = O(1),
\end{align*}
we further write
\begin{align*}
\|\bT^{-1/2}(\bD - \bT)(\breve{\bX}\bW - \bY_0)\|_{\mathrm{F}}
&\leq \|\bT^{-1/2}\|_2\|\bD - \bT\|_2\|\breve{\bX}\bW - \bY_0\|_{\mathrm{F}}\\
&\leq \|\bT^{-1/2}\|_2\|\bD - \bT\|_2
\|\bT^{-1/2}\|_2\|\bA -\bP_0\|_2\|\bT^{-1/2}\|_2\|\bY_0\|_{\mathrm{F}}\|\widetilde\bDelta_n^{-1}\|_2\\
&\quad + \frac{1}{2}\|\bT^{-1/2}\|_2\|\bD - \bT\|_2\|\bT^{-1}\|_2\|\bT - \bD\|_2\|\bY_0\|_{\mathrm{F}}\\
&\quad + \|\bT^{-1/2}\|_2\|\bD - \bT\|_2\|\breve\bR\|_{\mathrm{F}}\\
& = O_{\prob_0}\left\{(n\rho_n)^{-1/2}(\log n)^2\right\}.
\end{align*}
Therefore, we can write
\[
\|\widetilde{\bX}\bW - \rho_n^{1/2}\bX_0 - (\bA - \bP_0)\bT^{-1/2}\bY_0(\bY_0\transpose\bY_0)^{-1}\|_{\mathrm{F}} = O_{\prob_0}\left\{(n\rho_n)^{-1/2}(\log n)^2\right\}.
\]
Since for each fixed $i$,
\begin{align*}
\be_i\transpose(\bA - \bP_0)\bT^{-1/2}\bY_0(\bY_0\transpose\bY_0)^{-1}
& = \sum_{j = 1}^n\left(\frac{A_{ij} - \rho_n\bx_{0i}\transpose\bx_{0j}}{\sqrt{n\rho_n\bx_{0j}\transpose\bmu_n}}\right)\by_{0j}\transpose\widetilde{\bDelta}_n^{-1}
 = \sum_{j = 1}^n\left(\frac{A_{ij} - \rho_n\bx_{0i}\transpose\bx_{0j}}{\sqrt{n\rho_n\bx_{0j}\transpose\bmu_n}}\right)\frac{\bx_{0j}\transpose\widetilde{\bDelta}_n^{-1}}{\sqrt{n\bx_{0j}\transpose\bmu_n}}\\
& = \rho_n^{-1/2}\sum_{j = 1}^n\left({A_{ij} - \rho_n\bx_{0i}\transpose\bx_{0j}}\right)\frac{\bx_{0j}\transpose\widetilde{\bDelta}_n^{-1}}{n\bx_{0j}\transpose\bmu_n},
\end{align*}
the proof is completed by observing that
\[
\sup_{i,j\in[n]}\left\|\frac{\widetilde{\bDelta}_n^{-1}\bx_{0j}}{n\bx_{0j}\transpose\bmu_n}\right\|_2\leq \sup_{i,j\in[n]}\frac{\|\widetilde\bDelta_n^{-1}\|_2\|\bx_{0j}\|_2}{n\bx_{0j}\transpose\bmu_n} \lesssim\frac{1}{n}.
\]
\end{proof}


\section{Proofs of Theorems \ref{thm:asymptotic_normality_OS_Laplacian} and \ref{thm:convergence_OS_Laplacian}} 
\label{sec:proof_of_theorems_OSLaplacian}

\subsection{Proof of Theorem \ref{thm:asymptotic_normality_OS_Laplacian}} 
\label{sub:proof_of_theorem_thm:asymptotic_normality_os_laplacian}


\begin{proof}[\bf Proof of Theorem \ref{thm:asymptotic_normality_OS_Laplacian}]
Let $(\bW)_{n = 1}^\infty = (\bW_n)_{n = 1}^\infty\subset\mathbb{O}(d)$ be the sequence of orthogonal matrices satisfying \eqref{eqn:linearization_property}. 
Define a function $\bh:\mathbb{R}^n\times\mathbb{R}^{n\times d}\to\mathbb{R}^d$ by
\[
\bh(\bx, \bZ) = [h_1(\bx,\bZ),\ldots,h_d(\bx,\bZ)]\transpose = \frac{\bx}{\sqrt{(1/n)\sum_{j = 1}^n\bx\transpose\bz_j}},\quad \text{where }\bZ = [\bz_1,\ldots,\bz_n]\transpose\in\mathbb{R}^{n\times d}.
\]
Simple algebra shows that for $k = 1,\ldots,d$
\begin{align*}
\frac{\partial h_k}{\partial \bx\transpose}(\rho_n^{1/2}\bx_{0i},\rho_n^{1/2}\bX_0) & = \rho_n^{-1/2}\left(\bx_{0i}\transpose\bmu_n\right)^{-3/2}\left(\frac{1}{n}\sum_{j = 1}^n\bx_{0i}\transpose\bx_{0j}\be_k\transpose - \frac{1}{2n}\sum_{j = 1}^n\be_k\transpose\bx_{0i}\bx_{0j}\transpose\right),\\
\frac{\partial h_k}{\partial \bz_j\transpose}(\rho_n^{1/2}\bx_{0i},\rho_n^{1/2}\bX_0)& = -\frac{1}{2n\sqrt{\rho_n}}\left(\bx_{0i}\transpose\bmu_n\right)^{-3/2}\be_k\transpose\bx_{0i}\bx_{0i}\transpose,\\
\frac{\partial^2 h_k}{\partial\bx\partial \bx\transpose}(\rho_n^{1/2}\bx_{0i},\rho_n^{1/2}\bX_0)& = \rho_n^{-1}(\bmu_n\transpose\bx_{0i})^{-3/2}\left\{-\frac{1}{2}\be_k\bmu_n\transpose - \frac{1}{2}\bmu_n\be_k\transpose + \frac{3}{4}(\be_k\transpose\bx_{0i})(\bmu_n\transpose\bx_{0i})^{-1}\bmu_n\bmu_n\transpose\right\},\\
\frac{\partial^2 h_k}{\partial\bx\partial \bz_j\transpose}(\rho_n^{1/2}\bx_{0i},\rho_n^{1/2}\bX_0)& = -\frac{1}{2n\rho_n}(\bmu_n\transpose\bx_{0i})^{-3/2}\left\{\be_k\bx_{0i}\transpose + (\be_k\transpose\bx_{0i})\eye - \frac{3}{2n}(\be_k\transpose\bx_{0i})(\bmu_n\transpose\bx_{0i})^{-1}\bx_{0i}\bx_{0i}\transpose\right\},\\
\frac{\partial^2 h_k}{\partial\bz_l\partial \bz_j\transpose}(\rho_n^{1/2}\bx_{0i},\rho_n^{1/2}\bX_0)& = \frac{3}{4n^2\rho_n}(\be_k\transpose\bx_{0i})(\bmu_n\transpose\bx_{0i})^{-5/2}\bx_{0i}\bx_{0i}\transpose.
\end{align*}
Note that
\[
\sup_{j\in[n]}\left\|\frac{\partial^2 h_k}{\partial\bx\partial\bz_j\transpose}(\rho_n^{1/2}\bx_{0i},\rho_n^{1/2}\bX_0)\right\|_{\mathrm{F}} = O\left(\frac{1}{n\rho_n}\right),\quad\sup_{j,l\in[n]}\left\|\frac{\partial^2 h_k}{\partial\bz_l\partial\bz_j\transpose}(\rho_n^{1/2}\bx_{0i},\rho_n^{1/2}\bX_0)\right\|_{\mathrm{F}} = O\left(\frac{1}{n^2\rho_n}\right).
\]
It follows from Taylor's expansion that
\begin{align*}
\bh(\bW\transpose\widehat\bx_i, \widetilde\bX\bW) &= \bh(\rho_n^{1/2}\bx_{0i}, \rho_n^{1/2}\bX_0) + \frac{\partial \bh}{\partial\bx\transpose}(\rho_n^{1/2}\bx_{0i}, \rho_n^{1/2}\bX_0)(\bW\transpose\widehat\bx_i - \rho_n^{1/2}\bx_{0i})\\
&\quad + \sum_{j = 1}^n\frac{\partial \bh}{\partial\bz_j\transpose}(\rho_n^{1/2}\bx_{0i}, \rho_n^{1/2}\bX_0)(\bW\transpose\widetilde\bx_j - \rho_n^{1/2}\bx_{0j})
 + \bR_{\bx_i} + \sum_{j = 1}^n\bR_{\bx_i\bz_j} + \sum_{j = 1}^n\sum_{l = 1}^n\bR_{\bz_j\bz_l},
\end{align*}
where 
\[
\max_{i\in[n]}\|\bR_{\bx_i}\|\lesssim\frac{(\log n)^{1\vee\omega}}{\rho_n^2n} ,\quad\sup_{i,j\in[n]}\|\bR_{\bx_i\bz_j}\|\lesssim \frac{(\log n)^{1\vee\omega}}{n^2\rho_n^2},\quad\sup_{i,j,l\in[n]}\|\bR_{\bz_l\bz_j}\|\lesssim \frac{(\log n)^{1\vee\omega}}{n^3\rho_n^2}
\] 
provided that
\[
\|\widetilde\bX\bW - \rho_n^{1/2}\bX_0\|_{2\to\infty} \leq C_c\frac{(\log n)^{(1\vee\omega)/2}}{\sqrt{n\rho_n}},\quad
\|\widehat\bX\bW - \rho_n^{1/2}\bX_{0}\|_{2\to\infty} \leq C_c\frac{(\log n)^{(1\vee\omega)/2}}{\sqrt{n\rho_n}}
\]
for some constant $C_c > 0$. 
Note that by Theorem \ref{thm:asymptotic_normality_OS}, we have
\[
\|\widehat\bX\bW - \rho_n^{1/2}\bX_0\|_{2\to\infty}\leq \sum_{k = 1}^d\max_{i\in[n]}\left|\frac{1}{n\sqrt{\rho_n}}\sum_{j = 1}^n\frac{(A_{ij} - \rho_n\bx_{0i}\transpose\bx_{0j})[\bG_n(\bx_{0i})^{-1}\bx_{0j}]_k}{\bx_{0i}\transpose\bx_{0j}(1 - \rho_n\bx_{0i}\transpose\bx_{0j})}\right| + O_{\prob_0}\left(\frac{(\log n)^{1\vee\omega}}{n\rho_n^{5/2}}\right).
\]
By Hoeffding's inequality and the union bound, we see that
\[
\max_{i\in[n]}\left|\frac{1}{n\sqrt{\rho_n}}\sum_{j = 1}^n\frac{(A_{ij} - \rho_n\bx_{0i}\transpose\bx_{0j})[\bG_n(\bx_{0i})^{-1}\bx_{0j}]_k}{\bx_{0i}\transpose\bx_{0j}(1 - \rho_n\bx_{0i}\transpose\bx_{0j})}\right| = O_{\prob_0}\left(\sqrt{\frac{\log n}{n\rho_n}}\right).
\]
Thus, we conclude that $\|\widehat\bX\bW - \rho_n^{1/2}\bX_0\|_{2\to\infty} = O_{\prob_0}((n\rho_n)^{-1/2}(\log n)^{(1\vee\omega)/2})$. Invoking this fact and Lemma \ref{lemma:two_to_infinity_error_ASE}, we see that
\begin{align*}
\sqrt{n}(\bW\transpose\widehat\by_i - \by_{0i})& = \bh(\bW\transpose\widehat\bx_i, \widetilde\bX\bW) - \bh(\rho_n^{1/2}\bx_{0i}, \rho_n^{1/2}\bX_0)\\
&= \rho_n^{-1/2}\left(\bmu_n\transpose\bx_{0i}\right)^{-3/2}\left\{\frac{1}{n}\sum_{j = 1}^n\left(\bx_{0i}\transpose\bx_{0j}\eye_d - \frac{1}{2}\bx_{0i}\bx_{0j}\transpose\right)\right\}(\bW\transpose\widehat\bx_i - \rho_n^{1/2}\bx_{0i})\\
&\quad + \bR_{i1}^{(\mathrm{L})} + \bR_{i2}^{(\mathrm{L})}, 
\end{align*}
where
\begin{align*}
\bR_{i1}^{(\mathrm{L})} &= \sum_{j = 1}^n\boldsymbol{\Xi}_{ij}(\bW\transpose\widetilde\bx_j - \rho_n^{1/2}\bx_{0j}),
\quad
\boldsymbol{\Xi}_{ij} = [\bxi_{ij1},\ldots,\bxi_{ijd}]\transpose = -\frac{1}{2n\sqrt{\rho_n}}\left(\bx_{0i}\transpose\bmu_n\right)^{-3/2}\bx_{0i}\bx_{0i}\transpose,
\end{align*}
and
\[
\max_{i\in[n]}\|\bR_{i2}^{(\mathrm{L})}\|\leq \max_{i\in[n]}\|\bR_{\bx_i}\| + n\max_{j\in[n]}\|\bR_{\bx_i\bz_j}\| + n^2\max_{j,l}\|\bR_{\bz_j\bz_l}\| = O_{\prob_0}\left(\frac{(\log n)^{1\vee\omega}}{n\rho_n^2}\right).
\]
By an argument that is similar to the proof of Lemma \ref{lemma:R12k_analysis}, we see that 
\[
\max_{i\in[n]}\|\bR_{i1}^{(\mathrm{L})}\|\lesssim \sum_{k = 1}^d\max_{i\in[n]}\left|\sum_{j = 1}^n\boldsymbol{\xi}_{ijk}\transpose(\bW\transpose\widetilde\bx_j - \rho_n^{1/2}\bx_{0j})\right| = O_{\prob_0}\left(\frac{(\log n)^{(1\vee\omega)/2}}{n\rho_n}\right).
\]
Hence we conclude that
\begin{align*}
\sqrt{n}(\bW\transpose\widehat\by_i - \by_{0i}) &= \rho_n^{-1/2}\left(\bmu_n\transpose\bx_{0i}\right)^{-3/2}\left\{\frac{1}{n}\sum_{j = 1}^n\left(\bx_{0i}\transpose\bx_{0j}\eye_d - \frac{1}{2}\bx_{0i}\bx_{0j}\transpose\right)\right\}(\bW\transpose\widehat\bx_i - \rho_n^{1/2}\bx_{0i}) + \bR_i^{(\mathrm{L})}\\
&=\rho_n^{-1/2}\frac{1}{\sqrt{\bmu_n\transpose\bx_{0i}}}\left(\eye_d - \frac{\bx_{0i}\bmu_n\transpose}{2\bmu_n\transpose\bx_{0i}}\right)(\bW\transpose\widehat\bx_i - \rho_n^{1/2}\bx_{0i}) + \bR_i^{(\mathrm{L})},
\end{align*}
where $\max_{i\in[n]}\|\bR_i^{(\mathrm{L})}\| = O_{\prob_0}((n\rho_n^2)^{-1}(\log n)^{1\vee\omega})$. 
This further implies that 
\[
\sum_{i = 1}^n\|\bR_i^{(\mathrm{L})}\|^2 = O_{\prob_0}\left((n\rho_n^4)^{-1}{(\log n)^{2(1\vee\omega)}}\right).
\]
The proof is thus completed. 
\end{proof}

\subsection{Proof of Theorem \ref{thm:convergence_OS_Laplacian}} 
\label{sub:proof_of_theorem_thm:convergence_os_laplacian}


\begin{proof}[\bf Proof of Theorem \ref{thm:convergence_OS_Laplacian}]
Let $(\bW)_{n = 1}^\infty = (\bW_n)_{n = 1}^\infty\subset\mathbb{O}(d)$ be the sequence of orthogonal matrices satisfying \eqref{eqn:linearization_property}. 
Denote 
\[
\bgamma_{ij} = \frac{1}{n\sqrt{\rho_n}}(\bmu_n\transpose\bx_{0i})^{-1/2}\left(\eye_d - \frac{\bx_{0i}\bmu_n\transpose}{\bmu_n\transpose\bx_{0i}}\right)\frac{\bG_n(\bx_{0i})^{-1}\bx_{0j}}{\bx_{0i}\transpose\bx_{0j}(1 - \rho_n\bx_{0i}\transpose\bx_{0j})}. 
\]
First note that $\|\bG_n(\bx_{0i})^{-1}\|_2\leq\|\bDelta^{-1}\|_2$ for sufficiently large $n$, and hence,
\begin{align}\label{eqn:gammaij_bound}
\sup_{i,j\in[n]}\|\bgamma_{ij}\|\leq\frac{1}{n\sqrt{\rho_n}}\delta^{-1/2}\left(1 + \frac{1}{\delta}\right)\sup_{i,j\in[n]}\frac{\|\bG_n(\bx_{0i})^{-1}\|\|\bx_{0j}\|}{\bx_{0i}\transpose\bx_{0j}(1 - \rho_n\bx_{0i}\transpose\bx_{0j})}\lesssim\frac{1}{n\sqrt{\rho_n}}.
\end{align}
Also observe that
\begin{align*}
\expect_0\left(\sum_{i = 1}^n\left\|\sum_{j = 1}^n(A_{ij} - \rho_n\bx_{0i}\transpose\bx_{0j})\bgamma_{ij}\right\|^2\right)
&= \sum_{i = 1}^n\sum_{a = 1}^n\sum_{b = 1}^n\expect_0\left\{(A_{ia} - \rho_n\bx_{0i}\transpose\bx_{0a})(A_{ib} - \rho_n\bx_{0i}\transpose\bx_{0b})\bgamma_{ia}\transpose\bgamma_{ib}\right\}\\
&
= \frac{1}{n}\sum_{i = 1}^n\mathrm{tr}\left[\frac{1}{(\bmu_n\transpose\bx_{0i})}\left(\eye_d - \frac{\bx_{0i}\bmu_n\transpose}{\bmu_n\transpose\bx_{0i}}\right)\bG_n(\bx_{0i})^{-1}\left(\eye_d - \frac{\bx_{0i}\bmu_n\transpose}{\bmu_n\transpose\bx_{0i}}\right)\transpose\right]\\
&
= \frac{1}{n}\sum_{i = 1}^n\mathrm{tr}\{\widetilde \bG_n(\bx_{0i})\}.
\end{align*}
Denote
\[
\widehat\bR_i^{(\mathrm{L})} = (\bmu_n\transpose\bx_{0i})^{-1/2}\left(\eye_d - \frac{\bx_{0i}\bmu_n\transpose}{2\bx_{0i}\transpose\bmu_n}\right)\widehat\bR_i + \rho_n^{1/2}\bR_i^{(\mathrm{L})}.
\]
Clearly, $\sum_{i = 1}^n\|\widehat\bR_i^{(\mathrm{L})}\|^2 = O_{\prob_0}((n\rho_n^5)^{-1}(\log n)^2)$. 
By Theorem \ref{thm:asymptotic_normality_OS_Laplacian} and Lemma \ref{lemma:convergence_score_function}, we can write
\begin{align*}
&n\rho_n\left\|\widehat\bY\bW - \bY_0\right\|_{\mathrm{F}}^2\\
&\quad = \sum_{i = 1}^n\left\|\sum_{j = 1}^n(A_{ij} - \rho_n\bx_{0i}\transpose\bx_{0j})\bgamma_{ij}\right\|^2 + 2\sum_{i = 1}^n(\widehat\bR_i^{(\mathrm{L})})\transpose\sum_{j = 1}^n(A_{ij} - \rho_n\bx_{0i}\transpose\bx_{0j})\bgamma_{ij} + \sum_{i = 1}^n\|\widehat\bR_i^{(\mathrm{L})}\|^2\\
&\quad = \frac{1}{n}\sum_{i = 1}^n\mathrm{tr}\{\widetilde\bG_n(\bx_{0i})\} + 2\sum_{i = 1}^n(\widehat\bR_i^{(\mathrm{L})})\transpose\sum_{j = 1}^n(A_{ij} - \rho_n\bx_{0i}\transpose\bx_{0j})\bgamma_{ij} + o_{\prob_0}(1) + O_{\prob_0}\left(\frac{(\log n)^{2(1\vee\omega)}}{n\rho_n^5}\right)\\
&\quad = \frac{1}{n}\sum_{i = 1}^n\mathrm{tr}\{\widetilde\bG_n(\bx_{0i})\} + 2\sum_{i = 1}^n(\widehat\bR_i^{(\mathrm{L})})\transpose\sum_{j = 1}^n(A_{ij} - \rho_n\bx_{0i}\transpose\bx_{0j})\bgamma_{ij} + o_{\prob_0}(1).
\end{align*}
By Cauchy-Schwarz inequality and Lemma \ref{lemma:convergence_score_function},
\begin{align*}
\left|\sum_{i = 1}^n(\widehat\bR_i^{(\mathrm{L})})\transpose\sum_{j = 1}^n(A_{ij} - \rho_n\bx_{0i}\transpose\bx_{0j})\bgamma_{ij}\right|
&\leq \sum_{i = 1}^n\|\widehat\bR_i^{(\mathrm{L})}\|\left\|\sum_{j = 1}^n(A_{ij} - \rho_n\bx_{0i}\transpose\bx_{0j})\gamma_{ijk}\right\|\\
&\leq \left(\sum_{i = 1}^n\|\widehat\bR_i^{(\mathrm{L})}\|^2\right)^{1/2}\left\{\sum_{i = 1}^n\left\|\sum_{j = 1}^n(A_{ij} - \rho_n\bx_{0i}\transpose\bx_{0j})\bgamma_{ij}\right\|^2\right\}^{1/2}\\
& = o_{\prob_0}(1).
\end{align*}
Furthermore, by condition \eqref{eqn:strong_convergence_measure} and Lemma \ref{lemma:uniform_convergence_G}, we see that
\[
\frac{1}{n}\sum_{i = 1}^n\mathrm{tr}\{\widetilde\bG_n(\bx_{0i})\} = \int \mathrm{tr}\{\widetilde\bG_n(\bx)\}F_n(\mathrm{d}\bx)\to \int \mathrm{tr}\{\widetilde\bG(\bx)\}F(\mathrm{d}\bx).
\]
This completes the proof of the first part of the theorem. For the second part, we see that
\begin{align*}
\frac{1}{\bmu_n\transpose\bx_{0i}}\left(\eye_d - \frac{\bx_{0i}\bmu_n\transpose}{2\bmu_n\transpose\bx_{0i}}\right)\bG(\bx_{0i})^{-1}\left(\eye_d - \frac{\bx_{0i}\bmu_n\transpose}{2\bmu_n\transpose\bx_{0i}}\right) = \widetilde\bG_n(\bx_{0i})\to \widetilde\bG(\bx_{0i}).
\end{align*}
The result directly follows from the asymptotic normality of $\sqrt{n}(\bW\transpose\widehat\bx_i - \rho_n^{1/2}\bx_{0i})$. The proof is thus completed. 
\end{proof}


\section{Positive Definite Stochastic Block Models} 
\label{sec:positive_definite_stochastic_block_models}

In this section, we show that the asymptotic covariance matrix of the one-step estimator \eqref{eqn:one_step_estimator} and the ASE, under the conditions of Theorem \ref{thm:asymptotic_normality_OS}, are identical, when the underlying random dot product graph coincides with a stochastic block model with a positive definite block probability matrix.  This is established in Theorem \ref{thm:SBM_normality_OS} below.
\begin{theorem}\label{thm:SBM_normality_OS}
Let $\bA\sim\mathrm{RDPG}(\bX_0)$ with a sparsity factor $\rho_n$ for some $\bX_0 = [\bx_{01},\ldots,\bx_{0n}]\transpose\in\calX^n$. Assume that the conditions of Theorem \ref{thm:asymptotic_normality_OS} hold, and denote $\rho = \lim_{n\to\infty}\rho_n$. Further assume that there exist $d$ linearly independent vectors $\bnu_1,\ldots,\bnu_d\in\calX$ and a probability vector $\bpi = [\pi_1,\ldots,\pi_d]$, such that for each $i\in[n]$, $\bx_{0i} = \bnu_k$ for some $k\in [d]$ with
\[
\frac{1}{n}\sum_{i = 1}^n\mathbbm{1}(\bx_{0i} = \bnu_k) = \pi_k,\quad k \in [d].
\]
Let $\tau$ be the cluster assignment function $\tau:[n]\to [d]$ such that $\tau(i) = k$ if and only if $\bx_{0i} = \bnu_k$. 
Denote $\widehat\bX = [\widehat{\bx}_1,\ldots,\widehat{\bx}_n]\transpose$  the one-step estimator \eqref{eqn:one_step_estimator} based on an initial estimator $\widetilde{\bX}$ that satisfies the approximate linearization property. Then there exists a sequence of orthogonal matrices $(\bW)_{n = 1}^\infty = (\bW_n)_{n = 1}^\infty\subset\mathbb{O}(d)$, such that for each fixed $i\in[n]$, 
\[
\sqrt{n}(\bW\transpose\widehat{\bx}_i - \bnu_k) \overset{\calL}{\to}\mathrm{N}(\zero, \bSigma(\bnu_k))
\]
if $\tau(i) = k$, where $\bSigma(\bnu_k)$ is the same as the asymptotic covariance matrix of the $i$th row of the ASE, and can be computed using the formula in Theorem \ref{thm:ASE_limit_theorem}:
\[
\bSigma(\bnu_k) = \bDelta^{-1}\left[\sum_{l = 1}^d\pi_l\{\bnu_k\transpose\bnu_l(1 - \rho\bnu_k\transpose\bnu_l)\}\bnu_l\bnu_l\transpose\right]\bDelta^{-1},\quad\text{and}\quad \bDelta = \sum_{l = 1}^d\pi_l\bnu_l\bnu_l\transpose.
\]
In other words, the asymptotic covariance matrix of the one-step estimator is the same as that of the ASE. 
\end{theorem}
\begin{proof}[\bf Proof]
According to Theorem \ref{thm:convergence_OS} in the manuscript, we have
\[
\sqrt{n}(\bW\transpose\widehat{\bx}_i - \bnu_k) \overset{\calL}{\to}\mathrm{N}(\zero, \bG(\bnu_k)^{-1})
\]
if $\bx_{0i}$ matches with $\bnu_k$, where
\[
\bG(\bnu_k) = \sum_{l = 1}^d\frac{\pi_l\bnu_l\bnu_l\transpose}{\bnu_k\transpose\bnu_l(1 - \rho\bnu_k\transpose\bnu_l)}.
\]
It is therefore sufficient to show that $\bG(\bnu_k)^{-1} = \bSigma(\bnu_k)$ for all $k\in[K]$. 
Since $\bnu_1,\ldots,\bnu_d$ are linearly independent,
it follows that $\bN_0 = [\bnu_1,\ldots,\bnu_d]\transpose$ is invertible. Thus, we can write
\[
\bDelta = \sum_{l = 1}^d\pi_l\bnu_l\bnu_l\transpose = \bN_0\transpose\mathrm{diag}(\bpi)\bN_0,\quad
\bDelta^{-1} = \bN_0^{-1}\mathrm{diag}(\bpi)^{-1}\bN_0^{-\mathrm{T}},
\]
where $\bN_0^{-\mathrm{T}}$ is the shorthand notation for $(\bN_0^{-1})\transpose = (\bN_0\transpose)^{-1}$.
Therefore,
\begin{align*}
\bG(\bnu_k)^{-1} 
& = \left[\sum_{l = 1}^d\frac{\pi_l\bnu_l\bnu_l\transpose}{\bnu_k\transpose\bnu_l(1 - \rho\bnu_k\transpose\bnu_l)}\right]^{-1} = \left[\bN_0\transpose\mathrm{diag}\left\{\frac{\pi_1}{\bnu_k\transpose\bnu_1(1 - \rho\bnu_k\transpose\bnu_1)},\ldots,\frac{\pi_d}{\bnu_k\transpose\bnu_d(1 - \rho\bnu_k\transpose\bnu_d)}\right\}\bN_0\right]^{-1}\\
& = \bN_0^{-1}\mathrm{diag}\left\{\frac{\bnu_k\transpose\bnu_1(1 - \rho\bnu_k\transpose\bnu_1)}{\pi_1},\ldots,\frac{\bnu_k\transpose\bnu_d(1 - \rho\bnu_k\transpose\bnu_d)}{\pi_d}\right\}\bN_0^{-{\mathrm{T}}}\\
& = \bDelta^{-1}\bN_0\transpose\mathrm{diag}(\bpi)\mathrm{diag}\left\{\frac{\bnu_k\transpose\bnu_1(1 - \rho\bnu_k\transpose\bnu_1)}{\pi_1},\ldots,\frac{\bnu_k\transpose\bnu_d(1 - \rho\bnu_k\transpose\bnu_d)}{\pi_d}\right\}\mathrm{diag}(\bpi)\bN_0\bDelta^{-1}\\
& = \bDelta^{-1}\bN_0\transpose\mathrm{diag}\left\{{\pi_1}{\bnu_k\transpose\bnu_1(1 - \rho\bnu_k\transpose\bnu_1)},\ldots,{\pi_d}{\bnu_k\transpose\bnu_d(1 - \rho\bnu_k\transpose\bnu_d)}\right\}\bN_0\bDelta^{-1}\\
& = \bDelta^{-1}\left[\sum_{l = 1}^d\pi_l\{\bnu_k\transpose\bnu_l(1 - \rho\bnu_k\transpose\bnu_l)\}\bnu_l\bnu_l\transpose\right]\bDelta^{-1} = \bSigma(\bnu_k).
\end{align*}
The proof is thus completed.
\end{proof}


\section{Further Discussion of Sparse Graphs} 
\label{sec:further_discussion_of_sparse_graphs}

This section provides further discussion on the decaying rate of the sparsity factor $\rho_n$. The discussion is motivated by the comparison of different conditions of the average degrees in the random graph model for different estimators. In Theorem \ref{thm:asymptotic_normality_OS} and Theorem \ref{thm:asymptotic_normality_OS_Laplacian}, which establish the asymptotic characterizations of the one-step estimators, we have imposed the assumption that the sparsity factor $\rho_n$ for the random dot product graph is either constantly $1$, or converges to $0$ with the requirement that $n\rho_n$ grows at a polynomial rate of $n$. Here we use the polynomial rate of $n$ to describe the case where $n\rho_n\gtrsim n^{\eta}$ for a strictly positive constant $\eta > 0$. In contrast, the limit theorem of the ASE (\emph{i.e.}, Theorem \ref{thm:ASE_limit_theorem}) only requires that $n\rho_n$ grows at a polynomial rate of $\log n$. 
From the graph theory perspective, 
$n\rho_n$ controls the average expected degree of the random graph model, \emph{i.e.}, the sparsity level of the graph. Therefore, the limit theorem of the ASE (Theorem \ref{thm:ASE_limit_theorem}) allows sparser graphs because the average expected degree grows at a polynomial rate of $\log n$, whereas the limit theorems of the one-step estimators (Theorems \ref{thm:asymptotic_normality_OS} and \ref{thm:asymptotic_normality_OS_Laplacian}) require the graphs to be relatively denser as the average expected degree  grows at a polynomial rate of $n$. The graphs in the latter scenario are only considered as moderately sparse. 

This relatively stronger sparsity condition is a consequence of the proof technique employed here. The condition we propose may be weakened. Nevertheless, the proof strategy is a standard approach for establishing the asymptotic normality of the one-step estimators (see Section 5.7 of \citealp{van2000asymptotic}), and the standard proof strategy inevitably leads to the requirement that $n\rho_n$ must be lower bounded by a polynomial of $n$. 
To illustrate this result, we consider the following simplified problem: Assume that the underlying random dot product graph is of dimension $1$, \emph{i.e.}, the latent position matrix $\bX_0$ is an $n\times 1$ column vector, and suppose we focus on estimating a single latent position $x_{0i}$ with the knowledge of the rest of the latent positions $(x_{0j})_{j\neq i}$. Formally, we take the initial estimator $\widetilde\bX$ for the entire latent position matrix to be of the form
\[
e_j\transpose\widetilde\bX = \left\{
\begin{aligned}
&\rho_n^{1/2}x_{0i},&\quad\text{if }j\neq i,\\
&\be_i\transpose\bX^{\mathrm{(ASE)}},&\quad\text{if }j = i.
\end{aligned}
\right..
\]
Here we choose $\widetilde{x}_i$ for $x_{0i}$ as the $i$th row of the ASE, \emph{i.e.}, $\widetilde{x}_i = \be_i\transpose\bX^{\mathrm{ASE}}$, but it is also possible to consider more general estimators. Denote $\widehat{x}_i^{\mathrm{(ASE)}} = \be_i\transpose\widehat{\bX}^{\mathrm{(ASE)}}$. Then the $i$th row of the one-step estimator \eqref{eqn:one_step_estimator} can be written as
\begin{align*}
\widehat{x}_i 
& = \widehat{x}_i^{\mathrm{(ASE)}} + \left\{\frac{1}{n}\sum_{j\neq i}\frac{\rho_n^{1/2}x_{0j}}{\widehat{x}_i^{\mathrm{(ASE)}}(1 - \rho_n^{1/2}\widehat{x}_i^{\mathrm{(ASE)}}x_{0j})} + \frac{1}{n(1 - (\widehat{x}_i^{\mathrm{(ASE)}}) ^ 2)}\right\}^{-1}\\
&\quad\times \left\{ \frac{1}{n}\sum_{j\neq i}\frac{A_{ij} - \rho_n^{1/2}\widehat{x}_i^{\mathrm{(ASE)}}x_{0j}}{\widehat{x}_i^{\mathrm{(ASE)}}(1 - \rho_n^{1/2}\widehat{x}_i^{\mathrm{(ASE)}} x_{0j} )} - \frac{\widehat{x}_i^{\mathrm{(ASE)}}}{n(1 - (\widehat{x}_i^{\mathrm{(ASE)}}) ^ 2)} \right\}.
\end{align*}
Since $x_i^{\mathrm{(ASE)}} = \rho_n^{1/2}x_{0i} + O_{\prob_0}(n^{-1/2}) = O_{\prob_0}(\rho_n^{1/2})$ by the asymptotic normality and the fact that $n\rho_n\geq1$, it follows that $1 - (x_i^{\mathrm{(ASE)}})^2$ stays bounded away from $0$ and $1$ with probability going to $1$, and hence, 
\[
\frac{1}{n(1 - (\widehat{x}_i^{\mathrm{(ASE)}}) ^ 2)} = O_{\prob_0}(n^{-1}),\quad\text{and}\quad
\frac{\widehat{x}_i^{\mathrm{(ASE)}}}{n(1 - (\widehat{x}_i^{\mathrm{(ASE)}}) ^ 2)} = O_{\prob_0}(n^{-1}).
\]
It follows that
\begin{align*}
\widehat{x}_i 
& = \widehat{x}_i^{\mathrm{(ASE)}} + 
\left[
\left\{\frac{1}{n}\sum_{j\neq i}\frac{\rho_n^{1/2}x_{0j}}{\widehat{x}_i^{\mathrm{(ASE)}}(1 - \rho_n^{1/2}\widehat{x}_i^{\mathrm{(ASE)}}x_{0j})}\right\}^{-1} + O_{\prob_0}(n^{-1})\right]
\\&\quad\times 
\left\{ \frac{1}{n}\sum_{j\neq i}\frac{(A_{ij} - \rho_n^{1/2}\widehat{x}_i^{\mathrm{(ASE)}}x_{0j})}{\widehat{x}_i^{\mathrm{(ASE)}}(1 - \rho_n^{1/2}\widehat{x}_i^{\mathrm{(ASE)}} x_{0j}) } + O_{\prob_0}(n^{-1}) \right\}\\
& = \widehat{x}_i^{\mathrm{(ASE)}} + \left\{\frac{1}{n}\sum_{j\neq i}\frac{\rho_n^{1/2}x_{0j}}{\widehat{x}_i^{\mathrm{(ASE)}}(1 - \rho_n^{1/2}\widehat{x}_i^{\mathrm{(ASE)}}x_{0j})}\right\}^{-1}
\left\{
\frac{1}{n}\sum_{j\neq i}\frac{A_{ij} - \rho_n^{1/2}\widehat{x}_i^{\mathrm{(ASE)}}x_{0j}}{\widehat{x}_i^{\mathrm{(ASE)}}(1 - \rho_n^{1/2}\widehat{x}_i^{\mathrm{(ASE)}} x_{0j} )}\right\} + O_{\prob_0}(n^{-1}).
\end{align*}
Now we use the notation of Section 5.7 of \cite{van2000asymptotic}. Denote
\begin{align*}
\Psi_n(x) & = \frac{1}{n}\sum_{j\neq i}\frac{A_{ij} - \rho_n^{1/2}xx_{0j}}{x(1 - \rho_n^{1/2}xx_{0j})},
\dot{\Psi}_{n,0} 
= -\frac{1}{n}\sum_{j\neq i}\frac{\rho_n^{1/2}x_{0j}}{\widehat{x}_i^{\mathrm{(ASE)}}(1 - \rho_n^{1/2}\widehat{x}_i^{\mathrm{(ASE)}}x_{0j})},
\dot{\Psi}_0
 = -\int_\calX \frac{x_1}{1 - \rho x_{0i}x_1}F(\mathrm{d}x_1).
\end{align*}
Then the above equation can be written as
\[
\widehat{x}_i = \widehat{x}_i^{\mathrm{(ASE)}} - \dot{\Psi}_{n,0}^{-1}\Psi_n(\widehat{x}^{\mathrm{(ASE)}}_i) + O_{\prob_0}(n^{-1}),
\]
and following the method in \cite{van2000asymptotic},
\begin{align*}
\dot{\Psi}_{n,0}\sqrt{n}(\widehat{x}_i - \rho_n^{1/2}x_{0i}) &= \dot{\Psi}_{n,0}\sqrt{n}(\widehat{x}_i^{\mathrm{(ASE)}} - \rho_n^{1/2}x_{0i}) - \sqrt{n}\Psi_n(\widehat{x}_i^{\mathrm{(ASE)}}) + O_{\prob_0}(n^{-1/2})\\
& = \dot{\Psi}_{n,0}\sqrt{n}(\widehat{x}_i^{\mathrm{(ASE)}} - \rho_n^{1/2}x_{0i})
- \sqrt{n}\left\{\Psi_n(\widehat{x}_i^{\mathrm{(ASE)}}) - \Psi_n(\rho_n^{1/2}x_{0i})\right\}\\
&\quad - \sqrt{n}\Psi_n(\rho_n^{1/2}x_{0i}) + O_{\prob_0}(n^{-1/2})\\
& = (\dot{\Psi}_{n,0} - \dot{\Psi}_0)\sqrt{n}(\widehat{x}_i^{\mathrm{(ASE)}} - \rho_n^{1/2}x_{0i})\\
&\quad + \dot{\Psi}_0\sqrt{n}(\widehat{x}_i^{\mathrm{(ASE)}} - \rho_n^{1/2}x_{0i})
- \sqrt{n}\left\{\Psi_n(\widehat{x}_i^{\mathrm{(ASE)}}) - \Psi_n(\rho_n^{1/2}x_{0i})\right\}\\ 
&\quad - \sqrt{n}\Psi_n(\rho_n^{1/2}x_{0i}) + O_{\prob_0}(n^{-1/2})\\
& = (\dot{\Psi}_{n,0} - \dot{\Psi}_0) O_{\prob_0}(1)\\
&\quad + \dot{\Psi}_0\sqrt{n}(\widehat{x}_i^{\mathrm{(ASE)}} - \rho_n^{1/2}x_{0i})
- \sqrt{n}\left\{\Psi_n(\widehat{x}_i^{\mathrm{(ASE)}}) - \Psi_n(\rho_n^{1/2}x_{0i})\right\}\\ 
&\quad - \sqrt{n}\Psi_n(\rho_n^{1/2}x_{0i}) + O_{\prob_0}(n^{-1/2}).
\end{align*}
To establish the asymptotic normality of $\sqrt{n}(\widehat{x}_i - \rho_n^{1/2}x_{0i})$, the two main standard ingredients employed in Section 5.7 of \cite{van2000asymptotic} are:
\begin{itemize}
  \item[(a)] $\dot{\Psi}_{n,0}\overset{\prob_0}{\to} \dot{\Psi}_0$. This condition holds without requiring that $n\rho_n$ grows polynomially in $n$. Instead, we only need $n\rho_n\to\infty$. Similar to the proof of Theorem \ref{thm:OSE_single_vertex} in Section \ref{sec:proof_of_theorem_thm:ose_single_vertex}, if we denote 
  \[
  \Gamma_j(x) = \frac{\rho_n^{1/2}x_{0j}}{x(1 - \rho_n^{1/2}xx_{0j})}, 
  \]
  then the derivative is
  \[
  \Gamma'_j(x) = -\frac{\rho_n^{1/2}x_{0j}(1 - 2\rho_n^{1/2}xx_{0j})}{x^2(1 - \rho_n^{1/2}xx_{0j})^2},
  \]
  and this implies that there exists a sufficiently small $\eps >0$, such that for all $x\in\mathbb{R}$ with $\sqrt{n}|x - \rho_n^{1/2}x_{0i}| < \eps(n\rho_n)^{1/2}$,
  \begin{align*}
  1 - 2\rho_n^{1/2}xx_{0j} &\leq 1 + 2\rho_n^{1/2}|x - \rho_n^{1/2}x_{0i}|x_{0j} + 2\rho_nx_{0i}x_{0j}\leq 1 + 4\rho_n = O(1),\\
  1 - \rho_n^{1/2}xx_{0j} &\geq 1 - \rho_n^{1/2}|x - \rho_n^{1/2}x_{0i}|x_{0j} - \rho_nx_{0i}x_{0j}
  \geq 1 - \rho_n(\eps + x_{0i}x_{0j}) = O(1),\\
  x ^ 2&\geq (x - \rho_n^{1/2}x_{0i} + \rho_n^{1/2}x_{0i})^2\geq (\rho_n^{1/2}x_{0i} - |x - \rho_n^{1/2}x_{0i}|)^2\\
  &\geq (\rho_n^{1/2}x_{0i} - \rho_n^{1/2}\eps)^2 = \rho_n(x_{0i} - \eps)^2\gtrsim \rho_n,
  \end{align*}
  and hence,
  \[
  \max_{j\in[n]}\sup_{x:\sqrt{n}|x - \rho_n^{1/2}x_{0i}| < \eps(n\rho_n)^{1/2}}|\Gamma_j'(x)|\lesssim \rho_n^{-1/2}.
  \]
  Therefore, by the mean-value theorem, over the event $\{\sqrt{n}|\widehat{x}_i^{(\mathrm{ASE})} - \rho_n^{1/2}x_{0i}| < \eps(n\rho_n)^{1/2}\}$, 
  \begin{align*}
  |\dot{\Psi}_{n,0} - \dot{\Psi}_0|
  & \leq\frac{1}{n}\sum_{j \neq i}|\Gamma_j(\widehat{x}^{\mathrm{(ASE)}}_i) - \Gamma_j(\rho_n^{1/2}x_{0i})| + \left|\frac{1}{n}\sum_{j \neq i}\Gamma_j(\rho_n^{1/2}x_{0i}) - \dot{\Psi}_0\right|\\
  & \leq \max_{j\in[n]}\sup_{x:\sqrt{n}|x - \rho_n^{1/2}x_{0i}| < \eps(n\rho_n)^{1/2}}|\Gamma_j'(x)||\widehat{x}_i^{\mathrm{(ASE)}} - \rho_n^{1/2}x_{0i}| + \left|\frac{1}{n}\sum_{j \neq i}\Gamma_j(\rho_n^{1/2}x_{0i}) - \dot{\Psi}_0\right|\\
  &\leq C\rho_n^{-1/2}|\widehat{x}_i^{\mathrm{(ASE)}} - \rho_n^{1/2}x_{0i}| + o(1)
  \end{align*}
  for some constant $C > 0$.
  This is because the second term goes to $0$ as $n\to\infty$ due condition \eqref{eqn:strong_convergence_measure} in the main text. Therefore, for any $t > 0$ and sufficiently large $n$,
  \begin{align*}
    \prob_0\left(|\dot{\Psi}_{n,0} - \dot{\Psi}_0| > t\right)
    & = \prob_0\left(|\dot{\Psi}_{n,0} - \dot{\Psi}_0| > t, 
        \sqrt{n}|\widehat{x}_i^{\mathrm{(ASE)}} - \rho_n^{1/2}x_{0i}| < \eps(n\rho_n)^{1/2}
        \right)\\
    &\quad + 
    \prob_0\left(|\dot{\Psi}_{n,0} - \dot{\Psi}_0| > t, 
        \sqrt{n}|\widehat{x}_i^{\mathrm{(ASE)}} - \rho_n^{1/2}x_{0i}| \geq \eps(n\rho_n)^{1/2}\right)\\
    &\leq  \prob_0\left(C\rho_n^{-1/2}|\widehat{x}_i^{(\mathrm{ASE})} - \rho_n^{1/2}x_{0i}| > t/2, \sqrt{n}|\widehat{x}_i^{\mathrm{(ASE)}} - \rho_n^{1/2}x_{0i}| < \eps(n\rho_n)^{1/2}\right)\\
    &\quad + \prob_0(\sqrt{n}|\widehat{x}_i^{(\mathrm{ASE})} - \rho_n^{1/2}x_{0i}| \geq \eps(n\rho_n)^{1/2})\\
    &\leq \prob_0\left(\sqrt{n}|\widehat{x}_i^{(\mathrm{ASE})} - \rho_n^{1/2}x_{0i}| > \frac{(n\rho_n)^{1/2}t}{2C}\right)
    \\&\quad
     + \prob_0(\sqrt{n}|\widehat{x}_i^{(\mathrm{ASE})} - \rho_n^{1/2}x_{0i}| \geq \eps(n\rho_n)^{1/2})
    \to 0
  \end{align*}
  where we have used the fact $\sqrt{n}(\widehat{x}_i^{\mathrm{(ASE)}} - \rho_n^{1/2}x_{0i}) = O_{\prob_0}(1)$ and $n\rho_n\to\infty$.
  This completes the verification of $\dot{\Psi}_{n,0}\overset{\prob_0}{\to}\dot{\Psi}_0$.
  \item[(b)] $\sqrt{n}\{\Psi_n(\widehat{x}_i^{\mathrm{(ASE)}}) - \Psi_n(\rho_n^{1/2}x_{0i})\} - \dot{\Psi}_{0}\sqrt{n}(\widehat{x}_i^{\mathrm{(ASE)}} - \rho_n^{1/2}x_{0i}) = o_{\prob_0}(1)$. We remark that this condition holds provided that $n\rho_n^{3}\to 0$ as $n\to\infty$. Observe that
  \begin{align*}
  \Psi'_n(x) & = 
  \frac{1}{n}\sum_{j\neq i}\frac{-\rho_n^{1/2}xx_{0j}(1 - \rho_n^{1/2}xx_{0j}) - (A_{ij} - \rho_n^{1/2}xx_{0j})(1 - 2\rho_n^{1/2}xx_{0j})}{x^2(1 - \rho_n^{1/2}xx_{0j})^2}\\
  & = -\frac{1}{n}\sum_{j\neq i}\Gamma_j(x) - \frac{1}{n}\sum_{j\neq i}\frac{(A_{ij} - \rho_n^{1/2}xx_{0j})(1 - 2\rho_n^{1/2}xx_{0j})}{x^2(1 - \rho_n^{1/2}xx_{0j})^2}
  \end{align*}
  and that
  \begin{align*}
  \dot{\Psi}_{0}
  & = \left\{\dot{\Psi}_{0} + \frac{1}{n}\sum_{j\neq i}\Gamma_j(\rho_n^{1/2}x_{0j})\right\}
   - \left\{\frac{1}{n}\sum_{j\neq i}\Gamma_j(\rho_n^{1/2}x_{0j}) +  \Psi_n'(\rho_n^{1/2}x_{0i})\right\} + \Psi'(\rho_n^{1/2}x_{0i})\\
  & = o(1) - \frac{1}{n}\sum_{j\neq i}\frac{(A_{ij} - \rho_nx_{0i}x_{0j})(1 - 2\rho_nx_{0i}x_{0j})}{\rho_nx_{0i}^2(1 - \rho_nx_{0i}x_{0j}) ^ 2} + \Psi_n'(\rho_n^{1/2}x_{0i})\\
  & = o_{\prob_0}(1) + \Psi_n'(\rho_n^{1/2}x_{0i})
  \end{align*}
  by condition \eqref{eqn:strong_convergence_measure} and the fact that
  \begin{align*}
  \var\left\{\frac{1}{n}\sum_{j\neq i}\frac{(A_{ij} - \rho_nx_{0i}x_{0j})(1 - 2\rho_nx_{0i}x_{0j})}{\rho_nx_{0i}^2(1 - \rho_nx_{0i}x_{0j}) ^ 2}\right\}
  &= \frac{1}{n^2}\sum_{j \neq i}\frac{\rho_nx_{0i}x_{0j}(1 - \rho_nx_{0i}x_{0j})(1 - 2\rho_nx_{0i}x_{0j})^2}{\rho_n^2x_{0i}^4(1 - \rho_nx_{0i}x_{0j})^4}\\
  &\lesssim \frac{1}{n\rho_n}\to 0.
  \end{align*}
  Then we apply the fact that $\sqrt{n}(\widehat{x}_i^{(\mathrm{ASE})} - \rho_n^{1/2}x_{0i}) = O_{\prob_0}(1)$ and the mean-value theorem
   to derive
  \begin{align*}
  &\sqrt{n}\{\Psi_n(\widehat{x}_i^{\mathrm{(ASE)}}) - \Psi_n(\rho_n^{1/2}x_{0i})\} - \dot{\Psi}_{0}\sqrt{n}(\widehat{x}_i^{\mathrm{(ASE)}} - \rho_n^{1/2}x_{0i})\\
  &\quad = \sqrt{n}\{\Psi_n(\widehat{x}_i^{\mathrm{(ASE)}}) - \Psi_n(\rho_n^{1/2}x_{0i})\}
  - \Psi'_n(\rho_n^{1/2}x_{0i})\sqrt{n}(\widehat{x}_i^{\mathrm{(ASE)}} - \rho_n^{1/2}x_{0i})\\
  &\quad\quad - o_{\prob_0}(1)\sqrt{n}(\widehat{x}_i^{\mathrm{(ASE)}} - \rho_n^{1/2}x_{0i})\\
  &\quad = \sqrt{n}\{\Psi_n(\widehat{x}_i^{\mathrm{(ASE)}}) - \Psi_n(\rho_n^{1/2}x_{0i})\}
  - \Psi'_n(\rho_n^{1/2}x_{0i})\sqrt{n}(\widehat{x}_i^{\mathrm{(ASE)}} - \rho_n^{1/2}x_{0i}) + o_{\prob_0}(1)\\
  &\quad = \Psi_n''(\widetilde{x}_i) \sqrt{n}(\widehat{x}_i^{\mathrm{(ASE)}} - \rho_n^{1/2}x_{0i}) ^ 2 + o_{\prob_0}(1)\\
  &\quad = \Psi_n''(\widetilde{x}_i) O_{\prob_0}(n^{-1/2}) + o_{\prob_0}(1),
  \end{align*}
  where $\widetilde{x}_i$ lies between $\rho_n^{1/2}x_{0i}$ and $\widehat{x}_i^{(\mathrm{ASE})}$.
  Since the above derivation is equivalent, we see that ingredient (b) holds if and only if $\Psi''(\widetilde{x}_i) = o_{\prob_0}(n^{1/2})$. Thus we compute
  \begin{align*}
  \Psi_n''(\widetilde{x}_i)
  & = -\frac{1}{n}\sum_{j\neq i}\Gamma'_j(\widetilde{x}_i)
   - \frac{\rho_n^{1/2}}{\widetilde{x}_i^2}
   \frac{1}{n}
   \sum_{j\neq i}
   \frac{x_{0j}(1 - 2\rho_n^{1/2}\widetilde{x}_ix_{0j}) + 2(A_{ij} - \rho_n^{1/2}\widetilde{x}_ix_{0j})x_{0j}}{(1 - \rho_n^{1/2}\widetilde{x}_ix_{0j})^2}
   \\
  &\quad - \frac{1}{\widetilde{x}_i^3}\frac{1}{n}\sum_{j\neq i}\frac{2(A_{ij} - \rho_n^{1/2}\widetilde{x}_ix_{0j})(1 - 2\rho_n^{1/2}\widetilde{x}_ix_{0j})^2 (1 - \rho_n^{1/2}\widetilde{x}_ix_{0j})}{(1 - \rho_n^{1/2}\widetilde{x}_ix_{0j})^4}
  \\
  & = -\frac{1}{n}\sum_{j\neq i}\Gamma'_j(\widetilde{x}_i)
   - \frac{\rho_n^{1/2}}{n\widetilde{x}_i^2}
   \sum_{j\neq i}
   \frac{x_{0j}(1 - 2\rho_n^{1/2}\widetilde{x}_ix_{0j}) + 2(A_{ij} - \rho_n^{1/2}\widetilde{x}_{i}x_{0j})x_{0j}}{(1 - \rho_n^{1/2}\widetilde{x}_ix_{0j})^2}
   \\
  &\quad - 
  \frac{1}{n\widetilde{x}_i^3}
  \sum_{j\neq i}\frac{2\rho_n^{1/2}(\rho_n^{1/2}x_{0i} - \widetilde{x}_i)x_{0j}(1 - 2\rho_n^{1/2}\widetilde{x}_ix_{0j})^2 (1 - \rho_n^{1/2}\widetilde{x}_ix_{0j})}{(1 - \rho_n^{1/2}\widetilde{x}_ix_{0j})^4}\\
  &\quad - \frac{1}{n\widetilde{x}_i^3}
  \sum_{j\neq i}\frac{2(A_{ij} - \rho_nx_{0i}x_{0j})(1 - 2\rho_n^{1/2}\widetilde{x}_ix_{0j})^2 (1 - \rho_n^{1/2}\widetilde{x}_ix_{0j})}{(1 - \rho_n^{1/2}\widetilde{x}_ix_{0j})^4}.
  \end{align*}
  Similarly, there exists a sufficiently small $\eps > 0$, such that over the event $\{\sqrt{n}|\widehat{x}_i^{\mathrm{(ASE)}} - \rho_n^{1/2}x_{0i}| < \eps(n\rho_n)^{1/2}\}$, we have
  \begin{align*}
  |\rho_n^{1/2}x_{0i} - \widetilde{x}_i|
  &
  \leq |\rho_n^{1/2}x_{0i} - \widehat{x}_i^{\mathrm{(ASE)}}| < \eps\rho_n^{1/2},
  \\
  |1 - 2\rho_n^{1/2}\widetilde{x}_ix_{0j}| &\leq 1 + 2\rho_n^{1/2}|\widetilde{x}_i - \rho_n^{1/2}x_{0i}|x_{0j} + 2\rho_nx_{0i}x_{0j}\leq 1 + 4\rho_n = O(1),\\
  1 - \rho_n^{1/2}\widetilde{x}_ix_{0j} &\geq 1 - \rho_n^{1/2}|\widetilde{x}_i - \rho_n^{1/2}x_{0i}|x_{0j} - \rho_nx_{0i}x_{0j}
  \geq 1 - \rho_n(\eps + x_{0i}x_{0j}) = O(1),\\
  |\widetilde{x}_i|&\geq -|\widetilde{x}_i - \rho_n^{1/2}x_{0i}| + \rho_n^{1/2}x_{0i}
  \geq \rho_n^{1/2}x_{0i} - \rho_n^{1/2}\eps\gtrsim \rho_n^{1/2},
  \end{align*}
  and
  \[
  \max_{j\in[n]}|\Gamma_j'(\widetilde{x}_i)|\lesssim\rho_n^{-1/2}.
  \]
  Since the event $\{\sqrt{n}|\widehat{x}_i^{\mathrm{(ASE)}} - \rho_n^{1/2}x_{0i}| < \eps(n\rho_n)^{1/2}\}$ has probability going to $1$, then with probability going to $1$,
  \begin{align*}
  \Psi_n''(\widetilde{x}_i)
  & = -\frac{1}{n}\sum_{j\neq i}\Gamma'_j(\widetilde{x}_i)
   - \frac{\rho_n^{1/2}}{\widetilde{x}_i^2}
   \frac{1}{n}
   \sum_{j\neq i}
   \frac{x_{0j}(1 - 2\rho_n^{1/2}\widetilde{x}_ix_{0j}) + 2(A_{ij} - \rho_n^{1/2}\widetilde{x}_{i}x_{0j})x_{0j}}{(1 - \rho_n^{1/2}\widetilde{x}_ix_{0j})^2}
   \\
  &\quad - 
  \frac{\rho_n^{1/2}}{\widetilde{x}_i^3}\frac{1}{n}\sum_{j\neq i}\frac{2(\rho_n^{1/2}x_{0i} - \widetilde{x}_i)x_{0j}(1 - 2\rho_n^{1/2}\widetilde{x}_ix_{0j})^2 (1 - \rho_n^{1/2}\widetilde{x}_ix_{0j})}{(1 - \rho_n^{1/2}\widetilde{x}_ix_{0j})^4}\\
  &\quad - \frac{1}{\widetilde{x}_i^3}\frac{1}{n}\sum_{j\neq i}\frac{2(A_{ij} - \rho_nx_{0i}x_{0j})(1 - 2\rho_n^{1/2}\widetilde{x}_ix_{0j})^2 (1 - \rho_n^{1/2}\widetilde{x}_ix_{0j})}{(1 - \rho_n^{1/2}\widetilde{x}_ix_{0j})^4}\\
  & = O(\rho_n^{-1/2}) + O(\rho_n^{-1/2}) + O(\rho_n^{-1}) + O(\rho_n^{-3/2}) = O(\rho_n^{-3/2}).
  \end{align*}
  Hence, we see that $\Psi''_n(\widetilde{x}_i) = O_{\prob_0}(\rho_n^{-3/2})$. Since we need to require $\Psi''_n(\widetilde{x}_i) = o_{\prob_0}(\sqrt{n})$ in order to establish the asymptotic normality of the one-step estimator, it is hence essential to require that $\Psi''_n(\widetilde{x}_i) = O_{\prob_0}(\rho_n^{-3/2}) = o_{\prob_0}(\sqrt{n})$, which in turn requires that $\rho_n^3n\to\infty$. This is equivalent to require that $n\rho_n = \{n^2(n\rho_n^3)\}^{1/3} = n^{2/3}(n\rho_n^3)^{1/3}\gg n^{2/3}$. 
\end{itemize}
Through the above derivation, we observe that in order the standard technique introduced in Section 5.7 of \cite{van2000asymptotic} works for establishing the asymptotic normality of the one-step estimator, it is necessary to require that $n\rho_n$ is lower bounded by a polynomial of $n$, \emph{i.e.}, the graph is moderately sparse. 

We also remark that the condition $n\rho_n$ being lower bounded by a polynomial of $n$ is not a sufficient condition for the asymptotic normality of the one-step estimator. The following theorem illustrates this result by providing an example showing that the asymptotic normality of the one-step estimator also occurs under  specific setups when $(\log n)^4/(n\rho_n)\to 0$. The proof technique employed below requires a non-standard treatment of the score function and the Fisher information matrix in contrast to the classical technique in \cite{van2000asymptotic}. Note that  generalizing   the technique below to random dot product graphs  with latent dimension $d$ greater than $1$ is non-trivial. 

\begin{theorem}\label{thm:Sparse_ER_graph}
Let $\bA\sim\mathrm{RDPG}(\bX_0)$ with a sparsity factor $\rho_n$, where the latent position matrix is $\bX_0 = [p,\ldots,p]\transpose$ for some $p\in (0, 1)$. Fix $i\in[n]$ and let $\widehat{\bX}^{\mathrm{(ASE)}}$ be the ASE of $\bA$ into $\mathbb{R}^1$. Consider an initial estimator $\widetilde\bX$ of the form
\[
e_j\transpose\widetilde\bX = \left\{
\begin{aligned}
&\rho_n^{1/2}p,&\quad\text{if }j\neq i,\\
&e_i\transpose\bX^{\mathrm{(ASE)}},&\quad\text{if }j = i.
\end{aligned}
\right.
\]
Let $\widehat\bX = [\widehat{x}_1,\ldots,\widehat{x}_n]\transpose$ be the one-step estimator defined by \eqref{eqn:one_step_estimator} with the initial estimator $\widetilde\bX = [\widetilde{x}_{1},\ldots,\widehat{x}_n]\transpose$ defined as above. Under the condition that $\rho_n\to 0$ but $(\log n)^4/(n\rho_n)\to 0$, we have
\[
\sqrt{n}(\widehat{x}_j - \rho_n^{1/2}p)\overset{\calL}{\to}\mathrm{N}(0, 1),\quad j = 1,\ldots,n.
\]
\end{theorem}

\begin{proof}[\bf Proof of Theorem \ref{thm:Sparse_ER_graph}]
We first consider $\widehat{x}_i$. Write
\begin{align*}
\widehat{x}_i & = \widetilde{x}_i + \left\{\frac{1}{n}\sum_{j = 1}^n\frac{\widetilde{x}_j^2}{\widetilde{x}_i\widetilde{x}_j(1 - \widetilde{x}_i\widetilde{x}_j)}\right\}^{-1}
\left\{
\frac{1}{n}\sum_{j = 1}^n\frac{(A_{ij} - \widetilde{x}_i\widetilde{x}_j)\widetilde{x}_j}{\widetilde{x}_i\widetilde{x}_j(1 - \widetilde{x}_i\widetilde{x}_j)}
\right\}\\
& = \widetilde{x}_i + \left\{\frac{1}{n}\sum_{j = 1}^n\frac{\widetilde{x}_j}{(1 - \widetilde{x}_i\widetilde{x}_j)}\right\}^{-1}
\left\{
\frac{1}{n}\sum_{j = 1}^n\frac{(A_{ij} - \widetilde{x}_i\widetilde{x}_j)}{(1 - \widetilde{x}_i\widetilde{x}_j)}
\right\}\\
& = \widetilde{x}_i + \left\{\frac{1}{n}\sum_{j\neq i}\frac{\rho_n^{1/2}p}{1 - \rho_n^{1/2}\widetilde{x}_ip} + \frac{\widetilde{x}_i}{n(1 - \widetilde{x}_i^2)}\right\}^{-1}\left\{\frac{1}{n}\sum_{j \neq i}\frac{(A_{ij} - \rho_n^{1/2}\widetilde{x}_ip)}{(1 - \rho_n^{1/2}\widetilde{x}_ip)} - \frac{\widetilde{x}_i^2}{n(1 - \widetilde{x}_i^2)}\right\}\\
& = \widetilde{x}_i + \left\{\frac{p}{1 - \rho_n^{1/2}\widetilde{x}_ip} - \frac{p}{n(1 - \rho_n^{1/2}\widetilde{x}_ip)} + \frac{\widetilde{x}_i}{n\rho_n^{1/2}(1 - \widetilde{x}_i^2)}\right\}^{-1}\\
&\quad\qquad\quad\times\left\{\frac{1}{n\rho_n^{1/2}}\sum_{j \neq i}\frac{(A_{ij} - \rho_n^{1/2}\widetilde{x}_ip)}{(1 - \rho_n^{1/2}\widetilde{x}_ip)} - \frac{\widetilde{x}_i^2}{n\rho_n^{1/2}(1 - \widetilde{x}_i^2)}\right\}.
\end{align*}
Observe that $\widetilde{x}_i = \rho_n^{1/2}p + O_{\prob_0}(n^{-1/2})$, $1 - \rho_n^{1/2}\widetilde{x}_ip = 1 - \rho_np^2 + o_{\prob_0}(1)$, $1 - \widetilde{x}_i^2 =1 - \rho_np^2 + o_{\prob_0}(1)$ by the continuous mapping theorem, we see that
\begin{align*}
 - \frac{p}{n(1 - \rho_n^{1/2}\widetilde{x}_ip)} + \frac{\widetilde{x}_i}{n\rho_n^{1/2}(1 - \widetilde{x}_i^2)}
& = \frac{-p}{n\{1 - \rho_np^2 + o_{\prob_0}(1)\}} + \frac{\rho_n^{1/2}p + O_{\prob_0}(n^{-1/2})}
{n\rho_n^{1/2}\{1 - \rho_np^2 + o_{\prob_0}(1)\}}\\ 
& = \frac{-p}{n\{1 - \rho_np^2 + o_{\prob_0}(1)\}} + \frac{\rho_n^{1/2}[p + O_{\prob_0}\{(n\rho_n)^{-1/2}\}]}
{n\rho_n^{1/2}\{1 - \rho_np^2 + o_{\prob_0}(1)\}}
 = O_{\prob_0}(n^{-1}),
\\
\frac{\widetilde{x}_i^2}{n\rho_n^{1/2}(1 - \widetilde{x}_i^2)}
& = \frac{\rho_n[p + O_{\prob_0}\{(n\rho_n)^{-1/2}]^2\}}{n\rho_n^{1/2}\{1 - \rho_np^2 + o_{\prob_0}(1)\}} = O_{\prob_0}\left(\frac{\rho_n^{1/2}}{n}\right).
\end{align*}
Therefore, we proceed to compute
\begin{align*}
\widehat{x}_i & =
\widetilde{x}_i + \left\{\frac{p}{1 - \rho_n^{1/2}\widetilde{x}_ip} - \frac{p}{n(1 - \rho_n^{1/2}\widetilde{x}_ip)} + \frac{\widetilde{x}_i}{n\rho_n^{1/2}(1 - \widetilde{x}_i^2)}\right\}^{-1}\\
&\quad\qquad\quad\times\left\{\frac{1}{n\rho_n^{1/2}}\sum_{j \neq i}\frac{(A_{ij} - \rho_n^{1/2}\widetilde{x}_ip)}{(1 - \rho_n^{1/2}\widetilde{x}_ip)} - \frac{\widetilde{x}_i^2}{n\rho_n^{1/2}(1 - \widetilde{x}_i^2)}\right\}\\
& = \widetilde{x}_i + \left\{\frac{p}{1 - \rho_n^{1/2}\widetilde{x}_ip} + O_{\prob_0}(n^{-1})\right\}^{-1}\left\{\frac{1}{n\rho_n^{1/2}}\sum_{j \neq i}\frac{(A_{ij} - \rho_n^{1/2}\widetilde{x}_ip)}{(1 - \rho_n^{1/2}\widetilde{x}_ip)} + O_{\prob_0}(n^{-1})\right\}\\
& = \widetilde{x}_i + \left\{\frac{1 - \rho_n^{1/2}p\widetilde{x}_i}{p} + o_{\prob_0}(1)\right\}\left\{\frac{1}{n\rho_n^{1/2}}\sum_{j \neq i}\frac{(A_{ij} - \rho_n^{1/2}\widetilde{x}_ip)}{(1 - \rho_n^{1/2}\widetilde{x}_ip)} + O_{\prob_0}(n^{-1})\right\},
\end{align*}
where the last inequality is due to the continuous mapping theorem. Since
\begin{align*}
\frac{1}{n\rho_n^{1/2}}\sum_{j\neq i}\frac{(A_{ij} - \rho_n^{1/2}\widetilde{x}_ip)}{(1 - \rho_n^{1/2}\widetilde{x}_ip)}
& = \frac{1}{n\rho_n^{1/2}}\sum_{j\neq i}\frac{(A_{ij} - \rho_np^2) + (\rho_np^2 - \rho_n^{1/2}\widetilde{x}_ip)}{(1 - \rho_n^{1/2}\widetilde{x}_ip)}\\
& = \frac{1}{(1 - \rho_n^{1/2}\widetilde{x}_ip)}\sum_{j\neq i}\frac{(A_{ij} - \rho_np^2)}{n\rho_n^{1/2}} + \frac{n - 1}{n\rho_n^{1/2}}\frac{(\rho_np^2 - \rho_n^{1/2}\widetilde{x}_ip)}{(1 - \rho_n^{1/2}\widetilde{x}_ip)}\\
& = \frac{O_{\prob_0}(n^{-1/2})}{1 - \rho_np + o_{\prob_0}(1)}
   + \frac{n - 1}{n}\frac{p(\rho_n^{1/2}p - \widetilde{x}_i)}{\{1 - \rho_np^2 + o_{\prob_0}(1)\}}
 = O_{\prob_0}(n^{-1/2}),
\end{align*}
it follows that
\begin{align*}
\widehat{x}_i & =
\widetilde{x}_i + \frac{1}{n\rho_n^{1/2}}\sum_{j\neq i}\frac{(A_{ij} - \rho_n^{1/2}\widetilde{x}_ip)}{p} + \frac{1 - \rho_n^{1/2}p\widetilde{x}_i}{p}O_{\prob_0}(n^{-1}) + o_{\prob_0}(n^{-1/2})\\
& = \widetilde{x}_i + \frac{1}{n\rho_n^{1/2}}\sum_{j\neq i}\frac{(A_{ij} - \rho_n^{1/2}\widetilde{x}_ip)}{p} + o_{\prob_0}(n^{-1/2})
 = \frac{1}{n}\widetilde{x}_i + \frac{1}{n\rho_n^{1/2}}\sum_{j\neq i}\frac{A_{ij}}{p} + o_{\prob_0}(n^{-1/2}).
\end{align*}
Hence we conclude that
\begin{align*}
\sqrt{n}(\widehat{x}_i - \rho_n^{1/2}p) & = \frac{1}{\sqrt{(n - 1)\rho_n}}\sum_{j\neq i}\frac{A_{ij} - \rho_np^2}{p} + o_{\prob_0}(1)
\overset{\calL}{\to}\mathrm{N}(0, 1 - \rho p^2) = \mathrm{N}(0, 1)
.
\end{align*}
by the central limit theorem.
This establishes the asymptotic normality for $\widehat{x}_j$ with $j = i$.

\noindent
Next we consider $\widehat{x}_k$ with $k\neq i$. Write
\begin{align*}
\widehat{x}_k
& = \widetilde{x}_k + \left\{\frac{1}{n}\sum_{j = 1}^n\frac{\widetilde{x}_j}{\widetilde{x}_k(1 - \widetilde{x}_k\widetilde{x}_j)}\right\}^{-1}
\left\{\frac{1}{n}\sum_{j = 1}^n\frac{A_{kj} - \widetilde{x}_k\widetilde{x}_j}{\widetilde{x}_k(1 - \widetilde{x}_k\widetilde{x}_j)}\right\}
\\
& = \rho_n^{1/2}p + \left\{\frac{1}{n(1 - \rho_np^2)} + \frac{\widetilde{x}_i}{n\rho_n^{1/2}p(1 - \rho_n^{1/2}p\widetilde{x}_i)} + \frac{1}{n}\sum_{j\notin\{i,k\}}
\frac{\rho_n^{1/2}p}{\rho_n^{1/2}p(1 - \rho_np^2)}
\right\}^{-1}\\
&\qquad\qquad\quad\times
\left\{\frac{1}{n}\sum_{j\notin\{i,k\}}\frac{A_{kj} - \rho_np^2}{\rho_n^{1/2}p(1 - \rho_np^2)} + \frac{1}{n}\frac{A_{ki} - \rho_n^{1/2}p\widetilde{x}_i}{\rho_n^{1/2}p(1 - \rho_n^{1/2}p\widetilde{x}_{i})} - \frac{\rho_n^{1/2}p}{n(1 - \rho_np^2)}\right\}\\
& =\rho_n^{1/2}p + 
\left\{
\frac{1}{n(1 - \rho_np^2)} + \frac{\rho_n^{1/2}[p + O_{\prob_0}\{(n\rho_n)^{-1/2}\}]}{n\rho_n^{1/2}p(1 - \rho_np^2 + O_{\prob_0}\{(n\rho_n)^{-1/2}\})} + \frac{n - 2}{n}
\frac{1}{1 - \rho_np^2}
\right\}^{-1}\\
&\qquad\qquad\quad\times
\left\{
\frac{1}{n}\sum_{j\notin\{i,k\}}\frac{A_{kj} - \rho_np^2}{\rho_n^{1/2}p(1 - \rho_np^2)} + \frac{1}{n}\frac{A_{ki} - \rho_n^{1/2}p\widetilde{x}_i}{\rho_n^{1/2}p(1 - \rho_n^{1/2}p\widetilde{x}_{i})} - \frac{\rho_n^{1/2}p}{n(1 - \rho_np^2)}
\right\}\\
& = \rho_n^{1/2}p + \left\{o_{\prob_0}(1) + {(1 - \rho_np^2)}
\right\}
\left\{
\frac{1}{n\rho_n^{1/2}}\sum_{j\notin \{i,k\}}\frac{A_{kj} - \rho_np^2}{p(1 - \rho_np^2)} + O_{\prob_0}\left(\frac{1}{\sqrt{n}\sqrt{n\rho_n}}\right)
\right\}\\
& = \rho_n^{1/2}p + o_{\prob_0}(n^{-1/2}) + \frac{1}{n\rho_n^{1/2}}\sum_{j\neq \{i,k\}}\frac{A_{kj} - \rho_np^2}{p}.
\end{align*}
Therefore,
\[
\sqrt{n}(\widehat{x}_k - \rho_n^{1/2}p) = \frac{1}{\sqrt{n\rho_n}}\sum_{j\neq k}\frac{A_{kj} - \rho_np^2}{p} + o_{\prob_0}(1)\overset{\calL}{\to}\mathrm{N}(0, 1)
\]
because $\var_0(A_{kj}/\rho_n^{1/2}) = p^2(1 - \rho_np^2)\to p^2$ as $\rho_n\to 0$. The proof is thus completed.
\end{proof}


\section{Additional Simulated Example} 
\label{sec:additional_simulated_example}

\subsection{A Sparser Stochastic Block Model} 
\label{sub:a_sparser_stochastic_block_model}
In the manuscript, we mainly consider general dense random dot product graphs in simulated examples in Section \ref{sec:numerical_examples}. 
This section presents the analysis of a synthetic example of a sparser stochastic block model. In particular, we adopt the basic simulation setup in \ref{sub:a_simulated_example} and make modification as follows. The underlying sampling model is still a random dot product graph, but the sparsity level of the graph changes with the number of vertices. The block probability matrix is given by a constant multiple of $\bB$, namely,
\[
\bB = (\bX_0^*)(\bX_0^*)\transpose,\quad\text{where}\quad
(\bX_0^*)\transpose{} = 
\alpha_n
\begin{bmatrix*}
0.3 & 0.3 & 0.6\\
0.3 & 0.6 & 0.3\\
\end{bmatrix*},
\]
and $\alpha_n$ is a scaling factor that varies with the number of vertices $n$ specified later. The cluster assignment function $\tau:[n]\to[3]$ is such that
\[
\frac{1}{n}\sum_{i = 1}^n\mathbbm{1}\{\tau(i) = 1\}\to 0.3,\quad
\frac{1}{n}\sum_{i = 1}^n\mathbbm{1}\{\tau(i) = 2\}\to 0.3,\quad
\frac{1}{n}\sum_{i = 1}^n\mathbbm{1}\{\tau(i) = 3\}\to 0.4,
\]
as $n\to\infty$. The specification of $\tau$ remains the same as in Section \ref{sub:a_simulated_example}. Let $\bpi = [0.3, 0.3, 0.4]\transpose$. The scaling factor $\alpha_n$ is defined as follows:
\[
\alpha_n = \sqrt{\frac{\calD}{n\bpi\transpose\widetilde{\bB}\bpi}},
\quad\text{where}\quad
\widetilde{\bB} = \begin{bmatrix*}
0.3 & 0.3 \\
0.3 & 0.6 \\
0.6 & 0.3
\end{bmatrix*}\begin{bmatrix*}
0.3 & 0.3 & 0.6\\
0.3 & 0.6 & 0.3\\
\end{bmatrix*}.
\]
It is designed such that the average expected degree of vertices in the resulting graph is fixed at a constant $\calD$ when the number of vertices $n$ varies. Here we set $\calD = 300$. The remaining settings are the same as those in Section \ref{sub:a_simulated_example}, and we compute the Rand indices of the GMM-based clustering algorithm applied to the four estimates (the ASE, the OSE-A, the LSE, and the OSE-L). The results are tabulated in Table \ref{table:sparseSBM_simulation_RI_GMM} and visualized in Figure \ref{fig:RI_sparseSBM_K3}. 
Note that in contrast to Table \ref{table:SBM_simulation_RI_GMM} and Figure \ref{fig:RI_SBM_K3} in the manuscript, the Rand indices here decrease as the number of vertices $n$ increases. This is because the overall sparsity of the graphs increase as $n$ increases, and the clustering accuracy fundamentally depends on the overall sparsity of the stochastic block model (see, for example, \citealp{zhang2016minimax}). We see that the one-step estimators (both for the latent positions and for the population LSE) outperform the spectral estimators (the ASE and the LSE) when $n$ increases. The differences in the Rand indices are statistically significant at level $\alpha = 0.01$ for $n\geq 900$. In particular, when $n \geq 900$, the OSE-A and the OSE-L yield better results than the 
ASE and the LSE, respectively.  These numerical results are in accordance with the fact that asymptotically, the ASE and the LSE are dominated by the OSE-A and OSE-L, respectively, even when the stochastic block model exhibits increasing sparsity level as the number of vertices increases. 
\begin{table}[htbp]
\centering%
\caption{Rand indices of the GMM-based clustering algorithm using different estimates for Section \ref{sub:a_sparser_stochastic_block_model}. The number of vertices $n$ ranges over $\{500, 600, \ldots, 1200\}$, and for each $n$, the Rand indices are averaged over $10000$ Monte Carlo replicates of adjacency matrices, with the standard errors included in parentheses. }
\vspace*{1ex}
\begin{tabular}{c c c c c}
    \hline\hline
  Estimates & ASE & OSE-A & LSE & OSE-L \\
  \hline
  $n = 500$  & {\bf 0.99942} ($1.3\times10^{-5}$) & 0.99709 ($2.8\times10^{-5}$)     & 0.99879 ($1.9\times 10^{-5}$) & 0.99708 ($2.8\times 10^{-5}$)     \\
  $n = 600$  & {\bf 0.99557} ($3.3\times10^{-5}$) & 0.99551 ($3.3\times10^{-5}$)     & 0.99295 ($4.5\times 10^{-5}$) & 0.99548 ($3.4\times 10^{-5}$)   \\
  $n = 700$  & { 0.99011} ($4.7\times10^{-5}$) & {\bf 0.99015} ($4.7\times10^{-5}$)     & 0.98585 ($6.0\times 10^{-5}$) & 0.99008 ($4.8\times 10^{-5}$)      \\
  $n = 800$  & { 0.98467} ($5.4\times10^{-5}$) & {\bf 0.98470} ($5.5\times10^{-5}$)     & 0.97885 ($7.0\times 10^{-5}$) & 0.98464 ($5.5\times 10^{-5}$)      \\
  $n = 900$  & 0.97939 ($6.1\times10^{-5}$)    & { 0.97982} ($5.9\times10^{-5}$)& 
  0.97258 ($7.6\times 10^{-5}$) & {\bf 0.97990} ($5.8\times 10^{-5}$)    \\
  $n = 1000$ & 0.97516 ($6.7\times10^{-5}$)    & { 0.97565} ($6.4\times10^{-5}$)& 
  0.96761 ($8.7\times 10^{-5}$) & {\bf 0.97576} ($6.4\times 10^{-5}$)\\
  $n = 1100$ & 0.97138 ($7.0\times10^{-5}$)    & { 0.97179} ($6.8\times10^{-5}$)& 
  0.96249 ($9.2\times 10^{-5}$) & {\bf 0.97192} ($6.7\times 10^{-5}$)      \\
  $n = 1200$ & 0.96816 ($6.7\times10^{-5}$)    & { 0.96848} ($6.6\times10^{-5}$)& 
  0.95886 ($8.7\times 10^{-5}$) & {\bf 0.96866} ($6.3\times 10^{-5}$)  \\
    \hline\hline
  \end{tabular}%
\label{table:sparseSBM_simulation_RI_GMM}
\end{table}
\begin{figure}[t]
  \centerline{\includegraphics[width=.9\textwidth]{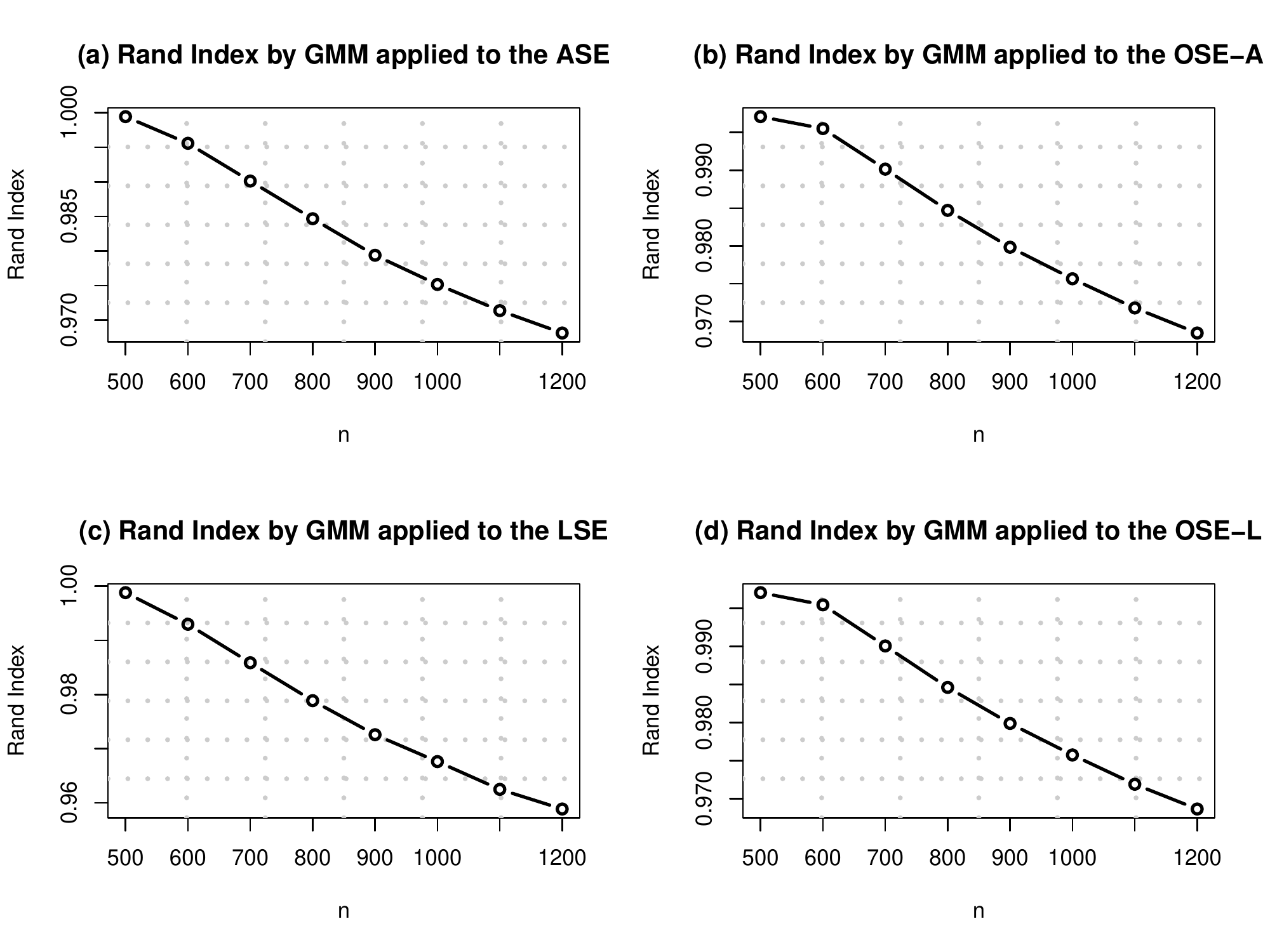}}
  \caption{The Rand indices of the GMM-model-based clustering method applied to different estimates (the ASE, the OSE-A, the LSE, and the OSE-L) when the number of vertices $n$ ranges in $\{500,600,\ldots,1200\}$. The results are averaged based on $10000$ Monte Carlo replicates. }
  \label{fig:RI_sparseSBM_K3}
\end{figure}

\subsection{Beyond Stochastic Block Models} 
\label{sub:beyond_stochastic_block_models}

We have considered numerical examples of random dot product graphs that can be classified as  stochastic block models (with positive semidefinite block probability matrices), and the primary focus there is vertex clustering. In this subsection, we  study a  numerical example of a random dot product graph that does not belong to the standard stochastic block model, and   focus on the inference for   latent positions rather than vertex clustering.

Consider a random dot product graph with $n$ vertices and latent dimension $d = 1$, where the latent position $x_{0i}$ for the $i$th vertex is set to
\[
x_{0i} = 0.8\sin\left\{\frac{\pi(i - 1)}{n - 1}\right\} + 0.1,\quad i = 1,2,\ldots,n.
\]
Let $\bX_0 = [x_{01},\ldots,x_{0n}]\transpose$ and suppose an adjacency matrix $\bA$ is generated from $\mathrm{RDPG}(\bX_0)$. The four estimators involved are:
\begin{itemize}
  \item The ASE, denoted by $\widehat{\bX}^{(\mathrm{ASE})}$;
  \item The one-step estimator initialized at the ASE (OSE-A), denoted by $\widehat{\bX}$;
  \item The LSE, denoted by $\breve{\bX}$;
  \item The one-step estimator for the population LSE, denoted by $\widehat{\bY}$.
\end{itemize}
We focus on the following two objectives:
\begin{itemize}
  \item[(i)] The comparison between the ASE and the OSE-A, and the comparison between the LSE and the OSE-L, via the sum of squares errors. This amounts to the comparison between
  \[
  SSE_{\mathrm{ASE}} = \|\widehat{\bX}^{\mathrm{(ASE)}}\bW - \bX_0\|_2^2\quad\text{and}\quad
  SSE_{\mathrm{OSE-A}} = \|\widehat{\bX}\bW - \bX_0\|_2^2,
  \]
  and the comparison between 
  \[
  SSE_{\mathrm{LSE}} = \|\breve{\bX}\bW - \bY_0\|_2^2\quad\text{and}\quad
  SSE_{\mathrm{OSE-L}} = \|\widehat{\bY}\bW - \bY_0\|_2^2.
  \]
  \item[(ii)] Performance of vertex-wise confidence intervals (CIs) for latent positions and the coordinates of the population LSE derived from Theorem \ref{thm:convergence_OS} and Theorem \ref{thm:convergence_OS_Laplacian}. In preparation for doing so, we need to derive the vertex-wise confidence intervals for the latent positions and the population LSE. Let $\widehat{\bX} = [\widehat{x}_1,\ldots,\widehat{x}_n]\transpose$ be the one-step estimator initialized at the ASE. By Theorem \ref{thm:convergence_OS}, 
  \[
  \sqrt{n}(|\widehat{x}_i| - x_{0i})\overset{\calL}{\to} \mathrm{N}(0, G(x_{0i})^{-1}),
  \]
  where
  \[
  G(x_{0i}) = \int\frac{x_1}{x_{0i}(1 - x_{0i}x_1)}F(\mathrm{d}x_1).
  \]
  To compute a $1 - \alpha$ confidence interval for $x_{0i}$, we need to estimate $G(x_{0i})$ using the one-step estimator $\widehat{\bX}$ because neither $x_{0i}$ nor the function form of $G$ is accessible from the data:
  \[
  \widehat{G}(\widehat{x}_i) = \frac{1}{n}\sum_{j = 1}^n\frac{\widehat{x}_j}{\widehat{x}_i(1 - \widehat{x}_i\widehat{x}_j)}
  \]
  By the Slutsky's theorem, a $1 - \alpha$ confidence interval for $x_{0i}$ is given by
  \begin{align}\label{eqn:CI_OSEA}
  \left(|\widehat{x}_i| - q_{z}(1 - \alpha/2)\sqrt{\frac{1}{\widehat{G}(\widehat{x}_i)n}},
  |\widehat{x}_i| + q_{z}(1 - \alpha/2)\sqrt{\frac{1}{\widehat{G}(\widehat{x}_i)n}}
  \right),
  \end{align}
  where $q_z(1 - \alpha/2)$ is the $1 - \alpha/2$ quantile of the standard normal distribution. Similarly, the asymptotic normality
  \[
  n(|\widehat{y}_i| - y_{0i})\overset{\calL}{\to}\mathrm{N}(0, \widetilde{G}(x_{0i}))
  \]
  from Theorem \ref{thm:convergence_OS_Laplacian} can be employed to construct a $1 - \alpha$ confidence interval for the coordinate $y_{0i}$ of the population LSE $\bY_0$, where $\widehat{\bY} = [\widehat{y}_1,\ldots,\widehat{y}_n]\transpose$ is the one-step estimator for the population LSE. The corresponding asymptotic variances can be estimated using the one-step estimator $\widehat{\bX}$ for $\bX$:
  \[
  \frac{1}{4\widehat{\mu}\widehat{x}_{i}\widehat{G}(\widehat{x}_i)},\quad\text{where}\quad
  \widehat{\mu} = \frac{1}{n}\sum_{j = 1}^n\widehat{x}_j.
  \]
  Therefore, a $1 - \alpha$ confidence interval for $y_{0i}$ is given by
  \begin{align}\label{eqn:CI_OSEL}
  \left(|\widehat{y}_i| - q_{z}(1 - \alpha/2)\sqrt{\frac{1}{4n^2\widehat{\mu}\widehat{x}_{i}\widehat{G}(\widehat{x}_i)}},
  |\widehat{y}_i| + q_{z}(1 - \alpha/2)\sqrt{\frac{1}{4n^2\widehat{\mu}\widehat{x}_{i}\widehat{G}(\widehat{x}_i)}}
  \right).
  \end{align}
\end{itemize}

 We draw  $1000$ independent adjacency matrices   from the aforementioned random dot product graph model. For objective (i), we aim  to compare   different estimators via the sum of squares errors (SSEs). The SSEs are computed across the $1000$ Monte Carlo replicates, and the boxplots of $SSE_{\mathrm{ASE}},SSE_{\mathrm{OSE-A}}, nSSE_{\mathrm{LSE}}, nSSE_{\mathrm{OSE-L}}$ are presented in Figure \ref{fig:boxplot_SSE}. Note that $SSE_{\mathrm{LSE}}$ and $SSE_{\mathrm{OSE-L}}$ are scaled by a factor of $n$ (the number of vertices) such that they are comparable with $SSE_{\mathrm{ASE}}$ and $\mathrm{SSE}_{\mathrm{OSE-A}}$ in the same boxplot. We can clearly see that the sum of squares errors of the OSE-A is smaller than that of the ASE, and the sum of squares errors of the OSE-L is also smaller than that of the LSE. The difference between $\mathrm{SSE}_{\mathrm{ASE}}$ and $\mathrm{SSE}_{\mathrm{OSE-A}}$ and that between $\mathrm{SSE}_{\mathrm{LSE}}$ and $\mathrm{SSE}_{\mathrm{OSE-L}}$ are both statistically significant at level $\alpha = 0.01$. These results are in accordance with the theory in Section \ref{sec:an_efficient_one_step_estimator} and \ref{sub:a_plug_in_estimator_for_the_normalized_laplacian_matrix} of the manuscript. 
\begin{figure}[htbp]
  \centerline{\includegraphics[width=1\textwidth]{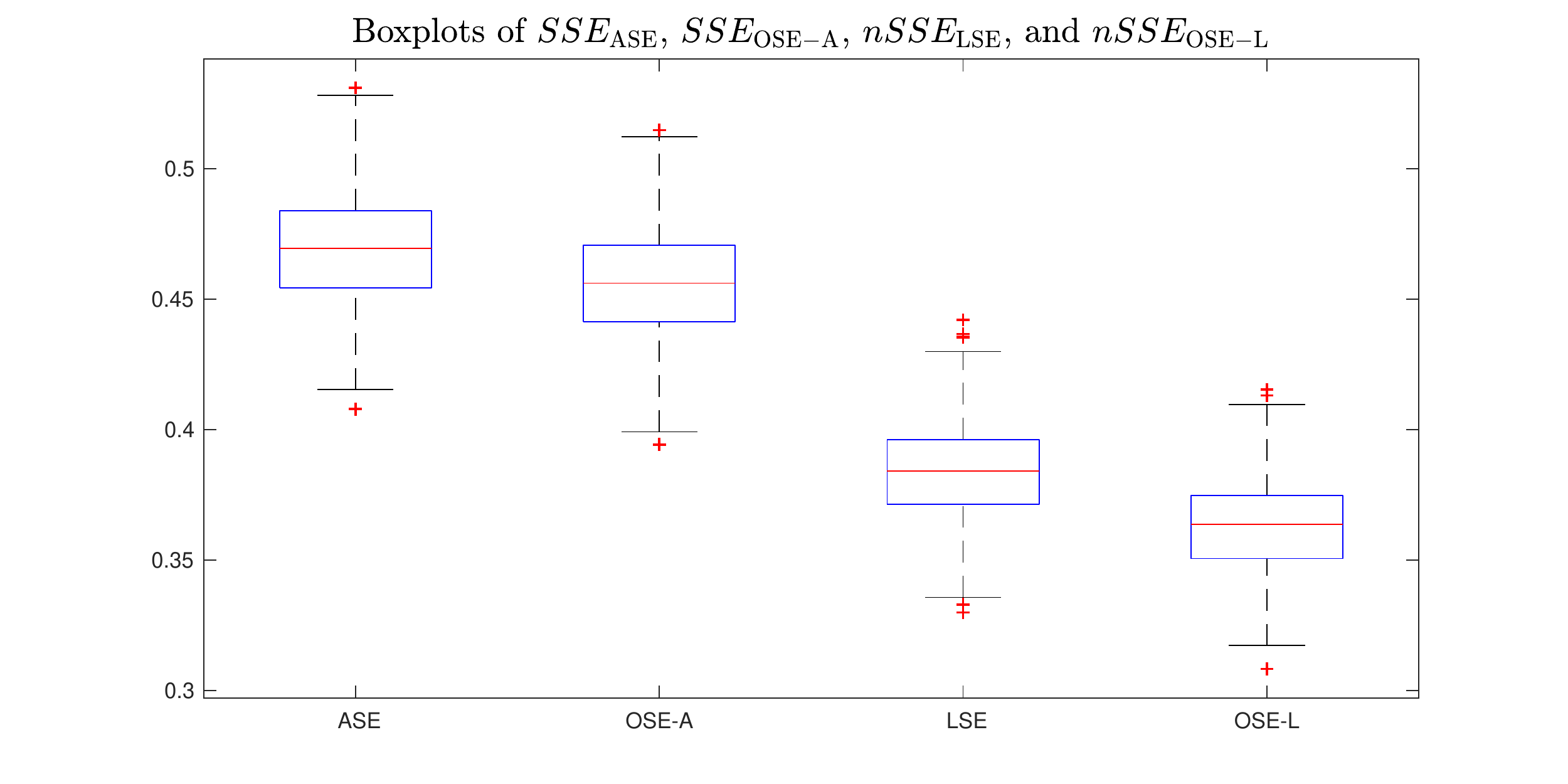}}
  \caption{The boxplots of $SSE_{\mathrm{ASE}}$, $SSE_{\mathrm{OSE-A}}$, $nSSE_{\mathrm{LSE}}$, and $nSSE_{\mathrm{OSE-L}}$ across $1000$ Monte Carlo replicates.}
  \label{fig:boxplot_SSE}
\end{figure}

For objective (ii), we
construct the vertex-wise 95\% confidence intervals for both $\bX_0$ and $\bY_0$ based on each realization of the adjacency matrix, and compute the corresponding empirical coverage probabilities for each vertex $i \in [n]$. The results are visualized in Figures \ref{fig:Coverage_probability_OSEA} and \ref{fig:Coverage_probability_OSEL}, respectively. These figures clearly demonstrate that the empirical coverage probabilities concentrate near the nominal coverage probability, verifying the usefulness of the proposed method for obtaining vertex-wise confidence intervals. 
\begin{figure}[htbp]
  \centerline{\includegraphics[width=1\textwidth]{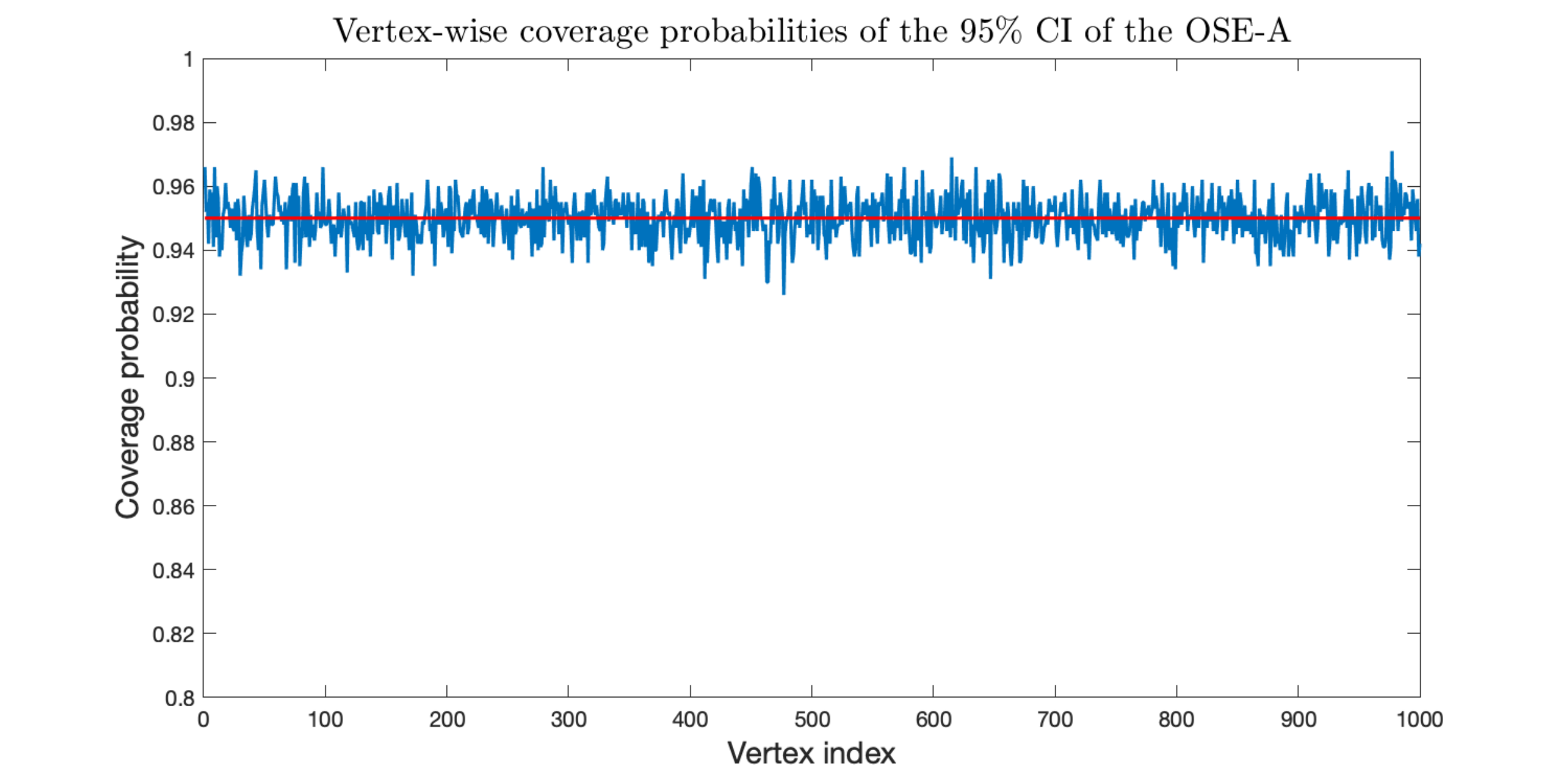}}
  \caption{Coverage probabilities of the vertex-wise confidence intervals for the latent positions $\bX_0 = [x_{01},\ldots,x_{0n}]\transpose$. The red horizontal line correspond to the nominal coverage probability $95\%$. }
  \label{fig:Coverage_probability_OSEA}
\end{figure}
\begin{figure}[htbp]
  \centerline{\includegraphics[width=1\textwidth]{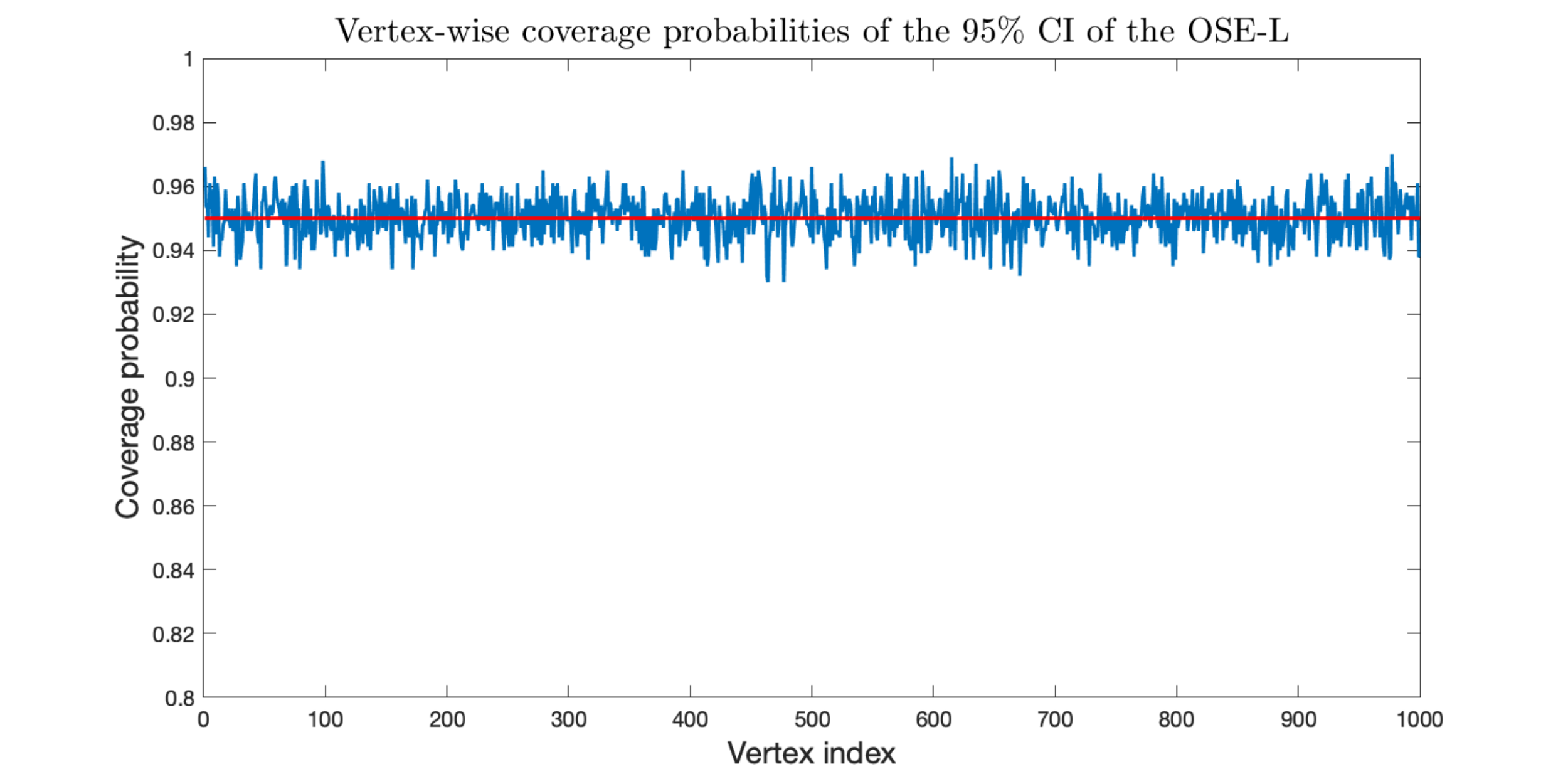}}
  \caption{Coverage probabilities of the vertex-wise confidence intervals for the population LSE $\bY_0 = [y_{01},\ldots,y_{0n}]\transpose$. The red horizontal line correspond to the nominal coverage probability $95\%$. }
  \label{fig:Coverage_probability_OSEL}
\end{figure}

We also randomly select one realization of the adjacency matrix from the aforementioned random dot product graph model   and compute the vertex-wise confidence intervals for $[x_{01},\ldots,x_{0n}]\transpose$ and $[y_{01},\ldots,y_{0n}]$ in Figures \ref{fig:CI_OSEA} and \ref{fig:CI_OSEL}, respectively. The numerical results presented further consolidate the theory developed in Sections \ref{sec:an_efficient_one_step_estimator} and \ref{sub:a_plug_in_estimator_for_the_normalized_laplacian_matrix} of the manuscript.
\begin{figure}[htbp]
  \centerline{\includegraphics[width=1\textwidth]{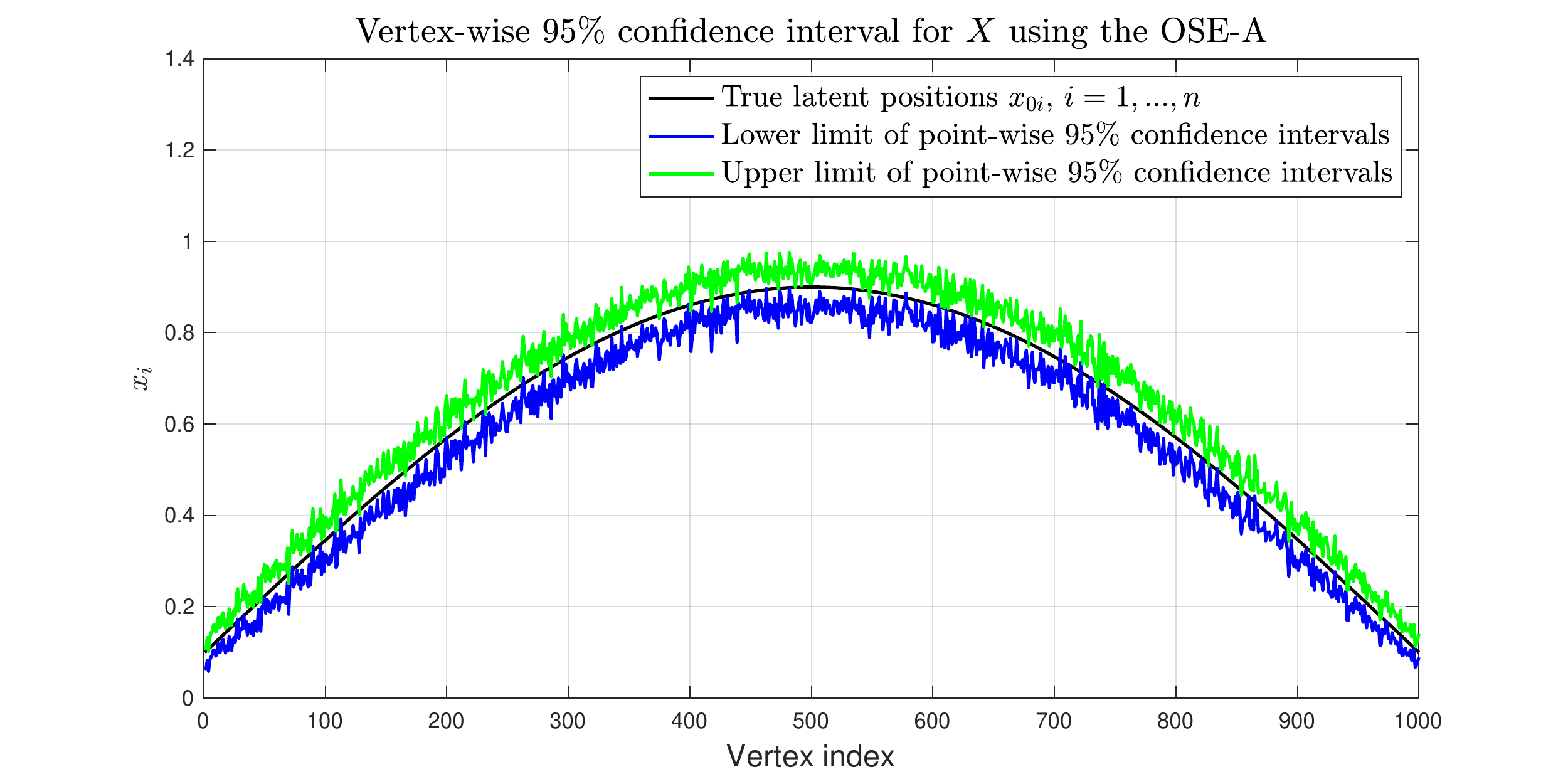}}
  \caption{A realization of the vertex-wise $95\%$ confidence intervals for $\bX_0 = [x_{01},\ldots,x_{0n}]\transpose$. }
  \label{fig:CI_OSEA}
\end{figure}
\begin{figure}[htbp]
  \centerline{\includegraphics[width=1\textwidth]{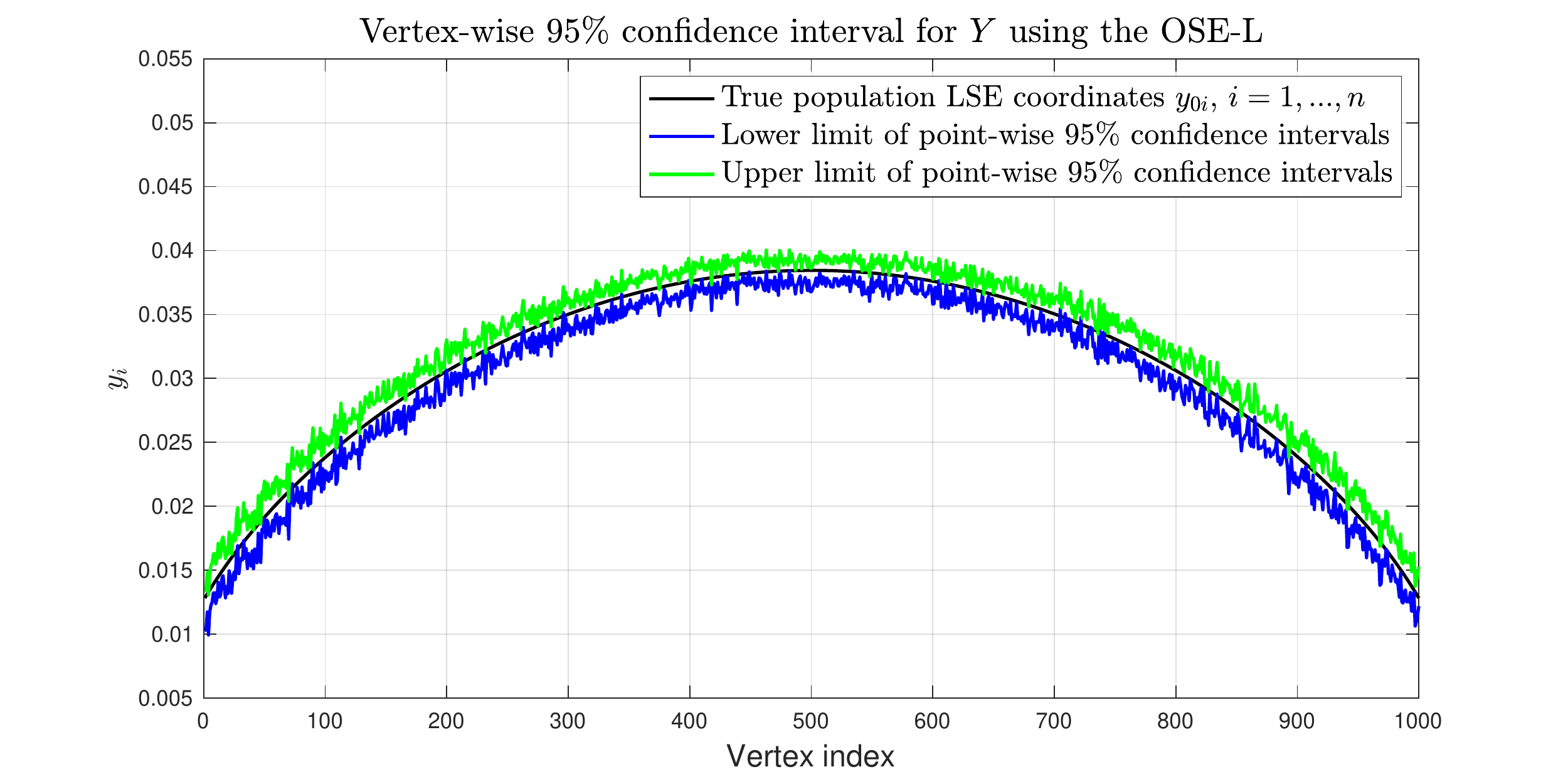}}
  \caption{A realization of the vertex-wise $95\%$ confidence intervals for $\bY_0 = [y_{01},\ldots,y_{0n}]\transpose$. }
  \label{fig:CI_OSEL}
\end{figure}

\end{appendices}

\bibliographystyle{apalike}
\bibliography{reference1,reference2}

\begin{thebibliography}{}

\bibitem[Athreya et~al., 2018a]{JMLR:v18:17-448}
Athreya, A., Fishkind, D.~E., Tang, M., Priebe, C.~E., Park, Y., Vogelstein,
  J.~T., Levin, K., Lyzinski, V., Qin, Y., and Sussman, D.~L. (2018a).
\newblock Statistical inference on random dot product graphs: a survey.
\newblock {\em Journal of Machine Learning Research}, 18(226):1--92.

\bibitem[Athreya et~al., 2016]{athreya2016limit}
Athreya, A., Priebe, C.~E., Tang, M., Lyzinski, V., Marchette, D.~J., and
  Sussman, D.~L. (2016).
\newblock A limit theorem for scaled eigenvectors of random dot product graphs.
\newblock {\em Sankhya A}, 78(1):1--18.

\bibitem[Athreya et~al., 2018b]{athreya2018estimation}
Athreya, A., Tang, M., Park, Y., and Priebe, C.~E. (2018b).
\newblock On estimation and inference in latent structure random graphs.
\newblock {\em arXiv preprint arXiv:1806.01401}.

\bibitem[Bhattacharya and Dunson, 2011]{doi:10.1093/biomet/asr013}
Bhattacharya, A. and Dunson, D.~B. (2011).
\newblock Sparse {B}ayesian infinite factor models.
\newblock {\em Biometrika}, pages 291--306.

\bibitem[Bickel and Doksum, 2015]{bickel2015mathematical}
Bickel, P.~J. and Doksum, K.~A. (2015).
\newblock {\em Mathematical statistics: basic ideas and selected topics},
  volume~2.
\newblock CRC Press.

\bibitem[Boucheron et~al., 2003]{boucheron2003}
Boucheron, S., Lugosi, G., and Massart, P. (2003).
\newblock Concentration inequalities using the entropy method.
\newblock {\em Ann. Probab.}, 31(3):1583--1614.

\bibitem[Boucheron et~al., 2013]{boucheron2013concentration}
Boucheron, S., Lugosi, G., and Massart, P. (2013).
\newblock {\em Concentration inequalities: A nonasymptotic theory of
  independence}.
\newblock Oxford university press.

\bibitem[Cape et~al., 2019]{cape2019signal}
Cape, J., Tang, M., and Priebe, C.~E. (2019).
\newblock Signal-plus-noise matrix models: eigenvector deviations and
  fluctuations.
\newblock {\em Biometrika}, 106(1):243--250.

\bibitem[Caron and Fox, 2017]{doi:10.1111/rssb.12233}
Caron, F. and Fox, E.~B. (2017).
\newblock Sparse graphs using exchangeable random measures.
\newblock {\em Journal of the Royal Statistical Society: Series B (Statistical
  Methodology)}, 79(5):1295--1366.

\bibitem[Chernoff, 1952]{chernoff1952}
Chernoff, H. (1952).
\newblock A measure of asymptotic efficiency for tests of a hypothesis based on
  the sum of observations.
\newblock {\em Ann. Math. Statist.}, 23(4):493--507.

\bibitem[Chernoff, 1956]{10.2307/2236974}
Chernoff, H. (1956).
\newblock Large-sample theory: Parametric case.
\newblock {\em The Annals of Mathematical Statistics}, 27(1):1--22.

\bibitem[Chung, 2001]{chung2001course}
Chung, K.~L. (2001).
\newblock {\em A course in probability theory}.
\newblock Academic press.

\bibitem[Eckart and Young, 1936]{Eckart1936}
Eckart, C. and Young, G. (1936).
\newblock The approximation of one matrix by another of lower rank.
\newblock {\em Psychometrika}, 1(3):211--218.

\bibitem[Erd{\"o}s et~al., 2013]{erdos2013}
Erd{\"o}s, L., Knowles, A., Yau, H.-T., and Yin, J. (2013).
\newblock Spectral statistics of erdős–rényi graphs i: Local semicircle
  law.
\newblock {\em Ann. Probab.}, 41(3B):2279--2375.

\bibitem[Fraley and Raftery, 2002]{fraley2002model}
Fraley, C. and Raftery, A.~E. (2002).
\newblock Model-based clustering, discriminant analysis, and density
  estimation.
\newblock {\em Journal of the American Statistical Association},
  97(458):611--631.

\bibitem[Fraley et~al., 2012]{fraley2012mclust}
Fraley, C., Raftery, A.~E., Murphy, T.~B., and Scrucca, L. (2012).
\newblock mclust version 4 for r: normal mixture modeling for model-based
  clustering, classification, and density estimation.
\newblock Technical report, Technical report.

\bibitem[Gao et~al., 2017]{10.5555/3122009.3153016}
Gao, C., Ma, Z., Zhang, A.~Y., and Zhou, H.~H. (2017).
\newblock Achieving optimal misclassification proportion in stochastic block
  models.
\newblock {\em J. Mach. Learn. Res.}, 18(1):1980–2024.

\bibitem[Geng et~al., 2019]{doi:10.1080/01621459.2018.1458618}
Geng, J., Bhattacharya, A., and Pati, D. (2019).
\newblock Probabilistic community detection with unknown number of communities.
\newblock {\em Journal of the American Statistical Association},
  114(526):893--905.

\bibitem[Girvan and Newman, 2002]{Girvan7821}
Girvan, M. and Newman, M. E.~J. (2002).
\newblock Community structure in social and biological networks.
\newblock {\em Proceedings of the National Academy of Sciences},
  99(12):7821--7826.

\bibitem[Hoff et~al., 2002]{doi:10.1198/016214502388618906}
Hoff, P.~D., Raftery, A.~E., and Handcock, M.~S. (2002).
\newblock Latent space approaches to social network analysis.
\newblock {\em Journal of the American Statistical Association},
  97(460):1090--1098.

\bibitem[Hoffman and Wielandt, 2003]{hoffman2003variation}
Hoffman, A.~J. and Wielandt, H.~W. (2003).
\newblock The variation of the spectrum of a normal matrix.
\newblock In {\em Selected Papers Of Alan J Hoffman: With Commentary}, pages
  118--120. World Scientific.

\bibitem[Kosorok, 2008]{kosorok2008introduction}
Kosorok, M.~R. (2008).
\newblock {\em Introduction to empirical processes and semiparametric
  inference.}
\newblock Springer.

\bibitem[{Leang} and {Johnson}, 1997]{567705}
{Leang}, C.~C. and {Johnson}, D.~H. (1997).
\newblock On the asymptotics of m-hypothesis bayesian detection.
\newblock {\em IEEE Transactions on Information Theory}, 43(1):280--282.

\bibitem[Lei and Rinaldo, 2015]{lei2015}
Lei, J. and Rinaldo, A. (2015).
\newblock Consistency of spectral clustering in stochastic block models.
\newblock {\em Ann. Statist.}, 43(1):215--237.

\bibitem[Lyzinski et~al., 2014]{lyzinski2014}
Lyzinski, V., Sussman, D.~L., Tang, M., Athreya, A., and Priebe, C.~E. (2014).
\newblock Perfect clustering for stochastic blockmodel graphs via adjacency
  spectral embedding.
\newblock {\em Electron. J. Statist.}, 8(2):2905--2922.

\bibitem[Mao et~al., 2020]{doi:10.1080/01621459.2020.1751645}
Mao, X., Sarkar, P., and Chakrabarti, D. (2020).
\newblock Estimating mixed memberships with sharp eigenvector deviations.
\newblock {\em Journal of the American Statistical Association}, 0(0):1--13.

\bibitem[Marshall and Olkin, 1990]{Marshall1990}
Marshall, A.~W. and Olkin, I. (1990).
\newblock Matrix versions of the cauchy and kantorovich inequalities.
\newblock {\em aequationes mathematicae}, 40(1):89--93.

\bibitem[Mele et~al., 2019]{mele2019spectral}
Mele, A., Hao, L., Cape, J., and Priebe, C.~E. (2019).
\newblock Spectral inference for large stochastic blockmodels with nodal
  covariates.
\newblock {\em arXiv preprint arXiv:1908.06438}.

\bibitem[Merris, 1994]{merris1994laplacian}
Merris, R. (1994).
\newblock Laplacian matrices of graphs: a survey.
\newblock {\em Linear algebra and its applications}, 197:143--176.

\bibitem[{Neil} et~al., 2013]{6623779}
{Neil}, J., {Uphoff}, B., {Hash}, C., and {Storlie}, C. (2013).
\newblock Towards improved detection of attackers in computer networks: New
  edges, fast updating, and host agents.
\newblock In {\em 2013 6th International Symposium on Resilient Control Systems
  (ISRCS)}, pages 218--224.

\bibitem[Nocedal and Wright, 2006]{nocedal2006numerical}
Nocedal, J. and Wright, S. (2006).
\newblock {\em Numerical optimization}.
\newblock Springer Science \& Business Media.

\bibitem[Oliveira, 2009]{oliveira2009concentration}
Oliveira, R.~I. (2009).
\newblock Concentration of the adjacency matrix and of the laplacian in random
  graphs with independent edges.
\newblock {\em arXiv preprint arXiv:0911.0600}.

\bibitem[Pecaric et~al., 1996]{PECARIC1996455}
Pecaric, J.~E., Puntanen, S., and Styan, G.~P. (1996).
\newblock Some further matrix extensions of the cauchy-schwarz and kantorovich
  inequalities, with some statistical applications.
\newblock {\em Linear Algebra and its Applications}, 237-238:455 -- 476.
\newblock Linear Algebra and Statistics: In Celebration of C. R. Rao's 75th
  Birthday (September 10, 1995).

\bibitem[Priebe et~al., 2017]{priebe2017semiparametric}
Priebe, C.~E., Park, Y., Tang, M., Athreya, A., Lyzinski, V., Vogelstein,
  J.~T., Qin, Y., Cocanougher, B., Eichler, K., Zlatic, M., et~al. (2017).
\newblock Semiparametric spectral modeling of the drosophila connectome.
\newblock {\em arXiv preprint arXiv:1705.03297}.

\bibitem[Rand, 1971]{doi:10.1080/01621459.1971.10482356}
Rand, W.~M. (1971).
\newblock Objective criteria for the evaluation of clustering methods.
\newblock {\em Journal of the American Statistical Association},
  66(336):846--850.

\bibitem[Rockov\'a and George, 2016]{doi:10.1080/01621459.2015.1100620}
Rockov\'a, V. and George, E.~I. (2016).
\newblock Fast bayesian factor analysis via automatic rotations to sparsity.
\newblock {\em Journal of the American Statistical Association},
  111(516):1608--1622.

\bibitem[Rohe et~al., 2011]{rohe2011}
Rohe, K., Chatterjee, S., and Yu, B. (2011).
\newblock Spectral clustering and the high-dimensional stochastic blockmodel.
\newblock {\em Ann. Statist.}, 39(4):1878--1915.

\bibitem[{Rubin-Delanchy} et~al., 2016]{7745482}
{Rubin-Delanchy}, P., {Adams}, N.~M., and {Heard}, N.~A. (2016).
\newblock Disassortativity of computer networks.
\newblock In {\em 2016 IEEE Conference on Intelligence and Security Informatics
  (ISI)}, pages 243--247.

\bibitem[Rubin-Delanchy et~al., 2017]{rubin2017statistical}
Rubin-Delanchy, P., Cape, J., Tang, M., and Priebe, C.~E. (2017).
\newblock A statistical interpretation of spectral embedding: the generalised
  random dot product graph.
\newblock {\em arXiv preprint arXiv:1709.05506}.

\bibitem[Sarkar and Bickel, 2015]{sarkar2015}
Sarkar, P. and Bickel, P.~J. (2015).
\newblock Role of normalization in spectral clustering for stochastic
  blockmodels.
\newblock {\em Ann. Statist.}, 43(3):962--990.

\bibitem[Sussman et~al., 2012]{sussman2012consistent}
Sussman, D.~L., Tang, M., Fishkind, D.~E., and Priebe, C.~E. (2012).
\newblock A consistent adjacency spectral embedding for stochastic blockmodel
  graphs.
\newblock {\em Journal of the American Statistical Association},
  107(499):1119--1128.

\bibitem[Sussman et~al., 2014]{6565321}
Sussman, D.~L., Tang, M., and Priebe, C.~E. (2014).
\newblock Consistent latent position estimation and vertex classification for
  random dot product graphs.
\newblock {\em IEEE Transactions on Pattern Analysis and Machine Intelligence},
  36(1):48--57.

\bibitem[Tang et~al., 2017a]{doi:10.1080/10618600.2016.1193505}
Tang, M., Athreya, A., Sussman, D.~L., Lyzinski, V., Park, Y., and Priebe,
  C.~E. (2017a).
\newblock A semiparametric two-sample hypothesis testing problem for random
  graphs.
\newblock {\em Journal of Computational and Graphical Statistics},
  26(2):344--354.

\bibitem[Tang et~al., 2017b]{tang2017}
Tang, M., Athreya, A., Sussman, D.~L., Lyzinski, V., and Priebe, C.~E. (2017b).
\newblock A nonparametric two-sample hypothesis testing problem for random
  graphs.
\newblock {\em Bernoulli}, 23(3):1599--1630.

\bibitem[Tang and Priebe, 2018]{tang2018}
Tang, M. and Priebe, C.~E. (2018).
\newblock Limit theorems for eigenvectors of the normalized {L}aplacian for
  random graphs.
\newblock {\em Ann. Statist.}, 46(5):2360--2415.

\bibitem[Tang et~al., 2013]{tang2013}
Tang, M., Sussman, D.~L., and Priebe, C.~E. (2013).
\newblock Universally consistent vertex classification for latent positions
  graphs.
\newblock {\em Ann. Statist.}, 41(3):1406--1430.

\bibitem[{Tang} et~al., 2019]{8570772}
{Tang}, R., {Ketcha}, M., {Badea}, A., {Calabrese}, E.~D., {Margulies}, D.~S.,
  {Vogelstein}, J.~T., {Priebe}, C.~E., and {Sussman}, D.~L. (2019).
\newblock Connectome smoothing via low-rank approximations.
\newblock {\em IEEE Transactions on Medical Imaging}, 38(6):1446--1456.

\bibitem[Van~der Vaart, 2000]{van2000asymptotic}
Van~der Vaart, A.~W. (2000).
\newblock {\em Asymptotic statistics}, volume~3.
\newblock Cambridge university press.

\bibitem[Ward et~al., 2011]{doi:10.1146/annurev.polisci.12.040907.115949}
Ward, M.~D., Stovel, K., and Sacks, A. (2011).
\newblock Network analysis and political science.
\newblock {\em Annual Review of Political Science}, 14(1):245--264.

\bibitem[Wasserman and Faust, 1994]{wasserman1994social}
Wasserman, S. and Faust, K. (1994).
\newblock {\em Social network analysis: Methods and applications}, volume~8.
\newblock Cambridge university press.

\bibitem[Xie and Xu, 2019]{xie2019optimal}
Xie, F. and Xu, Y. (2019).
\newblock Optimal bayesian estimation for random dot product graphs.
\newblock {\em arXiv preprint arXiv:1904.12070}.

\bibitem[Young and Scheinerman, 2007]{young2007random}
Young, S.~J. and Scheinerman, E.~R. (2007).
\newblock Random dot product graph models for social networks.
\newblock In {\em International Workshop on Algorithms and Models for the
  Web-Graph}, pages 138--149. Springer.

\bibitem[Yu et~al., 2015]{doi:10.1093/biomet/asv008}
Yu, Y., Wang, T., and Samworth, R.~J. (2015).
\newblock A useful variant of the davis–kahan theorem for statisticians.
\newblock {\em Biometrika}, 102(2):315--323.

\bibitem[Zhang et~al., 2016]{zhang2016minimax}
Zhang, A.~Y., Zhou, H.~H., et~al. (2016).
\newblock Minimax rates of community detection in stochastic block models.
\newblock {\em The Annals of Statistics}, 44(5):2252--2280.

\bibitem[Zhu and Ghodsi, 2006]{ZHU2006918}
Zhu, M. and Ghodsi, A. (2006).
\newblock Automatic dimensionality selection from the scree plot via the use of
  profile likelihood.
\newblock {\em Computational Statistics \& Data Analysis}, 51(2):918 -- 930.

\end{thebibliography}
\end{document}